%% file: QuantumWebsoftypeQ.tex
\documentclass[reqno]{amsart}
\usepackage[utf8]{inputenc}
\usepackage{amssymb}
\usepackage{verbatim}
\usepackage{dsfont}
\usepackage[all]{xy}
\usepackage[usenames]{color}
\usepackage{amsfonts}
\usepackage{latexsym}
\usepackage{graphicx}
\usepackage[usenames,dvipsnames]{xcolor}
\usepackage{tikz}
\usepackage{tikz-cd}
\usepackage{todonotes}
\usepackage{hyperref}
\allowdisplaybreaks[1]

\usepackage[capitalize]{cleveref}
\crefname{theorem}{Theorem}{Theorems}
\crefname{fact}{Fact}{Facts}
\crefname{note}{Note}{Notes}
\crefname{lemma}{Lemma}{Lemmas}
\crefname{alg}{Algorithm}{Algorithms}
\crefname{remark}{Remark}{Remarks}
\crefname{example}{Example}{Examples}
\crefname{prop}{Proposition}{Propositions}
\crefname{conj}{Conjecture}{Conjectures}
\crefname{cor}{Corollary}{Corollaries}
\crefname{defn}{Definition}{Definitions}
\crefname{equation}{\!\!}{\!\!} 

\usepackage[left=1in, top=1in, right=1in, bottom=1in]{geometry}

\usetikzlibrary{decorations.markings}
\usetikzlibrary{decorations.pathreplacing}
\usetikzlibrary{arrows,shapes,positioning}
\tikzstyle directed=[postaction={decorate,decoration={markings, mark=at position #1 with {\arrow[scale=1]{>}}}}]
\tikzstyle rdirected=[postaction={decorate,decoration={markings, mark=at position #1 with {\arrow[scale=1]{<}}}}]
\usepackage[all]{xy}
\usepackage{ytableau}
\newcommand{\webs}{\mathfrak{q}(n)\text{-}\mathbf{Web}_{\up}}
\newcommand{\qwebs}{\mathfrak{q}\text{-}\mathbf{Web}_{\up}}

\newcommand{\websupdown}{\mathfrak{q}(n)\text{-}\mathbf{Web}_{\up\down}}
\newcommand{\qwebsupdown}{\mathfrak{q}\text{-}\mathbf{Web}_{\up\down}}

\newcommand{\websk}{\mathfrak{q}(n)\text{-}\mathbf{Web}_{\up, \k}}
\newcommand{\qwebsk}{\mathfrak{q}\text{-}\mathbf{Web}_{\up, \k}}

\newcommand{\websupdownk}{\mathfrak{q}(n)\text{-}\mathbf{Web}_{\up\down, \k}}
\newcommand{\qwebsupdownk}{\mathfrak{q}\text{-}\mathbf{Web}_{\up\down, \k}}

\newcommand{\qwebsK}{\mathfrak{q}\text{-}\mathbf{Web}_{\up, \K}}

\newcommand{\qwebsupdownK}{\mathfrak{q}\text{-}\mathbf{Web}_{\up\down, \K}}

\newcommand{\beq}{\begin{equation}}
\newcommand{\eeq}{\end{equation}}
\newcommand{\bea}{\begin{eqnarray*}}
\newcommand{\eea}{\end{eqnarray*}}
\newcommand{\bi}{\begin{itemize}}
\newcommand{\ei}{\end{itemize}}
\newcommand{\be}{\begin{enumerate}}
\newcommand{\ee}{\end{enumerate}}
\newcommand{\bc}{\begin{center}}
\newcommand{\ec}{\end{center}}
\newcommand{\bt}{\begin{tikzpicture}}
\newcommand{\et}{\end{tikzpicture}}
\newcommand{\nfrac}[2]{\genfrac{}{}{0pt}{}{#1}{#2}}
\newcommand{\bma}{\begin{bmatrix}}
\newcommand{\ema}{\end{bmatrix}}

\let\oldstar\star

\renewcommand{\star}{\circledast}
\newcommand{\fqt}{\mathfrak{q}}

\newcommand{\mods}{U_{q}\left( \mathfrak{q}_{n}\right)\text{-}\mathbf{Mod}_{\mathcal{S}}}

\newcommand{\modsupdown}{U_{q}(\mathfrak{q}_{n})\text{-}\mathbf{Mod}_{\mathcal{S},\mathcal{S}^*}}

\newcommand{\modsk}{U_{q}\left( \mathfrak{q}_{n}\right)\text{-}\mathbf{Mod}_{\mathcal{S}, \k}}
\newcommand{\modsupdownk}{U_{q}(\mathfrak{q}_{n})\text{-}\mathbf{Mod}_{\mathcal{S},\mathcal{S}^*, \k}}

\newcommand{\modsupdownK}{U_{q}(\mathfrak{q}_{n})\text{-}\mathbf{Mod}_{\mathcal{S},\mathcal{S}^*, \K}}

\newcommand{\lcap}{
\begin{tikzpicture}[baseline = 3pt, scale=0.5, color=\clr]
        \draw[-,thick] (1,0) to[out=up, in=right] (0.53,0.5) to[out=left, in=right] (0.47,0.5);
        \draw[->,thick] (0.49,0.5) to[out=left,in=up] (0,0);
\end{tikzpicture}
}

\newcommand{\rcap}{
\begin{tikzpicture}[baseline = 3pt, scale=0.5, color=\clr]
        \draw[<-,thick] (1,0) to[out=up, in=right] (0.53,0.5) to[out=left, in=right] (0.47,0.5);
        \draw[-,thick] (0.49,0.5) to[out=left,in=up] (0,0);
\end{tikzpicture}
}

\newcommand{\clr}{rgb:black,1;blue,4;red,1}
\newcommand{\wdot}{ node[circle, draw, color=\clr, fill=white, thick, inner sep=0pt, minimum width=4pt]{}}

\makeatletter

\newtheorem{theorem}{Theorem}[subsection]

\newtheorem{lemma}[theorem]{Lemma}
\newtheorem{proposition}[theorem]{Proposition}

\newtheorem{definition}[theorem]{Definition}

\newtheorem{remark}[theorem]{Remark}

\@addtoreset{equation}{section}

\newcommand{\End}{\operatorname{End}}
 
\newcommand{\Hom}{\operatorname{Hom}}

\newcommand{\Z}{\mathbb{Z}}

\newcommand{\C}{\mathbb{C}}
\newcommand{\K}{\mathbb{K}}

\newcommand{\cat}{\mathcal}

\newcommand{\arup}[1]{\stackrel{#1}{\longrightarrow}}
\newcommand{\larup}[1]{\stackrel{#1}{\longleftarrow}}

\newcommand{\up}{\uparrow}
\newcommand{\down}{\downarrow}
\newcommand{\ob}[1]{\mathsf{#1}}

\newcommand{\unit}{\mathds{1}}

\newcommand{\q}{\mathfrak{q}}
\newcommand{\0}{\bar{0}}
\renewcommand{\1}{\bar{1}}
\newcommand{\II}{\mathcal{I}}

\newcommand{\ibar}{{-\imath}}
\newcommand{\jbar}{{-\jmath}}

\newcommand{\vareps}{\varepsilon}

\renewcommand{\k}{\Bbbk}
\renewcommand{\K}{\mathbb{K}}

\newcommand{\ev}{\operatorname{ev}}
\newcommand{\coev}{\operatorname{coev}}

\newcommand{\fq}{\mathfrak{q}}
\newcommand{\fh}{\mathfrak{h}}

\newcommand{\fb}{\mathfrak{b}}
\newcommand{\Id}{\operatorname{Id}}

\newcommand{\svec}{\mathfrak{svec}}

\newcommand{\p}[1]{{|#1|}}

\newcommand{\U}{U}

\newcommand{\SP}{\mathcal{SP}}
\newcommand{\Udot}{\dot{\U}_{q}}
\newcommand{\bUdot}{\dot{\mathbf{U}}_{q}}
\newcommand{\etilde}{E}
\newcommand{\ftilde}{F}
\newcommand{\htilde}{h}
\newcommand{\uddoublearrow}{\vert}

\renewcommand{\AA}{\mathcal{A}}
\newcommand{\tq}{\tilde{q}}
\newcommand{\TT}{\mathcal{T}}
\newcommand{\B}{\mathcal{B}}
\newcommand{\qtbinom}[2]{\genfrac{[}{]}{0pt}{}{#1}{#2}_{t}}
\newcommand{\qbinom}[2]{\genfrac{[}{]}{0pt}{}{#1}{#2}_{q}}

\renewcommand{\ss}{\scriptsize}
\newcommand{\Udoti}{\dot{\U}_{q}}
\newcommand{\bUdoti}{\dot{\mathbf{U}}_{q}}

\newcommand{\qwebseven}{\mathfrak{q}\text{-}\mathbf{Web}_{\up, \text{ev}}}

\newcommand{\qwebsupdowneven}{\mathfrak{q}\text{-}\mathbf{Web}_{\up\down, \text{ev}}}
\newcommand{\HC}{\ensuremath{\operatorname{HC}}}
\newcommand{\BC}{\ensuremath{\operatorname{BC}}}
\newcommand{\mt}{\mathtt{t}}
\newcommand{\md}{\mathtt{d}}
\newcommand{\Cl}{\operatorname{Cl}}

\tikzstyle{none}=[]

\begin{document}

\title{Quantum Webs of Type Q}

\author{Gordon C. Brown }
\address{Fort Worth, TX}
\email{gcbrown@utexas.edu}
\author{Nicholas J. Davidson}
\address{Department of Mathematics \\
          Reed College \\
         Portland, OR 97202}
\email{njd@reed.edu}
\author{Jonathan R. Kujawa}
\address{Department of Mathematics \\
          University of Oklahoma \\
          Norman, OK 73019}
\thanks{Research of the first author was partially supported by NSA grant H98230-11-1-0127.   Part of the research of the second author was completed while in residence at the Mathematical Science Research Institute, funded by NSF Grant DMS-1440140. Research of the third author was partially supported by NSA grant H98230-11-1-0127 and a Simons Collaboration Grant for Mathematicians.}\
\email{kujawa@math.ou.edu}

\date{\today}

\begin{abstract}  Webs are combinatorial diagrams used to encode homomorphisms between representations of Lie (super)algebras and related objects.   This paper extends the theory of webs to the quantum group of type Q.  We define a monoidal supercategory of quantum type $Q$ webs and show it admits a full, essentially surjective functor onto the monoidal supercategory of $U_q(\q_n)$-modules generated by the quantum symmetric powers of the natural representation and their duals.  We also show that a certain subcategory of the web category is a ribbon category and discuss applications to the representation theory of $U_q(\q_n)$ and to invariants of oriented, framed links.  \end{abstract}

\maketitle  


\section{Introduction} 

\subsection{Background}\label{SS:background}  Let $\fq_{n}$ denote the complex Lie superalgebra of type $Q$.  In 1992 Olshanski introduced a corresponding quantized enveloping superalgebra, $U_{q}(\fq_{n})$,  over the field of rational functions $\C (q)$ \cite{Olshanski}.  This is a Hopf superalgebra and so one can consider $V_{n}^{\otimes d}$, the $d$-fold tensor product of the natural supermodule for $U_{q}(\fq_{n})$.   At the same time Olshanski introduced the Hecke-Clifford superalgebra $\HC_{d}(q)$, defined an action on $V_{n}^{\otimes d}$, and proved  $\HC_{d}(q)$ is in Schur-Weyl duality with $U_{q}(\fq_{n})$.  In this way Olshanski provided a quantized version of Sergeev's results for the complex Lie superalgebra $\fq_{n}$ \cite{Sergeev}.  Generalizing Olshanski's work, Benkart-Guay-Jung-Kang-Wilcox introduced the quantum walled Brauer-Clifford superalgebra $BC_{r,s}(q)$, showed it acts on mixed tensor space $V_{n}^{\otimes r} \otimes (V_{n}^{*})^{\otimes s}$ and again established a Schur-Weyl duality with $U_{q}(\fq_{n})$ \cite{BGJKW}.  More recently Chang-Wang obtained a Howe duality for $U_{q}(\fq_{n})$ \cite{ChangWang}.

In a somewhat different direction, the authors of \cite{GJKK,GJKKashK} investigated the polynomial representations of $U_{q}(\fq_{n})$.  In particular, they proved these representations are semisimple, they admit a classical limit which yields the polynomial representations of $\fq_{n}$, and their characters coincide with their classical limits.  One of the main results of those papers is a theory of crystals (in the sense of Kashiwara) for polynomial representations which includes a combinatorial rule for tensor products.  The polynomial representations of $U_{q}(\fq_{n})$ were also studied in \cite{GJKK2,DW}.

Despite these important results, the representation theory of $U_{q}(\fq_{n})$ is still mysterious.  Even fundamental questions like the existence of canonical isomorphisms $M \otimes N \to N \otimes M$ for finite-dimensional $U_{q}(\fq_{n})$-supermodules remain open.  For polynomial representations the fact $M \otimes N$ and $N \otimes M$ are isomorphic can be extracted from character considerations and the results of \cite{GJKK,GJKKashK}.  However obtaining a canonical family of isomorphisms or, more generally, a universal $R$-matrix, remains an open problem. Similarly, there is little known about the characters of the non-polynomial simple supermodules, what happens when $q$ is specialized to a root of unity, and so on.

\subsection{Main Results}\label{SS:mainresults}  The category of finite-dimensional $U_{q}(\fq_{n})$-supermodules is closed under tensor products and hence is a monoidal supercategory.  Here the adjective ``super'' reflects the fact the $\Hom$-spaces have a $\Z_{2}=\Z/2\Z$-grading and the category obeys a graded version of the interchange law. For details see \cref{E:superinterchange} and the adjoining discussion.  The main goal of the present paper is to provide a complete combinatorial model for several natural monoidal supercategories of representations for $U_{q}\left(\fq_{n} \right)$.  

In \cref{D:UpwardWebs} we introduce a diagrammatic supercategory $\qwebs$.  It is defined as a monoidal supercategory via generators and relations.  The objects are words from the set 
\[
\left\{ \uparrow_{k} \mid k\in \Z_{\geq 0} \right\}.  
\]  The generating morphisms are certain diagrams which we call dots, merges, and splits.  The $\Z_{2}$-grading on morphisms is given by declaring dots to have parity $\1$ and the merges and splits to have parity $\0$.  As customary for diagrammatic categories, composition is given by vertical concatenation and tensor product is given by horizontal concatenation.  A diagram obtained by a finite sequence of these operations is called a \emph{web} and a general morphism is a linear combination of webs.  Note, our convention is to read diagrams from bottom to top.   For example, the following web is a morphism from $\uparrow_{4}\uparrow_{9}\uparrow_{6}\uparrow_{7}$ to $\uparrow_{6}\uparrow_{5}\uparrow_{1}\uparrow_{4}\uparrow_{8}\uparrow_{2}$:
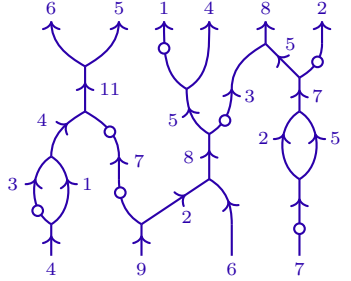
\[
\xy
(0,0)*{
\bt[color=\clr, scale=1.2]
	\node at (2,0) {\scriptsize $4$};
	\node at (3,0) {\scriptsize $9$};
	\node at (4,0) {\scriptsize $6$};
	\node at (4.75,0) {\scriptsize $7$};
	\node at (2,2.9) {\scriptsize $6$};
	\node at (2.75,2.9) {\scriptsize $5$};
	\node at (3.25,2.9) {\scriptsize $1$};
	\node at (3.75,2.9) {\scriptsize $4$};
	\node at (4.375,2.9) {\scriptsize $8$};
	\node at (5,2.9) {\scriptsize $2$};
	\draw [thick, directed=0.65] (2,0.15) to (2,0.5);
	\draw [thick, directed=0.65] (3,0.15) to (3,0.5);
	\draw [thick, ] (4,0.15) to (4,0.5);
	\draw [thick, directed=0.75] (4.75,0.15) to (4.75,1);
	\draw [thick, directed=0.65] (2,0.5) [out=30, in=330] to (2,1.25);
	\draw [thick, directed=0.65] (2,0.5) [out=150, in=210] to (2,1.25);
	\draw [thick, directed=0.65] (2,1.25) [out=90, in=210] to (2.375,1.75);
	\draw [thick, ] (3,0.5) [out=150, in=270] to (2.75,1);
	\draw [thick, directed=0.65] (3,0.5) to (3.75,1);
	\draw [thick, ] (2.75,1) to (2.75,1.25);
	\draw [thick, directed=1] (3.5,2) [out=30, in=270] to (3.75,2.75);
	\draw [thick, directed=1] (3.5,2) [out=150, in=270] to (3.25,2.75);
	\draw [thick, directed=0.01] (2.75,1.25) [out=90, in=330] to (2.375,1.75);
	\draw [thick, directed=0.65] (2.375,1.75) to (2.375,2.25);
	\draw [thick, directed=1] (2.375,2.25) [out=30, in=270] to (2.75,2.75);
	\draw [thick, directed=1] (2.375,2.25) [out=150, in=270] to (2,2.75);
	\draw [thick, directed=0.45] (4,0.5) [out=90, in=330] to (3.75,1);
	\draw [thick, directed=0.65] (3.75,1) to (3.75,1.5);
	\draw [thick, directed=0.65] (3.75,1.5) [out=150, in=270] to (3.5,2);
	\draw [thick, ] (3.75,1.5) [out=30, in=270] to (4,2);
	\draw [thick, directed=0.01] (4,2) [out=90, in=210] to (4.375,2.5);
	\draw [thick, directed=0.65] (4.75,1) [out=30, in=330] to (4.75,1.75);
	\draw [thick, directed=0.65] (4.75,1) [out=150, in=210] to (4.75,1.75);
	\draw [thick, directed=0.65] (4.75,1.75) to (4.75,2.125);
	\draw [thick, directed=0.65] (4.75,2.125) to (4.375,2.5);
	\draw [thick, ] (4.75,2.125) [out=30, in=270] to (5,2.5);
	\draw [thick, directed=1] (5,2.5) to (5,2.75);
	\draw [thick, directed=1] (4.375,2.5) to (4.375,2.75);
	\node at (1.575,0.95) {\scriptsize $3$}; 
	\node at (2.4,0.95) {\scriptsize $1$}; 
	\node at (2.975,1.2) {\scriptsize $7$}; 
	\node at (2.65,2) {\scriptsize $11$}; 
	\node at (1.9,1.65) {\scriptsize $4$}; 
	\node at (3.5,0.575) {\scriptsize $2$}; 
	\node at (3.525,1.25) {\scriptsize $8$}; 
	\node at (3.35,1.65) {\scriptsize $5$}; 
	\node at (4.2,1.925) {\scriptsize $3$}; 
	\node at (4.95,1.925) {\scriptsize $7$}; 
	\node at (5.15,1.45) {\scriptsize $5$}; 
	\node at (4.35,1.45) {\scriptsize $2$}; 
	\node at (4.65,2.5) {\scriptsize $5$}; 
	\draw (1.85,0.65) \wdot;
	\draw (2.765,0.85) \wdot;
	\draw (2.65,1.525) \wdot;
	\draw (4.75,0.45) \wdot;
	\draw (4.95,2.3) \wdot;
	\draw (3.925,1.65) \wdot;
	\draw (3.25,2.45) \wdot;
\et
};
\endxy .
\]
Morphisms are subject to an explicit list of diagrammatic relations.  See \cref{D:UpwardWebs} for details. 

Let $S_{q}^{d}\left(V_{n} \right)$ denote the $d$th quantum symmetric power of the natural $U_{q}(\fq_{n})$-supermodule $V_{n}$.  Let $\mods$ denote the full subcategory of finite-dimensional $U_{q}\left(\fq_{n} \right)$-supermodules consisting of arbitrary tensor products of $S_{q}^{d}\left(V_{n} \right)$ for various $d \geq 0$, where by definition $S_{q}^{0}\left(V_{n} \right)$ is understood to be the trivial supermodule and $S^{1}_{q}\left(V_{n} \right) \cong V_{n}$.  Also, we have $\Lambda_{q}^{d}\left(V_{n} \right) \cong S^{d}_{q}\left(V_{n} \right)$ for all $d \geq 0$ so the quantum exterior powers are covered by our setup.  In what follows it is convenient to write $S^{\uparrow_{d}}_{q}\left(V_{n} \right)$ for $S^{d}_{q}\left(V_{n} \right)$  and $S^{\downarrow_{d}}_{q}\left(V_{n} \right)$ for the dual supermodule $S^{d}_{q}\left(V_{n} \right)^{*}$.

Our first main result is contained in \cref{psi-functor,T:fullness} and shows there exists a  full functor of monoidal supercategories, 
\[
\Psi_{n}: \qwebs\to \mods,
\] for every $n \geq 1$.
The functor is given on the object $\up_{\ob{a}} := \up_{a_{1}}\dotsb \up_{a_{t}}$ by 
\[
\Psi_{n}\left(\up_{\ob{a}} \right)= S_{q}^{\uparrow_{a_{1}}}(V_{n}) \otimes\dotsb \otimes S_{q}^{\uparrow_{a_{t}}}(V_{n}).
\]  That is, there is a \emph{single} diagrammatic category $\qwebs$ which describes tensor products of symmetric powers of the natural supermodule for $U_{q}(\fq_{n})$ for \emph{all $n \geq 1$}.

In \cref{D:orientedwebs} we introduce the diagrammatic supercategory of oriented webs, $\qwebsupdown$.  It is again defined as a monoidal supercategory via generators and relations.  Objects in $\qwebsupdown$ are words  from the set 
\[
\left\{ \uparrow_{k}, \downarrow_{k} \mid k\in \Z_{\geq 0} \right\}.  
\]  The generating morphisms are dots, merges, and splits, along with leftward oriented cups and caps.  Dots are again declared to have parity $\1$ and the other generating morphisms are declared to have parity $\0$.  Dots, merges, and splits satisfy the defining relations of $\qwebs$.  We require cups and caps to satisfy a straightening rule which makes $\uparrow_{k}$ and $\downarrow_{k}$ dual to each other (see \cref{straighten-zigzag}) and require clockwise oriented circles labelled by $1$ to be equal to zero (see \cref{delete-bubble}).  The final key relation is the requirement certain leftward oriented crossing morphisms be invertible (see \cref{E:leftwardcrossing}).  That is, we assume the existence of additional generating morphisms which provide two-sided inverses to leftward crossings under composition.  The existence of these inverses force the additional relations we require in $\qwebsupdown$. 

Let $\modsupdown$ be the full subcategory of finite-dimensional $U_{q}(\fq_{n})$-supermodules consisting of arbitrary tensor products of $S_{q}^{d}\left(V_{n} \right)$ for $d\geq 0$ along with their duals. Extending the previous result, we prove in \cref{T:Psiupdown} the single diagrammatic category $\qwebsupdown$ admits a full functor of monoidal supercategories
\[
\Psi_{n}: \qwebsupdown \to \modsupdown
\] for every $n \geq 1$.  On generating objects the functor is given by
\[
\Psi_{n}\left(\uparrow_{k} \right) = S_{q}^{\uparrow_{k}}\left(V_{n} \right) \hspace{.25in} \text{and}  \hspace{.25in} \Psi_{n}\left(\downarrow_{k} \right) = S_{q}^{\downarrow_{k}}\left(V_{n} \right).
\] Because $\Psi_n$ is monoidal, an arbitrary object in $\qwebsupdown$ is then sent to the corresponding tensor product of $S_{q}^{k}\left(V_{n} \right)$'s and their duals.  For example, $\Psi_{n}\left(\downarrow_{6}\downarrow_{2}\uparrow_{9}\right) = S^{\downarrow_{6}\downarrow_{2}\uparrow_{9}}_{q}\left(V_{n} \right) := S^{\downarrow_{6}}_{q}\left(V_{n} \right) \otimes S^{\downarrow_{2}}_{q}\left(V_{n} \right) \otimes S^{\uparrow_{9}}_{q}\left(V_{n} \right)$.

Our third main result gives a precise description of the kernel of $\Psi_{n}$.   Namely, for a fixed $n \geq 1$ set $\ell  = (n+1)(n+2)/2$. In \cref{SS:EquivalencesofCategories} we fix a parity $\0$ element 
\[
e_{\lambda(n)} \in \End_{\qwebs}\left(\uparrow_{1}^{\otimes \ell } \right) \cong \End_{\qwebsupdown}\left(\uparrow_{1}^{\otimes \ell } \right),
\] which can be made explicit using the construction given in \cite{JN} (the isomorphism follows from \cref{T:Thetafunctor}.)  Define $\webs$ and $\websupdown$ to be $\qwebs$ and $\qwebsupdown$, respectively, with the additional relation 
\[
e_{\lambda(n)} =0.
\]  In \cref{T:maintheorem} we prove the functors $\Psi_{n}$ induce functors which we call by the same name,
\begin{align*}
\Psi_{n}&: \webs \to \mods,\\
\Psi_{n}&: \websupdown \to \modsupdown.
\end{align*}  Moreover, these functors are equivalences of monoidal supercategories.  In short, $\webs$ and $\websupdown$ are complete combinatorial models for $\mods$ and $\modsupdown$, respectively. 

Before describing some applications of our main theorems, we would like to mention we first learned the trick of defining monoidal categories via inverting morphisms from Jon Brundan.  Certainly the definition of our diagrammatic categories, many of the diagrammatic calculations, and the verification of the existence of the functors $\Phi_{n}$ would all be much less pleasant if we were to instead use the full set of relations implied by this inversion.  We also remark our webs are analogous to those used elsewhere in the literature to describe monoidal (super)categories.  For example, see \cite{CKM,TVW,RT} and the references therein for recent examples.  While some of our results and calculations have obvious parallels with those in the literature, we also encounter new difficulties which require different techniques.

\subsection{Applications}\label{SS:consequences}  Given a tuple of nonnegative integers $c=(c_{1},\dotsc ,c_{t})$, set $|c| = c_{1}+\dotsb +c_{t}$ and let $\uddoublearrow_{c}= \uddoublearrow_{c_{1}}\uddoublearrow_{c_{2}}\dotsb \uddoublearrow_{c_{t}}$ be an object of $\qwebsupdown$, where $\uddoublearrow$ denotes either a $\uparrow$ or $\downarrow$.  In \cref{T:EndomorphismAlgebras} we prove the map induced by the functor,
\[
\Psi_{n}: \End_{\qwebsupdown}\left(\uddoublearrow_{c} \right) \to \End_{U_{q}\left(\fq_{n} \right)}\left(S_{q}^{\uddoublearrow_{c}}\left(V_{n} \right) \right),
\] is a surjective superalgebra homomorphism.  Moreover this map is an isomorphism whenever $|c| < (n+1)(n+2)/2$.  When the object is $\uparrow_{1}^{k}$, then by \cref{P:sergeev-isomorphism} the superalgebra  $\End_{\qwebsupdown}\left(\uparrow_{1}^{k} \right)$ is isomorphic to the Hecke-Clifford superalgebra.  When it is $\uparrow_{1}^{r}\downarrow_{1}^{s}$, then by \cref{T:quantumBCalgebra} the superalgebra  $\End_{\qwebsupdown}\left(\uparrow_{1}^{r}\downarrow_{1}^{s} \right)$ is isomorphic to the quantum walled Brauer-Clifford superalgebra.  In this way our results generalize those of \cite{Olshanski,BGJKW}.  Moreover, in these cases our bound on when $\Psi_{n}$ is an isomorphism is sharp and improves on those given previously.  It would be interesting to describe other endomorphism superalgebras in $\qwebsupdown$, say by generators and relations. 

Let $\qwebsupdowneven$ be the subcategory consisting of all objects and all morphisms which can be written as a linear combination of webs with no dots.  We prove in \cref{T:ribboncategory} this is a ribbon category.  Pushing this structure forward with the functor $\Psi_{n}$ yields a ribbon category of $U_{q}\left(\fq_{n} \right)$-supermodules.  In particular, for any two tuples of nonnegative integers $a$ and $b$ we have isomorphisms $\uddoublearrow_{a}\otimes \uddoublearrow_{b} \to \uddoublearrow_{b} \otimes \uddoublearrow_{a}$ in $\qwebsupdown$. Thus we may apply the functor $\Psi_{n}$ to obtain an isomorphism
\[
S^{\uddoublearrow_{a}}_{q}\left(V_{n} \right) \otimes S^{\uddoublearrow_{b}}_{q}\left(V_{n} \right) \to S^{\uddoublearrow_{b}}_{q}\left(V_{n} \right) \otimes S^{\uddoublearrow_{a}}_{q}\left(V_{n} \right).
\]  This gives isomorphisms for all $U_{q}\left(\fq_{n} \right)$-supermodules which are arbitrary tensor products of symmetric powers of the natural supermodule and their duals.  As far as we are aware this is the first instance of canonical braiding-type isomorphisms for $U_{q}\left(\fq_{n} \right)$-supermodules beyond tensor products of the natural supermodule and its dual.  Moreover, as these are not necessarily polynomial supermodules, this establishes isomorphisms not covered by the results of \cite{GJKK,GJKKashK}.

Given a ribbon category such as $\qwebsupdowneven$ there are established techniques for defining invariants of oriented, framed links and other objects in low-dimensional topology (e.g.\ see \cite{Turaev}).  As a consequence of setting the clockwise oriented circle equal to zero in relation \cref{delete-bubble}, the objects in $\qwebsupdowneven$ have categorical dimension zero and, hence, give trivial invariants.  Undaunted, we show there are two ways to obtain nontrivial invariants from $\qwebsupdowneven$.  In \cref{SS:mtraces} we prove $\qwebsupdowneven$ admits an essentially unique family of modified traces in the sense of \cite{GKPM}.  Using these we describe how to define topological invariants.  In particular, we explain how to recover the colored HOMFLY-PT polynomials where edges are colored by partitions with one part.  In \cref{SS:bubblespop} we see we could choose to omit the left relation in \cref{delete-bubble} and instead declare the clockwise oriented circle equal to $2/(q-q^{-1})$.  This specialization defines a monoidal supercategory which, while no longer related to the representations of $U_{q}\left(\fq_{n} \right)$, does naturally lead to a ribbon category with nonzero categorical dimensions and, hence, to nontrivial topological invariants.  We leave calculating these invariants as an open question.  Interestingly, the calculations in \cref{SS:bubblespop} show $0$ and $2/(q-q^{-1})$ are the only possible specialization of the clockwise oriented circle labelled by $1$.

\section{Notation and Preliminaries}

\subsection{Notation and Preliminaries}\label{SS:NotationsandConventions}
Unless otherwise stated, the ground field is $\K = \C ((q))$ the field of formal Laurent series in the variable $q$.  In \cref{SS:rationalfunctions} we will describe how to obtain our main results over $\C (q)$.  We fix a primitive fourth root of unity $\sqrt{-1} \in \mathbb{C}$.  For short we set $\tq = q-q^{-1}$. Given positive integers $a,b$ and $t \in \K$ not a root of unity, define elements of $\K$ by 
\begin{align*}
	[a]_{t} &= \frac{t^{a}-t^{-a}}{t-t^{-1 }},\\
	[a]_{t}! &= [a]_{t}[a-1]_{t}\dotsb [2]_{t}[1]_{t},\\
	\qtbinom{a+b}{b} &= \frac{[a+b]_{t}!}{[a]_{t}![b]_{t}!}.
\end{align*}
When $t=q$ these are the quantum integers, factorials, and binomial coefficients, respectively.

Unless otherwise stated, throughout we work in the category of $\K$-superspaces, $\svec$, where by definition a \emph{superspace} is a $\Z_{2}$-graded $\K$-vector space.  We refer to the degree of a homogeneous element $w \in W = W_{\0} \oplus W_{\1 }$ as its \emph{parity} and write $\p{w} \in \Z_{2}$ for the parity.  An element of $W_{\0}$ (resp.\ $W_{\1}$) is often called \emph{even} (resp.\ \emph{odd}) for short.  We view $\K$ as a superspace concentrated in parity $\0$. Given two superspaces $V$ and $W$, the space of $\K$-linear maps $\Hom_{\K}(V,W)$ is naturally $\Z_{2}$-graded by declaring $f: V \to W$ to have parity $r \in \Z_{2}$ if $f(V_{s}) \subseteq W_{r+s}$ for all $s \in \Z_{2}$.  Similarly the vector space $V \otimes W = V \otimes_{\K} W$ has a natural grading given by $\p{v \otimes w} = \p{v}+\p{w}$ for all homogeneous $v \in V$, $w \in W$.

We assume the reader is familiar with the basics of super mathematics and the use of the \emph{sign rule}: whenever two homogenous elements trade positions in a formula, a sign given by the product of the parities of the elements should be included.  For example, if $A$ and $B$ are superalgebras, then $A \otimes B$ is again a superalgebra with the product given by $(a_{1} \otimes b_{1})(a_{2} \otimes b_{2}) = (-1)^{\p{a_{2}}\p{b_{1}}}a_{1}a_{2} \otimes b_{1}b_{2}$ for all homogeneous $a_{1},a_{2} \in A$ and $b_{1},b_{2} \in B$.  Similarly, the braiding isomorphism $V \otimes W \to W \otimes V$ in the category of $\K$-superspaces is given by the graded flip map $v \otimes w \mapsto (-1)^{\p{v}\p{w}}w \otimes v$ for all homogeneous $v \in V$ and $w \in W$. Here and elsewhere unlabelled tensor products are taken over the relevant ground field.  Also, as we did here, we commonly give a formula for homogeneous elements only. The general case is  obtained by linearity.  

By a \emph{supercategory} we mean a category enriched in $\svec$. Similarly, a \emph{superfunctor} is a functor enriched in $\svec$.  Unless otherwise stated, all functors in this paper will be superfunctors which preserve the parity of homogeneous morphisms.  Given two supercategories $\cat{A}$ and $\cat{B}$, there is a supercategory $\cat{A}\boxtimes\cat{B}$ whose objects are pairs $(\ob a,\ob b)$ of objects $\ob a\in\cat{A}$ and $\ob b\in\cat{B}$ and whose morphisms are given by the tensor product of superspaces $\Hom_{\cat{A}\boxtimes\cat{B}}((\ob a,\ob b),(\ob a',\ob b'))=\Hom_\cat{A}(\ob a,\ob a')\otimes\Hom_\cat{B}(\ob b,\ob b')$ with composition defined using the braiding in $\svec$.  

By a \emph{monoidal supercategory} we mean a supercategory $\cat{A}$ equipped with a functor $-\otimes-:\cat{A}\boxtimes\cat{A}\to\cat{A}$, a unit object $\unit$, and even supernatural isomorphisms $(-\otimes-)\otimes-\arup{\sim}-\otimes(-\otimes-)$ and $\unit\otimes-\arup{\sim}-\larup{\sim}-\otimes\unit$ called \emph{coherence maps} satisfying certain axioms analogous to the ones for a monoidal category. In particular, if $f$ and $g$ are homogenous morphisms in a monoidal supercategory, then the \emph{super-interchange law} is the identity
\begin{equation}\label{E:superinterchange}
(f\otimes g)\circ(h\otimes k)=(-1)^{\p{g}\p{h}}(f\circ h)\otimes(g\circ k).
\end{equation}

A monoidal supercategory is called \emph{strict} if its coherence maps are identities.  
A \emph{monoidal functor} between two monoidal supercategories $\cat{A}$ and $\cat{B}$ is a functor $F:\cat{A}\to\cat{B}$ equipped with an even supernatural isomorphism $(F-)\otimes (F-)\arup{\sim}F(-\otimes-)$ and an even isomorphism $\unit_\cat{B}\arup{\sim}F\unit_\cat{A}$ satisfying axioms analogous to the ones for a monoidal functor.

\section{The quantized  enveloping superalgebra of type Q}

\subsection{The Lie superalgebra \texorpdfstring{$\fq_{n}$}{q(n)}}\label{SS:Liesuperalgebras}

For $n \geq 1$, define an index set $I_{n|n} = \{\pm i \mid 1 \leq i \leq n\}$.  Given $i \in I_{n|n}$, define its parity $\p{i} \in \Z_2$ by $\p{i} = \0$ if $i > 0$, and $\p{i} = \1$ otherwise.  We also order $I_{n|n}$ via the usual ordering on integers.  Fix a complex superspace $V= V_{\0} \oplus V_{\1}$ with $\dim_{\C}(V_{\0}) = \dim_{\C}(V_{\1})=n$.   Fix a homogeneous basis $\left\{v_{i} \mid i \in I_{n|n} \right\}$ for $V$ with $\p{v_{i}}=\p{i}$ for all $i \in I_{n|n}$.  Realized as matrices with respect to our choice of basis,  $\fq_{n}$ is the set of all $2n \times 2n$ matrices of the following form: 
\begin{equation}\label{E:qndef}
\fq(n)=\left\{\left(\begin{matrix}A & B \\
B & A 
\end{matrix} \right) \mid A,B \text{ are } n \times n \text{ complex matrices} \right\}.
\end{equation}
The $\Z_{2}$-grading is given by declaring $\q_{n,\0}$ (resp.\ $\q_{n,\1}$) to be the subspace of all such matrices with $B=0$ (resp.\ $A=0$). The Lie bracket is given by the graded commutator bracket. That is, for homogeneous $x,y \in \fq_{n}$, $[x,y] = xy-(-1)^{\p{x} \p{y}}yx$.  Observe $\fq_{n}$ has a homogenous basis given by $e_{i,j}^{\0}:=e_{i,j}+e_{\ibar,\jbar}$ and $e_{i,j}^{\1 }:=e_{\ibar,j}+e_{i,\jbar}$ for $1\leq i,j\leq n$, where $e_{i,j}$ denotes the corresponding matrix unit for all $i,j \in I_{n|n}$. Observe, $\p{e_{i,j}^{\varepsilon}} = \varepsilon$ for all $1 \leq i,j \leq n$ and $\varepsilon \in \Z_{2}$.

Fix the Cartan subalgebra of $\fh \subseteq \fq_{n}$ consisting of matrices as in \cref{E:qndef} with $A$ and $B$ diagonal.  For $i=1, \dotsc , n$, let $\varepsilon_{i}: \fh_{\0} \to \C$ be defined by $\varepsilon_{i}(e^{\0}_{j,j}) =\delta_{i,j}$. Then $\varepsilon_{1}, \dotsc , \varepsilon_{n}$ gives a basis for $\fh_{\0}^{*}$.  It is convenient for later calculations to set the notation $\varepsilon_{-i}=\varepsilon_{i}$ for any $i \geq 0$.  Define a bilinear form on $\fh_{\0}^{*}$ by $(\varepsilon_{i}, \varepsilon_{j}) = \delta_{i,j}$. Set 
\begin{equation}\label{E:XTdef}
X(T) = X(T_{n}) = \oplus_{i=1}^{n} \Z\varepsilon_{i} \subseteq \fh_{\0}^{*}.
\end{equation}
Let $X(T)_{\geq 0} = \left\{\lambda \in X(T) \mid (\lambda, \varepsilon_{i}) \geq 0 \text{ for all $i=1,\dotsc ,n$} \right\}$.  For $1 \leq i,j \leq n$, write $\alpha_{i,j}=\vareps_i-\varepsilon_{j}$.

Fix the Borel subalgebra $\fb \subseteq \fq_{n}$ consisting of matrices as in \cref{E:qndef} with $A$ and $B$ upper triangular. Corresponding to our choices the set of roots, positive roots, and simple roots are $\left\{\alpha_{i,j} \mid 1 \leq i \neq j \leq n \right\}$,  $\left\{\alpha_{i,j}  \mid 1 \leq i<j \leq n \right\}$, and $\left\{\alpha_{i} := \alpha_{i,i+1} \mid 1 \leq i \leq n-1 \right\}$, respectively. 

Let $U(\fq_{n})$ denote the enveloping superalgebra of $\fq_{n}$ over $\C$.  A $U(\fq_{n})$-supermodule is a $\C$-superspace $M = M_{\0} \oplus M_{\1 }$ where the action of $U(\fq_{n})$ on $M$ is compatible with grading in the sense that if $x \in U(\fq_{n})_{r}$ and $m \in M_{s}$, then $x.m \in M_{r+s}$.  We allow for all $U(\fq_{n})$-supermodule homomorphisms, not just those which preserve the $\Z_{2}$-grading.  However, using the rule for superspaces (see \cref{SS:NotationsandConventions}) there is a $\Z_{2}$-grading on $\Hom$-spaces which makes the category of $U(\fq_{n})$-supermodules a supercategory. Since $U(\fq_{n})$ is a Hopf superalgebra the category of $U(\fq_{n})$-supermodules is a monoidal supercategory.

\subsection{The quantized enveloping superalgebra \texorpdfstring{$U_{q}(\fq_{n})$}{Uq}}\label{SS:quantizedenvelopingalgebra}

For $n \geq 1$ let $U_q(\q_n)$ denote the quantized enveloping algebra of the Lie superalgebra $\q_n$.  This is an associative $\K$-superalgebra originally defined by Olshanski in \cite{Olshanski}.  We describe it using the presentation given in \cite{GJKKashK}.   As an associative algebra $U_{q}(\q_n)$ has homogeneous generators $E_i, \bar{E_i}, F_i, \bar{F}_i$ $(1 \leq i \leq n-1)$ and $K_j, K_j^{-1}, \bar{K}_j$ $(1 \leq j \leq n)$.  The generators $E_i, F_i, K_j, K_j^{-1}$ are even, and their barred counterparts $\bar{E}_i, \bar{F}_i, \bar{K}_{j}$ are odd.  The generators satisfy a rather lengthy list of relations which we omit; see \cite[Definition 1.1]{GJKKashK} for details.  We will introduce an idempotent version in \cref{SS:IdempotentAlgebra}.  When comparing our algebra to the one in \cite{GJKKashK} the reader should note we write  $K_{i}^{\pm 1}$ and $\bar{K}_{i}$ for their elements $q^{\pm k_i}$ and $k_{\bar{i}}$, respectively.  It is also useful to note it is remarked at the end of  \cite[Definition 1.1]{GJKKashK} that $U_q(\q_n)$ is generated by its even generators together with $\bar{K}_1$.  In particular, superalgebra homomorphisms need only be specified on this smaller list of generators.

By definition, the \textit{natural supermodule} for $U_q(\q_n)$ is the $\K$-superspace, $V_{n}=V_{q,n}$, with fixed homogeneous basis $\left\{v_{i} \mid i \in I_{n|n} \right\}$ and $\Z_{2}$-grading given by $\p{v_{i}}=\p{i}$ for all $i \in  I_{n|n}$.  The action of the generators of $U_q(\q_n)$ are given by:
$$\begin{array}{lll}
K_j v_i = q^{(\varepsilon_{i},\varepsilon_{j})} v_i, & E_j v_i = \delta_{i, j+1} v_{j} + \delta_{i, -(j+1)} v_{-j}, & F_j v_i = \delta_{i,j} v_{j+1} + \delta_{i,-j} v_{-(j+1)} ,\\
\bar{K}_j v_i = \delta_{i,j} v_{-j} + \delta_{i,-j} v_j & \bar{E}_j v_i = \delta_{i,j+1} v_{-j} + \delta_{i,-(j+1)} v_{j} , & \bar{F}_j v_i = \delta_{i,j} v_{-(j+1)} + \delta_{i,-j} v_{j+1}.
\end{array}
$$

As shown in \cite{Olshanski}, the superalgebra $U_q(\q_n)$ is actually a Hopf superalgebra.  The coproduct $\Delta :  U_q(\q_n)  \to U_q(\q_n) \otimes U_q(\q_n)$, counit $\varepsilon: U_{q}(\fq_{n}) \to \K$, and antipode $S : U_q(\q_n) \to U_q(\q_n)$ are given as follows.  In the presentation used here the coproduct given on the generators of $U_q(\q_n)$ by:

$$\begin{array}{ll}
\Delta(E_i) = E_i \otimes K_i^{-1} K_{i+1} + 1 \otimes E_i, & \Delta(F_i) = F_i \otimes 1 + K_i K_{i+1}^{-1} \otimes F_i , \\
\Delta(K_i) = K_i \otimes K_i, & \Delta(\bar{K}_1) =  \bar{K}_1 \otimes K_1 + K_1^{-1} \otimes \bar{K}_1.
\end{array}$$
The value of the coproduct on the other odd generators is rather complicated (see \cite[Section 2]{GJKK}). Since we will not need these values, we omit them. The counit is given by 
\[
\begin{array}{ll}
\varepsilon(E_{i})=\varepsilon(F_{i})=\varepsilon(\bar{E}_{i})=\varepsilon(\bar{F_{i}})=\varepsilon(\bar{K}_{i})=0, & \varepsilon (K_{i})=1.
\end{array}
\]
The value of the antipode on the generators is given by:
$$\begin{array}{llll}
S(E_i) = -E_i K_i K_{i+1}^{-1}, & S(F_i) = -K_i^{-1} K_{i+1} F_i &
S(K_j) = K_j^{-1} & S(\bar{K}_1) = -\bar{K}_1.
\end{array}$$
Again, the value of the antipode on the other generators is complicated and not needed for our purposes so we omit them.  We remind the reader the map $S$ is an anti-automorphism of superalgebras, meaning $S(ab) = (-1)^{\p{a} \p{b}} S(b) S(a)$ for all homogeneous $a,b \in U_q(\q_n)$.

As we do for $U(\q_n)$, when studying $U_{q}(\fq_{n})$-supermodules we consider \emph{all} supermodule homomorphisms, not just those which preserve the $\Z_{2}$-grading. Using the same rule as for $\K$-superspaces (see \cref{SS:NotationsandConventions}) there is a $\Z_{2}$-grading on the $\Hom$-spaces for the category of $U_{q}(\fq)$-supermodules which makes it a supercategory.  Moreover, the Hopf superalgebra structure on $U_{q}(\fq_{n})$ makes this a monoidal supercategory.

A check using the presentation in \cite{GJKKashK} verifies there is a superalgebra homomorphism  $\tau : U_q(\q_n) \to U_q(\q_n)$ given on generators by:
\begin{equation}\label{E:tau}
\begin{array}{lll}
\tau(E_i) = F_i, & \tau(F_i) = E_i, & \tau(K_i) = K_i^{-1}, \\
\tau(\bar{E}_i) = -\sqrt{-1} \bar{F}_i, & \tau({\bar{F}_i}) = -\sqrt{-1} \bar{E}_i, & \tau({\bar{K}_j}) = \sqrt{-1} \bar{K}_j.
\end{array}
\end{equation}  Since $\tau$ has order four it is an isomorphism.
Moreover $\tau$ satisfies $\Delta \circ \tau = \tau \otimes \tau \circ \sigma \circ \Delta$, where $\sigma : U_q(\q_n) \otimes U_q(\q_n) \to U_q(\q_n) \otimes U_q(\q_n)$ is the flip map $a\otimes b \mapsto (-1)^{\p{a} \p{b}} b \otimes a$.  
Given a $U_q(\q_n)$-supermodule $M$ we can define a new action by $x \cdot_\tau m = \tau(x)m$ for all $x \in U_q(\q_n)$ and $m \in M$.  This equips the superspace $M$ with a new $U_q(\q_n)$-supermodule structure. We call the resulting supermodule the \textit{$\tau$-twist} of $M$ and denote it by $M^{\tau}$.  

Since the antipode $S$ also satisfies $\Delta \circ S = S \otimes S \circ \sigma \circ \Delta$, it follows 
\begin{equation}\label{E:TauSDelta}
\Delta \circ S \circ \tau = S \otimes S \circ \tau \otimes \tau \circ \Delta.
\end{equation}

\subsection{Polynomial Representations of \texorpdfstring{$U(\fq_{n})$}{Uq}}  Let $\TT_q$ denote the category of polynomial representations of $U_q(\q_n)$.  By definition this is the full Serre subcategory generated by all finite-dimensional $U_{q}(\fq_{n})$-supermodules whose composition factors appear as composition factors of $V_{q}^{\otimes d}$ for various $d\geq 0$ (where $V_{q}^{\otimes 0}=\K$ is the trivial supermodule by definition). 

The Cartan subalgebra of $U_{q}(\fq_{n})$ is defined to be the subsuperalgebra generated by $K_{i}^{\pm 1}$, $\bar{K}_{i}$ for $i=1, \dotsc , n$.  Given a polynomial representation $M$ and $\lambda \in X(T_{n})$, the $\lambda$-weight space of $M$ with respect to this choice of Cartan subalgebra is 
\[
M_{\lambda} =\{m \in M \mid  K_{i}m = q^{(\lambda, \varepsilon_{i})}m \text{ for $i=1, \dotsc , n$}\}.
\]
For any polynomial supermodule $M$ we have
\[
M = \bigoplus_{\lambda \in X(T)_{\geq 0}} M_{\lambda}.
\]   Given $\lambda = \sum_{i=1}^{n} \lambda_{i}\varepsilon_{i} \in X(T)_{\geq 0}$, define the \emph{total weight of $\lambda$} to be $|\lambda| = \sum_{i=1}^{n} \lambda_{i}$.  Given $d \geq 0$ we say a supermodule in $\TT_{q}$ is a polynomial representation of degree $d$ if all of its weights have total weight $d$.  This is equivalent to requiring the composition factors which occur in $M$ also occur in $V_{n}^{\otimes d}$.  Recall, if $d\geq 0$, then a partition of $d$ is a weakly decreasing sequence of nonnegative integers which sums to $d$.  A partition is \emph{strict} if the positive integers strictly decrease.  Let $\Lambda^{+}_{n}$ denote the set of strict partitions with no more than $n$ nonzero parts. We identify $\Lambda_{n}^{+}$ as a subset of $X(T)$ via the map $(\lambda_{1}, \dotsc , \lambda_{n}) \mapsto \sum_{i=1}^{n}\lambda_{i}\varepsilon_{i}$.

 Combining results from \cite{GJKK, GJKKashK} we have the following theorem.  When comparing with the results therein, the reader is reminded we are working over the field of formal Laurent series.

\begin{theorem}\label{T:Grantcharovtheorem}  The following statements are true:
	\begin{enumerate}
		\item The category $\TT_{q}$ is completely reducible and the Hopf superalgebra structure on $U_{q}(\fq_{n})$ makes $\TT_{q}$ into a monoidal supercategory.
		\item For each $\lambda \in \Lambda^{+}_{n}$ there exists a simple $U_{q}(\fq_{n})$-supermodule labelled by $\lambda$, $L_{q}(\lambda) = L_{q, n}(\lambda)$.  The set 
		\[
		\left\{L_{q}(\lambda) \mid \lambda \in \Lambda^{+}_{n} \right\}
		\]
		is a complete irredundant set of simple supermodules in $\TT_{q}$.
		\item Each $L_{q}(\lambda)$ admits a $q=1$ specialization, $L_{q=1}(\lambda)$, which is a $U(\fq_{n})$-supermodule and the characters and dimensions of $L_{q}(\lambda)$ and $L_{q=1}(\lambda)$ coincide.   Moreover,  $L_{q=1}(\lambda)$ is isomorphic to $L(\lambda)$, the simple $U(\fq_{n})$-supermodule labelled by $\lambda$.
		\item Supermodules in $\TT_{q}$ are completely determined by their characters. In particular, the multiplicity of  $L_{q}(\nu)$ in $L_{q}(\lambda) \otimes_{\K} L_{q}(\mu)$  equals the multiplicity of $L(\nu)$  as a summand of $L(\lambda) \otimes_{\C } L(\mu)$.
	\end{enumerate}
	\end{theorem}

If $M$ is a polynomial supermodule of $U_q(\q_n)$, the antipode $S$ of $U_q(\q_n)$ affords a supermodule structure on the dual space $M^* = \Hom_{\K}(M,\K)$ and, in turn, on $M^{\oldstar}:=\left(M^{*} \right)^{\tau}$.  
Explicitly, the action of a homogeneous $x \in U_{q}(\fq_{n})$ on a homogeneous $f \in M^{\oldstar}$ is given by $(xf)(v) = (-1)^{|x||f|} f(S(\tau(x))v)$.  Since the characters of $M^{\oldstar}$ and $M$ coincide, the following is immediate from the fact polynomial representations are completely reducible and the simple supermodules in $\TT_{q}$ are determined by their characters.
\begin{proposition}\label{P:dualtwist}
If $M$ is a polynomial representation of $U_q(\q_n)$, then $M^{\oldstar} \cong M$.
\end{proposition}

\subsection{Schur-Weyl-Sergeev-Olshanski Duality}\label{SS:Olshanskiduality}
If $A$ and $B$ are split semisimple superalgebras and $S$ and $T$ are simple $A$- and $B$-supermodules, respectively, then their outer tensor product, $S \boxtimes T$, is the direct sum of at most two simple $A \otimes B$-supermodules and all summands are isomorphic.  As a matter of notation we write $S \star T$ for the simple $A \otimes B$-supermodule which appears as a direct summand in $S \otimes T$.

For each $k \geq 1$  Olshanski introduced the \emph{Hecke-Clifford superalgebra} $\HC_{k}(q)$ and showed it satisfies a double centralizer theorem with respect to the $U_{q}(\fq_{n})$-supermodule $V_{n}^{\otimes k}$ \cite{Olshanski}.  By definition it is the associative $\K$-superalgebra generated by the elements $T_{1}, \dotsc , T_{k-1}, c_{1}, \dotsc , c_{k}$ subject to the relations listed in \cref{L:srelations}.  The $\Z_{2}$-grading is given by declaring $T_{1}, \dotsc , T_{k-1}$ to be even and $c_{1}, \dotsc , c_{k}$ to be odd.  The following theorem summarizes the results we will need.

\begin{theorem}\label{T:OlshanskiDuality}  Let $k \geq 1$ be fixed.
	
	\begin{enumerate}
		\item  There is a surjective superalgebra homomorphism 
		\[
		\psi: \HC_{k}(q) \to \End_{U_{q}(\fq_{n})}\left(V_{n}^{\otimes k} \right).
		\]  This map is an isomorphism if $n \gg k$.
		\item  The superalgebra $\HC_{k}(q)$ is semisimple.  In fact, since the base field is $\K$,  $\HC_{k}(q)$ is split semisimple.
		\item  As a $U_{q}(\fq_{n}) \otimes \HC_{k}(q)$-supermodule,
		
\begin{equation}\label{E:Olshanskidecomposition}
		V_{n}^{\otimes k} \cong \bigoplus_{\lambda \in \Lambda^{+}_{n,k}}  L_{q, n}(\lambda) \star D^{\lambda}, 
\end{equation}
where $\Lambda^{+}_{n,k}$ is the set of all strict partitions of $k$ with not more than $n$ parts, and where 
\[
\left\{D^{\lambda} \mid \lambda \in \Lambda^{+}_{n,k} \right\}
\]   is the set of simple $\HC_{k}(q)$-supermodules appearing in $V_{n}^{\otimes k}$. 
	\end{enumerate}
	
\end{theorem}

\begin{proof}  These results are already known so we only point to the relevant places in the literature.  Statement (a.) can be found in \cite[Theorem 3.28]{BGJKW}.  The first part of statement~(b.) holds by \cite[Proposition 2.1]{JN}.  Since $\K$ is known to contain the square roots of $[a]_{q^{2}}$ for all $a \geq 1$ (e.g.\ see the remark after \cite[Proposition 1.7]{GJKKashK}) it follows from \cite[Theorem 6.7]{JN} that $\HC_{k}(q)$ is split semisimple.  That a direct sum of the form given in (c.) holds follows from the theory of double centralizers for superalgebras; see \cite[Proposition 3.5]{ChengWangBook} for details.  That the listed simple supermodules are the ones which occur in this direct sum can be deduced by specializing to $q=1$ and applying \cref{T:Grantcharovtheorem} and the analogous statement in classical Sergeev duality (e.g.\ see \cite[Theorem 3.46]{ChengWangBook}). 
\end{proof}

Let $\SP (k)$ be the set of strict partitions of $k$.  Since $\psi$ is an isomorphism for $n$ sufficiently large it follows from the previous theorem that the simple $\HC_{k}(q)$-supermodules are parameterized by $\SP (k)$.  The super analogue of the Artin-Wedderburn theorem (e.g.\ see \cite[Chapter 3]{ChengWangBook}) implies there is an isomorphism of superalgebras,  
\begin{equation}\label{E:ArtinWedderburn} 
\HC_{k}(q)\cong \bigoplus_{\lambda\in\SP(k)}M(E_{\lambda}),
\end{equation} where $E_{\lambda}$ equals $\K$ or a Clifford algebra on one generator and, if we forget the $\Z_{2}$-grading, $M(E_{\lambda})$ is a matrix ring with entries from $E_{\lambda}$.  For our purposes the important point is each $M(E_{\lambda})$ is a simple superalgebra and so is generated as an ideal by any nonzero homogeneous element.

 Elements of  $M(E_{\lambda})$ are distinguished by the fact they act nontrivially on an $\HC_{k}(q)$-supermodule if and only if the $D^{\lambda}$-isotypic component of the supermodule is nontrivial.  In particular, consider $V_{n}^{\otimes k}$.  In this case it follows from \cref{E:Olshanskidecomposition} that if $L_{q, n}(\lambda)$ appears in $V_{n}^{\otimes k}$, then a nonzero element of $M(E_{\lambda})$ will act nontrivially on $V_{n}^{\otimes k}$ and the image under the action of this element will be a direct sum of $L_{q, n}(\lambda)$'s as a $U_{q}(\fq_{n})$-module.  On the other hand, given $a \in \HC_{k}(q)$, if $a = \sum_{\gamma \in \SP (k)} x_{\gamma}$ with $x_{\gamma}$ an element of $M(E_{\gamma})$, then the image of $a$ on $V_{n}^{\otimes k}$ will contain a nontrivial $L_{q, n}(\gamma)$-isotypic component if and only if $x_{\gamma}$ is nonzero.

For each $\lambda \in \SP (k)$, fix $e_{\lambda} \in \HC_{k}(q)$ to be an even nonzero element of $M(E_{\lambda})$.  Such elements are constructed in \cite{JN}.  Since their precise form is not needed here we do not describe them explicitly.   As the $M(E_{\lambda})$ are simple superalgebras, the two-sided ideal generated by any nonzero homogeneous element will contain and be generated by $e_{\lambda}$.    Combining this discussion with \cref{T:OlshanskiDuality} yields the following result.  Given a strict partition $\lambda$ let $\ell(\lambda)$ equal the number of nonzero parts in $\lambda$.

\begin{proposition}\label{P:OlshanskiKernel}  Let $\psi: \HC_{k}(q) \to \End_{\U_{q}(\fq_{n})}\left(V_{n}^{\otimes k} \right)$ be the superalgebra homomorphism given in \cref{T:OlshanskiDuality}.  The kernel of $\psi$ is generated by the set 
\[
\left\{e_{\lambda} \mid \ell (\lambda) > n \right\}.
\]  In particular, $\psi$ is injective if and only if $k < (n+1)(n+2)/2$. 
\end{proposition}

\subsection{Quantum symmetric powers} \label{SS:Symmetric}  
For any mathematical statement $P$ we adopt the convention $\delta_{P}=1$ if the statement $P$ is true and $\delta_{P}=0$ otherwise.  For example $\delta_{i=j}=\delta_{i,j}$ is the usual Kronecker delta function.    Given $a,b \in I_{n|n}$, let 
\begin{equation}\label{E:phidef}
\varphi(a,b)=(-1)^{\p{b}}\delta_{a=\pm b}.
\end{equation}

Let $T \in \End_{U_q(\q_n)}(V_n^{\otimes 2})$ denote the image of $T_1 \in \HC_2(q)$ under the superalgebra homomorphism $\psi$ from \cref{T:OlshanskiDuality}.  A formula for this map can be found in \cite[Theorem 3.16]{BGJKW}.  Explicitly, for any $a,b \in I_{n|n}$ the action of $T$ is given by
\begin{equation} \label{E:Braiding11}
T(v_a \otimes v_b)= q^{\varphi(a,b)} (-1)^{\p{a}\p{b}}v_b  \otimes v_a + \delta_{a < b} \tq v_a \otimes v_b + \delta_{-a < b} \tq (-1)^{\p{b} } v_{-a} \otimes v_{-b}.
\end{equation}

Let $\mathcal{T}(V_n) = \bigoplus_{d \geq 0} V_{n}^{\otimes d}$ denote the tensor superalgebra for $V_n$.  Let $I$ denote the two-sided ideal of $\mathcal{T}(V_n)$ generated by all elements of the form $qv_a \otimes v_b - T(v_a \otimes v_b)$. We call  $S_q(V_n)=\mathcal{T}(V_n)/I$ the \emph{quantum symmetric superalgebra}.  Since $\mathcal{T}(V_{n})$ is $\Z$-graded by total degree and $I$ is homogenous, it follows the quantum symmetric superalgebra admits a $\Z$-grading, $S_{q}(V_{n}) = \bigoplus_{d \geq 0} S_{q}^{d}(V_{n})$.  For $d \geq 0$ we call $S_q^d(V_n)$ the \emph{$d$-th quantum symmetric power} of $V_n$.  Since the action of $U_{q}(\fq_{n})$ respects the $\Z$-grading, it follows $S_{q}^{d}(V_{n})$ is a $U_{q}(\fq_{n})$-supermodule for all $d \geq  0$. 

For brevity in what follows we write $v_{a_1} \cdots v_{a_d}$  for the image of $v_{a_1} \otimes \cdots \otimes v_{a_d}$ under the canonical map $\mathcal{T}(V_n) \to S_q(V_n)$. 

\begin{lemma}\label{L:AlternatePresentationofSq} As an associative superalgebra ${S}_q(V_n)$ has a presentation given by the generators $\{v_a \mid a \in I_{n|n}\}$ with parity given by $\p{v_{a}}=\p{a}$ and subject to the relation
\[
q v_a v_b = q^{\varphi(a,b)} (-1)^{\p{a}\p{b}}v_b v_a + \delta_{a < b} \tq v_a v_b + \delta_{-a < b} \tq (-1)^{\p{b} } v_{-a}  v_{-b}.
\] 

Moreover, the set 
\[
\left\{ \prod_{a \in I_{n|n}}  v_{a}^{d_{a}} \mid d_{a} \geq 0 \text{ if } \p{v_{a}}=\0 \text{ and } d_{a} \in \{0,1 \} \text{ if } \p{v_{a}} = \1 \right\}
\] is a spanning set for $S_{q}(V_{n})$, where the product is taken in the order on $I_{n|n}$ fixed in \cref{SS:NotationsandConventions}. 
\end{lemma}

\begin{proof}  In light of the explicit description of the action of $T$ given in \cref{E:Braiding11} the description of $S_{q}(V_{n})$ by generators and relations is nothing but a reformulation of the definition of $S_{q}(V_{n})$.  A straightforward argument using the given relation along with induction on $\Z$-degree shows the given monomials span $S_{q}(V_{n})$. 
\end{proof}

\begin{remark}\label{R:Sqproperties}  In \cref{P:basicpropertiesofAq}(e.)  it will be shown the above set is a basis for $S_{q}(V_{n})$. As a consequence $S_{q}(V_{n})$ is a flat deformation of the supersymmetric algebra $S(V_{n})$ and  $S^{d}_{q}(V_{n}) \cong L_{q}((d))$ as $U_{q}(\fq_{n})$-supermodules where $(d)$ is the partition of $d$ with one part.
\end{remark}

\subsection{The superalgebra \texorpdfstring{$\AA_{q}$}{Aq} and Quantum Howe Duality of Type Q}\label{SS:Aqdef}  


\begin{definition}\label{D:Aqdefinition}  Fix $m,n \geq 1$. Let $\AA_{q}=\AA_{q}(V_{m} \star V_{n})$ be the associative $\K$-superalgebra generated by the set
\[
\left\{t_{a,b} \mid (a,b) \in I_{m|m} \times I_{n|n} \right\}
\] with the parity of $t_{a,b}$ given by $\p{t_{a,b}}= \p{a}+\p{b} \in \Z_{2}$.   These generators are subject to the following relations for all $a,b \in I_{n|n}$:
\begin{equation}\label{E:AqRel1}
t_{a,b} =t_{-a,-b}, 
\end{equation}
\begin{align}\label{E:AqRel2}
		q^{\varphi(a,c)}&(-1)^{(\p{a}+\p{b})(\p{c}+\p{d})}t_{a,b}t_{c,d} + \delta_{c < a}\tq (-1)^{\p{c}+(\p{b}+\p{c})(\p{c}+\p{d})} t_{c,b}t_{a,d} + \delta_{c < -a}\tq (-1)^{\p{c}+(\p{b}+\p{c}+\1)(\p{c}+\p{d})} t_{-c, b}t_{-a, d} \notag \\ 
		&= q^{\varphi(b,d)}t_{c,d}t_{a,b} + \delta_{b<d}\tq (-1)^{\p{b}+(\p{b}+\p{d})(\p{c}+\p{b})}t_{c,b}t_{a,d} - \delta_{-b < d}\tq (-1)^{\p{b}+(\p{b}+\p{c}+\1)(\p{b}+\p{d}+\1)}t_{c,-b}t_{a,-d}.
\end{align}
\end{definition}

\begin{remark}\label{R:AlternateAq}
For calculations it is helpful to observe relation \cref{E:AqRel2} is equivalent to the following list of relations.   For $a,b,c,d > 0$ and $a \leq c$ let
\begin{align}
q^{\delta_{a,c}}t_{a,b}t_{c,d} &= q^{\delta_{b,d}}t_{c,d}t_{a,b} + \delta_{b<d}\tq t_{c,b}t_{a,d} + \tq t_{c,-b}t_{a,-d}, \label{E:AltAqRel2a} \\
q^{\delta_{a,c}}t_{a,b}t_{c,-d} &= q^{-\delta_{b,d}}t_{c,-d}t_{a,b}  - \delta_{d < b}\tq t_{c,-b}t_{a,d}, \label{E:AltAqRel2b}\\
q^{\delta_{a,c}}t_{a,-b}t_{c,d} &= q^{\delta_{b,d}}t_{c,d}t_{a,-b} + \tq t_{c,-b}t_{a,d} + \delta_{b < d}\tq t_{c,b}t_{a,-d}, \label{E:AltAqRel2c}\\
q^{\delta_{a,c}}t_{a,-b}t_{c,-d} &= -q^{-\delta_{b,d}}t_{c,-d}t_{a,-b} + \delta_{d<b}\tq t_{c,-b}t_{a,-d}.\label{E:AltAqRel2d}
\end{align}
\end{remark}
The superalgebra $\AA_{q}$ appears as a deformation of the polynomial functions on $\fq (n)$ in \cite[Definition 5.3]{BGJKW}.  It is isomorphic to the algebra $\AA_q(\q_m,\q_n)$ which is used by \cite{ChangWang} to establish a Howe duality between $U_q(\q_m)$ and $U_q(\q_n)$, where the isomorphism is established below.  We summarize their results here.
	
Set $s = \max(m, n)$ and $r = \min(m,n)$.  Define elements $t_{a,b}' \in U_q(\q_s)^*$ via $x v_b = \sum_{a \in I_{s|s}} t_{a,b}'(x) v_a$ for all $x \in U_q(\q_s)$ and $b \in I_{s|s}$.  Then, \cite{ChangWang} defines $\AA_q(\q_m,\q_n)$ to be the $\Z$-graded subsuperalgebra of $U_q(\q_s)^*$ generated by $\{t'_{a,b} \mid (a,b) \in I_{m|m} \times I_{n|n}\}$, where the element $t'_{a,b}$ is declared to have $\Z$-degree equal to 1.  The authors deduce that $\AA_q(\q_m,\q_n)$ is a $U_q(\q_m) \otimes U_q(\q_n)$-supermodule, where the action of an element $x \otimes y$ on $t'_{a,b}$ is given by:
\begin{equation} \label{E:CWAction}
\left(x \otimes y \cdot t'_{a,b} \right)(z) = (-1)^{\p{x}(\p{a} + \p{b}) + \p{y}(\p{a} + \p{b} + \p{z})} t'_{a,b}\left(S(x) z y \right),
\end{equation}
where $x \in U_q(\q_m)$, $y \in U_q(\q_n)$, and $z \in U_q(\q_s)$.  In the product $S(x) z y$, view $x$ and $y$ as elements of $U_q(\q_s)$ via the embeddings $U_q(\q_m), U_q(\q_n) \hookrightarrow U_q(\q_s)$ given by sending generators to elements of the same name.

\begin{theorem} 
\cite[Theorem 3.5]{ChangWang} There is a degree preserving isomorphism of $\Z$-graded $U_q(\q_m) \otimes U_q(\q_n)$-supermodules:
\begin{equation} \label{E:HoweDualityDecomp} 
	\AA_q(\q_m, \q_n) \cong \bigoplus_{\lambda \in \Lambda_r^+} L_{q, m}(\lambda)^* \star L_{q, n}(\lambda),
	\end{equation}
	where the $\Z$-degree $d$ subspace of the supermodule on the right is $\bigoplus_{\lambda \in \Lambda_r^+, |\lambda| = d}  L_{q, m}(\lambda)^* \star L_{q, n}(\lambda)$.  
\end{theorem}

The following proposition summarizes what we need to know about $\AA_{q}$ and establishes the isomorphism $\AA_q \to \AA_q(\q_m, \q_n)$.    Define a lexicographic order on $I_{m|m} \times I_{n|n}$ by $(a,b) \leq (c,d)$ if $a < c$ or if $a=c$ and $b \leq d$. 


\begin{proposition}\label{P:basicpropertiesofAq}  The following statements hold true for $\AA_{q}$.
	\begin{enumerate}
		\item The superalgebra $\AA_{q}$ is a $\Z$-graded superalgebra where the generators are in degree one and the direct summands are finite-dimensional $\K$-vector spaces.
		\item If $\p{t_{a,b}}=\1$, then $t_{a,b}^{2}=0$.
		\item  As a superalgebra $\AA_{q}$ is isomorphic to $\AA_{q}(\fq_{m}, \fq_{n})$.
		\item  The set 
		\begin{equation}\label{E:Aqbasis}
		\B=\left\{ \prod_{(a,b) \in I_{m|m}\times I_{n|n}} t_{a,b}^{d_{a,b}} \mid a >0, d_{a,b} \in \Z_{\geq 0} \text{ if $\p{t_{a,b}}=\0$, and } d_{a,b} \in \{0,1 \} \text{ if $\p{t_{a,b}}=\1$}   \right\},
		\end{equation}
		where the product is taken with respect to the lexicographic order, is a $\K$-basis for $\AA_{q}$.
		\item  For each $k \in \left\{1, \dotsc , m \right\}$ there is an injective superalgebra homomorphism $\rho_k : S_q(V_n) \to \AA_q$ given by $v_b \mapsto t_{k,b}$.  Moreover the set given in \cref{L:AlternatePresentationofSq} forms a homogenous basis for $S_{q}(V_{n})$.
		\item  There is a  parity preserving vector space isomorphism $\bar{\rho} : S_q(V_n)^{\otimes m} \to \AA_q$ given by $x_1 \otimes \cdots \otimes x_m \mapsto \rho_1(x_1) \cdots \rho_m(x_m)$.
			\end{enumerate}
	
\end{proposition}

\begin{proof}  Statement (a.) is immediate from the fact the defining relations are homogeneous when the generators are declared to have integer degree one. Statement (b.) follows from a straightforward calculation using \cref{E:AltAqRel2d}.    
	
We prove (c.) and (d.) simultaneously.  First, an argument by induction on the $\Z$-degree using the defining relations along with (b.) shows the set given in \cref{E:Aqbasis} spans $\AA _{q}$.  By \cite[Proposition 3.3]{ChangWang} there is a surjective map of $\Z$-graded superalgebras $\eta:\AA_{q} \to \AA_{q}(\fq_{m}, \fq_{n})$ given by $t_{a,b}\mapsto (-q)^{\operatorname{abs}(a)}(-1)^{\p{a}+\p{b}}t'_{a,b}$, where $\operatorname{abs}(a) = a$ when $a > 0$, and $-a$ when $a < 0$.
From \cref{E:HoweDualityDecomp}, the $U_q(\q_n)$-supermodule $\AA_q(\q_m, \q_n)_d$ consisting of all vectors of $\AA_q(\q_m,\q_n)$ with $\Z$-degree $d$ is the direct sum of degree $d$ polynomial representations of $U_q(\q_n)$.    From \cref{T:Grantcharovtheorem} the dimension $\AA_q(\q_m,\q_n)_{d}$ is the same as its $q = 1$ analogue (c.f.\   \cite[Theorem 3.1]{ChengWangDuality}). That is, $\dim_{\K} \AA_q(\q_m,\q_n)_d = \dim_{\C} S^{d}\left(V_{m}^* \star V_{n} \right)$ which, in turn, equals the number of elements in $\B$ which have $\Z$-degree $d$.  Since the map $\eta$ is surjective and preserves $\Z$-degree, this implies the elements of $\Z$-degree $d$ in $\B$ are linearly independent, so $\B$ is a basis for $\AA_q$.  This proves (d.) and also shows $\eta$ is a superalgebra isomorphism, proving (c.).  
	
To verify (e.), compare the relations from \cref{L:AlternatePresentationofSq} and \cref{E:AqRel2} to verify that $\rho_k$ is a well-defined superalgebra homomorphism.  Since the distinct monomials in the spanning set given in \cref{L:AlternatePresentationofSq} map to distinct elements of the basis $\B$, which are linearly independent, it follows that $\rho_k$ is injective and that the monomials in the set \cref{L:AlternatePresentationofSq} are actually a basis of $S_q(V_n)$ as claimed in \cref{R:Sqproperties}.

Finally, (f.) follows from the observation the linear map $\bar{\rho}$ defines a bijection between a basis for $S_{q}(V_{n})^{\otimes m}$ and $\B$. 
\end{proof}

Transporting the action of $U_q(\q_m) \otimes U_q(\q_n)$ on $\AA_q(\q_m,\q_n)$ from \cref{E:CWAction} through the isomorphism given in \cref{P:basicpropertiesofAq}(c.) gives a $U_{q}(\fq_{m}) \otimes U_{q}(\fq_{n})$-supermodule structure to $\AA_{q}$.  Twisting this action via the algebra automorphism $\tau \otimes 1$ yields a new $U_{q}(\fq_{m}) \otimes U_{q}(\fq_{n})$-supermodule structure on $\AA_{q}.$ \textit{From here forward, we use this $\tau$-twisted structure}.  

Combining \cite[Theorem 4.2]{ChangWang} with \cref{P:dualtwist} yields the following result.
\begin{theorem}\label{P:decomposition} Let $r = \operatorname{min}(m,n)$. There is a multiplicity-free decomposition of $U_q(\q_m) \otimes U_q(\q_n)$-supermodules:
	$$
\AA_q \cong \bigoplus_{\lambda \in \Lambda_r^+} L_m(\lambda) \star L_n(\lambda).
$$
\end{theorem}
 In what follows we identify $U_q(\q_m)$ and $U_q(\q_n)$ with the subsuperalgebras $U_q(\q_m) \otimes 1$ and $1 \otimes U_q(\q_n)$ of $U_q(\q_m) \otimes U_q(\q_n)$.   Via this identification, we consider $\AA_q$ as a supermodule for both $U_q(\q_m)$ and $U_q(\q_n)$, where the actions commute in the graded sense.  
 In order to avoid confusing elements of $U_q(\q_m)$ and $U_q(\q_n$), we decorate the generators of $U_q(\q_n)$ with primes (e.g. $E_r', F_r', \dotsc$).  
  
The action of the generators of $U_{q}(\q_m)$ on the generators of $\AA_{q}$ is given below.  For indices $1 \leq r < m$, $1 \leq s \leq m$, and $(a,b) \in I_{m|m} \times I_{n|n}$, 
\begin{align*}\label{E:qmaction}
	E_r  \cdot t_{a,b}&= \delta_{a,r+1} t_{r,b} + \delta_{a,-(r+1)} t_{-r, b}, &
	\bar{E}_r \cdot t_{a,b} &= \sqrt{-1}(-1)^{\p{a} + \p{b}} \left( \delta_{a, r+1} t_{-r, b} + \delta_{a, -(r+1)} t_{r, b} \right), \\
	F_r  \cdot t_{a,b} &= \delta_{a,r} t_{r+1, b} + \delta_{a, -r} t_{-(r+1), b},&
	\bar{F}_r \cdot t_{a,b} &= \sqrt{-1} (-1)^{\p{a} + \p{b}} \left( \delta_{a, r} t_{-(r+1), b} + \delta_{a, -r} t_{r+1, b} \right), \\
	K_s  \cdot t_{a,b} &= q^{(\varepsilon_{s}, \varepsilon_{a})} t_{a,b}, & 
	 \bar{K}_s  \cdot t_{a,b} &=  \sqrt{-1} (-1)^{\p{a} + \p{b}} \left( \delta_{a,s} t_{-s,b} + \delta_{a,-s} t_{s,b} \right).
\end{align*}
The formulas above can be calculated as follows.  For $E_r, F_r, K_s$, and $\bar{K}_1$, use the values of $S$ and $\tau$ given in \cref{SS:quantizedenvelopingalgebra} to directly calculate the action on $t_{a,b}$.  Next, for every  $b \in I_{n|n}$, there is an injective linear map
$$\xi_b : V_m \to \AA_q, \quad v_a \mapsto  \left( \sqrt{-1} \right) ^{\p{a} + \p{b}} t_{a,b},$$   
Because the elements $E_r, F_r, K_s$, and $\bar{K}_1$ generate $U_q(\q_m)$ as an algebra, their action on $t_{a,b}$ given above can be used to verify that $\xi_b$ is a $U_q(\q_m)$-module homomorphism.  The action of $\bar{E}_r$, $\bar{F}_r$, and $\bar{K}_s$ ($s > 1$) on $t_{a,b}$ can be calculated by pushing their action on $v_a$ across the homomorphism $\xi_b$.  

The action of the generators of $U_q(\q_n)$ on the generators of $\AA_q$ can be calculated directly and is given below.  For all indices $ 1 \leq r \leq n-1$ and $1 \leq s \leq n$, for every $(a,b) \in I_{m|m} \times I_{n|n}$, 
\begin{align*}
	E_r' \cdot t_{a,b} &= \delta_{b,r+1} t_{a,r} + \delta_{b,-(r+1)} t_{a,-r}, 
	& \bar{E}_r'  \cdot t_{a,b} &=  \delta_{b,r+1}  t_{a,-r} + \delta_{b,-(r+1)} t_{a,r},\\
	 F_r '\cdot t_{a,b} &= \delta_{b,r} t_{a,r+1} + \delta_{b, -r} t_{a,-(r+1)}, 
	& \bar F_r '\cdot t_{a,b} &= \delta_{b,r} t_{a,-(r+1)} + \delta_{b, -r} t_{a,r+1},\\
	 K_r' \cdot t_{a,b} &= q^{(\varepsilon_{r}, \varepsilon_{b})} t_{a,b},
	&  \bar K_r ' \cdot t_{a,b} &=  \delta_{b,r} t_{a,-r} + \delta_{b, -r} t_{a,r}. 
\end{align*}

This action on the generators of $\AA_q$ extends to the general case as follows.  If $f$ and $g$ are homogeneous elements of $\AA_q$, given some element $x$ of $U_q(\q_m)$ or $U_q(\q_n)$, the action of $x$ on the product $fg$ is given by
\begin{align*}
	x \cdot fg &= \sum (-1)^{\p{x_{(2)} }\p{f}} (x_{(1)}  \cdot f )\cdot (x_{(2)} \cdot g ) , 
\end{align*}
where we write $\Delta(x)=\sum x_{(1)} \otimes x_{(2)}$ 
using Sweedler notation. In the case where $x \in U_q(\q_n)$, this can be verified by just considering the definition of the action.  For $x \in U_q(\q_m)$, one also needs \cref{E:TauSDelta}.


\begin{lemma}\label{L:weight-space-isomorphism}  Let 
	\[
	\AA_{q} = \bigoplus_{\lambda \in X(T_{m})} \AA_{q, \lambda}\quad 
	\] be the decomposition into weight spaces with respect to the Cartan subalgebra of $U_{q}(\fq_{m})$.  Then $\AA_{q,\lambda} \neq 0$ if and only if $\lambda \in X(T_{m})_{\geq 0}$.

	Furthermore, if $\lambda = \sum_{i=1}^{m}\lambda_{i}\varepsilon_{i} \in X(T_{m})_{\geq 0}$, then as a $U_{q}(\fq_{n})$-supermodule  
	\begin{equation}\label{E:symmisom}
	\AA_{q,\lambda} \cong S_{q}^{\lambda_1}(V_{n})\otimes\cdots\otimes S_{q}^{\lambda_m}(V_{n}),
	\end{equation}
\end{lemma}

\begin{proof}  The fact $\AA_{q,\lambda} \neq 0$ if and only if $\lambda \in X(T_{m})_{\geq 0}$ is immediate from the fact $\B$ is a basis of homogeneous weight vectors for $\AA_{q}$ as a $U_{q}(\fq_{m})$-supermodule. From \cref{P:basicpropertiesofAq}(f.) there is a linear isomorphism $\bar{\rho} : S_q(V_n)^{\otimes m} \overset{\sim}{\to} \AA_q$.  It is straightforward to check this is an isomorphism of $U_q(\q_n)$-supermodules.   By decomposing each $S_q(V_n)$ as a direct sum of symmetric powers, we have a $U_q(\q_n)$-supermodule decomposition
	$$
	\AA_q \cong \bigoplus_{\lambda \in X(T)_{\geq 0}} S^{\lambda_1}_q(V_n) \otimes \cdots \otimes S_q^{\lambda_m}(V_n). 
	$$
Furthermore, under the isomorphism $\bar{\rho}$ the image of $S^{\lambda_1}_q(V_n) \otimes \cdots \otimes S_q^{\lambda_m}(V_n)$ is precisely the $\lambda$-weight space for the action of $U_q(\q_m)$ on $\AA_q$.
\end{proof}

Let us consider an example which will play a role in what follows.  For any fixed $m \geq 1$, consider the weight $\lambda  = \varepsilon_{1}+\dotsb +\varepsilon_{m}$ for $U_{q}(\fq_{m})$.  Using \cref{R:Sqproperties} we identify $V_{n}$ with $S^1_{q}(V_{n})$ and, hence, $V_n^{\otimes m}  \subset S_q(V_n)^{\otimes m}$.  The restriction of $\bar{\rho}$ from \cref{P:basicpropertiesofAq}(f.)~then defines a $U_q(\q_n)$-homomorphism  $V_n^{\otimes m} \to \AA_{q}$  given by
$$
v_{b_1} \otimes v_{b_2} \otimes \dotsb \otimes v_{b_m} \mapsto t_{1,b_1}t_{2,b_1} \dotsb t_{m,b_m}.$$
A direct calculation using the fact $\B$ is a basis of weight vectors for $\AA_{q}$ shows the image of this homomorphism is precisely the $\lambda$-weight space with respect to the action of $U_{q}(\fq_{m})$.  That is, for every $m \geq 1$ the above map defines a parity preserving isomorphism of $U_q(\q_n)$-supermodules 
\begin{equation}\label{E:tensorspaceisomorphism}
V^{\otimes m}_n \overset{\sim}{\to} \AA_{q, \varepsilon_{1}+\dotsb +\varepsilon_{m}}.
\end{equation}

We end this section by recording a lemma which will be used later.  The proof is a straightforward calculation using the action of $U_{q}(\fq_{m})$ on $\AA_{q}$.

\begin{lemma}\label{L:Aqcalculation}  Fix $m,n$, fix $1 \leq r < m$, and let $x_{1}, \dotsc , x_{b} \in I_{n|n}$.  Then for any $1 \leq a \leq b$ the action of $E_{r}^{(a)} = \frac{E_{r}^{a}}{[a]_{q}!} \in U_{q}(\fq_{m})$ is given by 
\[
E_{r}^{(a)} \cdot \prod_{k=1}^{b} t_{r+1,x_{k}} = \sum_{} q^{\gamma(i_{1}, \dotsc , i_{b})} t_{i_{1}, x_{1}}t_{i_{2}, x_{2}}\dotsb t_{i_{b},x_{b}},
\] where the sum is over all tuples $(i_{1}, \dotsc , i_{b}) \in \left\{r, r+1 \right\}^{b}$ consisting of precisely $a$ $r$'s and $b-a$ $r+1$'s.  The function $\gamma : \left\{r, r+1 \right\}^{b} \to \Z$ is given by 
\begin{equation*}
\gamma (i_{1}, \dotsc , i_{b}) = |\left\{(p,q) \mid 1 \leq p < q \leq b, i_{p}=r, i_{q}=r+1 \right\}|.
\end{equation*}
\end{lemma}

\section{The supercategory  \texorpdfstring{$\bUdot(\fq_{m})$}{Udot} }\label{S:IdempotentAlgebra}

\subsection{The supercategory  \texorpdfstring{$\bUdot(\fq_{m})$}{Udot}}\label{SS:IdempotentAlgebra}
We now introduce  a $\K$-linear supercategory $\bUdot(\fq_{m})$ via generators and relations. When writing compositions of morphisms we often write them as products (e.g.\ $fg$ for $f \circ g$).   To lighten notation we use the same name for morphisms between different objects and leave the objects implicit when there is no risk of confusion.  As an application of this convention we can define the divided powers of the even generating morphisms by  $E_{i}^{(r)} := E_{i}^{r}/[r]_{q}!$ and  $F_{i}^{(r)} := F_{i}^{r}/[r]_{q}!$ for $r \geq 1$.  When a statement requires we specify the domain and/or range, we do so by pre- and/or post-composing with the relevant identity morphisms. In what follows we use the notation established in \cref{SS:Liesuperalgebras}.

\begin{definition}\label{D:dotU2}
	
	Let $\bUdot(\fq_{m})$ be the $\K$-linear supercategory with set of objects $X(T_{m})$ and with the morphisms generated by $E_i, \bar{E}_{i}: \lambda \to \lambda+\alpha_{i}$, $F_{i}, \bar{F}_{i}: \lambda \to \lambda-\alpha_{i}$, and $\bar{K}_{j}: \lambda \to \lambda$ for all  $\lambda\in X(T_{m})$, $i=1, \dotsc, m-1$, $j=1, \dotsc , m$.  For all $i,j$ the parities of $E_{i}$ and $F_{i}$ are even and the parities of $\bar{E}_{i}$, $\bar{F}_{i}$, and $\bar{K}_{j}$ are odd. We write $1_{\lambda}: \lambda \to \lambda$ for the identity morphism. 
	
	The morphisms in $\bUdot(\fq_{m})$ are subject to the following relations for all objects $\lambda$ and all admissible $i,j$:

	\begin{align*}
		&(E_i F_j - F_j E_i)1_{\lambda} =  \delta_{i,j} \frac{q^{(\lambda, \alpha_{i})} - q^{-(\lambda, \alpha_{i})}}{q - q^{-1}}1_{\lambda}, \\ 
		%
		%
		&E_i E_j - E_j E_i = F_i F_j - F_j F_i = 0 \quad \text{if $|i-j| > 1$},\\
		%
		%
		%
		%
		&E_i^2 E_j - (q + q^{-1}) E_i E_j E_i + E_j E_i^2 = 0 \quad \text{ if $|i -j| = 1$} ,\\
		%
		%
		%
		&F_i^2 F_j - (q + q^{-1}) F_i F_j F_i + F_j F_i^2 = 0 \quad \text{ if $|i -j| = 1$}, \\
		%
		%
		%
		%
		&(\bar{K}_i \bar{K}_j + \bar{K}_j \bar{K}_i)1_{\lambda} = \delta_{i,j} 2 \frac{q^{2(\lambda, \varepsilon_{i})} - q^{-2(\lambda, \varepsilon_{i})}}{q^2 - q^{-2}}1_{\lambda}, \\
		%
		%
		%
		&(\bar{K}_i E_i - q E_i \bar{K}_i)1_{\lambda} = q^{-(\lambda, \varepsilon_{i})}\bar{E}_i 1_{\lambda} \quad (q\bar{K}_i E_{i-1} - E_{i-1} \bar{K}_i)1_{\lambda} = - q^{-(\lambda+\alpha_{i-1}, \varepsilon_{i})} \bar{E}_{i-1}1_{\lambda}, \\
		%
		%
		&\bar{K}_i E_j - E_j \bar{K}_i = 0 \quad \text{if $j \neq i,i-1$}, \\
		%
		%
		%
		&(\bar{K}_i F_i - q F_i \bar{K}_i)1_{\lambda} = -q^{(\lambda, \varepsilon_{i})}\bar{F}_{i}1_{\lambda} \quad (q\bar{K}_i F_{i-1} - F_{i-1} \bar{K}_i)1_{\lambda} = q^{(\lambda-\alpha_{i-1}, \varepsilon_{i})}  \bar{F}_{i-1}1_{\lambda},\\
		&\bar{K}_i F_j - F_j \bar{K}_i = 0 \quad \text{if $j \neq i,i-1$} ,\\
		%
		%
		%
		&(E_i \bar{F}_j - \bar{F}_j E_i)1_{\lambda} = \delta_{i,j} (q^{-(\lambda, \varepsilon_{i+1})}\bar{K}_i - q^{-(\lambda, \varepsilon_{i})}\bar{K}_{i+1})1_{\lambda}\\
		%
		&(\bar{E}_i F_j - F_j \bar{E}_i)1_{\lambda} = \delta_{i,j} (q^{(\lambda, \varepsilon_{i+1})}\bar{K}_i  - q^{(\lambda, \varepsilon_{i})}\bar{K}_{i+1})1_{\lambda} ,\\  
		%
		& E_i \bar{E}_i - \bar{E}_i E_i = F_i \bar{F}_i - \bar{F}_i F_i = 0 ,\\
		%
		%
		%
		%
		%
		&E_i E_{i+1} - q E_{i+1} E_i  = \bar{E}_i \bar{E}_{i+1} + q\bar{E}_{i+1} \bar{E}_i, \\ 
		&qF_{i+1} F_i - F_i F_{i+1} = \bar{F}_i \bar{F}_{i+1} + q \bar{F}_{i+1} \bar{F}_i ,\\
		&E_i^2 \bar{E}_j - (q+q^{-1}) E_i \bar{E}_j E_i + \bar{E}_j E_i^2 = 0 \quad \text{if $|i - j| = 1$}, \\
		&F_i^2 \bar{F}_j - (q+q^{-1}) F_i \bar{F}_j F_i + \bar{F}_j F_i^2 = 0 \quad \text{if $|i - j| = 1$}.
		%
		%
		%
	\end{align*}
	
\end{definition}

From the supercategory $\bUdot (\fq_{m})$ we can define a locally unital superalgebra,
\[
\Udot(\q_{m})=\bigoplus_{\lambda,\mu\in X(T_{m})} \Hom_{\bUdot(\q_{m})}(\lambda, \mu),
\] with distinguished family of idempotents $\left\{1_{\lambda}  \mid \lambda \in X(T_{m}) \right\}$. Comparing with the presentation given in \cite{GJKK}, the superalgebra $\Udot (\fq_{m})$ is seen to be an idempotent version of the quantized enveloping superalgebra $U_{q}(\fq_{m})$.  

\begin{definition}\label{D:Udotgeq0}   Let $\bUdot(\q_{m})_{\geq 0}$ be  the quotient of  $\bUdot(\q_{m})$ given by setting $1_{\lambda}=0$ for all $\lambda \not\in X(T_{m})_{\geq 0}$.  That is,  $\bUdot(\q_{m})_{\geq 0}=\bUdot (\q_{m})/\mathcal{I}$, where $\mathcal{I}$ is the two-sided tensor ideal generated by $\left\{ 1_{\lambda} \mid \lambda \notin X(T_{m})_{\geq 0} \right\}$. Let $\Udot (\fq_{m})_{\geq 0}$ be the corresponding locally unital algebra.
\end{definition}

 If $M$ is a $U_{q}(\fq_{m})$-supermodule which has a weight space decomposition $M = \oplus_{\lambda \in X(T_{m})} M_{\lambda}$, then $M$ is naturally a $\Udot(\q_{m})$-supermodule. The action of an idempotent is given by declaring $1_{\lambda}m = \delta_{\lambda, \mu}m$ for all $\mu \in X(T_{m})$ and all $m \in M_{\mu}$.   The action of $E_{i}1_{\lambda}, \bar{E}_{i}1_{\lambda}, F_{i}1_{\lambda}, \bar{F}_{i}1_{\lambda}, \bar{K}_{i}1_{\lambda}$ on $M_{\lambda}$ is given by the action $E_{i}, \bar{E}_{i}, F_{i}, \bar{F}_{i}, \bar{K}_{i} \in U_{q}(\fq_{m})$, respectively, and by zero on $M_{\mu}$ for $\mu \neq \lambda$.  Furthermore,  if  $M = \oplus_{\lambda \in X(T_{m})} M_{\lambda}$ is a supermodule for the idempotented superalgebra $\Udot(\q_{m})$, then there is a functor of monoidal supercategories $F:\bUdot(\q_{m}) \to \svec$ defined on objects by $F(\lambda) = M_{\lambda}$.  These three concepts are known to be equivalent (for example  see \cite[Section 2.4]{BrownKujawa} for further details in a similar setting).  Going forward we freely use whichever is convenient.

\subsection{A functor from the supercategory  \texorpdfstring{$\bUdot(\fq_{m})$}{Udot} to \texorpdfstring{$U_{q}(\fq_{n})$-supermodules}{Udot}}\label{SS:superfunctor}

Define $\mods$ to be the monoidal supercategory of $U_{q}(\fq_{n})$-supermodules tensor generated by $\left\{ S_{q}^{d}(V_{n})\mid d \geq 0 \right\}$.  This is the full subcategory of $U_{q}(\fq_{n})$-supermodules consisting of objects of the form 
\[
S_{q}^{d_{1}}(V_{n}) \otimes \dotsb \otimes S_{q}^{d_{t}}(V_{n}),
\] where $t, d_{1}, \dotsc , d_{t} \in  \Z_{\geq 0}$. By convention both the empty tensor product and $S_{q}^{0}(V_{n})$ are the trivial supermodule, $\K$.

For every $m,n \geq 1$,  $\AA_{q} = \AA_{q}(V_{m} \star V_{n})$ is a weight supermodule for $U_{q}(\fq_{m})$ with weights lying in $X(T_{m})_{\geq 0}$ and, hence, defines a representation of $\Udoti(\fq_{m})_{\geq 0}$.  Since the actions of $U_{q}(\fq_{m})$ and $U_{q}(\fq_{n})$ mutually commute it follows the action of any element of $\Udoti(\fq_{m})_{\geq 0}$ is by a $U_{q}(\fq_{n})$-linear map.  Write 
\[
\phi : \Udot(\fq_{m})_{\geq 0} \to \End_{U_{q}\left(\fq_{n} \right)}\left(\AA_{q}\right)
\] for this representation.  As discussed in the previous section, having such a representation is equivalent to having a functor $\bUdot (\fq_{m})_{\geq 0} \to \svec$.  Since the action of $\Udot (\fq_{m})_{\geq 0}$ on $\AA_{q}$ commutes with the action of $U_{q}(\fq_{n})$, this implies the functor can be viewed as having codomain the category of $U_{q}(\fq_{n})$-supermodules.  The existence of this functor is summarized in the following result.

\begin{proposition}\label{Phi-up}
	For every $m,n \geq 1$ there exists a functor of supercategories 
	\begin{equation*}
		\Phi_{m,n}: \bUdot(\fqt_{m})_{\geq 0} \to \mods.
	\end{equation*}
	On objects, 
	\[
	\Phi_{m,n}(\lambda) = \AA_{q,\lambda} \cong S_{q}^{\lambda_{1}}(V_{n}) \otimes \dotsb \otimes S_{q}^{\lambda_{m}}(V_{n}),
	\]
	for all $\lambda = \sum_{i=1}^{m}\lambda_{i}\varepsilon_{i} \in X(T_{m})_{\geq 0}$.  On a morphism $x \in \Hom_{\bUdot (\fq_{m})}(\lambda, \mu)$,
	\[
	\Phi_{m,n} (x) = \phi(x).
	\]
\end{proposition}

The following stability result lets us be flexible in choosing our domain category.

\begin{remark}\label{R:Compatibility}  Let $m', m$ be positive integers with $m' \geq m$.  Given an element $\lambda = (\lambda_{1}, \dotsc , \lambda_{m}) \in X(T_{m})$ we view $\lambda  \in X(T_{m'})$ by extending by $m'-m$ zeros:  $\lambda= (\lambda_{1}, \dotsc , \lambda_{m}, 0, \dotsc , 0)$.  With this convention there is a functor of supercategories,
	\[
	\Theta_{m,m'}: \bUdot(\fq_{m}) \to \bUdot(\fq_{m'}),
	\]
	given by sending the objects and generating morphisms of $\bUdoti(\fq_{m})$ to the objects and morphisms of the same name in $\bUdoti(\fqt_{m'})$.  This induces a functor which we call by the same name,
	\[
	\Theta_{m,m'}: \bUdot(\fq_{m})_{\geq 0} \to \bUdot(\fq_{m'})_{\geq 0}.
	\]  For any fixed $n\geq 1$, the functors $\Phi_{m,n}$ and $\Phi_{m',n}$ are compatible in the sense $\Phi_{m',n}\circ \Theta_{m,m'}$ and $\Phi_{m,n}$ are canonically isomorphic. 

More generally for every $0 \leq a \leq m'-m$ we have functors $\Theta_{m+a,m'}:\bUdot(\fq_{m}) \to \bUdot(\fq_{m'})$ which send $\lambda = (\lambda_{1}, \dotsc , \lambda_{m})$ to $(0,\dotsc ,0, \lambda_{1}, \dotsc , \lambda_{m}, 0, \dotsc , 0)$ where there are $a$ zeros before $\lambda_{1}$ and the generators of $\bUdot (\fq_{m})$ are sent to the generators of the same name in $\bUdot (\fq_{m'})$ with an $a$ added to the subscript.  These functors also have an obvious compatiblity with the functors $\Phi_{m,n}$ and $\Phi_{m',n}$. 
\end{remark}

\section{Upward Webs}\label{S:Upward Webs of Type Q}

\subsection{String calculus for monoidal supercategories}\label{SS:stringcalculus}
Going forward we will make extensive use of the diagrammatic calculus for monoidal supercategories described in \cite{BrundanEllis} and used in, for example, \cite{BrownKujawa}.  Other than the appearance of odd morphisms it is very similar to the well-known diagrammatic calculus for monoidal categories (e.g. see \cite{Turaev}).  We only briefly describe it here and refer the reader to \cite{BrundanEllis} for details. 

A morphism $f:\ob a\to \ob b$ is drawn as
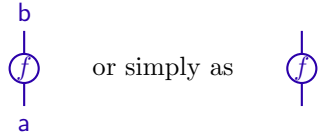
\begin{equation*}
    \begin{tikzpicture}[baseline = 12pt,scale=0.5,color=\clr,inner sep=0pt, minimum width=11pt]
        \draw[-,thick] (0,0) to (0,2);
        \draw (0,1) node[circle,draw,thick,fill=white]{$f$};
        \draw (0,-0.5) node{$\ob a$};
        \draw (0, 2.5) node{$\ob b$};
    \end{tikzpicture}
    \qquad\text{or simply as}\qquad
    \begin{tikzpicture}[baseline = 12pt,scale=0.5,color=\clr,inner sep=0pt, minimum width=11pt]
        \draw[-,thick] (0,0) to (0,2);
        \draw (0,1) node[circle,draw,thick,fill=white]{$f$};
    \end{tikzpicture}
\end{equation*}
when the objects are left implicit.  The convention used in this paper is to read diagrams from bottom to top.  On morphisms $f\otimes g$ and $f\circ g$ are given by horizontal and vertical concatenation, respectively:
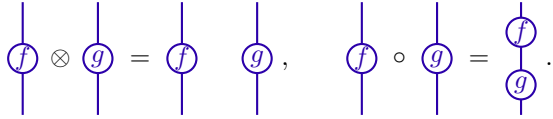
\begin{equation*}
    \begin{tikzpicture}[baseline = 19pt,scale=0.5,color=\clr,inner sep=0pt, minimum width=11pt]
        \draw[-,thick] (0,0) to (0,3);
        \draw[-,thick] (2,0) to (2,3);
        \draw (1,1.5) node[color=black]{$\otimes$};
        \draw (0,1.5) node[circle,draw,thick,fill=white]{$f$};
        \draw (2,1.5) node[circle,draw,thick,fill=white]{$g$};
    \end{tikzpicture}
    ~=~
    \begin{tikzpicture}[baseline = 19pt,scale=0.5,color=\clr,inner sep=0pt, minimum width=11pt]
        \draw[-,thick] (0,0) to (0,3);
        \draw[-,thick] (2,0) to (2,3);
        \draw (0,1.5) node[circle,draw,thick,fill=white]{$f$};
        \draw (2,1.5) node[circle,draw,thick,fill=white]{$g$};
    \end{tikzpicture}
    ~,\qquad
    \begin{tikzpicture}[baseline = 19pt,scale=0.5,color=\clr,inner sep=0pt, minimum width=11pt]
        \draw[-,thick] (0,0) to (0,3);
        \draw[-,thick] (2,0) to (2,3);
        \draw (1,1.5) node[color=black]{$\circ$};
        \draw (0,1.5) node[circle,draw,thick,fill=white]{$f$};
        \draw (2,1.5) node[circle,draw,thick,fill=white]{$g$};
    \end{tikzpicture}
    ~=~
    \begin{tikzpicture}[baseline = 19pt,scale=0.5,color=\clr,inner sep=0pt, minimum width=11pt]
        \draw[-,thick] (0,0) to (0,3);
        \draw (0,2.2) node[circle,draw,thick,fill=white]{$f$};
        \draw (0,0.8) node[circle,draw,thick,fill=white]{$g$};
    \end{tikzpicture}
    ~.
\end{equation*}
Pictures involving multiple products should be interpreted by \emph{first composing horizontally, then composing vertically}. For example, 
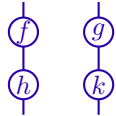
\begin{equation*}
    \begin{tikzpicture}[baseline = 19pt,scale=0.5,color=\clr,inner sep=0pt, minimum width=11pt]
        \draw[-,thick] (0,0) to (0,3);
        \draw[-,thick] (2,0) to (2,3);
        \draw (0,2.2) node[circle,draw,thick,fill=white]{$f$};
        \draw (2,2.2) node[circle,draw,thick,fill=white]{$g$};
        \draw (0,0.8) node[circle,draw,thick,fill=white]{$h$};
        \draw (2,0.8) node[circle,draw,thick,fill=white]{$k$};
    \end{tikzpicture}
\end{equation*}
should be interpreted as $(f\otimes g)\circ(h\otimes k)$. In general, this is \emph{not} the same as $(f\circ h)\otimes(g\circ k)$ because of the \emph{super-interchange law} given in \cref{E:superinterchange}.  Diagrammatically the super-interchange law becomes the identity
\begin{equation}\label{super-interchange}
    \begin{tikzpicture}[baseline = 19pt,scale=0.5,color=\clr,inner sep=0pt, minimum width=11pt]
        \draw[-,thick] (0,0) to (0,3);
        \draw[-,thick] (2,0) to (2,3);
        \draw (0,2) node[circle,draw,thick,fill=white]{$f$};
        \draw (2,1) node[circle,draw,thick,fill=white]{$g$};
    \end{tikzpicture}
    ~=~
    \begin{tikzpicture}[baseline = 19pt,scale=0.5,color=\clr,inner sep=0pt, minimum width=11pt]
        \draw[-,thick] (0,0) to (0,3);
        \draw[-,thick] (2,0) to (2,3);
        \draw (0,1.5) node[circle,draw,thick,fill=white]{$f$};
        \draw (2,1.5) node[circle,draw,thick,fill=white]{$g$};
    \end{tikzpicture}
    ~=(-1)^{\p{f}\p{g}}~
    \begin{tikzpicture}[baseline = 19pt,scale=0.5,color=\clr,inner sep=0pt, minimum width=11pt]
        \draw[-,thick] (0,0) to (0,3);
        \draw[-,thick] (2,0) to (2,3);
        \draw (0,1) node[circle,draw,thick,fill=white]{$f$};
        \draw (2,2) node[circle,draw,thick,fill=white]{$g$};
    \end{tikzpicture}
    ~.
\end{equation}

\subsection{Upward Webs of Type Q}\label{SS:UpwardWebs}  We use the diagrammatic calculus to introduce the following supercategory.  Recall, $\tq = q-q^{-1}$.

\begin{definition}\label{D:UpwardWebs}  Let $\qwebs$ be the $\K$-linear monoidal supercategory generated by the objects $\{\up_{k} \mid k \in \Z_{\geq 0}\}$ and by the generating morphisms
	
	
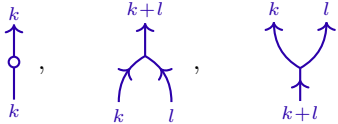
\begin{equation*}
		\xy
		(0,0)*{
			\begin{tikzpicture}[color=\clr, scale=1]
			\draw[color=\clr, thick, directed=1] (1,0) to (1,1);
			\node at (1,-0.15) {\scriptsize $k$};
			\node at (1,1.15) {\scriptsize $k$};
			\draw (1,0.5) \wdot;
			\end{tikzpicture}
		};
		\endxy \ ,\quad\quad
		\xy
		(0,0)*{
			\begin{tikzpicture}[color=\clr, scale=.35]
			\draw [color=\clr,  thick, directed=1] (0, .75) to (0,2);
			\draw [color=\clr,  thick, directed=.65] (1,-1) to [out=90,in=330] (0,.75);
			\draw [color=\clr,  thick, directed=.65] (-1,-1) to [out=90,in=210] (0,.75);
			\node at (0, 2.5) {\scriptsize $k\! +\! l$};
			\node at (-1,-1.5) {\scriptsize $k$};
			\node at (1,-1.5) {\scriptsize $l$};
			\end{tikzpicture}
		};
		\endxy \ ,\quad\quad
		\xy
		(0,0)*{
			\begin{tikzpicture}[color=\clr, scale=.35]
			\draw [color=\clr,  thick, directed=.65] (0,-0.5) to (0,.75);
			\draw [color=\clr,  thick, directed=1] (0,.75) to [out=30,in=270] (1,2.5);
			\draw [color=\clr,  thick, directed=1] (0,.75) to [out=150,in=270] (-1,2.5); 
			\node at (0, -1) {\scriptsize $k\! +\! l$};
			\node at (-1,3) {\scriptsize $k$};
			\node at (1,3) {\scriptsize $l$};
			\end{tikzpicture}
		};
		\endxy
	\end{equation*}
	for $k,l\in\Z_{>0}$.  We call these \emph{dots}, \emph{merges}, and \emph{splits}, respectively. The $\Z_{2}$-grading is given by declaring  merges and splits to have parity $\0$ and dots to have parity $\1$. The morphisms in $\qwebs$ are subject to the relations \cref{associativity,digon-removal,dot-collision,dots-past-merges,dumbbell-relation,square-switch,square-switch-dots,double-rungs-1,double-rungs-2}.	
\end{definition}

Before giving the defining relations for $\qwebs$, let us set some notation and terminology.  Objects in $\qwebs$ are words from the set $\left\{\uparrow_{k} \mid k \geq 0 \right\}$.  We write $\uparrow_{a}^{t}$ for the $t$-fold tensor product of $\uparrow_{a}$ with itself.  The empty word is the unit object and we write $\unit$ for this object.  Given a sequence of integers $a=(a_{1}, \dotsc , a_{r})$, for short we write $\uparrow_{a}$ for $\uparrow_{a_{1}}\dotsb \uparrow_{a_{r}}$.   We frequently use sans serif (e.g.\ $\ob{a}$) to denote an object of $\qwebs$ with the same letter in the ordinary font then denoting the corresponding sequence of integers.  For example, $a = (a_{1}, \dotsc, a_{r})$ is the tuple corresponding to the object $\ob{a}= \uparrow_{a} = \uparrow_{a_{1}}\dotsb \uparrow_{a_{r}}$.  When it will not cause confusion we may abuse notation by writing $\ob{a}$ when referring to the corresponding sequence of integers.
  
As we did with the generating morphisms above, we label the edges of a diagram by nonnegative integers to indicate the relevant object.  For short we sometimes say an edge has \emph{thickness $k$} if it is labelled by $k$.  Since vertical concatenation is composition and horizontal concatenation is tensor product, any finite sequence of these operations applied to merges, splits, and dots yields a diagram which is a morphism in $\qwebs$.  We call such a diagram a \emph{web}.    If $\ob{a}$ and $\ob{b}$ are objects in $\qwebs$, then we say a web is of \emph{type} $\ob{a}\to \ob{b}$ if it is a morphism from $\ob{a}$ to $\ob{b}$  For example, the generating merge depicted above is of type $\up_{k}\up_{l} \to \up_{k+l}$.

We identify any two webs which are related by a planar isotopy which fixes the top and bottom boundaries and all dots.  We adopt the convention an edge labeled by zero is understood to mean the edge is omitted.  We declare any web containing an edge labeled by a negative integer to be the zero morphism. When no confusion is possible we sometimes choose to suppress edge labels.  In particular, edges labelled by $1$ frequently have their labels omitted.

An arbitrary morphism from $\ob{a}$ to $\ob{b}$ is a linear combination of webs of type $\ob{a} \to \ob{b}$.  To write the defining relations for $\qwebs$ it is convenient to set the following shorthand. A \emph{ladder} is a web which is finite sequence of vertical and horizontal concatenations of identities, dots, and webs of the form
\begin{equation}\label{E:rightladder}
\xy
(0,0)*{\reflectbox{
		\bt[color=\clr, scale=.4]
		\draw [ thick, color=\clr,directed=1, directed=0.3 ] (2,-1) to (2,1);
		\draw [ thick, color=\clr,directed=0.6 ] (0,0) to (2,0);
		\draw [ thick, color=\clr,directed=1, directed=0.3 ] (0,-1) to (0,1);
		\node at (0,-1.5) {\reflectbox{\scriptsize $l$}};
		\node at (2,-1.5) {\reflectbox{\scriptsize $k$}};
		\node at (0,1.5) {\reflectbox{\scriptsize $l\! -\! j$}};
		\node at (2,1.5) {\reflectbox{\scriptsize $k\! +\! j$}};
		\node at (0.95,0.55) {\reflectbox{\scriptsize $j$ \ }};
		\et
}};
\endxy:= \ 
\xy
(0,0)*{\reflectbox{
		\bt[color=\clr, scale=.4]
		\draw [ thick, color=\clr,directed=1] (1,.75) to  (1,2);
		\draw [ thick, color=\clr,directed=0.75, looseness=1.25] (2,-1.5) to [out=90,in=0] (1,.75);
		\draw [ thick, color=\clr,directed=.65] (-1,-1.5) to (-1,-.25);
		\draw [ thick, blue, looseness=1.25] (-1,-0.25) to [out=180,in=270] (-2,1.5);
		\draw [ thick, color=\clr,directed=1] (-2,1.5) to (-2,2);
		\draw [ thick, color=\clr,directed=.55, looseness=1.5] (-1,-.25) to [out=0,in=180] (1,.75);
		\node at (-1,-2) {\reflectbox{\scriptsize $l$}};
		\node at (2,-2) {\reflectbox{\scriptsize $k$}};
		\node at (-2,2.5) {\reflectbox{\scriptsize $l\! -\! j$}};
		\node at (1,2.5) {\reflectbox{\scriptsize $k\! +\! j$}};
		\node at (-0.35,.8) {\reflectbox{\scriptsize $j$ \ }};
		\et
}};
\endxy,
\end{equation}
\begin{equation}\label{E:leftladder}
\xy
(0,0)*{
	\bt[color=\clr, scale=.4]
	\draw [ thick, color=\clr,directed=1, directed=0.3 ] (2,-1) to (2,1);
	\draw [ thick, color=\clr,directed=0.6 ] (0,0) to (2,0);
	\draw [ thick, color=\clr,directed=1, directed=0.3 ] (0,-1) to (0,1);
	\node at (0,-1.5) {\scriptsize $k$};
	\node at (2,-1.5) {\scriptsize $l$};
	\node at (0,1.5) {\scriptsize $k\! -\! j$};
	\node at (2,1.5) {\scriptsize $l\! +\! j$};
	\node at (0.95,0.55) {\scriptsize $j$ \ };
	\et
};
\endxy:= \ 
\xy
(0,0)*{
	\bt[color=\clr, scale=.4]
	\draw [ thick, color=\clr,directed=1] (1,.75) to (1,2);
	\draw [ thick, color=\clr,directed=0.75, looseness=1.25] (2,-1.5) to [out=90,in=0] (1,.75);
	\draw [ thick, color=\clr,directed=.65] (-1,-1.5) to (-1,-.25);
	\draw [ thick, blue, looseness=1.25] (-1,-0.25) to [out=180,in=270] (-2,1.5);
	\draw [ thick, color=\clr,directed=1] (-2,1.5) to (-2,2);
	\draw [ thick, color=\clr,directed=.55, looseness=1.5] (-1,-.25) to [out=0,in=180] (1,.75);
	\node at (-1,-2) {\scriptsize $k$};
	\node at (2,-2) {\scriptsize $l$};
	\node at (-2,2.5) {\scriptsize $k\! -\! j$};
	\node at (1,2.5) {\scriptsize $l\! +\! j$};
	\node at (-0.2,.8) {\scriptsize $j$ \ };
	\et
};
\endxy \ ,
\end{equation}
for $k,l\in\Z_{>0}$ and $j\in\Z_{\geq0}$. The edge which connects two vertical strands is called a \emph{rung}.
By choosing a suitable $k$, $l$, and $j$ every merge and split is itself a ladder.  A rung with a dot denotes the same web but with a dot placed on the strand labelled by $j$ in the above diagrams.

The morphisms in $\qwebs$ are subject to the following relations  for all nonnegative integers $h,k,l$.  We also impose the relations obtained by reflecting the webs in \cref{dots-past-merges} across a vertical axis while also exchanging $q$ and $q^{-1}$, and by reversing all rung orientations of the ladders in \cref{double-rungs-1} and \cref{double-rungs-2} (where edge labels at the top of the diagram are changed as needed to make a valid web).  


\input{UpwardQuantumWebRelations}

As a shorthand we declare the following diagram of type $\uparrow_{1}^{k} \to \uparrow_{k}$ to be understood as the vertical composition of $k$ merges:  
\begin{equation}\label{E:explosion}
\xy
(0,0)*{
	\bt[color=\clr, scale=.5]
	\draw [ thick, directed=1] (4,9) to (4,10);
	\draw [ thick, directed=0.75] (5,7.25) to [out=90,in=330] (4,9);
	\draw [ thick, directed=0.75] (3,7.25) to [out=90,in=210] (4,9);
	\node at (4.1, 7.75) { $\cdots$};
	\node at (4,10.5) {\scriptsize $k$};
	\node at (3,6.75) {\scriptsize $1$};
	\node at (5,6.75) {\scriptsize $1$};
	\et
};
\endxy \ .
\end{equation}
 By \cref{associativity} the resulting morphism is independent of how this is done.  In particular, it will frequently happen in calculations that several merges will be composed and by \cref{associativity} the result will equal the above diagram.  We use this observation (and its doppelg{\" a}nger for splits) without comment in what follows.

Finally, in later calculations we frequently save work by implicitly making use of the self-isomorphism of $\qwebs$ given on morphisms by reflecting diagrams across a vertical line through the center and exchanging $q$ and $q^{-1}$.

\subsection{Additional Relations}\label{SS:AdditionalRelations}  The defining relations of $\qwebs$ imply the following additional relations.

\begin{lemma}\label{additional-relations}
	The following relations holds in $\qwebs$.
	\begin{enumerate}
		\item 
		\beq\label{2-dots-zero}
		\xy
		(0,0)*{
			\begin{tikzpicture}[color=\clr, scale=.3]
			\draw [ color=\clr, thick, directed=1] (0,.75) to (0,2);
			\draw [ color=\clr, thick, directed=.85] (0,-2.75) to [out=30,in=330] (0,.75);
			\draw [ color=\clr, thick, directed=.85] (0,-2.75) to [out=150,in=210] (0,.75);
			\draw [ color=\clr, thick, directed=.65] (0,-4) to (0,-2.75);
			\node at (0,-4.5) {\scriptsize $2$};
			\node at (0,2.5) {\scriptsize $2$};
			\node at (-1.5,-1) {\scriptsize $1$};
			\node at (1.5,-1) {\scriptsize $1$};
			\draw (-0.85,-1) \wdot;
			\draw (0.85,-1) \wdot;
			\end{tikzpicture}
		};
		\endxy= \tq 
		\xy
		(0,0)*{
			\begin{tikzpicture}[color=\clr, scale=.3]
			\draw [ color=\clr, thick, directed=1] (0,-4.) to (0,2);
			\node at (0,-4.5) {\scriptsize $2$};
			\node at (0,2.5) {\scriptsize $2$};
			\end{tikzpicture}
		};
		\endxy
		\eeq
		\item 		For $k \geq 1$,
		\beq\label{dot-on-k-strand}
		\xy
		(0,0)*{
			\begin{tikzpicture}[color=\clr, scale=.3]
			\draw [ color=\clr, thick, directed=1] (0,.75) to (0,2);
			\draw [ color=\clr, thick, directed=.85] (0,-2.75) to [out=30,in=330] (0,.75);
			\draw [ color=\clr, thick, directed=.85] (0,-2.75) to [out=150,in=210] (0,.75);
			\draw [ color=\clr, thick, directed=.65] (0,-4) to (0,-2.75);
			\node at (0,-4.5) {\scriptsize $k$};
			\node at (0,2.5) {\scriptsize $k$};
			\node at (-1.5,-1) {\scriptsize $1$};
			\node at (2,-1) {\scriptsize $k\!-\!1$};
			\draw (-0.875,-1) \wdot;
			\end{tikzpicture}
		};
		\endxy=
		\xy
		(0,0)*{
			\begin{tikzpicture}[color=\clr, scale=.3]
			\draw [ color=\clr, thick, directed=1] (0,-4.) to (0,2);
			\node at (0,-4.5) {\scriptsize $k$};
			\node at (0,2.5) {\scriptsize $k$};
			\draw (0,-1) \wdot;
			\end{tikzpicture}
		};
		\endxy= 
		\xy
		(0,0)*{
			\begin{tikzpicture}[color=\clr, scale=.3]
			\draw [ color=\clr, thick, directed=1] (0,.75) to (0,2);
			\draw [ color=\clr, thick, directed=.85] (0,-2.75) to [out=30,in=330] (0,.75);
			\draw [ color=\clr, thick, directed=.85] (0,-2.75) to [out=150,in=210] (0,.75);
			\draw [ color=\clr, thick, directed=.65] (0,-4) to (0,-2.75);
			\node at (0,-4.5) {\scriptsize $k$};
			\node at (0,2.5) {\scriptsize $k$};
			\node at (-2,-1) {\scriptsize $k\!-\!1$};
			\node at (1.5,-1) {\scriptsize $1$};
			\draw (0.875,-1) \wdot;
			\end{tikzpicture}
		};
		\endxy.
		\eeq
		\item 
		For $k \geq 1$,
		\beq\label{dot-on-any-strand} 
		\xy
		(0,0)*{
			\begin{tikzpicture}[color=\clr, scale=.3]
			\draw [ color=\clr, thick, directed=1] (0,.75) to (0,2);
			\draw [ color=\clr, thick, directed=.85] (0,-2.75) to [out=30,in=330] (0,.75);
			\draw [ color=\clr, thick, directed=.85] (0,-2.75) to [out=150,in=210] (0,.75);
			\draw [ color=\clr, thick, directed=.65] (0,-4) to (0,-2.75);
			\node at (0.18, -1.25) {$\cdots$};
			\node at (0,-4.5) {\scriptsize $k$};
			\node at (0,2.5) {\scriptsize $k$};
			\node at (-1.5,-1) {\scriptsize $1$};
			\node at (1.5,-1) {\scriptsize $1$};
			\draw (-0.875,-1) \wdot;
			\end{tikzpicture}
		};
		\endxy=
		\xy
		(0,0)*{
			\begin{tikzpicture}[color=\clr, scale=.3]
			\draw [ color=\clr, thick, directed=1] (0,-4.) to (0,2);
			\node at (0,-4.5) {\scriptsize $k$};
			\node at (0,2.5) {\scriptsize $k$};
			\draw (0,-1) \wdot;
			\end{tikzpicture}
		};
		\endxy= 
		\xy
		(0,0)*{
			\begin{tikzpicture}[color=\clr, scale=.3]
			\draw [ color=\clr, thick, directed=1] (0,.75) to (0,2);
			\draw [ color=\clr, thick, directed=.85] (0,-2.75) to [out=30,in=330] (0,.75);
			\draw [ color=\clr, thick, directed=.85] (0,-2.75) to [out=150,in=210] (0,.75);
			\draw [ color=\clr, thick, directed=.65] (0,-4) to (0,-2.75);
			\node at (0.0, -1.25) {$\cdots$};
			\node at (0,-4.5) {\scriptsize $k$};
			\node at (0,2.5) {\scriptsize $k$};
			\node at (-1.5,-1) {\scriptsize $1$};
			\node at (1.5,-1) {\scriptsize $1$};
			\draw (0.875,-1) \wdot;
			\end{tikzpicture}
		};
		\endxy,
		\eeq
		where in \cref{dot-on-any-strand} the left and right diagrams have  $k$ parallel strands labelled by $1$.  Moreover, these diagrams are equal to any of the other, similar diagrams consisting of $k$ parallel strands labelled by $1$ with exactly one of those strands having a dot.
	\end{enumerate}
	
\end{lemma}

\begin{proof}
	The first relation follows by vertically composing the left side of \cref{2-dots-zero} with \cref{digon-removal}, rewriting the result using \cref{dumbbell-relation}, and simplifying to obtain the right hand side.  See \cite[Lemma 4.2.1]{BrownKujawa} for a similar calculation in the classical setting.  To prove \cref{dot-on-k-strand} one argues by inducting on $k$ using \cref{associativity,dots-past-merges}.  The base case of $k=2$ follows from \cref{2-dots-zero}.  As the details are similar to the analogous proof of \cite[Lemma 4.2.1]{BrownKujawa}, we omit them. 
\end{proof}

\begin{lemma}\label{Udot-relations}
	For $h,k,l,r,s \geq 0$, 
	\begin{enumerate}
		\item 
		\begin{equation}\label{rung-collision}
		\xy
		(0,0)*{
			\begin{tikzpicture}[color=\clr]
			\draw [color=\clr, thick, directed=.15, directed=1, directed=.55] (0,0) to (0,1.75);
			\node at (0,-0.15) {\scriptsize $k$};
			\node at (0,1.9) {\scriptsize $k\!+\!r\!+\!s$ \ };
			\draw [color=\clr, thick, directed=.15, directed=1, directed=.55] (1,0) to (1,1.75);
			\node at (1,-0.15) {\scriptsize $l$};
			\node at (1,1.9) {\scriptsize  \ $l\!-\!r\!-\!s$};
			\draw [color=\clr, thick, directed=.55] (1,0.5) to (0,0.5);
			\node at (0.5,0.25) {\scriptsize$r$};
			\node at (-0.4,0.875) {\scriptsize $k\!+\!r$};
			\draw [color=\clr, thick, directed=.55] (1,1.25) to (0,1.25);
			\node at (0.5,1.5) {\scriptsize$s$};
			\node at (1.4,0.875) {\scriptsize $l\!-\!r$};
			\end{tikzpicture}
		};
		\endxy=\qbinom{r+s}{s}
		\xy
		(0,0)*{
			\begin{tikzpicture}[color=\clr]
			\draw [color=\clr, thick, directed=.15, directed=1] (0,0) to (0,1.75);
			\node at (0,-0.15) {\scriptsize $k$};
			\node at (0,1.9) {\scriptsize $k\!+\!r\!+\!s$ \ };
			\draw [color=\clr, thick, directed=.15, directed=1] (1,0) to (1,1.75);
			\node at (1,-0.15) {\scriptsize $l$};
			\node at (1,1.9) {\scriptsize  \ $l\!-\!r\!-\!s$};
			\draw [color=\clr, thick, directed=.55] (1,0.75) to (0,0.75);
			\node at (0.5,1) {\scriptsize$r\!+\!s$};
			\end{tikzpicture}
		};
		\endxy,
		\end{equation}
		\item 
		\begin{equation}\label{square-switch-double-dots}
		\xy
		(0,0)*{
			\begin{tikzpicture}[color=\clr]
			\draw [color=\clr, thick, directed=.15, directed=1, directed=.55] (0,0) to (0,1.75);
			\node at (0,-0.15) {\scriptsize $k$};
			\node at (0,1.9) {\scriptsize $k$};
			\draw [color=\clr, thick, directed=.15, directed=1, directed=.55] (1,0) to (1,1.75);
			\node at (1,-0.15) {\scriptsize $l$};
			\node at (1,1.9) {\scriptsize $l$};
			\draw [color=\clr, thick, directed=.55] (0,0.5) to (1,0.5);
			\node at (0.5,0.25) {\scriptsize$1$};
			\node at (-0.4,0.875) {\scriptsize $k\!-\!1$};
			\draw [color=\clr, thick, directed=.55] (1,1.25) to (0,1.25);
			\node at (0.5,1.5) {\scriptsize$1$};
			\node at (1.4,0.875) {\scriptsize $l\!+\!1$};
			\draw  (0.75,0.5) \wdot;
			\draw  (0.25,1.25) \wdot;
			\end{tikzpicture}
		};
		\endxy+
		\xy
		(0,0)*{
			\begin{tikzpicture}[color=\clr]
			\draw [color=\clr, thick, directed=.15, directed=1, directed=.55] (0,0) to (0,1.75);
			\node at (0,-0.15) {\scriptsize $k$};
			\node at (0,1.9) {\scriptsize $k$};
			\draw [color=\clr, thick, directed=.15, directed=1, directed=.55] (1,0) to (1,1.75);
			\node at (1,-0.15) {\scriptsize $l$};
			\node at (1,1.9) {\scriptsize $l$};
			\draw [color=\clr, thick, directed=.55] (1,0.5) to (0,0.5);
			\node at (0.5,0.25) {\scriptsize$1$};
			\node at (-0.4,0.875) {\scriptsize $k\!+\!1$};
			\draw [color=\clr, thick, directed=.55] (0,1.25) to (1,1.25);
			\node at (0.5,1.5) {\scriptsize$1$};
			\node at (1.4,0.875) {\scriptsize $l\!-\!1$};
			\draw  (0.75,0.5) \wdot;
			\draw  (0.25,1.25) \wdot;
			\end{tikzpicture}
		};
		\endxy= \left[ k+l \right]_{q}
		\xy
		(0,0)*{
			\begin{tikzpicture}[color=\clr, scale=.3] 
			\draw [color=\clr, thick, directed=1] (1,-2.75) to (1,2.5);
			\draw [color=\clr, thick, directed=1] (-1,-2.75) to (-1,2.5);
			\node at (-1,3) {\scriptsize $k$};
			\node at (1,3) {\scriptsize $l$};
			\node at (-1,-3.15) {\scriptsize $k$};
			\node at (1,-3.15) {\scriptsize $l$};
			\end{tikzpicture}
		};
		\endxy \ +\tq  
		\xy
		(0,0)*{
			\begin{tikzpicture}[color=\clr, scale=.3] 
			\draw [color=\clr, thick, directed=1] (1,-2.75) to (1,2.5);
			\draw [color=\clr, thick, directed=1] (-1,-2.75) to (-1,2.5);
			\node at (-1,3) {\scriptsize $k$};
			\node at (1,3) {\scriptsize $l$};
			\node at (-1,-3.15) {\scriptsize $k$};
			\node at (1,-3.15) {\scriptsize $l$};
			\draw  (1,0) \wdot;
			\draw  (-1,0) \wdot;
			\end{tikzpicture}
		};
		\endxy
		\end{equation}
		\item

\beq\label{generalized-square-switch}
\xy
(0,0)*{
\bt[color=\clr]
	\draw[ color=\clr, thick, directed=.15, directed=1] (0,0) to (0,1.75);
	\node at (0,-0.2) {\scriptsize $k$};
	\node at (-0.2,1.9) {\scriptsize $k-r+1$};
	\draw[ color=\clr, thick, directed=.15, directed=1] (1,0) to (1,1.75);
	\node at (1,-0.2) {\scriptsize $l$};
	\node at (1.2,1.9) {\scriptsize $l+r-1$};
	\draw[ color=\clr, thick, directed=.55] (0,0.5) to (1,0.5);
	\node at (0.5,0.25) {\scriptsize $r$};
	\node at (-0.4,0.875) {\scriptsize{$k-r$}};
	\draw[ color=\clr, thick, directed=.55] (1,1.25) to (0,1.25);
	\node at (0.5,1.5) {\scriptsize$1$};
	\node at (1.4,0.875) {\scriptsize{$l+r$}};
\et
};
\endxy-
\xy
(0,0)*{
\bt[color=\clr]
	\draw[ color=\clr, thick, directed=.15, directed=1] (0,0) to (0,1.75);
	\node at (0,-0.2) {\scriptsize $k$};
	\node at (-0.2,1.9) {\scriptsize $k-r+1$};
	\draw[ color=\clr, thick, directed=.15, directed=1] (1,0) to (1,1.75);
	\node at (1,-0.15) {\scriptsize $l$};
	\node at (1.2,1.9) {\scriptsize $l+r-1$};
	\draw[ color=\clr, thick, directed=.55] (1,0.5) to (0,0.5);
	\node at (0.5,0.25) {\scriptsize $1$};
	\node at (-0.4,0.875) {\scriptsize $k\!+\!1$};
	\draw[ color=\clr, thick, directed=.55] (0,1.25) to (1,1.25);
	\node at (0.5,1.5) {\scriptsize $r$};
	\node at (1.4,0.875) {\scriptsize $l\!-\!1$};
\et
};
\endxy=[k-l+1-r]_{q}
\xy
(0,0)*{
\bt[color=\clr]
	\draw[ color=\clr, thick, directed=.15, directed=1] (0,0) to (0,1.75);
	\node at (0,-0.2) {\scriptsize $k$};
	\node at (-0.2,1.9) {\scriptsize $k-r+1$};
	\draw[ color=\clr, thick, directed=.15, directed=1] (1,0) to (1,1.75);
	\node at (1,-0.15) {\scriptsize $l$};
	\node at (1.2,1.9) {\scriptsize $l+r-1$};
	\node at (0.5,0.25) {\scriptsize $1$};
	\draw[ color=\clr, thick, directed=.55] (0,.85) to (1,.85);
	\node at (0.5,1.15) {\scriptsize $r-1$};
	\node at (1.4,0.875) {\scriptsize $l\!-\!1$};
\et
};
\endxy \ ,
\eeq

		\item 
		\begin{equation}\label{two-rung-switch}
		\xy
		(0,0)*{
			\begin{tikzpicture}[color=\clr]
			\draw [color=\clr, thick, directed=.15, directed=1] (-1,0) to (-1,1.75);
			\node at (-1,-0.15) {\scriptsize $h$};
			\node at (-1,1.9) {\scriptsize $h\!+\!2$};
			\draw [color=\clr, thick, directed=.15, directed=1, directed=.55] (0,0) to (0,1.75);
			\node at (0,-0.15) {\scriptsize $k$};
			\node at (0,1.9) {\scriptsize $k\!-\!1$};
			\draw [color=\clr, thick, directed=.15, directed=1] (1,0) to (1,1.75);
			\node at (1,-0.15) {\scriptsize $l$};
			\node at (1,1.9) {\scriptsize $l\!-\!1$};
			\draw [color=\clr, thick, rdirected=.55] (1,0.5) to (0,0.5);
			\node at (0.5,0.25) {\scriptsize$1$};
			\draw [color=\clr, thick, directed=.55] (0,1.25) to (-1,1.25);
			\node at (-0.5,1.5) {\scriptsize$1$};
			\end{tikzpicture}
		};
		\endxy =
		\xy
		(0,0)*{
			\begin{tikzpicture}[color=\clr]
			\draw [color=\clr, thick, directed=.15, directed=1] (-1,0) to (-1,1.75);
			\node at (-1,-0.15) {\scriptsize $h$};
			\node at (-1,1.9) {\scriptsize $h\!+\!2$};
			\draw [color=\clr, thick, directed=.15, directed=1, directed=.55] (0,0) to (0,1.75);
			\node at (0,-0.15) {\scriptsize $k$};
			\node at (0,1.9) {\scriptsize $k\!-\!1$};
			\draw [color=\clr, thick, directed=.15, directed=1] (1,0) to (1,1.75);
			\node at (1,-0.15) {\scriptsize $l$};
			\node at (1,1.9) {\scriptsize $l\!-\!1$};
			\draw [color=\clr, thick, rdirected=.55] (1,1.25) to (0,1.25);
			\node at (0.5,1.5) {\scriptsize$1$};
			\draw [color=\clr, thick, directed=.55] (0,0.5) to (-1,0.5);
			\node at (-0.5,0.25) {\scriptsize$1$};
			\end{tikzpicture}
		};
		\endxy,
		\end{equation}

		\item 
		\begin{equation}\label{two-rung-switch-with-dots}
		\xy
		(0,0)*{
			\begin{tikzpicture}[color=\clr]
			\draw [color=\clr, thick, directed=.15, directed=1] (-1,0) to (-1,1.75);
			\node at (-1,-0.15) {\scriptsize $h$};
			\node at (-1,1.9) {\scriptsize $h\!+\!2$};
			\draw [color=\clr, thick, directed=.15, directed=1, directed=.55] (0,0) to (0,1.75);
			\node at (0,-0.15) {\scriptsize $k$};
			\node at (0,1.9) {\scriptsize $k\!-\!1$};
			\draw [color=\clr, thick, directed=.15, directed=1] (1,0) to (1,1.75);
			\node at (1,-0.15) {\scriptsize $l$};
			\node at (1,1.9) {\scriptsize $l\!-\!1$};
			\draw [color=\clr, thick, rdirected=.25] (1,0.5) to (0,0.5);
			\node at (0.5,0.25) {\scriptsize$1$};
			\draw [color=\clr, thick, directed=.55] (0,1.25) to (-1,1.25);
			\node at (-0.5,1.5) {\scriptsize$1$};
			\draw (0.5,0.5) \wdot; 
			\end{tikzpicture}
		};
		\endxy =
		\xy
		(0,0)*{
			\begin{tikzpicture}[color=\clr]
			\draw [color=\clr, thick, directed=.15, directed=1] (-1,0) to (-1,1.75);
			\node at (-1,-0.15) {\scriptsize $h$};
			\node at (-1,1.9) {\scriptsize $h\!+\!2$};
			\draw [color=\clr, thick, directed=.15, directed=1, directed=.55] (0,0) to (0,1.75);
			\node at (0,-0.15) {\scriptsize $k$};
			\node at (0,1.9) {\scriptsize $k\!-\!1$};
			\draw [color=\clr, thick, directed=.15, directed=1] (1,0) to (1,1.75);
			\node at (1,-0.15) {\scriptsize $l$};
			\node at (1,1.9) {\scriptsize $l\!-\!1$};
			\draw [color=\clr, thick, rdirected=.25] (1,1.25) to (0,1.25);
			\node at (0.5,1.5) {\scriptsize$1$};
			\draw [color=\clr, thick, directed=.55] (0,0.5) to (-1,0.5);
			\node at (-0.5,0.25) {\scriptsize$1$};
			\draw (0.5,1.25) \wdot; 
			\end{tikzpicture}
		};
		\endxy,
		\end{equation}

		\item 
		\begin{equation}\label{double-rungs-3}
		\xy
		(0,0)*{
			\begin{tikzpicture}[color=\clr]
			\draw [color=\clr, thick, directed=.15, directed=1] (-1,0) to (-1,1.75);
			\node at (-1,-0.15) {\scriptsize $h$};
			\node at (-1,1.9) {\scriptsize $h\!+\!2$};
			\draw [color=\clr, thick, directed=.15, directed=1, directed=.55] (0,0) to (0,1.75);
			\node at (0,-0.15) {\scriptsize $k$};
			\node at (0,1.9) {\scriptsize $k\!-\!1$};
			\draw [color=\clr, thick, directed=.15, directed=1] (1,0) to (1,1.75);
			\node at (1,-0.15) {\scriptsize $l$};
			\node at (1,1.9) {\scriptsize $l\!-\!1$};
			\draw [color=\clr, thick, directed=.55] (1,0.5) to (0,0.5);
			\node at (0.5,0.25) {\scriptsize$1$};
			\draw [color=\clr, thick, directed=.55] (0,1.25) to (-1,1.25);
			\node at (-0.5,1.5) {\scriptsize$2$};
			\end{tikzpicture}
		};
		\endxy-
		\xy
		(0,0)*{
			\begin{tikzpicture}[color=\clr]
			\draw [color=\clr, thick, directed=.15, directed=1, directed=.55] (-1,0) to (-1,1.75);
			\node at (-1,-0.15) {\scriptsize $h$};
			\node at (-1,1.9) {\scriptsize $h\!+\!2$};
			\draw [color=\clr, thick, directed=.15, directed=1] (0,0) to (0,1.75);
			\node at (0,-0.15) {\scriptsize $k$};
			\node at (0,1.9) {\scriptsize $k\!-\!1$};
			\draw [color=\clr, thick, directed=.15, directed=1] (1,0) to (1,1.75);
			\node at (1,-0.15) {\scriptsize $l$};
			\node at (1,1.9) {\scriptsize $l\!-\!1$};
			\draw [color=\clr, thick, directed=.55] (0,0.5) to (-1,0.5);
			\node at (-0.5,0.25) {\scriptsize$1$};
			\draw [color=\clr, thick, directed=.55] (1,0.875) to (0,0.875);
			\node at (0.5,1.125) {\scriptsize$1$};
			\draw [color=\clr, thick, directed=.55] (0,1.25) to (-1,1.25);
			\node at (-0.5,1.5) {\scriptsize$1$};
			\end{tikzpicture}
		};
		\endxy+
		\xy
		(0,0)*{
			\begin{tikzpicture}[color=\clr]
			\draw [color=\clr, thick, directed=.15, directed=1] (-1,0) to (-1,1.75);
			\node at (-1,-0.15) {\scriptsize $h$};
			\node at (-1,1.9) {\scriptsize $h\!+\!2$};
			\draw [color=\clr, thick, directed=.15, directed=1, directed=.55] (0,0) to (0,1.75);
			\node at (0,-0.15) {\scriptsize $k$};
			\node at (0,1.9) {\scriptsize $k\!-\!1$};
			\draw [color=\clr, thick, directed=.15, directed=1] (1,0) to (1,1.75);
			\node at (1,-0.15) {\scriptsize $l$};
			\node at (1,1.9) {\scriptsize $l\!-\!1$};
			\draw [color=\clr, thick, directed=.55] (1,1.25) to (0,1.25);
			\node at (0.5,1.5) {\scriptsize$1$};
			\draw [color=\clr, thick, directed=.55] (0,0.5) to (-1,0.5);
			\node at (-0.5,0.25) {\scriptsize$2$};
			\end{tikzpicture}
		};
		\endxy=0,
		\end{equation}

		\item
		\begin{equation}\label{double-rungs-5}
		\xy
		(0,0)*{
			\begin{tikzpicture}[color=\clr]
			\draw [color=\clr, thick, directed=.15, directed=1] (-1,0) to (-1,1.75);
			\node at (-1,-0.15) {\scriptsize $h$};
			\node at (-1,1.9) {\scriptsize $h\!+\!2$};
			\draw [color=\clr, thick, directed=.15, directed=1, directed=.55] (0,0) to (0,1.75);
			\node at (0,-0.15) {\scriptsize $k$};
			\node at (0,1.9) {\scriptsize $k\!-\!1$};
			\draw [color=\clr, thick, directed=.15, directed=1] (1,0) to (1,1.75);
			\node at (1,-0.15) {\scriptsize $l$};
			\node at (1,1.9) {\scriptsize $l\!-\!1$};
			\draw [color=\clr, thick, directed=.85] (1,0.5) to (0,0.5);
			\draw (0.5,0.5) \wdot;
			\node at (0.5,0.25) {\scriptsize$1$};
			\draw [color=\clr, thick, directed=.55] (0,1.25) to (-1,1.25);
			\node at (-0.5,1.5) {\scriptsize$2$};
			\end{tikzpicture}
		};
		\endxy-
		\xy
		(0,0)*{
			\begin{tikzpicture}[color=\clr]
			\draw [color=\clr, thick, directed=.15, directed=1, directed=.55] (-1,0) to (-1,1.75);
			\node at (-1,-0.15) {\scriptsize $h$};
			\node at (-1,1.9) {\scriptsize $h\!+\!2$};
			\draw [color=\clr, thick, directed=.15, directed=1] (0,0) to (0,1.75);
			\node at (0,-0.15) {\scriptsize $k$};
			\node at (0,1.9) {\scriptsize $k\!-\!1$};
			\draw [color=\clr, thick, directed=.15, directed=1] (1,0) to (1,1.75);
			\node at (1,-0.15) {\scriptsize $l$};
			\node at (1,1.9) {\scriptsize $l\!-\!1$};
			\draw [color=\clr, thick, directed=.55] (0,0.5) to (-1,0.5);
			\node at (-0.5,0.25) {\scriptsize$1$};
			\draw [color=\clr, thick, directed=.85] (1,0.875) to (0,0.875);
			\draw (0.5,0.875) \wdot;
			\node at (0.5,1.125) {\scriptsize$1$};
			\draw [color=\clr, thick, directed=.55] (0,1.25) to (-1,1.25);
			\node at (-0.5,1.5) {\scriptsize$1$};
			\end{tikzpicture}
		};
		\endxy+
		\xy
		(0,0)*{
			\begin{tikzpicture}[color=\clr]
			\draw [color=\clr, thick, directed=.15, directed=1] (-1,0) to (-1,1.75);
			\node at (-1,-0.15) {\scriptsize $h$};
			\node at (-1,1.9) {\scriptsize $h\!+\!2$};
			\draw [color=\clr, thick, directed=.15, directed=1, directed=.55] (0,0) to (0,1.75);
			\node at (0,-0.15) {\scriptsize $k$};
			\node at (0,1.9) {\scriptsize $k\!-\!1$};
			\draw [color=\clr, thick, directed=.15, directed=1] (1,0) to (1,1.75);
			\node at (1,-0.15) {\scriptsize $l$};
			\node at (1,1.9) {\scriptsize $l\!-\!1$};
			\draw [color=\clr, thick, directed=.85] (1,1.25) to (0,1.25);
			\draw (0.5,1.25) \wdot ;
			\node at (0.5,1.5) {\scriptsize$1$};
			\draw [color=\clr, thick, directed=.55] (0,0.5) to (-1,0.5);
			\node at (-0.5,0.25) {\scriptsize$2$};
			\end{tikzpicture}
		};
		\endxy=0,
		\end{equation}

	\end{enumerate}
	In addition we have the equations obtained by reversing all rung orientations of the ladders in \cref{rung-collision} and relabelling the tops of diagrams as needed to make a valid web.  We also have the equations obtained from  \cref{double-rungs-3} by reversing all rung orientations, reflecting across the vertical line through the middle arrow, and both reversing and reflecting.  In this way \cref{double-rungs-3}, for example, actually represents four relations. 
\end{lemma}

\begin{proof}
The proof of \cref{rung-collision} follows from \cref{associativity} and \cref{digon-removal}.

 To prove \cref{square-switch-double-dots} we rewrite the first web on the left-hand side of \cref{square-switch-double-dots} using \cref{square-switch} as follows:
\begin{align*}
\mathord{\begin{tikzpicture}[baseline=-1ex,color=\clr, scale=.35]
		\node [style=none] (0) at (-2, 3) {};
		\node [style=none] (1) at (2, 3) {};
		\node [style=none] (2) at (-2, -3) {};
		\node [style=none] (3) at (2, -3) {};
		\node [style=none] (4) at (-2, -2) {};
		\node [style=none] (5) at (0, -2) {};
		\node [style=none] (6) at (0, -1) {};
		\node [style=none] (7) at (2, -1) {};
		\node [style=none] (8) at (2, 1) {};
		\node [style=none] (9) at (0, 1) {};
		\node [style=none] (10) at (0, 1.75) {};
		\node [style=none] (11) at (-2, 1.75) {};
		\node [style=none] (12) at (-1, 1.25) {\scriptsize{$1$}};
		\node [style=none] (13) at (-1, -1.25) {\scriptsize{$1$}};
		\node [style=none] (14) at (-2, -4) {\scriptsize{$k$}};
		\node [style=none] (15) at (2, -4) {\scriptsize{$l$}};
		\node [style=none] (18) at (-3, 0) {\scriptsize{$k-1$}};
		\node [style=none] (19) at (3, 0) {\scriptsize{$l+1$}};
		\draw [thick, <-] (0.center) to (2.center);
		\draw [thick, <-] (1.center) to (3.center);
		\draw [thick, -, directed=.5] (4.center) to (5.center);
		\draw [thick, -] (5.center) to (6.center);
		\draw [thick, -, directed=.5] (6.center) to (7.center);
		\draw [thick, -, directed=.6] (8.center) to (9.center);
		\draw [thick, -] (9.center) to (10.center);
		\draw [thick, -, directed=.5] (10.center) to (11.center);
		\draw  (-0.5, -2) \wdot;
		\draw  (-1.5, 1.75) \wdot;
\end{tikzpicture}}  
& =
\mathord{\begin{tikzpicture}[baseline=-1ex,color=\clr, scale=.35]
		\node [style=none] (0) at (-2, 3) {};
		\node [style=none] (1) at (2, 3) {};
		\node [style=none] (2) at (-2, -3) {};
		\node [style=none] (3) at (2, -3) {};
		\node [style=none] (4) at (-2, -2) {};
		\node [style=none] (5) at (0, -2) {};
		\node [style=none] (6) at (0, -1) {};
		\node [style=none] (7) at (2, -1) {};
		\node [style=none] (8) at (2, 1) {};
		\node [style=none] (9) at (0, 1) {};
		\node [style=none] (10) at (0, 1.75) {};
		\node [style=none] (11) at (-2, 1.75) {};
		\node [style=none] (12) at (-1, 1.25) {\scriptsize{$1$}};
		\node [style=none] (13) at (-1, -1.25) {\scriptsize{$1$}};
		\node [style=none] (14) at (-2, -4) {\scriptsize{$k$}};
		\node [style=none] (15) at (2, -4) {\scriptsize{$l$}};
		\node [style=none] (18) at (-3, 0) {\scriptsize{$k-1$}};
		\node [style=none] (19) at (3, 0) {\scriptsize{$l+1$}};
		\node [style=none] (20) at (-0.5, -0.25) {\scriptsize{$2$}};
		\draw [thick, <-] (0.center) to (2.center);
		\draw [thick, <-] (1.center) to (3.center);
		\draw [thick, -, directed=.5] (4.center) to (5.center);
		\draw [thick, -] (5.center) to (6.center);
		\draw [thick, -, rdirected=.5] (6.center) to (7.center);
		\draw [thick, -, rdirected=.6] (8.center) to (9.center);
		\draw [thick, -] (9.center) to (10.center);
		\draw [thick, -, directed=.5] (10.center) to (11.center);
		\draw [thick, -, directed=.75] (6.center) to (9.center);
		\draw  (-0.5, -2) \wdot;
		\draw  (-1.5, 1.75) \wdot;
\end{tikzpicture}}  
+  [1-l]_{q}[k]_{q}
\mathord{\begin{tikzpicture}[baseline=-1ex,color=\clr, scale=.35]
		\node [style=none] (0) at (-1, 3) {};
		\node [style=none] (3) at (-1, -3) {};
		\node [style=none] (4) at (1, 3) {};
		\node [style=none] (5) at (1.5, -2) {};
		\node [style=none] (6) at (1.5, 2) {};
		\node [style=none] (7) at (1, -3) {};
		\node [style=none] (14) at (-1, -4) {\scriptsize{$k$}};
		\node [style=none] (15) at (1, -4) {\scriptsize{$l$}};
		\draw [thick, <-] (0.center) to (3.center);
		\draw [thick,->] (7.center) to (4.center);
\end{tikzpicture}} \\
&= 
\mathord{\begin{tikzpicture}[baseline=-1ex,color=\clr, scale=.35]
		\node [style=none] (0) at (-2, 3) {};
		\node [style=none] (1) at (2, 3) {};
		\node [style=none] (2) at (-2, -3) {};
		\node [style=none] (3) at (2, -3) {};
		\node [style=none] (4) at (2, -2) {};
		\node [style=none] (5) at (0, -2) {};
		\node [style=none] (6) at (0, -0.6) {};
		\node [style=none] (7) at (-2, -2) {};
		\node [style=none] (8) at (-2, 1.75) {};
		\node [style=none] (9) at (0, 0.6) {};
		\node [style=none] (10) at (0, 1.75) {};
		\node [style=none] (11) at (2, 1.75) {};
		\node [style=none] (12) at (-1, 2) {\scriptsize{$1$}};
		\node [style=none] (13) at (-1, -2) {\scriptsize{$1$}};
		\node [style=none] (12) at (1, 2) {\scriptsize{$1$}};
		\node [style=none] (13) at (1, -2) {\scriptsize{$1$}};
		\node [style=none] (14) at (-2, -4) {\scriptsize{$k$}};
		\node [style=none] (15) at (2, -4) {\scriptsize{$l$}};
		\node [style=none] (18) at (-3, 0) {\scriptsize{$k-1$}};
		\node [style=none] (19) at (3, 0) {\scriptsize{$l-1$}};
		\node [style=none] (20) at (0.5, -0.25) {\scriptsize{$2$}};
		\draw [thick, <-] (0.center) to (2.center);
		\draw [thick, <-] (1.center) to (3.center);
		\draw [thick, -, directed=.6, looseness=1.25] (4.center) to [out=135, in=315] (6.center);
		\draw [thick, -, rdirected=.75, looseness=1.25] (6.center) to [out=225, in=45] (7.center);
		\draw [thick, -, rdirected=.35, looseness=1.25] (8.center) to [out=315,in=135] (9.center);
		\draw [thick, -, directed=.6, looseness=1.25] (9.center) to [out=45,in=225] (11.center);
		\draw [thick, -, directed=.75] (6.center) to (9.center);
		\draw  (-0.5, 1) \wdot;
		\draw  (-0.5, -1) \wdot;
\end{tikzpicture}}
+  [1-l]_{q}[k]_{q}
\mathord{\begin{tikzpicture}[baseline=-1ex,color=\clr, scale=.35]
		\node [style=none] (0) at (-1, 3) {};
		\node [style=none] (3) at (-1, -3) {};
		\node [style=none] (4) at (1, 3) {};
		\node [style=none] (5) at (1.5, -2) {};
		\node [style=none] (6) at (1.5, 2) {};
		\node [style=none] (7) at (1, -3) {};
		\node [style=none] (14) at (-1, -4) {\scriptsize{$k$}};
		\node [style=none] (15) at (1, -4) {\scriptsize{$l$}};
		\draw [thick, <-] (0.center) to (3.center);
		\draw [thick,->] (7.center) to (4.center);
\end{tikzpicture}} .
\end{align*}
Similarly we rewrite the second web on the left-hand side using \cref{square-switch,dumbbell-relation} as follows:
\begin{align*}
\mathord{\begin{tikzpicture}[baseline=-1ex,color=\clr, scale=.35]
		\node [style=none] (0) at (-2, 3) {};
		\node [style=none] (1) at (2, 3) {};
		\node [style=none] (2) at (-2, -3) {};
		\node [style=none] (3) at (2, -3) {};
		\node [style=none] (4) at (2, -2) {};
		\node [style=none] (5) at (0, -2) {};
		\node [style=none] (6) at (0, -1) {};
		\node [style=none] (7) at (-2, -1) {};
		\node [style=none] (8) at (-2, 1) {};
		\node [style=none] (9) at (0, 1) {};
		\node [style=none] (10) at (0, 1.75) {};
		\node [style=none] (11) at (2, 1.75) {};
		\node [style=none] (12) at (-1, 1.5) {\scriptsize{$1$}};
		\node [style=none] (13) at (-1, -1.5) {\scriptsize{$1$}};
		\node [style=none] (14) at (-2, -4) {\scriptsize{$k$}};
		\node [style=none] (15) at (2, -4) {\scriptsize{$l$}};
		\node [style=none] (18) at (-3, 0) {\scriptsize{$k+1$}};
		\node [style=none] (19) at (3, 0) {\scriptsize{$l-1$}};
		\draw [thick, <-] (0.center) to (2.center);
		\draw [thick, <-] (1.center) to (3.center);
		\draw [thick, -, directed=.5] (4.center) to (5.center);
		\draw [thick, -] (5.center) to (6.center);
		\draw [thick, -, directed=.75] (6.center) to (7.center);
		\draw [thick, -, directed=.75] (8.center) to (9.center);
		\draw [thick, -] (9.center) to (10.center);
		\draw [thick, -, directed=.5] (10.center) to (11.center);
		\draw  (0.5, -2) \wdot;
		\draw  (1.5, 1.75) \wdot;
\end{tikzpicture}}  
 &=  
\mathord{\begin{tikzpicture}[baseline=-1ex,color=\clr, scale=.35]
		\node [style=none] (0) at (-2, 3) {};
		\node [style=none] (1) at (2, 3) {};
		\node [style=none] (2) at (-2, -3) {};
		\node [style=none] (3) at (2, -3) {};
		\node [style=none] (4) at (2, -2) {};
		\node [style=none] (5) at (0, -2) {};
		\node [style=none] (6) at (0, -1) {};
		\node [style=none] (7) at (-2, -1) {};
		\node [style=none] (8) at (-2, 1) {};
		\node [style=none] (9) at (0, 1) {};
		\node [style=none] (10) at (0, 1.75) {};
		\node [style=none] (11) at (2, 1.75) {};
		\node [style=none] (12) at (-1, 1.5) {\scriptsize{$1$}};
		\node [style=none] (13) at (-1, -1.5) {\scriptsize{$1$}};
		\node [style=none] (14) at (-2, -4) {\scriptsize{$k$}};
		\node [style=none] (15) at (2, -4) {\scriptsize{$l$}};
		\node [style=none] (18) at (-3, 0) {\scriptsize{$k-1$}};
		\node [style=none] (19) at (3, 0) {\scriptsize{$l-1$}};
		\node [style=none] (20) at (0.5, -0.25) {\scriptsize{$2$}};
		\draw [thick, <-] (0.center) to (2.center);
		\draw [thick, <-] (1.center) to (3.center);
		\draw [thick, -, directed=.5] (4.center) to (5.center);
		\draw [thick, -] (5.center) to (6.center);
		\draw [thick, -, rdirected=.75] (6.center) to (7.center);
		\draw [thick, -, rdirected=.35] (8.center) to (9.center);
		\draw [thick, -] (9.center) to (10.center);
		\draw [thick, -, directed=.5] (10.center) to (11.center);
		\draw [thick, -, directed=.75] (6.center) to (9.center);
		\draw  (0.5, -2) \wdot;
		\draw  (1.5, 1.75) \wdot;
\end{tikzpicture}}
-  [k-1]_{q}[l]_{q}
\mathord{\begin{tikzpicture}[baseline=-1ex,color=\clr, scale=.35]
		\node [style=none] (0) at (-1, 3) {};
		\node [style=none] (3) at (-1, -3) {};
		\node [style=none] (4) at (1, 3) {};
		\node [style=none] (5) at (1.5, -2) {};
		\node [style=none] (6) at (1.5, 2) {};
		\node [style=none] (7) at (1, -3) {};
		\node [style=none] (14) at (-1, -4) {\scriptsize{$k$}};
		\node [style=none] (15) at (1, -4) {\scriptsize{$l$}};
		\draw [thick, <-] (0.center) to (3.center);
		\draw [thick,->] (7.center) to (4.center);
\end{tikzpicture}} \\
 &=  -
\mathord{\begin{tikzpicture}[baseline=-1ex,color=\clr, scale=.35]
		\node [style=none] (0) at (-2, 3) {};
		\node [style=none] (1) at (2, 3) {};
		\node [style=none] (2) at (-2, -3) {};
		\node [style=none] (3) at (2, -3) {};
		\node [style=none] (4) at (2, -2) {};
		\node [style=none] (5) at (0, -2) {};
		\node [style=none] (6) at (0, -0.6) {};
		\node [style=none] (7) at (-2, -2) {};
		\node [style=none] (8) at (-2, 1.75) {};
		\node [style=none] (9) at (0, 0.6) {};
		\node [style=none] (10) at (0, 1.75) {};
		\node [style=none] (11) at (2, 1.75) {};
		\node [style=none] (12) at (-1, 2) {\scriptsize{$1$}};
		\node [style=none] (13) at (-1, -2) {\scriptsize{$1$}};
		\node [style=none] (12) at (1, 2) {\scriptsize{$1$}};
		\node [style=none] (13) at (1, -2) {\scriptsize{$1$}};
		\node [style=none] (14) at (-2, -4) {\scriptsize{$k$}};
		\node [style=none] (15) at (2, -4) {\scriptsize{$l$}};
		\node [style=none] (18) at (-3, 0) {\scriptsize{$k-1$}};
		\node [style=none] (19) at (3, 0) {\scriptsize{$l-1$}};
		\node [style=none] (20) at (0.5, -0.25) {\scriptsize{$2$}};
		\draw [thick, <-] (0.center) to (2.center);
		\draw [thick, <-] (1.center) to (3.center);
		\draw [thick, -, directed=.6, looseness=1.25] (4.center) to [out=135, in=315] (6.center);
		\draw [thick, -, rdirected=.75, looseness=1.25] (6.center) to [out=225, in=45] (7.center);
		\draw [thick, -, rdirected=.35, looseness=1.25] (8.center) to [out=315,in=135] (9.center);
		\draw [thick, -, directed=.6, looseness=1.25] (9.center) to [out=45,in=225] (11.center);
		\draw [thick, -, directed=.75] (6.center) to (9.center);
		\draw  (-0.5, 1) \wdot;
		\draw  (-0.5, -1) \wdot;
\end{tikzpicture}}
 + \left( [2]_{q}[k]_{q}[l]_{q}-  [k-1]_{q}[l]_{q} \right)
\mathord{\begin{tikzpicture}[baseline=-1ex,color=\clr, scale=.35]
		\node [style=none] (0) at (-1, 3) {};
		\node [style=none] (3) at (-1, -3) {};
		\node [style=none] (4) at (1, 3) {};
		\node [style=none] (5) at (1.5, -2) {};
		\node [style=none] (6) at (1.5, 2) {};
		\node [style=none] (7) at (1, -3) {};
		\node [style=none] (14) at (-1, -4) {\scriptsize{$k$}};
		\node [style=none] (15) at (1, -4) {\scriptsize{$l$}};
		\draw [thick, <-] (0.center) to (3.center);
		\draw [thick,->] (7.center) to (4.center);
\end{tikzpicture}}
 +  \tq 
\mathord{\begin{tikzpicture}[baseline=-1ex,color=\clr, scale=.35]
		\node [style=none] (0) at (-1, 3) {};
		\node [style=none] (3) at (-1, -3) {};
		\node [style=none] (4) at (1, 3) {};
		\node [style=none] (5) at (1.5, -2) {};
		\node [style=none] (6) at (1.5, 2) {};
		\node [style=none] (7) at (1, -3) {};
		\node [style=none] (14) at (-1, -4) {\scriptsize{$k$}};
		\node [style=none] (15) at (1, -4) {\scriptsize{$l$}};
		\draw [thick, <-] (0.center) to (3.center);
		\draw [thick,->] (7.center) to (4.center);
		\draw (-1,0) \wdot;
		\draw (1,0) \wdot;
\end{tikzpicture}} .
\end{align*}  Adding the two and simplifying yields the right hand side of  \cref{square-switch-double-dots}, as desired.

To prove \cref{generalized-square-switch} one argues via induction using \cref{rung-collision,square-switch}.   The proof of \cref{two-rung-switch,two-rung-switch-with-dots} both follow from writing the rungs as merges and splits and applying \cref{associativity}.

The proof of \cref{double-rungs-3} follows by showing the middle diagram equals the sum of the other two webs and is as follows: 
\begin{align*}
\mathord{\begin{tikzpicture}[baseline=-1ex,color=\clr, scale=.35]
		\node [style=none] (0) at (-2, 3) {};
		\node [style=none] (1) at (-2.5, 2) {};
		\node [style=none] (3) at (-2, 4) {};
		\node [style=none] (8) at (-1, 1.5) {};
		\node [style=none] (10) at (-0.5, 4) {};
		\node [style=none] (12) at (0, -0.5) {};
		\node [style=none] (13) at (-0.5, -2) {};
		\node [style=none] (14) at (0.5, -1) {};
		\node [style=none] (16) at (-2, -1.5) {};
		\node [style=none] (17) at (-2.5, -4) {};
		\node [style=none] (18) at (-1.5, -2.5) {};
		\node [style=none] (19) at (-2.5, 0) {};
		\node [style=none] (20) at (1, -2) {};
		\node [style=none] (22) at (1.5, 4) {};
		\node [style=none] (23) at (1, -4) {};
		\node [style=none] (24) at (-1, -3) {};
		\node [style=none] (27) at (-1, -4) {};
		\node [style=none] (31) at (-1.5, 2.5) {};
		\node [style=none] (32) at (-0.5, 2.5) {};
		\node [style=none] (33) at (1.5, -1) {};
		\node [style=none] (34) at (-2.5, -2.5) {};
		\node [style=none] (35) at (0, 0.5) {};
		\node [style=none] (28) at (-2.5, -4.5) {\scriptsize{$h$}};
		\node [style=none] (29) at (-1, -4.5) {\scriptsize{$k$}};
		\node [style=none] (30) at (1, -4.5) {\scriptsize{$l$}};
		\node [style=none] (25) at (0.35, 0.75) {\scriptsize{$k$}};
		\node [style=none] (26) at (-1.75, -2.5) {\scriptsize{$1$}};
		\node [style=none] (40) at (0.75, -1) {\scriptsize{$1$}};
		\node [style=none] (21) at (-1.75, 2.25) {\scriptsize{$1$}};
		\draw [thick, <-](3.center) to (0.center);
		\draw [thick, -, looseness=1.25](0.center) to [out=180, in=90] (1.center);
		\draw [thick, -, looseness=1.25, rdirected=0.5](12.center) to [out=180, in=90] (13.center);
		\draw [thick, -, looseness=1.25](12.center) to [out=0, in=90] (14.center);
		\draw [thick, -, looseness=1.25](19.center) to [out=270, in=90] (16.center);
		\draw [thick, -, looseness=1.25, rdirected=0.6](16.center) to [out=0,in=90] (18.center);
		\draw [thick, -](23.center) to (20.center);
		\draw [thick, -, looseness=1.25, directed=0.75](20.center) to [out=180,in=270] (14.center);
		\draw [thick, -](27.center) to (24.center);
		\draw [thick, -, looseness=1.25](24.center) to [out=180,in=270] (18.center);
		\draw [thick, -, looseness=1.25](24.center) to [out=0, in=270] (13.center);
		\draw [thick, -, rdirected=0.5](1.center) to (19.center);
		\draw [thick, -, looseness=1.25](0.center) to [out=0,in=90] (31.center);
		\draw [thick, -, looseness=1.25](31.center) to [out=270,in=180] (8.center);
		\draw [thick, -, looseness=1.25](8.center) to [out=0, in=90] (32.center);
		\draw [thick, <-](10.center) to (32.center);
		\draw [thick, -, looseness=1.25](20.center) to [out=0, in=270] (33.center);
		\draw [thick, ->](33.center) to (22.center);
		\draw [thick, -, looseness=1.25](16.center) to [out=180,in=90] (34.center);
		\draw [thick, -](34.center) to (17.center);
		\draw [thick, -, looseness=1.25, rdirected=0.5](8.center) to [out=270, in=90] (35.center);
		\draw [thick, -](35.center) to (12.center);
\end{tikzpicture}}
=
\mathord{\begin{tikzpicture}[baseline=-1ex,color=\clr, scale=.35]
		\node [style=none] (0) at (-2, 3) {};
		\node [style=none] (1) at (-2.5, 2) {};
		\node [style=none] (3) at (-2, 4) {};
		\node [style=none] (8) at (-0.6125, 1) {};
		\node [style=none] (10) at (0, 4) {};
		\node [style=none] (12) at (0, -1) {};
		\node [style=none] (13) at (-0.5, -2) {};
		\node [style=none] (14) at (0.5, -2) {};
		\node [style=none] (16) at (-1.75, 1) {};
		\node [style=none] (17) at (-2.5, -4) {};
		\node [style=none] (18) at (-1.75, -2) {};
		\node [style=none] (19) at (-2.5, 0) {};
		\node [style=none] (20) at (1, -3) {};
		\node [style=none] (22) at (1.5, 4) {};
		\node [style=none] (23) at (1, -4) {};
		\node [style=none] (24) at (-1, -3) {};
		\node [style=none] (27) at (-1, -4) {};
		\node [style=none] (31) at (-1.25, 2) {};
		\node [style=none] (32) at (0, 2) {};
		\node [style=none] (33) at (1.5, -2) {};
		\node [style=none] (34) at (-2.5, -2) {};
		\node [style=none] (35) at (0, 0) {};
		\node [style=none] (28) at (-2.5, -4.5) {\scriptsize{$h$}};
		\node [style=none] (29) at (-1, -4.5) {\scriptsize{$k$}};
		\node [style=none] (30) at (1, -4.5) {\scriptsize{$l$}};
		\node [style=none] (25) at (0.35, 0.5) {\scriptsize{$k$}};
		\node [style=none] (40) at (0.75, -2) {\scriptsize{$1$}};
		\node [style=none] (21) at (-2, -1) {\scriptsize{$1$}};
		\node [style=none] (21) at (-1, 3) {\scriptsize{$2$}};
		\node [style=none] (26) at (-.65, 1.75) {\scriptsize{$1$}};
		\draw [thick, <-](3.center) to  (0.center);
		\draw [thick, -, looseness=1.25](0.center) to [out=180,in=90] (1.center);
		\draw [thick, -, looseness=1.25](12.center) to [out=180,in=90] (13.center);
		\draw [thick, -, looseness=1.25](12.center) to [out=0,in=90] (14.center);
		\draw [thick, -, rdirected=0.5](16.center) to (18.center);
		\draw [thick, -](23.center) to (20.center);
		\draw [thick, -, looseness=1.25, directed=0.75](20.center) to [out=180,in=270] (14.center);
		\draw [thick, -](27.center) to (24.center);
		\draw [thick, -, looseness=1.25](24.center) to [out=180,in=270] (18.center);
		\draw [thick, -, looseness=1.25](24.center) to [out=0,in=270] (13.center);
		\draw [thick, -](1.center) to (19.center);
		\draw [thick, -, looseness=1.25](0.center) to [out=0,in=90] (31.center);
		\draw [thick, -, looseness=1.25](31.center) to [out=0,in=180] (8.center);
		\draw [thick, -, looseness=1.25](8.center) to [out=0,in=270] (32.center);
		\draw [thick, <-](10.center) to (32.center);
		\draw [thick, -, looseness=1.25](20.center) to [out=0,in=270] (33.center);
		\draw [thick, ->](33.center) to (22.center);
		\draw [thick, -](34.center) to (17.center);
		\draw [thick, -,looseness=1.25,rdirected=0.5](8.center) to [out=270,in=90] (35.center);
		\draw [thick, -](35.center) to (12.center);
		\draw [thick, -, looseness=1.25](31.center) to [out=180,in=90] (16.center);
		\draw [thick, -](19.center) to (34.center);
\end{tikzpicture}}  
=
\mathord{\begin{tikzpicture}[baseline=-1ex,color=\clr, scale=.35]
		\node [style=none] (0) at (-1, 1) {};
		\node [style=none] (1) at (1, 1) {};
		\node [style=none] (2) at (1, -1) {};
		\node [style=none] (3) at (-1, -1) {};
		\node [style=none] (4) at (-2, 2) {};
		\node [style=none] (5) at (-2, 4) {};
		\node [style=none] (6) at (-3, 1) {};
		\node [style=none] (7) at (-3, -1) {};
		\node [style=none] (8) at (3, -1) {};
		\node [style=none] (9) at (2, -2) {};
		\node [style=none] (10) at (2, -4) {};
		\node [style=none] (11) at (-1, -4) {};
		\node [style=none] (12) at (-3, -4) {};
		\node [style=none] (13) at (3, 1) {};
		\node [style=none] (14) at (1, 4) {};
		\node [style=none] (15) at (3, 4) {};
		\node [style=none] (16) at (-3, -4.5) {\scriptsize{$h$}};
		\node [style=none] (17) at (-1, -4.5) {\scriptsize{$k$}};
		\node [style=none] (18) at (2, -4.5) {\scriptsize{$l$}};
		\node [style=none] (19) at (0, 1.5) {\scriptsize{$1$}};
		\node [style=none] (20) at (-1.5, 0) {\scriptsize{$1$}};
		\node [style=none] (21) at (0, -1.6) {\scriptsize{$k-1$}};
		\node [style=none] (22) at (1.5, 0) {\scriptsize{$k$}};
		\node [style=none] (23) at (1.6, -1.3) {\scriptsize{$1$}};
		\node [style=none] (24) at (-1.25, 2.25) {\scriptsize{$2$}};
		\draw [thick, -, rdirected=0.3] (0.center) to (1.center);
		\draw [thick, -] (1.center) to (2.center);
		\draw [thick, -, rdirected=0.3] (2.center) to (3.center);
		\draw [thick, -] (3.center) to (0.center);
		\draw [thick, -, looseness=1.25] (0.center) to [out=90,in=0] (4.center);
		\draw [thick, -, looseness=1.25] (4.center) to [out=180,in=90] (6.center);
		\draw [thick, ->] (4.center) to (5.center);
		\draw [thick, -] (6.center) to (7.center);
		\draw [thick, -, looseness=1.25] (2.center) to [out=270,in=180] (9.center);
		\draw [thick, -, looseness=1.25] (9.center) to [out=0,in=270] (8.center);
		\draw [thick, -] (9.center) to (10.center);
		\draw [thick, -] (11.center) to (3.center);
		\draw [thick, -] (12.center) to (7.center);
		\draw [thick, ->] (1.center) to (14.center);
		\draw [thick, <-] (15.center) to (13.center);
		\draw [thick, -] (13.center) to (8.center);
\end{tikzpicture}} 
&=
\mathord{\begin{tikzpicture}[baseline=-1ex,color=\clr, scale=.35]
		\node [style=none] (0) at (-1, 1) {};
		\node [style=none] (1) at (1, 1) {};
		\node [style=none] (2) at (1, -1) {};
		\node [style=none] (3) at (-1, -1) {};
		\node [style=none] (4) at (-2, 2) {};
		\node [style=none] (5) at (-2, 4) {};
		\node [style=none] (6) at (-3, 1) {};
		\node [style=none] (7) at (-3, -1) {};
		\node [style=none] (8) at (3, -1) {};
		\node [style=none] (9) at (2, -2) {};
		\node [style=none] (10) at (2, -4) {};
		\node [style=none] (11) at (-1, -4) {};
		\node [style=none] (12) at (-3, -4) {};
		\node [style=none] (13) at (3, 1) {};
		\node [style=none] (14) at (1, 4) {};
		\node [style=none] (15) at (3, 4) {};
		\node [style=none] (16) at (-3, -4.5) {\scriptsize{$h$}};
		\node [style=none] (17) at (-1, -4.5) {\scriptsize{$k$}};
		\node [style=none] (18) at (2, -4.5) {\scriptsize{$l$}};
		\node [style=none] (19) at (0, 1.6) {\scriptsize{$k-1$}};
		\node [style=none] (20) at (-2, 0) {\scriptsize{$k+1$}};
		\node [style=none] (21) at (0, -1.5) {\scriptsize{$1$}};
		\node [style=none] (22) at (1.5, 0) {\scriptsize{$0$}};
		\node [style=none] (23) at (1.6, -1.3) {\scriptsize{$1$}};
		\node [style=none] (24) at (-1.25, 2.25) {\scriptsize{$2$}};
		\draw [thick, -, directed=0.8] (0.center) to (1.center);
		\draw [thick, -] (1.center) to (2.center);
		\draw [thick, -, directed=0.8] (2.center) to (3.center);
		\draw [thick, -] (3.center) to (0.center);
		\draw [thick, -, looseness=1.25] (0.center) to [out=90,in=0] (4.center);
		\draw [thick, -, looseness=1.25] (4.center) to [out=180,in=90] (6.center);
		\draw [thick, ->] (4.center) to (5.center);
		\draw [thick, -] (6.center) to (7.center);
		\draw [thick, -, looseness=1.25] (2.center) to [out=270,in=180] (9.center);
		\draw [thick, -, looseness=1.25] (9.center) to [out=0,in=270] (8.center);
		\draw [thick, -] (9.center) to (10.center);
		\draw [thick, -] (11.center) to (3.center);
		\draw [thick, -] (12.center) to (7.center);
		\draw [thick, ->] (1.center) to (14.center);
		\draw [thick, <-] (15.center) to (13.center);
		\draw [thick, -] (13.center) to (8.center);
\end{tikzpicture}} 
+
\mathord{\begin{tikzpicture}[baseline=-1ex,color=\clr, scale=.35]
		\node [style=none] (0) at (-1, 1) {};
		\node [style=none] (1) at (1, 1) {};
		\node [style=none] (2) at (1, -1) {};
		\node [style=none] (3) at (-1, -1) {};
		\node [style=none] (4) at (-2, 2) {};
		\node [style=none] (5) at (-2, 4) {};
		\node [style=none] (6) at (-3, 1) {};
		\node [style=none] (7) at (-3, -1) {};
		\node [style=none] (8) at (3, -1) {};
		\node [style=none] (9) at (2, -2) {};
		\node [style=none] (10) at (2, -4) {};
		\node [style=none] (11) at (-1, -4) {};
		\node [style=none] (12) at (-3, -4) {};
		\node [style=none] (13) at (3, 1) {};
		\node [style=none] (14) at (1, 4) {};
		\node [style=none] (15) at (3, 4) {};
		\node [style=none] (16) at (-3, -4.5) {\scriptsize{$h$}};
		\node [style=none] (17) at (-1, -4.5) {\scriptsize{$k$}};
		\node [style=none] (18) at (2, -4.5) {\scriptsize{$l$}};
		\node [style=none] (19) at (0, 0.8) {\scriptsize{$k-2$}};
		\node [style=none] (23) at (1.6, -1.3) {\scriptsize{$1$}};
		\node [style=none] (24) at (-1.25, 2.25) {\scriptsize{$2$}};
		\node [style=none] (25) at (1, 0) {};
		\node [style=none] (26) at (-1, 0) {};
		\draw [thick, -] (1.center) to (2.center);
		\draw [thick, -, rdirected=0.3] (25.center) to (26.center);
		\draw [thick, -] (3.center) to (0.center);
		\draw [thick, -, looseness=1.25] (0.center) to [out=90,in=0] (4.center);
		\draw [thick, -, looseness=1.25] (4.center) to [out=180,in=90] (6.center);
		\draw [thick, ->] (4.center) to (5.center);
		\draw [thick, -] (6.center) to (7.center);
		\draw [thick, -, looseness=1.25] (2.center) to [out=270,in=180] (9.center);
		\draw [thick, -, looseness=1.25] (9.center) to [out=0,in=270] (8.center);
		\draw [thick, -] (9.center) to (10.center);
		\draw [thick, -] (11.center) to (3.center);
		\draw [thick, -] (12.center) to (7.center);
		\draw [thick, ->] (1.center) to (14.center);
		\draw [thick, <-] (15.center) to (13.center);
		\draw [thick, -] (13.center) to (8.center);
\end{tikzpicture}} \\  
&=
\mathord{\begin{tikzpicture}[baseline=-1ex,color=\clr, scale=.35]
		\node [style=none] (0) at (-1, 0.75) {};
		\node [style=none] (1) at (0, 1.75) {};
		\node [style=none] (2) at (0, -2) {};
		\node [style=none] (3) at (-1, -1) {};
		\node [style=none] (4) at (-2.25, 2) {};
		\node [style=none] (5) at (-2.25, 4) {};
		\node [style=none] (6) at (-3, 0.5) {};
		\node [style=none] (7) at (-3, -1) {};
		\node [style=none] (8) at (2, -2) {};
		\node [style=none] (9) at (1, -2.75) {};
		\node [style=none] (10) at (1, -4) {};
		\node [style=none] (11) at (-2, -4) {};
		\node [style=none] (12) at (-3, -4) {};
		\node [style=none] (13) at (2, 1.75) {};
		\node [style=none] (14) at (0, 4) {};
		\node [style=none] (15) at (2, 4) {};
		\node [style=none] (16) at (-3, -4.5) {\scriptsize{$h$}};
		\node [style=none] (17) at (-2, -4.5) {\scriptsize{$k$}};
		\node [style=none] (18) at (1, -4.5) {\scriptsize{$l$}};
		\node [style=none] (22) at (0.25, 0) {\scriptsize{$k+1$}};
		\node [style=none] (23) at (-0.5, -1.75) {\scriptsize{$1$}};
		\node [style=none] (24) at (-1.35, 2) {\scriptsize{$2$}};
		\node [style=none] (25) at (-2, -2) {};
		\draw [thick, -, looseness=1.25] (0.center) to [out=0,in=270] (1.center);
		\draw [thick, -, looseness=1.25] (2.center) to [out=90,in=0] (3.center);
		\draw [thick, -] (3.center) to (0.center);
		\draw [thick, -, looseness=1.25] (0.center) to [out=180,in=0] (4.center);
		\draw [thick, -, looseness=1.25] (4.center) to [out=180,in=90] (6.center);
		\draw [thick, ->] (4.center) to (5.center);
		\draw [thick, -] (6.center) to (7.center);
		\draw [thick, -, looseness=1.25] (2.center) to [out=270,in=180] (9.center);
		\draw [thick, -, looseness=1.25] (9.center) to [out=0,in=270] (8.center);
		\draw [thick, -] (9.center) to (10.center);
		\draw [thick, -] (12.center) to (7.center);
		\draw [thick, ->] (1.center) to (14.center);
		\draw [thick, <-] (15.center) to (13.center);
		\draw [thick, -] (13.center) to (8.center);
		\draw [thick, -, looseness=1.25] (3.center) to [out=180,in=90] (25.center);
		\draw [thick, -] (25.center) to (11.center);
\end{tikzpicture}}
+
\mathord{\begin{tikzpicture}[baseline=-1ex,color=\clr, scale=.35]
		\node [style=none] (0) at (-1, -0.5) {};
		\node [style=none] (1) at (0.25, 0.75) {};
		\node [style=none] (2) at (1, -1.5) {};
		\node [style=none] (3) at (-1, -1.5) {};
		\node [style=none] (4) at (-2.25, 2) {};
		\node [style=none] (5) at (-2.25, 4) {};
		\node [style=none] (6) at (-3, 0.5) {};
		\node [style=none] (7) at (-3, -1.5) {};
		\node [style=none] (8) at (3, -1.5) {};
		\node [style=none] (9) at (2, -2.5) {};
		\node [style=none] (10) at (2, -4) {};
		\node [style=none] (11) at (-1, -4) {};
		\node [style=none] (12) at (-3, -4) {};
		\node [style=none] (13) at (3, 1.75) {};
		\node [style=none] (14) at (0.25, 4) {};
		\node [style=none] (15) at (3, 4) {};
		\node [style=none] (16) at (-3, -4.5) {\scriptsize{$h$}};
		\node [style=none] (17) at (-1, -4.5) {\scriptsize{$k$}};
		\node [style=none] (18) at (2, -4.5) {\scriptsize{$l$}};
		\node [style=none] (23) at (1.5, -1.5) {\scriptsize{$1$}};
		\node [style=none] (24) at (-1.25, 2) {\scriptsize{$2$}};
		\node [style=none] (26) at (1, -0.5) {};
		\draw [thick, -, looseness=1.25] (0.center) to [out=0,in=180] (1.center);
		\draw [thick, -] (3.center) to (0.center);
		\draw [thick, -, looseness=1.25] (0.center) to [out=180,in=0] (4.center);
		\draw [thick, -, looseness=1.25] (4.center) to [out=180,in=90] (6.center);
		\draw [thick, ->] (4.center) to (5.center);
		\draw [thick, -] (6.center) to (7.center);
		\draw [thick, -, looseness=1.25] (2.center) to [out=270,in=180] (9.center);
		\draw [thick, -, looseness=1.25] (9.center) to [out=0,in=270] (8.center);
		\draw [thick, -] (9.center) to (10.center);
		\draw [thick, -] (12.center) to (7.center);
		\draw [thick, ->] (1.center) to (14.center);
		\draw [thick, <-] (15.center) to (13.center);
		\draw [thick, -] (13.center) to (8.center);
		\draw [thick, -] (11.center) to (3.center);
		\draw [thick, -] (26.center) to (2.center);
		\draw [thick, -, looseness=1.25] (26.center) to [out=90,in=0] (1.center);
\end{tikzpicture}} .
\end{align*}
The first web is the middle diagram of \cref{double-rungs-3} where the ladders have been drawn using merges and splits, the first equality follows from using \cref{associativity} on the two leftmost merges, the second equality is obtained by drawing the merges and splits which form the closed loop as ladders, the next equality is an application of \cref{generalized-square-switch}, the last equality given by writing ladders as merges and splits and, finally, these two diagrams are seen to be the other two diagrams in \cref{double-rungs-3} after writing those ladders in terms of merges and splits.  An identical series of calculations also proves \cref{double-rungs-5}.

Entirely similar arguments along with the use of the symmetry of $\qwebs$ when reflecting across a vertical line yield the other relations. We leave those to the reader.
\end{proof}

\subsection{Crossings and the Hecke-Clifford Algebra}\label{SS:crossings}

Define the \emph{upward over-crossing} morphism in $\End_{\qwebs}\left(\up_{1}^2 \right)$ by
\begin{equation}\label{E:overcrossingdef}
\xy
(0,0)*{
	\begin{tikzpicture}[scale=.3, color=\clr]
	\draw [ thick, ->] (-1,-1) to (1,1);
	\draw [ thick, ->] (-0.25,0.25) to (-1,1);
	\draw [ thick] (0.25,-0.25) to (1,-1);
	\node at (-1,1.5) {\ss $1$};
	\node at (1,1.5) {\ss $1$};
	\node at (-1,-1.5) {\ss $1$};
	\node at (1,-1.5) {\ss $1$};
	\end{tikzpicture}
};
\endxy
\ := \ 
\xy
(0,0)*{
	\begin{tikzpicture}[color=\clr, scale=.3]
	\draw [ thick, directed=.75] (0,0.25) to (0,1.25);
	\draw [ thick, directed=1] (0,1.25) to [out=30,in=270] (1,2.5);
	\draw [ thick, directed=1] (0,1.25) to [out=150,in=270] (-1,2.5); 
	\draw [ thick, directed=.65] (1,-1) to [out=90,in=330] (0,0.25);
	\draw [ thick, directed=.65] (-1,-1) to [out=90,in=210] (0,0.25);
	\node at (-1,3) {\scriptsize $1$};
	\node at (1,3) {\scriptsize $1$};
	\node at (-1,-1.5) {\scriptsize $1$};
	\node at (1,-1.5) {\scriptsize $1$};
	\node at (-0.75,0.75) {\scriptsize $2$};
	\end{tikzpicture}
};
\endxy
\ -q^{-1} \ 
\xy
(0,0)*{
	\begin{tikzpicture}[color=\clr, scale=.3]
	\draw [thick, directed=1] (-2,-1) to (-2,2.5);
	\draw [thick, directed=1] (-4,-1) to (-4,2.5);
	\node at (-2,-1.5) {\scriptsize $1$};
	\node at (-2,3) {\scriptsize $1$};
	\node at (-4,-1.5) {\scriptsize $1$};
	\node at (-4,3) {\scriptsize $1$};
	\end{tikzpicture}
};
\endxy \ .
\end{equation}

\begin{definition}\label{Sergeev-webs}
	Fix $k\in\Z_{>0}$. For $i=1, \dotsc , k$ and $j = 1, \dotsc , k-1$, define the morphisms $c_{i}, T_{j} \in\End_{\qwebs}(\up_{1}^k)$ by
	\[
	c_i=
	\xy
	(0,0)*{
		\bt[color=\clr, scale=1.25]
		\draw[thick, directed=1] (-0.75,0) to (-0.75,0.5);
		\draw[thick, directed=1] (-0.25,0) to (-0.25,0.5);
		\draw[thick, directed=1] (0.25,0) to (0.25,0.5);
		\draw[thick, directed=1] (0.75,0) to (0.75,0.5);
		\draw[thick, directed=1] (0,0) to (0,0.5);
		\node at (0,-0.15) {\scriptsize $1$};
		\node at (0,0.65) {\scriptsize $1$};
		\node at (-0.25,-0.15) {\scriptsize $1$};
		\node at (-0.25,0.65) {\scriptsize $1$};
		\node at (-0.75,-0.15) {\scriptsize $1$};
		\node at (-0.75,0.65) {\scriptsize $1$};
		\node at (0.25,-0.15) {\scriptsize $1$};
		\node at (0.25,0.65) {\scriptsize $1$};
		\node at (.75,-0.15) {\scriptsize $1$};
		\node at (.75,0.65) {\scriptsize $1$};
		\node at (-0.5,0.25) { \ $\cdots$};
		\node at (0.5,0.25) { \ $\cdots$};
		\draw (0,0.25) \wdot;
		\et
	};
	\endxy \ ,
	\quad T_j=
	\xy
	(0,0)*{
		\bt[color=\clr, scale=1.25]
		\draw[thick, directed=1] (-0.75,0) to (-0.75,0.5);
		\draw[thick, directed=1] (-0.25,0) to (-0.25,0.5);
		\draw[thick, directed=1] (0.75,0) to (0.75,0.5);
		\draw[thick, directed=1] (1.25,0) to (1.25,0.5);
		\draw[thick, directed=1] (0,0) to (0.5,0.5);
		\draw[thick, directed=1] (0.2,0.3) to (0,0.5);
		\draw[thick, ] (0.5,0) to (0.3,0.2);
		\node at (0,-0.15) {\scriptsize $1$};
		\node at (0,0.65) {\scriptsize $1$};
		\node at (0.5,-0.15) {\scriptsize $1$};
		\node at (0.5,0.65) {\scriptsize $1$};
		\node at (-0.25,-0.15) {\scriptsize $1$};
		\node at (-0.25,0.65) {\scriptsize $1$};
		\node at (-0.75,-0.15) {\scriptsize $1$};
		\node at (-0.75,0.65) {\scriptsize $1$};
		\node at (0.75,-0.15) {\scriptsize $1$};
		\node at (0.75,0.65) {\scriptsize $1$};
		\node at (1.25,-0.15) {\scriptsize $1$};
		\node at (1.25,0.65) {\scriptsize $1$};
		\node at (-0.5,0.25) { \ $\cdots$};
		\node at (1,0.25) { \ $\cdots$};
		\et
	};
	\endxy
	\]  where the dot in $c_{i}$ is on the $i$th strand and $T_{j}$ is the crossing of the $j$th and $(j+1)$st strands.
\end{definition}

Drawing the morphisms $T_i$ as crossings is justified by the following lemma.

\begin{lemma}\label{L:srelations} For any $k \geq 1$ the following relations hold in $\qwebs$ for all admissible $i,j$:
	\begin{enumerate}
		\item  $T_{j}^{2}=\tq T_j+1$.
		\item $T_{i}T_{j}=T_{j}T_{i}$ if $|i-j| >1$.
		\item  $T_{i}T_{i+1}T_{i}=T_{i + 1}T_{i}T_{i+1}$.
		\item $c_{i}^{2}=1$, $c_{i}c_{j}=-c_{j}c_{i}$ if $i \neq j$.
		\item  $T_{i}c_{j}=c_{j}T_{i}$ if $|i-j| >1$.
		\item  $T_{i}c_{i}=c_{i+1}T_{i}$.
		\item  $c_iT_i=T_ic_{i+1}+\tq (c_i-c_{i+1})$.
	\end{enumerate}
\end{lemma}

\begin{proof}
	Statements (b.) and (e.) follow immediately from the super-interchange law, while the rest may be verified by direct calculations involving \cref{digon-removal,square-switch,square-switch-dots}.
\end{proof}

The Hecke-Clifford superalgebra introduced by Olshanski in \cite{Olshanski} has even generators $T_{1}, \dotsc , T_{k-1}$, odd generators $c_{1}, \dotsc , c_{k}$, and defining relations precisely those listed in the previous lemma.   Thus there is a homomorphism of superalgebras,
\begin{equation}\label{E:Sergeevhomomorphism}
\xi_{k} :\HC_{k}(q) \to \End_{\qwebs}(\up_{1}^{k}),
\end{equation}
for any $k \geq 1$ which sends the generators of $\HC_{k}(q)$ to the morphisms in $\End_{\qwebs}(\up_{1}^{k})$ of the same name\footnote{To be precise, in \cite{Olshanski} the elements $c_{i}$ square to $-1$. Scaling those generators by $\sqrt{-1}$ yields the presentation used here.}.  In \cref{description-of-all-ones} we will see this map is an isomorphism. Meanwhile, for $x \in \HC_{k}(q)$ we abuse notation by writing $x \in \End_{\qwebs}(\up_{1}^{k})$ for $\xi_{k}(x)$.  In particular, if  $w=s_{i_{1}}\dotsb s_{i_{t}}$ is a reduced expression for $w \in \Sigma_{k}$, then there is a well-defined element $T_{w}:=T_{i_{1}}\dotsb T_{i_{t}} \in \HC_{k}(q)$ and we write $T_{w}$ for $\xi_{k}(T_{w})$ in $\End_{\qwebs}(\up_{1}^{k})$.

From part (a.) of the lemma it follows that $T_{j}^{-1}  = T_{j} - \tq$.  Consequently we extend our notation to include the \emph{upward under-crossing} morphism in $\End_{\qwebs}(\up_{1}^2)$, defined as
\begin{equation}\label{E:inverseupwardcrossing}
\xy
(0,0)*{\reflectbox{
		\begin{tikzpicture}[scale=.3, color=\clr]
		\draw [ thick, ->] (-1,-1) to (1,1);
		\draw [ thick, ->] (-0.25,0.25) to (-1,1);
		\draw [ thick] (0.25,-0.25) to (1,-1);
		\node at (-1,1.5) {\reflectbox{\ss $1$}};
		\node at (1,1.5) {\reflectbox{\ss $1$}};
		\node at (-1,-1.5) {\reflectbox{\ss $1$}};
		\node at (1,-1.5) {\reflectbox{\ss $1$}};
		\end{tikzpicture}
}};
\endxy
\ := \ 
\xy
(0,0)*{
	\begin{tikzpicture}[scale=.3, color=\clr]
	\draw [ thick, ->] (-1,-1) to (1,1);
	\draw [ thick, ->] (-0.25,0.25) to (-1,1);
	\draw [ thick] (0.25,-0.25) to (1,-1);
	\node at (-1,1.5) {\ss $1$};
	\node at (1,1.5) {\ss $1$};
	\node at (-1,-1.5) {\ss $1$};
	\node at (1,-1.5) {\ss $1$};
	\end{tikzpicture}
};
\endxy
\ - \tq \ 
\xy
(0,0)*{
	\begin{tikzpicture}[scale=.3, color=\clr]
	\draw [ thick, ->] (-1,-1) to (-1,1);
	\draw [ thick, ->] (1,-1) to (1,1);
	\node at (-1,1.5) {\ss $1$};
	\node at (1,1.5) {\ss $1$};
	\node at (-1,-1.5) {\ss $1$};
	\node at (1,-1.5) {\ss $1$};
	\end{tikzpicture}
};
\endxy
\ = \ 
\xy
(0,0)*{
	\begin{tikzpicture}[color=\clr, scale=.3]
	\draw [ thick, directed=.75] (0,0.25) to (0,1.25);
	\draw [ thick, directed=1] (0,1.25) to [out=30,in=270] (1,2.5);
	\draw [ thick, directed=1] (0,1.25) to [out=150,in=270] (-1,2.5); 
	\draw [ thick, directed=.65] (1,-1) to [out=90,in=330] (0,0.25);
	\draw [ thick, directed=.65] (-1,-1) to [out=90,in=210] (0,0.25);
	\node at (-1,3) {\scriptsize $1$};
	\node at (1,3) {\scriptsize $1$};
	\node at (-1,-1.5) {\scriptsize $1$};
	\node at (1,-1.5) {\scriptsize $1$};
	\node at (-0.75,0.75) {\scriptsize $2$};
	\end{tikzpicture}
};
\endxy
\ -q \ 
\xy
(0,0)*{
	\begin{tikzpicture}[color=\clr, scale=.3]
	\draw [thick, directed=1] (-2,-1) to (-2,2.5);
	\draw [thick, directed=1] (-4,-1) to (-4,2.5);
	\node at (-2,-1.5) {\scriptsize $1$};
	\node at (-2,3) {\scriptsize $1$};
	\node at (-4,-1.5) {\scriptsize $1$};
	\node at (-4,3) {\scriptsize $1$};
	\end{tikzpicture}
};
\endxy \ .
\end{equation} The fact the upward over-crossings satisfy \cref{L:srelations}(b.)-(c.) implies the upward under-crossings do as well.  More generally, any two webs of type $\up_{1}^{k} \to \up_{1}^{k}$ consisting of only upward over- and under-crossings which are equivalent as topological braids are equal as morphisms in $\qwebs$.

Given an element of the symmetric group on $k$ letters, $w \in S_{k}$, let $\ell(w)$ denote its length. 
\begin{lemma}\label{L:untwist-permutation}
	Let $k\in\Z_{>0}$ and $w\in S_k$. Then, 
	\begin{equation}\label{E:untwist-permutation}
	\xy
	(0,0)*{
		\bt[color=\clr, scale=.35]
		\draw [color=\clr,  thick, directed=1] (4,9) to (4,10);
		\draw [color=\clr,  thick, directed=0.75] (5,7.25) to [out=90,in=330] (4,9);
		\draw [color=\clr,  thick, directed=0.75] (3,7.25) to [out=90,in=210] (4,9);
		\draw [color=\clr,  thick ] (2.5,7.25) rectangle (5.5,5.75);
		\draw [color=\clr,  thick ] (3,5.75) to (3,4.75);
		\draw [color=\clr,  thick ] (5,5.75) to (5,4.75);
		\node at (4.1, 7.5) { $\cdots$};
		\node at (4.1, 5.25) { $\cdots$};
		\node at (4, 6.5) { $T_w$};
		\node at (4,10.5) {\scriptsize $k$};
		\node at (3,4.25) {\scriptsize $1$};
		\node at (5,4.25) {\scriptsize $1$};
		\et
	};
	\endxy  \ = \ q^{\ell(w)}
	\xy
	(0,0)*{
		\bt[color=\clr, scale=.35]
		\draw [ thick, directed=1] (4,9) to (4,10);
		\draw [ thick, directed=0.75] (5,7.25) to [out=90,in=330] (4,9);
		\draw [ thick, directed=0.75] (3,7.25) to [out=90,in=210] (4,9);
		\node at (4.1, 7.75) { $\cdots$};
		\node at (4,10.5) {\scriptsize $k$};
		\node at (3,6.75) {\scriptsize $1$};
		\node at (5,6.75) {\scriptsize $1$};
		\et
	};
	\endxy  .
	\end{equation}
\end{lemma}

\begin{proof}
	First, note 
	\beq\label{untwist-crossing}
	\xy
	(0,0)*{
		\bt[color=\clr, scale=.35]
		\draw [ thick, directed=1] (0,2.75) to (0,3.5);
		\draw [ thick, directed=0.75] (1,1.75) to [out=90,in=330] (0,2.75);
		\draw [ thick, directed=0.75] (-1,1.75) to [out=90,in=210] (0,2.75);
		\draw [thick] (-1,-1) to (1,1.75);
		\draw [thick] (-0.25,0.75) to (-1,1.75);
		\draw [thick] (1,-1) to (0.25,0);
		\node at (0,3.95) {\scriptsize $2$};
		\node at (-1,-1.45) {\scriptsize $1$};
		\node at (1,-1.45) {\scriptsize $1$};
		\et
	};
	\endxy \ = \ 
	\xy
	(0,0)*{
		\begin{tikzpicture}[color=\clr, scale=.35]
		\draw [ thick, ] (0,0) to (0,.75);
		\draw [ thick, ] (0,.75) to [out=30,in=270] (1,1.75);
		\draw [ thick, ] (0,.75) to [out=150,in=270] (-1,1.75); 
		\draw [ thick, directed=0.75] (1,-1) to [out=90,in=330] (0,0);
		\draw [ thick, directed=0.75] (-1,-1) to [out=90,in=210] (0,0);
		\draw [ thick, directed=1] (0,2.75) to (0,3.5);
		\draw [ thick, directed=0.75] (1,1.75) to [out=90,in=330] (0,2.75);
		\draw [ thick, directed=0.75] (-1,1.75) to [out=90,in=210] (0,2.75);
		\node at (-1.5,1.75) {\scriptsize $1$};
		\node at (1.5,1.75) {\scriptsize $1$};
		\node at (-1,-1.45) {\scriptsize $1$};
		\node at (1,-1.45) {\scriptsize $1$};
		\node at (0,3.95) {\scriptsize $2$};
		\end{tikzpicture}
	};
	\endxy \ -q^{-1} \ 
	\xy
	(0,0)*{
		\bt[color=\clr, scale=.35]
		\draw [ thick, directed=1] (0,2.75) to (0,3.5);
		\draw [ thick, directed=0.75] (1,1.75) to [out=90,in=330] (0,2.75);
		\draw [ thick, directed=0.75] (-1,1.75) to [out=90,in=210] (0,2.75);
		\draw [thick] (-1,-1) to (-1,1.75);
		\draw [thick] (1,-1) to (1,1.75);
		\node at (0,3.95) {\scriptsize $2$};
		\node at (-1,-1.45) {\scriptsize $1$};
		\node at (1,-1.45) {\scriptsize $1$};
		\et
	};
	\endxy \ = \ [2]_{q}
	\xy
	(0,0)*{
		\bt[color=\clr, scale=.35]
		\draw [ thick, directed=1] (4,9) to (4,10);
		\draw [ thick, directed=0.75] (5,7.25) to [out=90,in=330] (4,9);
		\draw [ thick, directed=0.75] (3,7.25) to [out=90,in=210] (4,9);
		\node at (4,10.5) {\scriptsize $2$};
		\node at (3,6.75) {\scriptsize $1$};
		\node at (5,6.75) {\scriptsize $1$};
		\et
	};
	\endxy \ -q^{-1} \ 
	\xy
	(0,0)*{
		\bt[color=\clr, scale=.35]
		\draw [ thick, directed=1] (4,9) to (4,10);
		\draw [ thick, directed=0.75] (5,7.25) to [out=90,in=330] (4,9);
		\draw [ thick, directed=0.75] (3,7.25) to [out=90,in=210] (4,9);
		\node at (4,10.5) {\scriptsize $2$};
		\node at (3,6.75) {\scriptsize $1$};
		\node at (5,6.75) {\scriptsize $1$};
		\et
	};
	\endxy \ = \ q
	\xy
	(0,0)*{
		\bt[color=\clr, scale=.35]
		\draw [ thick, directed=1] (4,9) to (4,10);
		\draw [ thick, directed=0.75] (5,7.25) to [out=90,in=330] (4,9);
		\draw [ thick, directed=0.75] (3,7.25) to [out=90,in=210] (4,9);
		\node at (4,10.5) {\scriptsize $2$};
		\node at (3,6.75) {\scriptsize $1$};
		\node at (5,6.75) {\scriptsize $1$};
		\et
	};
	\endxy \ .
	\eeq
	The first equality is by definition and the second follows from \cref{digon-removal}. Since $T_w=T_{i_1}\cdots T_{i_l}$ for any reduced expression $w=s_{i_1}\cdots s_{i_{\ell }}$, a straightforward calculation using \cref{associativity,untwist-crossing} proves the statement in general.
\end{proof}
\noindent  If one instead uses under-crossings a similar calculation yields the same formula with $q$ replaced by $q^{-1}$.  Similar formulas also hold when crossings are above a split instead of below a merge.

\subsection{The clasp idempotents}\label{SS:ClaspIdempotents}  We now introduce an important family of morphisms which we call \emph{clasps}.

\begin{definition}\label{clasps}  For $k\in\Z_{\geq 1}$, the $k$-th \emph{clasp} $\Cl_{k} \in\End_{\qwebs}(\up_{1}^{k})$ is given by
	\[
	\Cl_k
	=\frac{1}{[k]_{q}!}
	\xy
	(0,0)*{
		\begin{tikzpicture}[color=\clr, scale=.3]
		\draw [ thick, directed=.55] (0,-1) to (0,.75);
		\draw [ thick, directed=1] (0,.75) to [out=30,in=270] (1,2.5);
		\draw [ thick, directed=1] (0,.75) to [out=150,in=270] (-1,2.5); 
		\draw [ thick, directed=.65] (1,-2.75) to [out=90,in=330] (0,-1);
		\draw [ thick, directed=.65] (-1,-2.75) to [out=90,in=210] (0,-1);
		\node at (-1,3) {\scriptsize $1$};
		\node at (0.1,1.75) {$\cdots$};
		\node at (1,3) {\scriptsize $1$};
		\node at (-1,-3.3) {\scriptsize $1$};
		\node at (0.1,-2.35) {$\cdots$};
		\node at (1,-3.3) {\scriptsize $1$};
		\node at (0.75,-0.25) {\scriptsize $k$};
		\end{tikzpicture}
	};
	\endxy \in \End_{\qwebs}\left(\up_{1}^{k} \right).
	\]   
\end{definition}
A calculation using \cref{digon-removal} shows $\Cl_k$ is an idempotent for all $k \in \Z_{\geq 1}$. 
The following lemma shows clasps admit a recursion formula similar to the \emph{Jones-Wenzl projectors} in the Temperley-Lieb algebra (e.g., see \cite{W}).

\begin{lemma}\label{clasp-recursion}
	For $k\in\Z_{>1}$,
	\begin{equation}\label{E:clasp2}
	\Cl_k
	= \frac{[2]_{q}[k-1]_{q}}{[k]_{q}}
	\xy
	(0,0)*{
		\begin{tikzpicture}[color=\clr, scale=.3]
		\draw [ thick] (-2,-9) to (-2,-8);
		\draw [ thick] (0,-9) to (0,-8);
		\draw [ thick] (2,-9) to (2,-8);
		\draw [ thick] (4,-9) to (4,-5.25);
		\draw [ thick, directed=1] (-2,-0) to (-2,1);
		\draw [ thick, directed=1] (0,0) to (0,1);
		\draw [ thick, directed=1] (2,0) to (2,1);
		\draw [ thick, directed=1] (4,-2.75) to (4,1);
		\draw [ thick, directed=0.75] (-2,-5.85) to (-2,-2.15);
		\draw [ thick, directed=0.75] (0,-5.85) to (0,-2.15);
		\draw [ thick] (-2.3,-8) rectangle (2.3,-5.85);
		\draw [ thick] (-2.3,-2.15) rectangle (2.3,0);
		\draw [ thick] (1.7, -5.25) rectangle (4.3, -2.75);
		\draw [ thick, ] (2,-5.85) to (2,-5.25);
		\draw [ thick, ] (2,-2.75) to (2,-2.15);
		\node at (-2,-9.45) {\scriptsize $1$};
		\node at (0,-9.45) {\scriptsize $1$};
		\node at (2,-9.45) {\scriptsize $1$};
		\node at (4,-9.45) {\scriptsize $1$};
		\node at (-2,1.55) {\scriptsize $1$};
		\node at (0,1.55) {\scriptsize $1$};
		\node at (2,1.55) {\scriptsize $1$};
		\node at (4,1.55) {\scriptsize $1$};
		\node at (3,-4) {\small $\Cl_2$};
		\node at (-2.5,-4) {\scriptsize $1$};
		\node at (0.5,-4) {\scriptsize $1$};
		\node at (-1,0.25) { \ $\cdots$};
		\node at (-1,-4) { \ $\cdots$};
		\node at (-1,-8.75) { \ $\cdots$};
		\node at (0,-1.125) {\small $\Cl_{k-1}$};
		\node at (0,-6.975) {\small $\Cl_{k-1}$};
		\end{tikzpicture}
	};
	\endxy
	-\frac{[k-2]_{q}}{[k]_{q}}
	\xy
	(0,0)*{
		\begin{tikzpicture}[color=\clr, scale=.3]
		\draw [ thick, ] (-2,-7) to (-2,-6);
		\draw [ thick, ] (0,-7) to (0,-6);
		\draw [ thick, ] (2,-7) to (2,-6);
		\draw [ thick, directed=1] (-2,-2) to (-2,-1);
		\draw [ thick, directed=1] (0,-2) to (0,-1);
		\draw [ thick, directed=1] (2,-2) to (2,-1);
		\draw [ thick, directed=1] (4,-7) to (4,-1);
		\draw [ thick] (-2.3,-6) rectangle (2.3,-2);
		\node at (-2,-7.45) {\scriptsize $1$};
		\node at (0,-7.45) {\scriptsize $1$};
		\node at (2,-7.45) {\scriptsize $1$};
		\node at (4,-7.45) {\scriptsize $1$};
		\node at (-2,-.45) {\scriptsize $1$};
		\node at (0,-.45) {\scriptsize $1$};
		\node at (2,-.45) {\scriptsize $1$};
		\node at (4,-.45) {\scriptsize $1$};
		\node at (-1,-1.75) { \ $\cdots$};
		\node at (-1,-6.5) { \ $\cdots$};
		\node at (0,-4) { $\Cl_{k-1}$};
		\end{tikzpicture}
	};
	\endxy .
	\end{equation}
\end{lemma}

\begin{proof}
	The formula follows from an inductive argument and a calculation using \cref{square-switch}. See \cite[Lemma 2.13]{RT} and \cite[Lemma 2.12]{TVW} for similar arguments in other settings. Those proofs carry over essentially unchanged.
\end{proof}
We also have the following closed formula for clasps.

\begin{lemma}\label{L:claspsum} For $k\geq 1$ we have
	\begin{equation}\label{clasp-sum}
	\Cl_k \ = \ \frac{q^{\frac{-k(k-1)}{2}}}{[k]_{q}!}\sum_{\sigma\in S_k} q^{\ell(\sigma)}
	\xy
	(0,0)*{
		\bt[color=\clr, scale=.35]
		\draw [ thick ] (2.5,7.25) rectangle (5.5,4.25);
		\draw [ thick ] (3,4.25) to (3,3.25);
		\draw [ thick ] (5,4.25) to (5,3.25);
		\draw [ thick, directed=1] (3,7.25) to (3,8.25);
		\draw [ thick, directed=1 ] (5,7.25) to (5,8.25);
		\node at (4.1, 7.5) { $\cdots$};
		\node at (4.1, 3.75) { $\cdots$};
		\node at (4, 5.75) { $T_{\sigma}$};
		\node at (3,2.75) {\scriptsize $1$};
		\node at (5,2.75) {\scriptsize $1$};
		\node at (3,8.75) {\scriptsize $1$};
		\node at (5,8.75) {\scriptsize $1$};
		\et
	};
	\endxy \ .
	\end{equation}
\end{lemma}

\begin{proof} The result follows from an inductive argument using \cref{E:overcrossingdef,clasp-recursion}, the fact the $T_{w}$'s are defined using only over-crossings, and straightforward but somewhat delicate calculations.  See \cite[Section 3.2]{KL} for a similar argument.  Since all which is needed for this paper is the fact clasps can be written as a linear combination of $T_{w}$'s and this follows easily from the recursion formula, we omit the details of this calculation.
\end{proof}

The following result immediately follows from \cref{L:claspsum} and the fact that equivalent braids give equal morphisms in $\qwebs$.


\begin{lemma}\label{L:clasps-past-crossings}
	For $k\in\Z_{>0}$, we have
	\beq\label{clasps-past-crossings}
	\xy
	(0,0)*{
		\begin{tikzpicture}[color=\clr, scale=.3]
		\draw [ thick] (1.7, -5.25) rectangle (4.3, -2.75);
		\draw [ thick, ] (2,-6) to (2,-5.25);
		\draw [ thick, ->] (2,-2.75) to (2,-2) to (4,0) to (4,1);
		\draw [ thick] (4,-6) to (4,-5.25);
		\draw [ thick, ->] (4,-2.75) to (4,-2) to (6,0) to (6,1);
		\draw [ thick, ] (6,-6) to (6,-2) to (5,-1.5);
		\draw [ thick ] (4.25,-1.125) to (3.75,-0.875); 
		\draw [thick, ->] (3,-0.5) to (2,0) to (2,1);
		\node at (2,-6.6) {\scriptsize $1$};
		\node at (4,-6.6) {\scriptsize $1$};
		\node at (6,-6.6) {\scriptsize $1$};
		\node at (2,1.5) {\scriptsize $1$};
		\node at (4,1.5) {\scriptsize $1$};
		\node at (6,1.5) {\scriptsize $1$};
		\node at (3,-4) {\small $\Cl_k$};
		\node at (2.9,-5.6) { \ $\cdots$};
		\node at (2.9,-2.4) { \ $\cdots$};
		\node at (4.9,0.3) { \ $\cdots$};
		\end{tikzpicture}
	};
	\endxy =
	\xy
	(0,0)*{
		\begin{tikzpicture}[color=\clr, scale=.3]
		\draw [ thick] (3.7, 1) rectangle (6.3, 3.5);
		\draw [ thick, ->] (6,3.5) to (6,4.5);
		\draw [ thick, ] (2,-2.75) to (2,-2) to (4,0) to (4,1);
		\draw [ thick, ->] (4,3.5) to (4,4.5);
		\draw [ thick, ] (4,-2.75) to (4,-2) to (6,0) to (6,1);
		\draw [ thick, ] (6,-2.75) to (6,-2) to (5,-1.5);
		\draw [ thick ] (4.25,-1.125) to (3.75,-0.875); 
		\draw [thick, ->] (3,-0.5) to (2,0) to (2,4.5);
		\node at (2,-3.25) {\scriptsize $1$};
		\node at (4,-3.25) {\scriptsize $1$};
		\node at (6,-3.25) {\scriptsize $1$};
		\node at (2,5) {\scriptsize $1$};
		\node at (4,5) {\scriptsize $1$};
		\node at (6,5) {\scriptsize $1$};
		\node at (5,2.25) {\small $\Cl_k$};
		\node at (4.9,3.875) { \ $\cdots$};
		\node at (2.9,-2.4) { \ $\cdots$};
		\node at (4.9,0.3) { \ $\cdots$};
		\end{tikzpicture}
	};
	\endxy
	\eeq
	along with the relations obtained by reflecting the above about a vertical axis and/or by replacing all over-crossings with under-crossings.
\end{lemma}

\subsection{Braidings in  \texorpdfstring{$\qwebs$}{Oriented Webs}}

We now introduce braidings between arbitrary objects of $\qwebs$.  We first define the braiding between two generating objects and also introduce a shorthand diagram for this morphism.

\begin{definition}\label{D:Upbraiding} For $k,l\in\Z_{>0}$ define
	\[\beta_{\up_{k}, \up_{l}}=
	\xy 
	(0,0)*{
		\begin{tikzpicture}[scale=.3, color=\clr]
		\draw [ thick, ->] (-1,-1) to (1,1);
		\draw [ thick, ->] (-0.25,0.25) to (-1,1);
		\draw [ thick] (0.25,-0.25) to (1,-1);
		\node at (-1,1.5) {\ss $l$};
		\node at (1,1.5) {\ss $k$};
		\node at (-1,-1.5) {\ss $k$};
		\node at (1,-1.5) {\ss $l$};
		\end{tikzpicture}
	};
	\endxy :=\frac{1}{[k]_{q}!\,[l]_{q}!}
	\xy
	(0,0)*{
		\bt[color=\clr, scale=.35]
		\draw [ thick, directed=1] (0, .75) to (0,1.5);
		\draw [ thick, directed=0.75] (1,-1) to [out=90,in=330] (0,.75);
		\draw [ thick, directed=0.75] (-1,-1) to [out=90,in=210] (0,.75);
		\draw [ thick, directed=1] (4, .75) to (4,1.5);
		\draw [ thick, directed=0.75] (5,-1) to [out=90,in=330] (4,.75);
		\draw [ thick, directed=0.75] (3,-1) to [out=90,in=210] (4,.75);
		\draw [ thick, directed=0.75] (0,-6.5) to (0,-5.75);
		\draw [ thick, ] (0,-5.75) to [out=30,in=270] (1,-4);
		\draw [ thick, ] (0,-5.75) to [out=150,in=270] (-1,-4); 
		\draw [ thick, directed=0.75] (4,-6.5) to (4,-5.75);
		\draw [ thick, ] (4,-5.75) to [out=30,in=270] (5,-4);
		\draw [ thick, ] (4,-5.75) to [out=150,in=270] (3,-4); 
		\draw [ thick ] (5,-1) to (1,-4);
		\draw [ thick ] (3,-1) to (-1,-4);
		\draw [ thick ] (5,-4) to (3.4,-2.8);
		\draw [ thick ] (2.6,-2.2) to (2.2,-1.9); 
		\draw [ thick ] (1.8,-1.6) to (1,-1);
		\draw [ thick ] (3,-4) to (2.2,-3.4);
		\draw [ thick ] (1.8,-3.1) to (1.4,-2.8);
		\draw [ thick ] (0.6,-2.2) to (-1,-1);
		\node at (2.6, -0.5) {\scriptsize $1$};
		\node at (5.4, -0.5) {\scriptsize $1$};
		\node at (1.4, -0.5) {\scriptsize $1$};
		\node at (-1.4, -0.5) {\scriptsize $1$};
		\node at (2.6, -4.5) {\scriptsize $1$};
		\node at (5.4, -4.5) {\scriptsize $1$};
		\node at (1.4, -4.5) {\scriptsize $1$};
		\node at (-1.4, -4.5) {\scriptsize $1$};
		\node at (0.1, -0.65) { $\cdots$};
		\node at (4.1, -0.65) { $\cdots$};
		\node at (0.1, -4.6) { $\cdots$};
		\node at (4.1, -4.6) { $\cdots$};
		\node at (0,2) {\scriptsize $l$};
		\node at (4,2) {\scriptsize $k$};
		\node at (0,-7) {\scriptsize $k$};
		\node at (4,-7) {\scriptsize $l$};
		\et
	};
	\endxy \ .
	\]
\end{definition} 
\noindent This morphism is the \emph{over-crossing of type $\up_{k}\up_{l} \to \up_{l}\up_{k}$}.  There is an obvious under-crossing analogue which  is the two-sided inverse  to $\beta_{\up_{k}, \up_{l}}$ under composition,  thanks to \cref{L:clasps-past-crossings}.

\begin{lemma}\label{L:braidingforups}
	The following relations hold in $\qwebs$ for all $h,k,l\in\Z_{>0}$. 
	\begin{enumerate}
		\item 
		\[
		\xy
		(0,0)*{
			\bt[color=\clr, scale=1.25]
			\draw[thick, ->] (0,-0.5) to (0.5,0) to (0,0.5);
			\draw[thick, ] (0.5,-0.5) to (0.325, -0.325); 
			\draw [thick, ] (0.175,-0.175) to (0,0) to (0.175,0.175);
			\draw[thick, ->] (0.325,0.325) to (0.5,0.5);
			\node at (0,-0.65) {\scriptsize $k$};
			\node at (0,.65) {\scriptsize $k$};
			\node at (0.5,-.65) {\scriptsize $l$};
			\node at (0.5,.65) {\scriptsize $l$};
			\et
		};
		\endxy=
		\xy
		(0,0)*{
			\bt[color=\clr, scale=1.25]
			\draw[thick, directed=1] (0.5,-0.5) to (0.5,0.5);
			\draw[thick, rdirected=0.05] (0,0.5) to (0,-0.5);
			\node at (0,-.65) {\scriptsize $k$};
			\node at (0,.65) {\scriptsize $k$};
			\node at (0.5,-.65) {\scriptsize $l$};
			\node at (0.5,.65) {\scriptsize $l$};
			\et
		};
		\endxy=
		\xy
		(0,0)*{
			\bt[color=\clr, scale=1.25]
			\draw[thick, ->] (0.175,0.325) to (0,0.5);
			\draw[thick, ] (0,-0.5) to (0.175,-0.325);
			\draw[thick, ] (0.325,-0.175) to (0.5,0) to (0.325,0.175);
			\draw [thick, ->] (0.5,-0.5) to (0,0) to (0.5,0.5);
			\node at (0,-0.65) {\scriptsize $k$};
			\node at (0,.65) {\scriptsize $k$};
			\node at (0.5,-.65) {\scriptsize $l$};
			\node at (0.5,.65) {\scriptsize $l$};
			\et
		};
		\endxy ,
		\]
		\item 
		\[
		\xy
		(0,0)*{
			\bt[color=\clr, scale=1.25]
			\draw[thick, directed=1] (0,0) to (1,1) to (1,1.5);
			\draw[thick, ] (0.5,0) to (0.375,0.125);
			\draw[thick, ->] (0.125,0.375) to (0,0.5) to (0,1) to (0.5,1.5);
			\draw[thick, ] (1,0) to (1,0.5) to (0.875,0.625);
			\draw[thick, ] (0.625,0.875) to (0.375,1.125);
			\draw[thick, ->] (0.125,1.375) to (0,1.5);
			\node at (0,-0.15) {\scriptsize $h$};
			\node at (0,1.65) {\scriptsize $l$};
			\node at (0.5,-0.15) {\scriptsize $k$};
			\node at (0.5,1.65) {\scriptsize $k$};
			\node at (1,-0.15) {\scriptsize $l$};
			\node at (1,1.65) {\scriptsize $h$};
			\et
		};
		\endxy=
		\xy
		(0,0)*{
			\bt[color=\clr, scale=1.25]
			\draw[thick, directed=1] (0,0) to (0,0.5) to (1,1.5);
			\draw[thick, ] (0.5,0) to (1,0.5) to (1,1) to (0.875,1.125);
			\draw[thick, ->] (0.625,1.375) to (0.5,1.5);
			\draw[thick, ] (1,0) to (0.875,0.125);
			\draw[thick, ] (0.625,0.375) to (0.375,0.625);
			\draw[thick, ->] (0.125,0.875) to (0,1) to (0,1.5);
			\node at (0,-0.15) {\scriptsize $h$};
			\node at (0,1.65) {\scriptsize $l$};
			\node at (0.5,-0.15) {\scriptsize $k$};
			\node at (0.5,1.65) {\scriptsize $k$};
			\node at (1,-0.15) {\scriptsize $l$};
			\node at (1,1.65) {\scriptsize $h$};
			\et
		};
		\endxy ,
		\]
		\item 
		\[
		\xy
		(0,0)*{
			\bt[scale=.35, color=\clr]
			\draw [ thick, directed=1] (0, .75) to (0,1.5) to (2,3.5);
			\draw [ thick, directed=.65] (1,-1) to [out=90,in=330] (0,.75);
			\draw [ thick, directed=.65] (-1,-1) to [out=90,in=210] (0,.75);
			\draw [ thick, ] (2,-1) to (2,1.5) to (1.25,2.25);
			\draw [thick, ->] (0.75,2.75) to (0,3.5);
			\node at (2, 4) {\scriptsize $h\! +\! k$};
			\node at (-1,-1.5) {\scriptsize $h$};
			\node at (1,-1.5) {\scriptsize $k$};
			\node at (2,-1.5) {\scriptsize $l$};
			\node at (0,4) {\scriptsize $l$};
			\et
		};
		\endxy=
		\xy
		(0,0)*{
			\bt[scale=.35, color=\clr]
			\draw [ thick, directed=1] (0, .75) to (0,1.5);
			\draw [ thick, directed=.65] (1,-1) to [out=90,in=330] (0,.75);
			\draw [ thick, directed=.65] (-1,-1) to [out=90,in=210] (0,.75);
			\draw [ thick, ] (-3,-3) to (-1,-1);
			\draw [ thick, ] (-1,-3) to (1,-1);
			\draw [ thick, ] (1,-3) to (0.1,-2.4);
			\draw [thick, ] (-0.5,-2) to (-1.1,-1.6);
			\draw [ thick, ->] (-1.7, -1.2) to (-2,-1) to (-2,1.5);
			\node at (0, 2) {\scriptsize $h\! +\! k$};
			\node at (-1,-3.5) {\scriptsize $k$};
			\node at (-3,-3.5) {\scriptsize $h$};
			\node at (-2,2) {\scriptsize $l$};
			\node at (1,-3.5) {\scriptsize $l$};
			\et
		};
		\endxy \ ,\quad\quad
		\xy
		(0,0)*{\rotatebox{180}{
				\bt[scale=.35, color=\clr]
				\draw [ thick, rdirected=0.65] (0, .75) to (0,1.5);
				\draw [ thick, ] (1,-1) to [out=90,in=330] (0,.75);
				\draw [ thick, ] (-1,-1) to [out=90,in=210] (0,.75);
				\draw [ thick, <-] (-3,-3) to (-1,-1);
				\draw [ thick, <-] (-1,-3) to (1,-1);
				\draw [ thick, <-] (1,-3) to (0.1,-2.4);
				\draw [thick, ] (-0.5,-2) to (-1.1,-1.6);
				\draw [ thick, ] (-1.7, -1.2) to (-2,-1) to (-2,1.5);
				\node at (0, 2) {\rotatebox{180}{\scriptsize $h\! +\! k$}};
				\node at (-1,-3.5) {\rotatebox{180}{\scriptsize $h$}};
				\node at (-3,-3.5) {\rotatebox{180}{\scriptsize $k$}};
				\node at (-2,2) {\rotatebox{180}{\scriptsize $l$}};
				\node at (1,-3.5) {\rotatebox{180}{\scriptsize $l$}};
				\et
		}};
		\endxy=
		\xy
		(0,0)*{\rotatebox{180}{
				\bt[scale=.35, color=\clr]
				\draw [ thick, rdirected=0.15] (0, .75) to (0,1.5) to (2,3.5);
				\draw [ thick, <-] (1,-1) to [out=90,in=330] (0,.75);
				\draw [ thick, <-] (-1,-1) to [out=90,in=210] (0,.75);
				\draw [ thick, <-] (2,-1) to (2,1.5) to (1.25,2.25);
				\draw [thick, ] (0.75,2.75) to (0,3.5);
				\node at (2, 4) {\rotatebox{180}{\scriptsize $h\! +\! k$}};
				\node at (-1,-1.5) {\rotatebox{180}{\scriptsize $k$}};
				\node at (1,-1.5) {\rotatebox{180}{\scriptsize $h$}};
				\node at (2,-1.5) {\rotatebox{180}{\scriptsize $l$}};
				\node at (0,4) {\rotatebox{180}{\scriptsize $l$}};
				\et
		}};
		\endxy ,
		\]
		\item 
		\[
		\xy
		(0,0)*{
			\bt[color=\clr, scale=1.25]
			\draw[thick, directed=1] (0,0) to (0.5,1);
			\draw[thick, ] (0.5,0) to (0.325,0.35);
			\draw[thick, ->] (0.175,0.65) to (0,1);
			\node at (0,-0.15) {\scriptsize $k$};
			\node at (0,1.15) {\scriptsize $l$};
			\node at (0.5,-0.15) {\scriptsize $l$};
			\node at (0.5,1.15) {\scriptsize $k$};
			\draw
			(0.0875,0.175) \wdot;
			\et
		};
		\endxy=
		\xy
		(0,0)*{
			\bt[color=\clr, scale=1.25]
			\draw[thick, directed=1] (0,0) to (0.5,1);
			\draw[thick, ] (0.5,0) to (0.325,0.35);
			\draw[thick, ->] (0.175,0.65) to (0,1);
			\node at (0,-0.15) {\scriptsize $k$};
			\node at (0,1.15) {\scriptsize $l$};
			\node at (0.5,-0.15) {\scriptsize $l$};
			\node at (0.5,1.15) {\scriptsize $k$};
			\draw
			(0.4125,0.825) \wdot;
			\et
		}; 
		\endxy \ .
		\]
	\end{enumerate}
	In addition the relations obtained by reflecting the above about a vertical axis and/or by replacing all over-crossings with under-crossings also hold.
\end{lemma}

\begin{proof} Relations (a.) and (b.) follow by applying \cref{D:Upbraiding} to the crossings, using \cref{clasps-past-crossings} to reduce to the case of strands of thickness $1$ on which the analogous relation from \cref{L:srelations} can be used and, finally, \cref{digon-removal} to simplify.  In the case of (c.) we explain the equation on the left, as the one on the right and their under-crossing analogues are similar.  By replacing the over-crossing with its definition from \cref{D:Upbraiding} and by exploding the two legs of the clasp into strands of thickness $1$  (see the proof of \cref{sergeev-generators} for another time when something similar is done), this reveals a scalar multiple of the clasp $\Cl_{h+k}$ below a series of over-crossing by strands of thickness $1$. Using \cref{L:clasps-past-crossings} to move the clasp through these crossings along with \cref{digon-removal} to simplify as needed yields the desired result.  Finally, (d.) follows by using \cref{D:Upbraiding}, \cref{dot-on-any-strand}, and part (c.).  The other cases follow by applying the left-right symmetry of webs and/or by similar calculations using under-crossings.
\end{proof}

We now define the braiding for arbitrary objects in $\qwebs$.

\begin{definition}\label{D:generalupbraiding} Let $a=(a_{1}, \dotsc , a_{r})$ and $b=(b_{1}, \dotsc , b_{s})$  be tuples of nonnegative integers and let $\ob{a}=\up_{a}$ and $\ob{b}=\up_{b}$ be the corresponding objects of $\qwebs$.  Define 
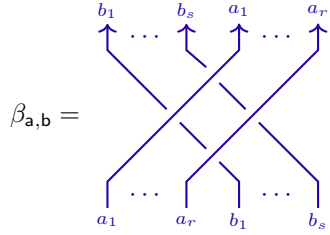
	\[
	\beta_{\ob{a},\ob{b}} =  
	\xy
	(0,0)*{
		\bt[color=\clr, scale=.35]
		\draw [ thick, directed=1] (0,-1) to (0,0) to (5,5) to (5,6);
		\draw [ thick, directed=1] (3,-1) to (3,0) to (8,5) to (8,6);
		\draw [ thick, ] (5,-1) to (5,0) to (4.25,0.75);
		\draw [ thick, ] (3.75,1.25) to (2.75,2.25);
		\draw [ thick, -> ] (2.25,2.75) to (0,5) to (0,6);
		\draw [ thick, ] (8,-1) to (8,0) to (5.75,2.25);
		\draw [ thick, ] (5.25,2.75) to (4.25,3.75);
		\draw [thick, -> ] (3.75,4.35) to (3,5) to (3,6);
		\node at (1.5,-0.5) {$\cdots$};
		\node at (1.5,5.5) {$\cdots$};
		\node at (6.5,-0.5) {$\cdots$};
		\node at (6.5,5.5) {$\cdots$};
		\node at (0,-1.5) {\scriptsize ${a}_1$};
		\node at (3,-1.5) {\scriptsize ${a}_r$};
		\node at (5,-1.5) {\scriptsize ${b}_1$};
		\node at (8,-1.5) {\scriptsize ${b}_{s}$};
		\node at (0,6.5) {\scriptsize ${b}_1$};
		\node at (3,6.5) {\scriptsize ${b}_{s}$};
		\node at (5,6.5) {\scriptsize ${a}_1$};
		\node at (8,6.5) {\scriptsize ${a}_r$};
		\et
	};
	\endxy
	\] 
\end{definition} 
\noindent  We refer to the above diagram as the \emph{over-crossing of type $\ob{a}\ob{b} \to \ob{b}\ob{a}$}.  There is an obvious under-crossing analogue which is the two-sided inverse under composition.

Let $\qwebseven$ denote the monoidal subcategory of $\qwebs$ consisting of all objects, and all morphisms which are linear combinations of diagrams with no dots. \cref{L:braidingforups} implies the following result.
\begin{theorem}\label{T:braiding}  The family of morphisms $\left\{ \beta^{\pm 1}_{\up_{\ob{a}}, \up_{\ob{b}}} \right\}$ define a  braiding on the monoidal category $\qwebseven$.
\end{theorem}

From \cref{L:braidingforups} we see dots move freely through crossings as long as they travel on the over-crossing strand.  However, \cref{L:srelations} already shows this fails if the dot travels along the under-crossing strand.  One might hope for a  ``thick strand'' version of the Hecke-Clifford relation given in \cref{L:srelations}(g.)  which describes how to move a dot under a crossing when the strands have thickness greater than $1$.  Small examples hint at such a formula.  As it is beyond the scope of this paper we leave it as an open question.

\section{Functors  \texorpdfstring{$\Pi_m$ and $\Psi$}{Pim and Psim}, and  \texorpdfstring{$\End_{\qwebs}(\up_{1}^k)$}{Endups}}

We now relate the diagrammatic supercategory $\qwebs$ to the supercategory $\bUdot (\fq_{m})_{\geq 0}$ introduced in \cref{SS:IdempotentAlgebra} and to the category of $U_{q}\left(\fq_{n} \right)$-supermodules $\mods$ introduced in \cref{SS:superfunctor}.

\subsection{The Functor  \texorpdfstring{$\Pi_m$}{Pim}}

\begin{proposition}\label{pi-functor}
For every $m\geq 1$ there exists a functor of monoidal supercategories 
\[
\Pi_m: \bUdot(\fq(m))_{\geq 0} \to\qwebs
\]
given on objects $\lambda = \sum_{i=1}^{m} \lambda_{i}\varepsilon_{i} \in X(T_{m})_{\geq 0}$ by 
\[
\Pi_m \left(\lambda \right) = \up_{\lambda} := \up_{\lambda_{1}} \up_{\lambda_{2}}\dotsb \up_{\lambda_{m}},
\]
and on the divided powers of the generating morphisms by 
\begin{equation*}
\Pi_m(e_i^{(j)}1_{\lambda})=
\xy
(0,0)*{
\begin{tikzpicture}[color=\clr ]
	\draw[ thick, directed=.25, directed=1] (0,0) to (0,1.5);
	\node at (0,-0.15) {\scriptsize $\lambda_{i}$};
	\node at (0,1.65) {\scriptsize $\lambda_{i}\!+\!j$};
	\draw[ thick, directed=.25, directed=1] (1,0) to (1,1.5);
	\node at (1,-0.15) {\scriptsize $\lambda_{i+1}$};
	\node at (1,1.65) {\scriptsize $\lambda_{i+1}\!-\!j$};
	\draw[ thick, directed=.55] (1,0.75) to (0,0.75);
	\node at (0.5,1) {\scriptsize$j$};
	\draw [thick, directed=1] (2,0) to (2,1.5);
	\draw [thick, directed=1] (2.75,0) to (2.75,1.5);
	\draw [thick, directed=1] (-1.75,0) to (-1.75,1.5);
	\draw [thick, directed=1] (-1,0) to (-1,1.5);
	\node at (-1.75,-0.15) {\scriptsize $\lambda_1$};
	\node at (-1,-0.15) {\scriptsize $\lambda_{i-1}$};
	\node at (-1.75,1.65) {\scriptsize $\lambda_1$};
	\node at (-1,1.65) {\scriptsize $\lambda_{i-1}$};
	\node at (-1.35,0.75) {$\cdots$};
	\node at (2,-0.15) {\scriptsize $\lambda_{i+2}$};
	\node at (2.75,-0.15) {\scriptsize $\lambda_m$};
	\node at (2,1.65) {\scriptsize $\lambda_{i+2}$};
	\node at (2.75,1.65) {\scriptsize $\lambda_m$};
	\node at (2.4,0.75) {$\cdots$};
\end{tikzpicture}
};
\endxy \ ,
\end{equation*}
\begin{equation*}
\Pi_m(f_i^{(j)}1_{\lambda})=
\xy
(0,0)*{
\begin{tikzpicture}[color=\clr ]
	\draw[ thick, directed=.25, directed=1] (0,0) to (0,1.5);
	\node at (0,-0.15) {\scriptsize $\lambda_{i}$};
	\node at (0,1.65) {\scriptsize $\lambda_{i}\!-\!j$};
	\draw[ thick, directed=.25, directed=1] (1,0) to (1,1.5);
	\node at (1,-0.15) {\scriptsize $\lambda_{i+1}$};
	\node at (1,1.65) {\scriptsize $\lambda_{i+1}\!+\!j$};
	\draw[ thick, directed=.55] (0,0.75) to (1,0.75);
	\node at (0.5,1) {\scriptsize$j$};
	\draw [thick, directed=1] (2,0) to (2,1.5);
	\draw [thick, directed=1] (2.75,0) to (2.75,1.5);
	\draw [thick, directed=1] (-1.75,0) to (-1.75,1.5);
	\draw [thick, directed=1] (-1,0) to (-1,1.5);
	\node at (-1.75,-0.15) {\scriptsize $\lambda_1$};
	\node at (-1,-0.15) {\scriptsize $\lambda_{i-1}$};
	\node at (-1.75,1.65) {\scriptsize $\lambda_1$};
	\node at (-1,1.65) {\scriptsize $\lambda_{i-1}$};
	\node at (-1.35,0.75) {$\cdots$};
	\node at (2,-0.15) {\scriptsize $\lambda_{i+2}$};
	\node at (2.75,-0.15) {\scriptsize $\lambda_m$};
	\node at (2,1.65) {\scriptsize $\lambda_{i+2}$};
	\node at (2.75,1.65) {\scriptsize $\lambda_m$};
	\node at (2.4,0.75) {$\cdots$};
\end{tikzpicture}
};
\endxy \ ,
\end{equation*}
\begin{equation*}
\Pi_m(\bar{e}_{i}^{(j)}1_{\lambda})=
\xy
(0,0)*{
\begin{tikzpicture}[color=\clr ]
	\draw[ thick, directed=.25, directed=1] (0,0) to (0,1.5);
	\node at (0,-0.15) {\scriptsize $\lambda_{i}$};
	\node at (0,1.65) {\scriptsize $\lambda_{i}\!+\!j$};
	\draw[ thick, directed=.25, directed=1] (1,0) to (1,1.5);
	\node at (1,-0.15) {\scriptsize $\lambda_{i+1}$};
	\node at (1,1.65) {\scriptsize $\lambda_{i+1}\!-\!j$};
	\draw[ thick, directed=.55] (1,0.75) to (0,0.75);
	\node at (0.5,1) {\scriptsize$j$};
	\draw  (0.25,0.75) \wdot;
	\draw [thick, directed=1] (2,0) to (2,1.5);
	\draw [thick, directed=1] (2.75,0) to (2.75,1.5);
	\draw [thick, directed=1] (-1.75,0) to (-1.75,1.5);
	\draw [thick, directed=1] (-1,0) to (-1,1.5);
	\node at (-1.75,-0.15) {\scriptsize $\lambda_1$};
	\node at (-1,-0.15) {\scriptsize $\lambda_{i-1}$};
	\node at (-1.75,1.65) {\scriptsize $\lambda_1$};
	\node at (-1,1.65) {\scriptsize $\lambda_{i-1}$};
	\node at (-1.35,0.75) {$\cdots$};
	\node at (2,-0.15) {\scriptsize $\lambda_{i+2}$};
	\node at (2.75,-0.15) {\scriptsize $\lambda_m$};
	\node at (2,1.65) {\scriptsize $\lambda_{i+2}$};
	\node at (2.75,1.65) {\scriptsize $\lambda_m$};
	\node at (2.4,0.75) {$\cdots$};
\end{tikzpicture}
};
\endxy \ ,
\end{equation*}
\begin{equation*}
\Pi_m(\bar{f}_{i}^{(j)}1_{\lambda})=
\xy
(0,0)*{
\begin{tikzpicture}[color=\clr ]
	\draw[ thick, directed=.25, directed=1] (0,0) to (0,1.5);
	\node at (0,-0.15) {\scriptsize $\lambda_{i}$};
	\node at (0,1.65) {\scriptsize $\lambda_{i}\!-\!j$};
	\draw[ thick, directed=.25, directed=1] (1,0) to (1,1.5);
	\node at (1,-0.15) {\scriptsize $\lambda_{i+1}$};
	\node at (1,1.65) {\scriptsize $\lambda_{i+1}\!+\!j$};
	\draw[ thick, directed=.55] (0,0.75) to (1,0.75);
	\node at (0.5,1) {\scriptsize$j$};
	\draw  (0.75,0.75) \wdot;
	\draw [thick, directed=1] (2,0) to (2,1.5);
	\draw [thick, directed=1] (2.75,0) to (2.75,1.5);
	\draw [thick, directed=1] (-1.75,0) to (-1.75,1.5);
	\draw [thick, directed=1] (-1,0) to (-1,1.5);
	\node at (-1.75,-0.15) {\scriptsize $\lambda_1$};
	\node at (-1,-0.15) {\scriptsize $\lambda_{i-1}$};
	\node at (-1.75,1.65) {\scriptsize $\lambda_1$};
	\node at (-1,1.65) {\scriptsize $\lambda_{i-1}$};
	\node at (-1.35,0.75) {$\cdots$};
	\node at (2,-0.15) {\scriptsize $\lambda_{i+2}$};
	\node at (2.75,-0.15) {\scriptsize $\lambda_m$};
	\node at (2,1.65) {\scriptsize $\lambda_{i+2}$};
	\node at (2.75,1.65) {\scriptsize $\lambda_m$};
	\node at (2.4,0.75) {$\cdots$};
\end{tikzpicture}
};
\endxy \ ,
\end{equation*}
\begin{equation*}
\Pi_m(\bar{K}_{i}1_{\lambda})=
\xy
(0,0)*{
\begin{tikzpicture}[scale=1.25, color=\clr ]
	\draw [thick, directed=1] (0,0) to (0,1);
	\draw [thick, directed=1] (0.5,0) to (0.5,1);
	\node at (0,-0.15) {\scriptsize $\lambda_1$};
	\node at (0.5,-0.15) {\scriptsize $\lambda_{i-1}$};
	\node at (0,1.15) {\scriptsize $\lambda_1$};
	\node at (0.5,1.15) {\scriptsize $\lambda_{i-1}$};
	\node at (0.275,0.5) {$\cdots$};
	\draw[thick, directed=1] (1,0) to (1,1);
	\node at (1,-0.15) {\scriptsize $\lambda_{i}$};
	\node at (1,1.15) {\scriptsize $\lambda_{i}$};
	\draw (1,0.5) \wdot;
	\draw [thick, directed=1] (1.5,0) to (1.5,1);
	\draw [thick, directed=1] (2,0) to (2,1);
	\node at (1.5,-0.15) {\scriptsize $\lambda_{i+1}$};
	\node at (2,-0.15) {\scriptsize $\lambda_m$};
	\node at (1.5,1.15) {\scriptsize $\lambda_{i+1}$};
	\node at (2,1.15) {\scriptsize $\lambda_m$};
	\node at (1.775,0.5) {$\cdots$};
\end{tikzpicture}
};
\endxy \ .
\]
Moreover, $\Pi_m$ is a full functor onto the full subcategory of $\qwebs$ consisting of objects $\left\{\up_{a} \mid a \in \Z_{\geq 0}^{m} \right\}$.
\end{proposition}
\begin{proof}
To show $\Pi_m$ is well-defined it suffices to verify the defining relations of $\bUdot(\fqt(m))_{\geq 0}$ hold in $\qwebs$. This follows by direct calculations using the defining relations of $\qwebs$ along with the identities proven in \cref{SS:AdditionalRelations}, as we briefly explain.  In \cref{D:dotU2}, the relation listed in the first line follows from \cref{double-rungs-1} when $i=j\pm 1$, by \cref{square-switch} when $i=j$, and by the super-interchange law for the other cases, the relations in the second line hold by the super-interchange law, the Serre relations in the third and fourth lines hold by \cref{double-rungs-1} and its reflection across a vertical line through the center, the relation on the fifth line follows by the super-interchange law when $i\neq j$ and by \cref{dot-collision} when $i=j$, the relations on the sixth and eighth lines follow by an application of \cref{dots-past-merges} (or the version obtained by reflecting it across its vertical axis), the relations on the seventh and ninth lines follows from the super-interchange law, the relations on the tenth and eleventh rows follow from \cref{double-rungs-2} when $i = j \pm  1$, \cref{square-switch-dots} when $i=j$, and from the super-interchange law for the other cases,   the relations on the twelveth line follow from \cref{associativity,dot-on-k-strand}, the relations on the thirteenth and fourtenth lines follow from \cref{double-rungs-1} or its reflection across the vertical line through its center, the ``odd'' Serre relations on the fifteenth and sixteenth lines follows from \cref{double-rungs-5}

That the image of $\Pi_m$ lies in the given subcategory is immediate.  
Since every merge and split is a special case of a ladder it follows we can obtain every possible diagram which is a merge, split, or dot tensored on the left and right by identity morphisms and which lies in this subcategory.  Since every web which is a morphism in this subcategory is a composition of such diagrams, it follows $\Pi_{m}$ is full.
\end{proof}

\begin{remark}\label{R:Compatiblity2}  Recall the functors $\Theta_{m',m}$ from \cref{R:Compatibility}. The $\Pi$ functors are compatible in the sense  $\Pi_{m'} \circ \Theta_{m',m}$ and $\Pi_{m}$ are isomorphic functors for all positive integers $m' \geq m$.
\end{remark}

\subsection{The Functor  \texorpdfstring{$\Psi$}{Psim}}

Going forward it will be convenient to write  $S_{q}^{\uparrow_{d}}\left(V_{n} \right)=S_{q}^{d}\left(V_{n} \right)$ for $d \geq 0$ and, more generally,
\[
S_{q}^{\ob{a}}\left(V_{n} \right)=S_{q}^{\uparrow_{a}}\left(V_{n} \right)  = S_{q}^{\uparrow_{a_{1}}}\left(V_{n} \right) \otimes \dotsb \otimes S_{q}^{\uparrow_{a_{r}}}\left(V_{n} \right)
\]
for $\ob{a}=\uparrow_{a}=\uparrow_{a_{1}}\dotsb \uparrow_{a_{r}}$.

\begin{proposition}\label{psi-functor}
For every $n \geq 1$ there exists an essentially surjective functor of monoidal supercategories, 
\[
\Psi_{n}: \qwebs\to\mods,
\]
given an object $\up_{\lambda} = \up_{\lambda_{1}}\dotsb \up_{\lambda_{t}}$ by 
\[
\Psi_{n}\left(\up_{\lambda} \right)=\AA_{q}(V_{m} \star V_{n})_{\lambda} \cong S_{q}^{\uparrow_{\lambda}}\left(V_{n} \right)=S_{q}^{\lambda_{1}}(V_{n}) \otimes\dotsb \otimes S_{q}^{\lambda_{t}}(V_{n}),
\] where we choose $m \geq t$ and $\lambda = (\lambda_{1}, \dotsc , \lambda_{t})$ is identified with the weight $(\lambda_{1}, \dotsc , \lambda_{t}, 0, \dotsc ,0)$ for $U_{q}\left(\fq_{m} \right)$ where there are $m-t$ zeros. The functor is independent of this choice of $m$.   

On morphisms the functor is defined by sending the dot, merge, and split, respectively, to the following maps: 
\begin{align}\label{dotmorphism}
\Psi_{n}\left(
\xy
(0,0)*{
\begin{tikzpicture}[scale=1.25, color=\clr ]
	\draw[thick, directed=1] (1,0) to (1,1);
	\node at (1,-0.15) {\scriptsize $k$};
	\node at (1,1.15) {\scriptsize $k$};
	\draw (1,0.5) \wdot ;
\end{tikzpicture}
};
\endxy\right)&=\Phi_{1,n}(\bar{K}_{1}1_{(k)}): S^{k}_{q}(V_{n}) \to S^{k}_{q}(V_{n}),\\ \label{mergemorphism}
\Psi_{n}\left(
\xy
(0,0)*{
\begin{tikzpicture}[scale=.35, color=\clr ]
	\draw [ thick, directed=1] (0, .75) to (0,2);
	\draw [ thick, directed=.65] (1,-1) to [out=90,in=330] (0,.75);
	\draw [ thick, directed=.65] (-1,-1) to [out=90,in=210] (0,.75);
	\node at (0, 2.5) {\scriptsize $k\! +\! l$};
	\node at (-1,-1.5) {\scriptsize $k$};
	\node at (1,-1.5) {\scriptsize $l$};
\end{tikzpicture}
};
\endxy\right)&=\Phi_{2,n}(\etilde_1^{(l)}1_{(k,l)}): S^{k}_{q}(V_{n})\otimes S^{l}_{q}(V_{n}) \to S^{k+l}_{q}(V_{n}),\\ \label{splitmorphism}
\Psi_{n}\left(
\xy
(0,0)*{
\begin{tikzpicture}[scale=.35, color=\clr ]
	\draw [ thick, directed=0.65] (0,-0.5) to (0,.75);
	\draw [ thick, directed=1] (0,.75) to [out=30,in=270] (1,2.5);
	\draw [ thick, directed=1] (0,.75) to [out=150,in=270] (-1,2.5); 
	\node at (0, -1) {\scriptsize $k\! +\! l$};
	\node at (-1,3) {\scriptsize $k$};
	\node at (1,3) {\scriptsize $l$};
\end{tikzpicture}
};
\endxy\right)&=\Phi_{2,n}(\etilde_1^{(k)}1_{(0,k+l)}): S^{k+l}_{q}(V_{n})\to S^{k}_{q}(V_{n})\otimes S^{l}_{q}(V_{n}).
\end{align}
Moreover, for any $m,n \geq 1$ we have $\Psi_{n} \circ \Pi_{m} = \Phi_{m,n}$.
\end{proposition}

\begin{proof} We first observe it follows from the action of $U_{q}(\fq_{m})$ on $\AA_{q}$ that, for all $a,b,c \geq 0$ and all $1 \geq r \geq b$,
\[
1 \otimes \Phi_{2,n} \left( E_{1}^{(r)}1_{(a,b)} \right) = \sum_{c \geq 0} \Phi_{3,n}\left(E_{2}^{(r)}1_{(c,a,b)} \right)
\] as maps $S_{q}^{c}(V_{n}) \otimes S_{q}^{a} \otimes S_{q}^{b}(V_{n}) \to S_{q}^{c}(V_{n}) \otimes S_{q}^{a+r} \otimes S_{q}^{b-r}(V_{n})$. Similar statements hold for the other generators of $U_{q}(\fq_{m})$.

Using the map $\bar{\rho}$ in \cref{P:basicpropertiesofAq}(f.) when $m=2$ we identify $S^{k}_{q}(V_{n}) \otimes S^{l}_{q}(V_{n})$ with the subspace of $\AA_{q}$ spanned by the monomials 
\[
\prod_{r=1}^{k} t_{1, x_{r}} \prod_{s=1}^{l} t_{2,y_{s}},
\] where $x_{1}, \dotsc x_{r}, y_{1}, \dotsc , y_{s} \in I_{n|n}$. Acting on this monomial with $ E_{1}^{(l)} \in U_{q}(\fq_{m})$ on $\AA_{q}$ and using the fact $E_{1} \cdot t_{1,x}=0$ for all $x \in I_{n|n}$ and \cref{L:Aqcalculation}  yields 
\[
 E_{1}^{(l)} \cdot \left(  \prod_{r=1}^{k} t_{1, x_{r}} \prod_{s=1}^{l} t_{2,y_{s}}\right) = \left(\prod_{r=1}^{k} t_{1, x_{r}}  \right) E_{1}^{(l)} \left(\prod_{s=1}^{l} t_{2,y_{s}}= \right) = \prod_{r=1}^{k} t_{1, x_{r}} \prod_{s=1}^{l} t_{1,y_{s}}.
\]  From this we see the map given by the merge is nothing but the multiplication map on $S_{q}(V_{n})$.  Therefore the left hand relation given in \cref{associativity} holds.  We next verify the right hand relation in \cref{associativity}. From the observation at the beginning of the proof this amounts to verifying the actions of 
\[
E_{2}^{(k)}1_{(h,0,k+l)}E_{1}^{(h)}1_{(0,h,k+l)}E_{2}^{(h)}1_{(0,0,h+k+l)} \;\; \text{ and } \;\;  E_{1}^{(h)}1_{(0,h+k,l)} E_{2}^{(h+k)}1_{(0,0,h+k+l)}
\]  on $\AA_{q}$ coincide. This follows from a calculation using \cref{L:Aqcalculation}.

Relation \cref{digon-removal} follows from the fact the left hand side is given by the action of 
\[
E^{(l)}_{1}1_{(k,l)}E^{(k)}_{1}1_{(0,k+l)} =  \qbinom{k+l}{l} E_{1}^{(k+l)}1_{(0,k+l)}
\]
on $\AA_{q}$ along with \cref{L:Aqcalculation}.  Relation \cref{dumbbell-relation} follows by a direct calculation using \cref{E:tensorspaceisomorphism}. The reader can find the relevant generating morphisms computed in \cref{L:functoronkeymorphisms}.

That the remaining relations defining $\qwebs$ are satisfied follows from using the observation at the beginning of the proof, expressing the relations terms of the action of $\bUdot (\fq_{m})$ on $\AA_{q}$, and showing the relations hold thanks to known relations of $\bUdot (\fq_{m})$.  For example, relation \cref{dot-collision} follows immediately from the fifth defining relation of $\bUdot(\fq_{m})$.  Relation \ref{dots-past-merges} follows from an inductive argument using the seventh and twelfth relations of $\bUdot (\fq_{m})$.  For the other relations we first observe the maps on $\AA_{q}$ given by the ladders \cref{E:rightladder} and \cref{E:leftladder} are given by the action of $E_{1}^{(j)}1_{(k,l)}$ and $F_{1}^{(j)}1_{(k,l)}$, respectively (c.f.\ \cref{R:astute} for the latter case).  This along with the observation at the beginning of the proof allows us to deduce relation \cref{square-switch} from the first relation for $\bUdot (\fq_{m})$.   The remaining $\qwebs$ relations drawn with ladders follow from the even and odd Serre relations.  The other relations obtained by reflection and/or rung reversal are similarly verified.  Taken all together the defining relations of $\qwebs$ hold and hence the functor $\Phi$ is well-defined.

Finally, that the functor is essentially surjective is clear.
\end{proof}

\begin{remark}\label{R:astute}  The astute reader will have noticed the functor $\Psi$ could have instead been defined on merges and splits via the action of the appropriate divided powers of $F_{1}$. As discussed in \cite[Remark 2.19]{RT} for type $A$, doing so yields the same maps and, hence, the same functor.
\end{remark}

In what follows, we will need explicit formulas for the the $U_q(\q_n)$-module homomorphisms which $\Psi_n$ associates to merges, splits, crossings, and dots of thin strands.  We record these homorphisms in the following lemma and in the proof we illustrate the techniques required to calculate these homorphisms.

\begin{lemma}\label{L:functoronkeymorphisms} Passing through the isomorphisms given in \cref{L:weight-space-isomorphism} yields the following statements:

\begin{enumerate}
\item  The morphism
\[
\Psi_{n}\left(
\xy
(0,0)*{
\begin{tikzpicture}[scale=1, color=\clr ]
	\draw[thick, directed=1] (1,0) to (1,1);
	\node at (1,-0.15) {\scriptsize $1$};
	\node at (1,1.15) {\scriptsize $1$};
	\draw (1,0.5) \wdot ;
\end{tikzpicture}
};
\endxy\right): V_{n} \to V_{n}
\] is given by $v_{b} \mapsto \sqrt{-1} \cdot (-1)^{\p{{b}}} v_{-b}$ for all $b \in I_{n|n}$. 
\item 
The morphism
\[
\Psi_{n}\left(
\xy
(0,0)*{
\begin{tikzpicture}[scale=.35, color=\clr ]
	\draw [ thick, directed=1] (0, .75) to (0,2);
	\draw [ thick, directed=.65] (1,-1) to [out=90,in=330] (0,.75);
	\draw [ thick, directed=.65] (-1,-1) to [out=90,in=210] (0,.75);
	\node at (0, 2.5) {\scriptsize $2$};
	\node at (-1,-1.5) {\scriptsize $1$};
	\node at (1,-1.5) {\scriptsize $1$};
\end{tikzpicture}
};
\endxy\right)
: V_n \otimes V_n \to S_q^2(V_n)
\]
is given by $v_a \otimes v_b \mapsto v_a v_b$ for all $a, b \in I_{n|n}$. 
\item The morphism 
\[
\Psi_{n}\left(
\xy
(0,0)*{
\begin{tikzpicture}[scale=.35, color=\clr ]
	\draw [ thick, directed=0.65] (0,-0.5) to (0,.75);
	\draw [ thick, directed=1] (0,.75) to [out=30,in=270] (1,2.5);
	\draw [ thick, directed=1] (0,.75) to [out=150,in=270] (-1,2.5); 
	\node at (0, -1) {\scriptsize $ 2$};
	\node at (-1,3) {\scriptsize $1$};
	\node at (1,3) {\scriptsize $1$};
\end{tikzpicture}
};
\endxy\right):
 S_q^2(V_n) \to V_n \otimes V_n\]
 is given by 

\[
v_a v_b \mapsto q^{-1} v_{a} \otimes v_{b} + T(v_a \otimes v_b)
\]
for all $a, b \in I_{n|n}$, where $T$ is the $U_q(\q_n)$-homomorphism from \cref{E:Braiding11}.
\item  The morphism 
\[
\Psi_{n}\left( \xy
(0,0)*{
\begin{tikzpicture}[scale=.3, color=\clr]
	\draw [ thick, ->] (-1,-1) to (1,1);
	\draw [ thick, ->] (-0.25,0.25) to (-1,1);
	\draw [ thick] (0.25,-0.25) to (1,-1);
	\node at (-1,1.5) {\ss $1$};
	\node at (1,1.5) {\ss $1$};
	\node at (-1,-1.5) {\ss $1$};
	\node at (1,-1.5) {\ss $1$};
	\end{tikzpicture}
};
\endxy \right): V_{n} \otimes V_{n} \to V_{n}\otimes V_{n}
\] is given by 
\[
v_{a} \otimes v_b \mapsto T(v_a \otimes v_b)
\] for all $a,b \in I_{n|n}$. 
\item  The morphism 
\[
\Psi_{n}\left( \xy
(0,0)*{\reflectbox{
\begin{tikzpicture}[scale=.3, color=\clr]
	\draw [ thick, ->] (-1,-1) to (1,1);
	\draw [ thick, ->] (-0.25,0.25) to (-1,1);
	\draw [ thick] (0.25,-0.25) to (1,-1);
	\node at (-1,1.5) {\reflectbox{\ss $1$}};
	\node at (1,1.5) {\reflectbox{\ss $1$}};
	\node at (-1,-1.5) {\reflectbox{\ss $1$}};
	\node at (1,-1.5) {\reflectbox{\ss $1$}};
	\end{tikzpicture}
}};
\endxy \right): V_{n} \otimes V_{n} \to V_{n}\otimes V_{n}
\] is given by 
\[
v_{a} \otimes v_b \mapsto   T^{-1}(v_a \otimes v_b) = T(v_a \otimes v_b) - \tq v_a \otimes v_b
\] for all $a,b \in I_{n|n}$.
\end{enumerate}

\end{lemma}
\begin{proof}
Note that (d.) and (e.) can be deduced directly from (b.) and (c.), along with \cref{E:overcrossingdef} and \cref{E:inverseupwardcrossing}.  The formula for (b.) is deduced in the proof of \cref{psi-functor}.  Parts (a.) and (c.) can be verified by direct calculations using the explicit formulas for the action of $U_{q}(\fq_{m})$ on $\AA_{q}$ given in \cref{SS:Aqdef} along with the identifications given in \cref{E:tensorspaceisomorphism,L:weight-space-isomorphism}.  

We verify (c.).  From \cref{splitmorphism}, the supermodule homomorphism in (c.) is defined by calculating the morphism $\Phi_{2,n}(E_1 1_{(0, 2)})$.  From the definition of the functor $\Phi_{2,n}$, this is the unique homomorphism making the following diagram commute:
$$
\begin{tikzcd}[column sep = huge]
S_q^2(V_n) \arrow{d}{\cong} \arrow{r}{\Phi_{2,n}(E_1 1_{(0, 2)})} & V_n \otimes V_n \arrow{d}{\cong} \\
\AA_{q,2\varepsilon_2} \arrow{r}{E_1} & \AA_{q, \varepsilon_1 + \varepsilon_2} 
\end{tikzcd}
$$ 
where the arrow labeled by $E_1$ is the $U_q(\q_n)$-module homomorphism induced by the action of $E_1 \in U_q(\q_2)$ on the $U_q(\q_2) \otimes U(\q_n)$-module $\AA_q = \AA_q(V_2 \star V_n)$, and the vertical isomorphisms are both given by \cref{L:weight-space-isomorphism}.  

Using the formulas from \cref{SS:Aqdef}, 
$$
E_1 \cdot ( t_{2a} t_{2b} ) =  \left(E_1 \cdot  t_{2a} \right) \left(K_1^{-1} K_{2} \cdot t_{2,b}\right) + t_{2,a}(E_1 \cdot t_{2,b}) = qt_{1a} t_{2,b} + t_{2,a} t_{1,b}.
$$
The monomial $t_{2,a} t_{1,b}$ is not in the order specified in \cref{P:basicpropertiesofAq}.  Applying the defining relations for $\AA_q$ we write this monomial in terms of monomials in the specified order:
$$
t_{2,a} t_{1,b} = -\tq t_{1,a}t_{2,b} +  (-1)^{\varphi(a,b)} (-1)^{\p{a} \p{b}} t_{1,b} t_{2,a} + \tq \delta_{a < b} t_{1,a} t_{2,b} + \delta_{-b < d} (-1)^{\p{b}} t_{1,-a} t_{2,-b}.
$$
Hence, 
$$
E_1 \cdot t_{2,a} t_{2,b} = q^{-1} t_{1,a}t_{2,b} + q^{\varphi(a,b)} (-1)^{\p{a} \p{b}} t_{1,b} t_{2,a} + \tq \delta_{a < b} t_{1,a} t_{2,b} + \delta_{-b < d} (-1)^{\p{b}} t_{1,-a} t_{2,-b}.
$$
After identifying $\AA_{q, 2 \varepsilon_2}$ with $S^2_q(V_n)$ and $\AA_{q,\varepsilon_1 + \varepsilon_2}$ with $V_n \otimes V_n$, it follows that $\Phi_{2,n}(E_1 1_{(0,2)})$ is the $U_q(\q_n)$-homomorphism 
\begin{align*}
v_a v_b &\mapsto q^{-1} v_a \otimes v_b + q^{\varphi(a,b)} (-1)^{\p{a} \p{b}} v_b \otimes v_a + \tq \delta_{a < b} v_a \otimes v_b + \delta_{-b < d} (-1)^{\p{b}} v_{-a} \otimes v_{-b}\\ 
&= q^{-1} v_a \otimes v_b + T(v_a \otimes v_b).
\end{align*}
\end{proof}

\begin{remark}\label{R:CompatibleWithOlshanski} By comparing with \cite[Section 3]{BGJKW} we see the formulas obtained above in (a.), (d.), and (e.) define the same $U_{q}(\fq_{n})$-supermodule homomorphisms used therein and in \cite{Olshanski}.
\end{remark}

\subsection{Description of \texorpdfstring{$\End_{\qwebs}(\up_{1}^k)$}{End(upk)}}\label{description-of-all-ones}

We next describe the endomorphism algebra $\End_{\qwebs}(\up_{1}^k)$ and show it is isomorphic to the Hecke-Clifford superalgebra.

\begin{lemma}\label{sergeev-generators}
The superalgebra $\End_{\qwebs}(\up_{1}^k)$ is generated by $c_1,\dots, c_k, T_1,\dots,T_{k-1}$.
\end{lemma}

\begin{proof}
It suffices to prove every web $w\in\End_{\qwebs}(\up_{1}^k)$ can be written as a linear combination of webs consisting solely of strands of thickness one which are upward over- and under-crossings and dots.  Given a web $w$,  by \cref{dot-on-k-strand} we may assume without loss of generality every dot in $w$ is on a $1$-strand. Next, every merge and split in $w$ can have its edges expanded into $1$-strands using \cref{digon-removal}.  For example, here is an expanded merge:
\[
\xy
(0,0)*{
\begin{tikzpicture}[scale=.35]
	\draw [color=\clr,  thick, directed=1] (-1,3) to (-1,6);
	\draw [color=\clr,  thick, rdirected=.55,looseness=1.25] (-1,3) to [out=200,in=90] (-3,0);
	\draw [color=\clr,  thick, rdirected=.55, looseness=1.25] (-1,3) to [out=340,in=90] (1,0);
	\node at (-1,6.35) {\scriptsize $h\! +\! l$};
	\node at (-3,-.35) {\scriptsize $h$};
	\node at (1,-.35) {\scriptsize $l$};
\end{tikzpicture}
};
\endxy
=
\frac{1}{[h]_{q}! \ [l]_{q}! \ [h+l]_{q}!}
\xy
(0,0)*{
\begin{tikzpicture}[scale=.35]
	\draw [color=\clr,  thick, directed=1] (-1,5.5) to (-1,6);
	\draw [color=\clr,  thick, directed=.65] (-1,3) to (-1,3.5);
	\draw [color=\clr,  thick, directed=.35] (-1,3.5) to [out=150,in=210] (-1,5.5);
	\draw [color=\clr,  thick, directed=.35] (-1,3.5) to [out=30,in=330] (-1,5.5);
	\draw [color=\clr,  thick, rdirected=.65] (-1,3) to (-1.45,2.333);
	\draw [color=\clr,  thick, rdirected=.55] (-2.55,.67) to (-3,0);
	\draw [color=\clr,  thick, directed=.75] (-2.55,.67) to [out=116.3,in=176.3] (-1.45,2.333);
	\draw [color=\clr,  thick, directed=.75] (-2.55,.67) to [out=-3.7,in=296.3] (-1.45,2.333);
	\draw [color=\clr,  thick, rdirected=.65] (-1,3) to (-.55,2.333);
	\draw [color=\clr,  thick, rdirected=.55] (0.55,.67) to (1,0);
	\draw [color=\clr,  thick, directed=.75] (0.55,.67) to [out=187.7,in=243.7] (-.55,2.333);
	\draw [color=\clr,  thick, directed=.75] (0.55,.67) to [out=63.7,in=363.7] (-.55,2.333);
	\draw [color=\clr, dashed] (-3,0.85) rectangle (1,5);
	\node at (-1,6.35) {\scriptsize $h\! +\! l$};
	\node at (-3,-.35) {\scriptsize $h$};
	\node at (1,-.35) {\scriptsize $l$};
	\node at (-1.85,4.5) {\scriptsize $1$};
	\node at (-0.15,4.5) {\scriptsize $1$};
	\node at (-0.95,4.5) {\small $\cdots$};
	\node at (0.05,1.55) { \rotatebox{33.7}{\small $\cdots$}};
	\node at (-0.7,1.15) {\scriptsize $1$};
	\node at (0.7,1.85) {\scriptsize $1$};
	\node at (-2,1.45) { \rotatebox{-33.7}{\small $\cdots$}};
	\node at (-1.3,1.15) {\scriptsize $1$};
	\node at (-2.7,1.85) {\scriptsize $1$};
\end{tikzpicture}
};
\endxy.
\]
As discussed at the end of \cref{SS:UpwardWebs}, the web enclosed by the dashed rectangle above is $[h+l]_{q}!\,\Cl_{h+l}$.  Having used this technique at every merge and split, the only edges labelled by $k >1$ begin and end with explosions into strands of thickness $1$.   Moreover, like the one in the dashed rectangle above, these are scalar multiples of clasps.  By  \cref{L:claspsum} each such clasp  can be rewritten as a sum of upward crossings.   That is, we have rewritten $w$ as a linear combination of webs involving strands of thickness $1$ involving only upward crossings and dots, as desired.
\end{proof}


\begin{proposition}\label{P:sergeev-isomorphism}  For every $k\in\Z_{>0}$ the map 
\[
\xi_k:\HC_{k}(q)  \to  \End_{\qwebs}(\up_{1}^k)
\] of \cref{E:Sergeevhomomorphism} is a superalgebra isomorphism.  Moreover, the map obtained by composing this homomorphism with $\Psi_{n}:\End_{\qwebs}(\up_{1}^{k}) \to \End_{\fq (n)}\left(V^{\otimes k} \right)$ coincides with the map $\psi$ given in \cref{T:OlshanskiDuality}.
\end{proposition}

\begin{proof} As discussed before \cref{E:Sergeevhomomorphism}, $\xi_{k}$ gives a well-defined homomorphism.  By \cref{sergeev-generators} $\xi_k$ is surjective.  

We now prove injectivity.  Fix $n > k$. Since by \cref{E:tensorspaceisomorphism}  $\AA_{q, \varepsilon_{1}+\dotsb \varepsilon_{k}} \cong V_n^{\otimes k}$, the functor $\Psi_{n}$ induces a map of superalgebras $\End_{\qwebs}(\up_{1}^k) \to \End_{U_{q}(\q_{n})}(V_n^{\otimes k})$ which we call by the same name.  Taken together with the superalgebra map $\psi: \HC_k(q) \to \End_{U_{q}(\q_{n})}(V_n^{\otimes k})$ from \cref{T:OlshanskiDuality} we have the following diagram of superalgebra maps.
\[
\xy
(0,0)*{
\begin{tikzpicture}
\node (S) at (0,0) {$\HC_{k}(q)$};
\node (P) at (3,0) {$\End_{U_{q}(\q_{n})}(V_n^{\otimes k})$};
\node (W) at (0,-2) {$\End_{\qwebs}(\up_{1}^k)$};
\path[->]
(S) edge [above] node {$\psi$} (P)
(S) edge [left] node {$\xi_k$} (W)
(W) edge [below] node {\quad $\Psi_{n}$} (P);
\end{tikzpicture}
};
\endxy \ .
\]  As discussed in \cref{R:CompatibleWithOlshanski}, the calculations in \cref{L:functoronkeymorphisms} imply this diagram commmutes.   The injectivity of $\xi_{k}$ follows from the fact that $\psi$ is an isomorphism by \cref{T:OlshanskiDuality,P:OlshanskiKernel}.
\end{proof}

\subsection{The fullness of \texorpdfstring{$\Psi_{n}$}{Psin} }\label{S:fullness}

Given a tuple of integers $a=(a_{1}, \dotsc , a_{r})$, let $|a| = a_{1}+\dotsb +a_{r}$.

\begin{theorem}\label{T:fullness} For every $n \geq 1$, the functor $\Psi_{n}: \qwebs \to \mods$ is full.
\end{theorem}

\begin{proof}  Let $a=(a_{1}, \dotsc , a_{r})$ and $b=(b_{1}, \dotsc , b_{s})$ be tuples of nonnegative integers.  We first observe $\Hom_{\qwebs}(\uparrow_{a},\uparrow_{b})=0$ unless $|a|=|b|$.  Likewise, by weight considerations $\Hom_{U_{q}\left(\fq_{n} \right)}\left( S_{q}^{\uparrow_{a}}(V_{n}), S_{q}^{\uparrow_{b}}(V_{n})\right)=0$ unless $|a|=|b|$.  Thus we may assume $|a|=|b|$ in what follows.

There is a map of superspaces 
\[
\Hom_{\qwebs}\left(  \uparrow_{a}, \uparrow_{b}\right) \xrightarrow{\alpha} \Hom_{\qwebs}\left( \up_{1}^{|a|}, \up_{1}^{|b|}\right)
\]
given by pre- and post-composing with merges and splits to explode all the boundary strands:
\[
\xy
(0,0)*{
\begin{tikzpicture} [scale=1, color=\clr]
	\draw[ thick] (-1.5,1.25) rectangle (0.5,.75);
	\draw[ thick, directed=0.55] (-1.25,0) to (-1.25,.75);
	\draw[ thick, directed=0.55] (-1.25,1.25) to (-1.25,2);
	\draw[ thick, directed=0.55] (.25,0) to (.25,.75);
	\draw[ thick, directed=0.55] (.25,1.25) to (.25,2);
	\node at (-0.5,1.75) {$\cdots$};
	\node at (-0.5,.25) {$\cdots$};
	\node at (-0.5,1) {$u$};
	\node at (-1.55,1.675) {\scriptsize ${b}_1$};
	\node at (-1.55,0.325) {\scriptsize ${a}_1$};
	\node at (0.65,1.675) {\scriptsize ${b}_{s}$};
	\node at (0.65,0.325) {\scriptsize ${a}_r$};
\end{tikzpicture}
};
\endxy
\mapsto
\xy
(0,0)*{
\begin{tikzpicture} [scale=1, color=\clr]
	\draw[ thick] (-1.5,1.25) rectangle (0.5,.75);
	\draw[ thick, directed=0.55] (-1.25,0) to (-1.25,.75);
	\draw[ thick, directed=0.55] (-1.25,1.25) to (-1.25,2);
	\draw[ thick, directed=0.55] (.25,0) to (.25,.75);
	\draw[ thick, directed=0.55] (.25,1.25) to (.25,2);
	\node at (-0.5,1.75) {$\cdots$};
	\node at (-0.5,.25) {$\cdots$};
	\node at (-0.5,1) {$u$};
	\node at (-1.55,1.675) {\scriptsize ${b}_1$};
	\node at (-1.55,0.325) {\scriptsize ${a}_1$};
	\node at (0.65,1.675) {\scriptsize ${b}_{s}$};
	\node at (0.65,0.325) {\scriptsize ${a}_r$};
	\draw [ thick, directed=.55] (-1,-.5) to [out=90,in=330] (-1.25,0);
	\draw [ thick, rdirected=.65] (-1.25,0) to [out=210,in=90] (-1.5,-.5);
	\draw [ thick, rdirected=.65] (0.25,0) to [out=330,in=90] (0.5,-.5);
	\draw [ thick, rdirected=.65] (0.25,0) to [out=210,in=90] (0,-.5);
	\draw [ thick, directed=.55] (-1.25,2) to [out=30,in=270] (-1,2.5);
	\draw [ thick, directed=.55] (-1.25,2) to [out=150,in=270] (-1.5,2.5);
	\draw [ thick, directed=.55] (0.25,2) to [out=30,in=270] (0.5,2.5);
	\draw [ thick, directed=.55] (0.25,2) to [out=150,in=270] (0,2.5);
	\node at (-1.25,2.65) {\small \  $\cdots$};
	\node at (-1.25,-.65) {\small  \ $\cdots$};
	\node at (0.25,2.65) {\small  \ $\cdots$};
	\node at (0.25,-.65) {\small  \ $\cdots$};
	\node at (-1,2.65) {\scriptsize $1$};
	\node at (-1.5,2.65) {\scriptsize $1$};
	\node at (-1,-.65) {\scriptsize $1$};
	\node at (-1.5,-.65) {\scriptsize $1$};
	\node at (0,2.65) {\scriptsize $1$};
	\node at (0.5,2.65) {\scriptsize $1$};
	\node at (0,-.65) {\scriptsize $1$};
	\node at (0.5,-.65) {\scriptsize $1$};
\end{tikzpicture}
};
\endxy.
\]  This map is an embedding.  Up to a scalar the left inverse is given by applying a complementary set of splits and merges to rejoin the strands and applying relation \cref{digon-removal}.  Since the map is given by composing and tensoring diagrams, via the functor there is corresponding map $\tilde{\alpha}$ which makes the following  diagram of superspace maps commute:

\begin{equation}\label{E:PsiBox}
\begin{tikzcd}
\Hom_{\qwebs}(\uparrow_{a}, \uparrow_{b})  \arrow[r, hook, "\alpha"] \arrow[d,  "\Psi_{n}"] & \Hom_{\qwebs}\left( \up_{1}^{|{a}|}, \up_{1}^{|{b}|}\right) \arrow[d,  "\Psi_{n}"] \\
\Hom_{U_{q}\left( \fq_{n}\right)}\left(  S_{q}^{\uparrow_{a}}(V_{n}),S_{q}^{\uparrow_{b}}(V_{n})\right)\arrow[r, hook, "\tilde{\alpha}"] & \Hom_{U_{q}(\fq_{n})}\left( V_{n}^{\otimes |{a}|}, V_{n}^{\otimes |{b}|}\right).
\end{tikzcd}
\end{equation}
Furthermore, $e_{{a}}= \Cl_{a_{1}} \otimes \dotsb \otimes \Cl_{a_{r}}, e_{{b}} = \Cl_{b_{1}} \otimes \dotsb \otimes \Cl_{b_{s}} \in \End_{\qwebs}(\uparrow_{1}^{|{a}|})$  are idempotents, the image of $\alpha$ is precisely $e_{{b}}\Hom_{\qwebs}(\uparrow^{{a}}, \uparrow^{{b}})e_{{a}}$, and the image of $\tilde{\alpha}$ is $\Psi_{n}(e_{{b}})\Hom_{q(n)}(V^{{a}}, V^{{b}})\Psi_{n}(e_{{a}})$.  By \cref{P:sergeev-isomorphism} and \cref{T:OlshanskiDuality}  the $\Psi_{n}$ on the right side of \cref{E:PsiBox} is surjective.  This along with a diagram chase implies $\Psi_{n}$ is surjective on the left side of the diagram.  Together with the discussion of the first paragraph it follows $\Psi_{n}$ is a full functor. 
\end{proof} 

\section{Oriented Webs}\label{S:OrientedWebs}
We next introduce oriented webs.  On the representation theory side this corresponds to including  duals.

\subsection{Oriented Webs}

\begin{definition}\label{D:orientedwebs}
The category $\qwebsupdown$ is the monoidal supercategory with generating objects 
\[
\{\up_{k}, \down_k \mid k \in \Z_{\geq 0}\}
\]
and generating morphisms
 
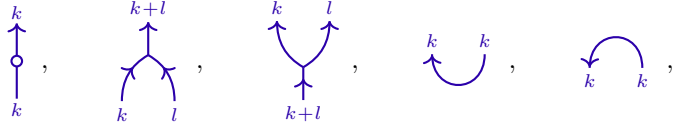
\[
\xy
(0,0)*{
\begin{tikzpicture}[color=\clr, scale=1]
	\draw[color=\clr, thick, directed=1] (1,0) to (1,1);
	\node at (1,-0.15) {\scriptsize $k$};
	\node at (1,1.15) {\scriptsize $k$};
	\draw (1,0.5) \wdot;
\end{tikzpicture}
};
\endxy \ ,\quad\quad
\xy
(0,0)*{
\begin{tikzpicture}[color=\clr, scale=.35]
	\draw [color=\clr,  thick, directed=1] (0, .75) to (0,2);
	\draw [color=\clr,  thick, directed=.65] (1,-1) to [out=90,in=330] (0,.75);
	\draw [color=\clr,  thick, directed=.65] (-1,-1) to [out=90,in=210] (0,.75);
	\node at (0, 2.5) {\scriptsize $k\! +\! l$};
	\node at (-1,-1.5) {\scriptsize $k$};
	\node at (1,-1.5) {\scriptsize $l$};
\end{tikzpicture}
};
\endxy \ ,\quad\quad
\xy
(0,0)*{
\begin{tikzpicture}[color=\clr, scale=.35]
	\draw [color=\clr,  thick, directed=.65] (0,-0.5) to (0,.75);
	\draw [color=\clr,  thick, directed=1] (0,.75) to [out=30,in=270] (1,2.5);
	\draw [color=\clr,  thick, directed=1] (0,.75) to [out=150,in=270] (-1,2.5); 
	\node at (0, -1) {\scriptsize $k\! +\! l$};
	\node at (-1,3) {\scriptsize $k$};
	\node at (1,3) {\scriptsize $l$};
\end{tikzpicture}
};
\endxy \ , \quad \quad 
\xy
(0,0)*{
\bt[scale=.35, color=\clr]
	\draw [ thick, looseness=2, directed=0.99] (1,2.5) to [out=270,in=270] (-1,2.5);
	\node at (-1,3) {\scriptsize $k$};
	\node at (1,3) {\scriptsize $k$};
\et
};
\endxy \ ,\quad\quad
\xy
(0,0)*{
\bt[scale=.35, color=\clr]
	\draw [ thick, looseness=2, directed=0.99] (1,2.5) to [out=90,in=90] (-1,2.5);
	\node at (-1,2) {\scriptsize $k$};
	\node at (1,2) {\scriptsize $k$};
\et
};
\endxy \ ,
\]
for all $k,l\in\Z_{>0}$.  We call these the \emph{dot}, \emph{merge}, \emph{split},  \emph{leftward cup}, and \emph{leftward cap}, respectively.  The parity is given by declaring the dot to be odd and the other generating morphisms to be even.

The relations imposed on the generators of $\qwebsupdown$ are the relations of $\qwebs$, declaring the morphisms given in \cref{E:leftwardcrossing} are invertible, and declaring equations \cref{straighten-zigzag,delete-bubble} hold.  
\end{definition}


The first relation is the following straightening rules for cups and caps:
\begin{equation}\label{straighten-zigzag}
\xy
(0,0)*{\reflectbox{\rotatebox{180}{
\bt[scale=1.25, color=\clr]
	\draw [ thick, looseness=2, ] (1,0) to [out=90,in=90] (0.5,0);
	\draw [ thick, looseness=2, ] (1.5,0) to [out=270,in=270] (1,0);
	\draw[ thick, directed=1] (0.5,0) to (0.5,-0.5);
	\draw[ thick, ] (1.5,0) to (1.5,0.5);
	\node at (0.5,-0.65) {\scriptsize \reflectbox{\rotatebox{180}{$k$}}};
	\node at (1.5,0.65) {\scriptsize \reflectbox{\rotatebox{180}{$k$}}};
\et
}}};
\endxy=
\xy
(0,0)*{
\bt[scale=1.25, color=\clr]
	\draw[thick, directed=1] (1,0) to (1,1);
	\node at (1,-0.15) {\scriptsize $k$};
	\node at (1,1.15) {\scriptsize $k$};
\et
};
\endxy \ \quad\quad\text{and}\quad\quad
\xy
(0,0)*{
\bt[scale=1.25, color=\clr]
	\draw [ thick, looseness=2, ] (1,0) to [out=90,in=90] (0.5,0);
	\draw [ thick, looseness=2, ] (1.5,0) to [out=270,in=270] (1,0);
	\draw[ thick, directed=1] (0.5,0) to (0.5,-0.5);
	\draw[ thick, ] (1.5,0) to (1.5,0.5);
	\node at (0.5,-0.65) {\scriptsize $k$};
	\node at (1.5,0.65) {\scriptsize $k$};
\et
};
\endxy=
\xy
(0,0)*{
\bt[scale=1.25,  color=\clr]
	\draw[thick, directed=1] (1,1) to (1,0);
	\node at (1,-0.15) {\scriptsize $k$};
	\node at (1,1.15) {\scriptsize $k$};
\et
};
\endxy \ .
\end{equation}

To state the second relation, we first define the \emph{leftward under-crossing} and \emph{leftward over-crossing} in $\qwebsupdown$ for all $k, l \geq 1$, respectively,  by 
\begin{equation}\label{E:leftwardcrossing}
\xy
(0,0)*{
\bt[scale=1.25, color=\clr]
	\draw[thick, rdirected=0.2] (0,0) to (0.2,0.2);
	\draw[thick ] (0.3,0.3) to (0.5,0.5);
	\draw[thick, directed=1] (0.5,0) to (0,0.5);
	\node at (0,-0.15) {\scriptsize $k$};
	\node at (0,0.65) {\scriptsize $l$};
	\node at (0.5,-0.15) {\scriptsize $l$};
	\node at (0.5,0.65) {\scriptsize $k$};
\et
};
\endxy := \ 
\xy
(0,0)*{
\bt[scale=1.25, color=\clr]
	\draw[thick, ] (0,0) to (0.5,0.5);
	\draw [ thick, directed=1] (0.5,0.5) to (0.5,1);
	\draw[thick, ] (0.5,0) to (0.3,0.2);
	\draw[thick, ] (0.2,0.3) to (0,0.5);
	\draw [ thick, looseness=2, ] (1,0) to [out=270,in=270] (0.5,0);
	\draw [thick, ] (1,1) to (1,0);
	\draw [ thick, ] (0,-0.5) to (0,0);
	\draw [ thick, looseness=2, ] (0,0.5) to [out=90,in=90] (-0.5,0.5);
	\draw [ thick, directed=1] (-0.5,0.5) to (-0.5,-0.5);
	\node at (-0.5,-0.65) {\scriptsize $k$};
	\node at (0,-0.65) {\scriptsize $l$};
	\node at (0.5,1.15) {\scriptsize $l$};
	\node at (1,1.15) {\scriptsize $k$};
\et
};
\endxy \ , \quad\quad
\xy
(0,0)*{
\bt[scale=1.25, color=\clr]
	\draw[thick, rdirected=0.1] (0,0) to (0.5,0.5);
	\draw[thick ] (0.5,0) to (0.33,0.18);
	\draw[thick, directed=1] (0.18,.33) to (0,0.5);
	\node at (0,-0.15) {\scriptsize $k$};
	\node at (0,0.65) {\scriptsize $l$};
	\node at (0.5,-0.15) {\scriptsize $l$};
	\node at (0.5,0.65) {\scriptsize $k$};
\et
};
\endxy := \ 
\xy
(0,0)*{
\bt[scale=1.25, color=\clr]
	\draw[thick, ] (0,0) to (0.2,0.2);
	\draw[thick, ] (0.3,0.3) to (0.5,0.5);
	\draw [ thick, directed=1] (0.5,0.5) to (0.5,1);
	\draw[thick, ] (0.5,0) to (0,0.5);
	\draw [ thick, looseness=2, ] (1,0) to [out=270,in=270] (0.5,0);
	\draw [thick, ] (1,1) to (1,0);
	\draw [ thick, ] (0,-0.5) to (0,0);
	\draw [ thick, looseness=2, ] (0,0.5) to [out=90,in=90] (-0.5,0.5);
	\draw [ thick, directed=1] (-0.5,0.5) to (-0.5,-0.5);
	\node at (-0.5,-0.65) {\scriptsize $k$};
	\node at (0,-0.65) {\scriptsize $l$};
	\node at (0.5,1.15) {\scriptsize $l$};
	\node at (1,1.15) {\scriptsize $k$};
\et
};
\endxy
.
\end{equation}  We impose on $\qwebsupdown$ the requirement every leftward under- and over-crossing be invertible. In other words, for every $k,l \geq 1$ we assume the existence of two additional generating morphisms of type $\up_{l}\down_{k} \to \down_{k}\up_{l}$ which are two-sided inverses under composition to the leftward over- and under-crossing morphisms, respectively.  We diagrammatically represent these inverses by
\begin{equation}\label{E:rightwardcrossing}
\xy
(0,0)*{
	\bt[yscale=1.25, xscale = -1.25, color=\clr]
	\draw[thick, rdirected=0.2] (0,0) to (0.2,0.2);
	\draw[thick ] (0.3,0.3) to (0.5,0.5);
	\draw[thick, directed=1] (0.5,0) to (0,0.5);
	\node at (0,-0.15) {\scriptsize $k$};
	\node at (0,0.65) {\scriptsize $l$};
	\node at (0.5,-0.15) {\scriptsize $l$};
	\node at (0.5,0.65) {\scriptsize $k$};
	\et
};
\endxy := \left( 
\xy
(0,0)*{
	\bt[yscale=1.25, xscale = 1.25, color=\clr]
	\draw[thick, rdirected=0.2] (0,0) to (0.2,0.2);
	\draw[thick ] (0.3,0.3) to (0.5,0.5);
	\draw[thick, directed=1] (0.5,0) to (0,0.5);
	\node at (0,-0.15) {\scriptsize $k$};
	\node at (0,0.65) {\scriptsize $l$};
	\node at (0.5,-0.15) {\scriptsize $l$};
	\node at (0.5,0.65) {\scriptsize $k$};
	\et
};
\endxy \right)^{-1}  , \quad\quad
\xy
(0,0)*{
	\bt[yscale=1.25, xscale = -1.25, color=\clr]
	\draw[thick, rdirected=0.1] (0,0) to (0.5,0.5);
	\draw[thick ] (0.5,0) to (0.33,0.18);
	\draw[thick, directed=1] (0.18,.33) to (0,0.5);
	\node at (0,-0.15) {\scriptsize $k$};
	\node at (0,0.65) {\scriptsize $l$};
	\node at (0.5,-0.15) {\scriptsize $l$};
	\node at (0.5,0.65) {\scriptsize $k$};
	\et
};
\endxy := \ \left(
\xy
(0,0)*{
	\bt[scale=1.25, color=\clr]
	\draw[thick, rdirected=0.1] (0,0) to (0.5,0.5);
	\draw[thick ] (0.5,0) to (0.33,0.18);
	\draw[thick, directed=1] (0.18,.33) to (0,0.5);
	\node at (0,-0.15) {\scriptsize $k$};
	\node at (0,0.65) {\scriptsize $l$};
	\node at (0.5,-0.15) {\scriptsize $l$};
	\node at (0.5,0.65) {\scriptsize $k$};
	\et
};
\endxy 
\right)^{-1},
\end{equation}
and we refer to these morphism as \textit{rightward over-crossing} and \textit{rightward under-crossings}, respectively.

Using these rightward crossings we define the \emph{rightward cap} and \emph{rightward cup} for every $k \geq 1$  by

\begin{equation}\label{rightward-cup-and-cap}
\xy
(0,0)*{
\bt[scale=.35, color=\clr]
	\draw [ thick, looseness=2, rdirected=-0.95] (1,4.5) to [out=90,in=90] (-1,4.5);
	\node at (-1.75,4.75) {\scriptsize $k$};
\et
};
\endxy := q^{k(k-1)}
\xy
(0,0)*{
\bt[scale=.35, color=\clr]
	\draw [ thick, looseness=2, directed=0.99] (1,4.5) to [out=90,in=90] (-1,4.5);
	\draw [ thick, directed=1 ] (-1,2.5) to (1,4.5);
	\draw [ thick, rdirected=-.85 ] (1,2.5) to (.25,3.25);
	\draw [ thick ] (-.25,3.75) to (-1,4.5);
	\node at (-1.75,2.55) {\scriptsize $k$};
\et
};
\endxy \   \quad\quad\text{and}\quad\quad
\xy
(0,0)*{
\bt[scale=.35, color=\clr]
	\draw [ thick, looseness=2, rdirected=-0.95 ] (1,2.5) to [out=270,in=270] (-1,2.5);
	\node at (-1.75,2.5) {\scriptsize $k$};
\et
};
\endxy := q^{-k(k-1)}
\xy
(0,0)*{
\bt[scale=.35, color=\clr]
	\draw [ thick, looseness=2, directed=0.99 ] (1,2.5) to [out=270,in=270] (-1,2.5);
	\draw [ thick,  ] (-1,2.5) to (-.25,3.25);
	\draw [ thick, directed=1 ] (.25,3.75) to (1,4.5);
	\draw [ thick, rdirected=-0.90] (1,2.5) to (-1,4.5);
	\node at (-1.75,4.75) {\scriptsize $k$};
\et
};
\endxy \ .
\end{equation} 
  The final relation we impose on $\qwebsupdown$ is that the clockwise bubbles of thickness $1$, with and without a dot, vanish:

\begin{equation}\label{delete-bubble}
\xy
(0,0)*{
\bt[scale=.35, color=\clr]
	\draw [ thick, looseness=2, ] (1,2.5) to [out=270,in=270] (-1,2.5);
	\draw [ thick, looseness=2, rdirected=0.05] (1,2.5) to [out=90,in=90] (-1,2.5);
	\node at (-1.5,3) {\scriptsize $1$};
\et
};
\endxy = 0 =
\xy
(0,0)*{
\bt[scale=.35, color=\clr]
	\draw [ thick, looseness=2, ] (1,2.5) to [out=270,in=270] (-1,2.5);
	\draw [ thick, looseness=2, rdirected=0.05] (1,2.5) to [out=90,in=90] (-1,2.5);
	\draw (-1,2.5) \wdot  ;
	\node at (-1.5,3) {\scriptsize $1$};
\et
};
\endxy.
\end{equation}

\subsection{Additional Morphisms and Relations}\label{SS:OrientedAdditionalRelations}

We define the downward dot, merge, and split as

\begin{equation*}
\xy
(0,0)*{
\bt[scale=1, color=\clr]
	\draw[thick, directed=1] (1,1) to (1,0);
	\node at (1,-0.15) {\scriptsize $k$};
	\node at (1,1.15) {\scriptsize $k$};
	\draw (1,0.5) \wdot;
\et
};
\endxy:=
\xy
(0,0)*{
\bt[scale=1, color=\clr]
	\draw [ thick, looseness=2, ] (1,0.25) to [out=90,in=90] (0.5,0.25);
	\draw [ thick, looseness=2, ] (1.5,-0.25) to [out=270,in=270] (1,-0.25);
	\draw [ thick, ] (1,-0.25) to (1,0.25);
	\draw[ thick, directed=1] (0.5,0.25) to (0.5,-0.75);
	\draw[ thick, ] (1.5,-0.25) to (1.5,0.75);
	\draw (1,0) \wdot;
	\node at (0.5,-0.9) {\scriptsize $k$};
	\node at (1.5,0.9) {\scriptsize $k$};
\et
};
\endxy \ , \hspace{.2in}
%
%
\xy
(0,0)*{
\bt[scale=.25, color=\clr]
	\draw [ thick, directed=1] (0,0.75) to (0,-0.5);
	\draw [ thick, rdirected=.45] (0,.75) to [out=30,in=270] (1,2.5);
	\draw [ thick, rdirected=.45] (0,.75) to [out=150,in=270] (-1,2.5); 
	\node at (0, -1) {\scriptsize $k\! +\! l$};
	\node at (-1,3) {\scriptsize $k$};
	\node at (1,3) {\scriptsize $l$};
\et
};
\endxy \ := 
\xy
(0,0)*{
\bt[scale=.25, color=\clr]
	\draw [ thick, ] (0, .75) to (0,1.5);
	\draw [ thick, looseness=2, ] (0,1.5) to [out=90,in=90] (-2,1.5);
	\draw [ thick, directed=1] (-2, 1.5) to (-2,-4);
	\draw [ thick, directed=.65] (1,-1) to [out=90,in=330] (0,.75);
	\draw [ thick, directed=.65] (-1,-1) to [out=90,in=210] (0,.75);
	\draw [ thick, looseness=2, ] (1,-1) to [out=270,in=270] (3,-1);
	\draw [ thick, looseness=1.5, ] (-1,-1) to [out=270,in=270] (5,-1);
	\draw [ thick, ] (3,-1) to (3,3);
	\draw [ thick, ] (5,-1) to (5,3);
	\node at (-2, -4.5) {\scriptsize $k\! +\! l$};
	\node at (3,3.5) {\scriptsize $k$};
	\node at (5,3.5) {\scriptsize $l$};
\et
};
\endxy,  \hspace{.2in}  \text{and}  \hspace{.2in}
\xy
(0,0)*{
\bt[scale=.25, color=\clr]
	\draw [ thick, rdirected=.55] (0, .75) to (0,2);
	\draw [ thick, rdirected=0.1] (1,-1) to [out=90,in=330] (0,.75);
	\draw [ thick, rdirected=0.1] (-1,-1) to [out=90,in=210] (0,.75);
	\node at (0, 2.5) {\scriptsize $k\! +\! l$};
	\node at (-1,-1.75) {\scriptsize $k$};
	\node at (1,-1.75) {\scriptsize $l$};
\et
};
\endxy  :=
\xy
(0,0)*{
\bt[scale=.25, color=\clr]
	\draw [ thick, directed=0.75] (0,0) to (0,.75);
	\draw [ thick, looseness=2, ] (0,0) to [out=270,in=270] (2,0);
	\draw [ thick, ] (2,0) to (2,5);
	\draw [ thick, ] (0,.75) to [out=30,in=270] (1,2.5);
	\draw [ thick, ] (0,.75) to [out=150,in=270] (-1,2.5); 
	\draw [ thick, looseness=2, ] (-1,2.5) to [out=90,in=90] (-3,2.5);
	\draw [ thick, looseness=1.5, ] (1,2.5) to [out=90,in=90] (-5,2.5);
	\draw [ thick, directed=1] (-3,2.5) to (-3,-1);
	\draw [ thick, directed=1] (-5,2.5) to (-5,-1);
	\node at (2, 5.5) {\scriptsize $k\! +\! l$};
	\node at (-3,-1.5) {\scriptsize $l$};
	\node at (-5,-1.5) {\scriptsize $k$};
\et
};
\endxy \ . 
\end{equation*}
Applying leftward cups and caps to these and using \cref{straighten-zigzag} to simplify shows dots, merges, and splits move freely over leftward cups and caps in the obvious sense.

We define downward over- and under-crossings as

\begin{equation*}
\xy
(0,0)*{
\bt[scale=1.25, color=\clr]
	\draw[thick, rdirected=0.1] (0,0) to (0.5,0.5);
	\draw[thick, rdirected=0.2] (0.5,0) to (0.3,0.2);
	\draw[thick, ] (0.2,0.3) to (0,0.5);
	\node at (0,-0.15) {\scriptsize $k$};
	\node at (0,0.65) {\scriptsize $l$};
	\node at (0.5,-0.15) {\scriptsize $l$};
	\node at (0.5,0.65) {\scriptsize $k$};
\et
};
\endxy := \ 
\xy
(0,0)*{
\bt[scale=1, color=\clr]
	\draw[thick, ] (0,0) to (0.5,0.5);
	\draw[thick, ] (0.5,0) to (0.3,0.2);
	\draw[thick, ] (0.2,0.3) to (0,0.5);
	\draw [ thick, looseness=2, ] (1,0) to [out=270,in=270] (0.5,0);
	\draw [ thick, looseness=1.5, ] (1.5,0) to [out=270,in=270] (0,0);
	\draw [thick, ] (1,1.25) to (1,0);
	\draw [thick, ] (1.5,1.25) to (1.5,0);
	\draw [ thick, looseness=2, ] (0,0.5) to [out=90,in=90] (-0.5,0.5);
	\draw [ thick, looseness=1.5, ] (0.5,0.5) to [out=90,in=90] (-1,0.5);
	\draw [ thick, directed=1] (-0.5,0.5) to (-0.5,-0.75);
	\draw [ thick, directed=1] (-1,0.5) to (-1,-0.75);
	\node at (-1,-0.9) {\scriptsize $k$};
	\node at (-0.5,-0.9) {\scriptsize $l$};
	\node at (1,1.4) {\scriptsize $l$};
	\node at (1.5,1.4) {\scriptsize $k$};
\et
};
\endxy \hspace{.2in} \text{and} \hspace{.2in}
\xy
(0,0)*{
\bt[scale=1.25, color=\clr]
	\draw[thick, rdirected=0.1] (0,0) to (0.2,0.2);
	\draw[thick, ] (0.3,0.3) to (0.5,0.5);
	\draw[thick, rdirected=0.1] (0.5,0) to (0,0.5);
	\node at (0,-0.15) {\scriptsize $k$};
	\node at (0,0.65) {\scriptsize $l$};
	\node at (0.5,-0.15) {\scriptsize $l$};
	\node at (0.5,0.65) {\scriptsize $k$};
\et
};
\endxy := \ 
\xy
(0,0)*{
\bt[scale=1, color=\clr]
	\draw[thick, ] (0,0) to (0.2,0.2);
	\draw[thick, ] (0.3,0.3) to (0.5,0.5);
	\draw[thick, ] (0.5,0) to (0,0.5);
	\draw [ thick, looseness=2, ] (1,0) to [out=270,in=270] (0.5,0);
	\draw [ thick, looseness=1.5, ] (1.5,0) to [out=270,in=270] (0,0);
	\draw [thick, ] (1,1.25) to (1,0);
	\draw [thick, ] (1.5,1.25) to (1.5,0);
	\draw [ thick, looseness=2, ] (0,0.5) to [out=90,in=90] (-0.5,0.5);
	\draw [ thick, looseness=1.5, ] (0.5,0.5) to [out=90,in=90] (-1,0.5);
	\draw [ thick, directed=1] (-0.5,0.5) to (-0.5,-0.75);
	\draw [ thick, directed=1] (-1,0.5) to (-1,-0.75);
	\node at (-1,-0.9) {\scriptsize $k$};
	\node at (-0.5,-0.9) {\scriptsize $l$};
	\node at (1,1.4) {\scriptsize $l$};
	\node at (1.5,1.4) {\scriptsize $k$};
\et
};
\endxy.
\end{equation*}
Having over- and under-crossings of strands of arbitrary orientation and labels we can mimic \cref{D:generalupbraiding} to define mutually inverse over- and under-crossing morphisms for arbitrary pairs of objects in $\qwebsupdown$.  Using these definitions it is straightforward to verify the relations in \cref{L:braidingforups} hold for all possible orientations.  As merges and splits go through upward crossings by \cref{L:braidingforups}(c.), this implies they go through leftward, rightward, and downward  crossings as well.  

Also, by applying leftward cups and caps to  \cref{associativity,digon-removal,dot-collision,dots-past-merges,dumbbell-relation,square-switch,square-switch-dots,double-rungs-1,double-rungs-2} one obtains a parallel set of relations on downward oriented diagrams.  We freely use these in what follows. When computing these the reader is advised to keep in mind the effect of the super interchange law when diagrams have multiple dots.  For example, \cref{dot-collision} becomes

\begin{equation}\label{reverse-dot-collision}
\xy
(0,0)*{
\begin{tikzpicture}[scale=.3, color=\clr] 
	\draw [ thick, rdirected=.05] (1,-2.75) to (1,2.5);
	\node at (1,3) {\scriptsize $k$};
	\node at (1,-3.25) {\scriptsize $k$};
	\draw (1,0.5) \wdot ;
	\draw (1,-0.5) \wdot ;
\end{tikzpicture}
};
\endxy=-[k]_{q^{2}}
\xy
(0,0)*{
\begin{tikzpicture}[scale=.3, color=\clr] 
	\draw [ thick, rdirected=.05] (1,-2.75) to (1,2.5);
	\node at (1,3) {\scriptsize $k$};
	\node at (1,-3.25) {\scriptsize $k$};
\end{tikzpicture}
};
\endxy.
\end{equation}

Since dots move freely across leftward cups and caps, and dots also move freely along the over strand in upward crossings by \cref{L:braidingforups}(d) it follows dots freely travel along the over-strand in leftward, rightward, and downward crossings as well.  However, we remind the reader when a dot moves through a crossing by traveling along the under-strand there is the Hecke-Clifford relation so some care must be taken when moving dots around web diagrams. 

For calculations it is useful to know the following relations involving strands of thickness $1$.
\begin{lemma}\label{L:thicknessonerelations}  The following relations hold in $\qwebsupdown$.  All edges are labelled by $1$.
\begin{enumerate}
\item 
\begin{equation}\label{E:inverseleftwardcrossing}
 \xy
(0,0)*{
\begin{tikzpicture}[scale=.3, color=\clr]
	\draw [ thick, <-] (-1,-1) to (1,1);
	\draw [ thick, ->] (-0.25,0.25) to (-1,1);
	\draw [ thick, -] (0.25,-0.25) to (1,-1);
	\end{tikzpicture}
};
\endxy
 \ = \ 
\xy
(0,0)*{\reflectbox{
\begin{tikzpicture}[scale=.3, color=\clr]
	\draw [ thick, ->] (-1,-1) to (1,1);
	\draw [ thick, -] (-0.25,0.25) to (-1,1);
	\draw [ thick, ->] (0.25,-0.25) to (1,-1);
	\end{tikzpicture}
}};
\endxy
 \ - \tq \ 
 \xy
(0,0)*{
\begin{tikzpicture}[scale=.3, color=\clr]
	\draw [ thick, <-, looseness=1.5] (-1,1) to [out=270, in=270] (1,1);
	\draw [ thick, <-, looseness=1.5] (-1,-1) to [out=90, in=90] (1,-1);
	\end{tikzpicture}
};
\endxy \ .
\end{equation}

\item 
\begin{equation}\label{E:inverserightwardcrossing}
 \xy
(0,0)*{
\begin{tikzpicture}[scale=.3, color=\clr]
	\draw [ thick, ->] (-1,-1) to (1,1);
	\draw [ thick, -] (-0.25,0.25) to (-1,1);
	\draw [ thick, ->] (0.25,-0.25) to (1,-1);
	\end{tikzpicture}
};
\endxy
 \ = \ 
\xy
(0,0)*{\reflectbox{
\begin{tikzpicture}[scale=.3, color=\clr]
	\draw [ thick, <-] (-1,-1) to (1,1);
	\draw [ thick, ->] (-0.25,0.25) to (-1,1);
	\draw [ thick, -] (0.25,-0.25) to (1,-1);
	\end{tikzpicture}
}};
\endxy
 \ - \tq  \ 
 \xy
(0,0)*{
\begin{tikzpicture}[scale=.3, color=\clr]
	\draw [ thick, ->, looseness=1.5] (-1,1) to [out=270, in=270] (1,1);
	\draw [ thick, ->, looseness=1.5] (-1,-1) to [out=90, in=90] (1,-1);
	\end{tikzpicture}
};
\endxy \ .
\end{equation}

\item 
\begin{equation}\label{E:inversedownwardcrossing}
\xy
(0,0)*{\reflectbox{
\begin{tikzpicture}[scale=.3, color=\clr]
	\draw [ thick, <-] (-1,-1) to (1,1);
	\draw [ thick, -] (-0.25,0.25) to (-1,1);
	\draw [ thick, ->] (0.25,-0.25) to (1,-1);
	\end{tikzpicture}
}};
\endxy
 \ = \ 
 \xy
(0,0)*{
\begin{tikzpicture}[scale=.3, color=\clr]
	\draw [ thick, <-] (-1,-1) to (1,1);
	\draw [ thick, -] (-0.25,0.25) to (-1,1);
	\draw [ thick, ->] (0.25,-0.25) to (1,-1);
	\end{tikzpicture}
};
\endxy
 \ - \tq  \ 
 \xy
(0,0)*{
\begin{tikzpicture}[scale=.3, color=\clr]
	\draw [ thick, ->] (-.5,1) to  (-.5,-1);
	\draw [ thick, ->] (.5,1) to (.5,-1);
	\end{tikzpicture}
};
\endxy \ .
\end{equation}

\item  
\begin{equation}\label{E:leftcircleiszero}
\begin{tikzpicture}[scale=.3, color=\clr,baseline = 0]
	\draw[thick, -, looseness=1.5] (-1,0) to [out=270, in=270] (1,0);
	\draw[thick, <-, looseness=1.5] (-1,0) to [out=90, in=90] (1,0);
\end{tikzpicture} = 0  \ .
\end{equation}

\item   
\begin{equation}\label{thickness-one-rightward-cup-and-cap}
\xy
(0,0)*{
\bt[scale=.35, color=\clr]
	\draw [ thick, looseness=2, directed=0.99] (1,4.5) to [out=90,in=90] (-1,4.5);
	\draw [ thick, directed=1 ] (0.15,3.65) to (1,4.5);
	\draw [ thick,  ] (-1,2.5) to (-0.15,3.35);
	\draw [ thick, rdirected=-.9 ] (1,2.5) to (-1,4.5);
\et
};
\endxy 
 \ = \
\xy
(0,0)*{
\bt[scale=.35, color=\clr]
	\draw [ thick, looseness=2, rdirected=-0.95] (1,4.5) to [out=90,in=90] (-1,4.5);
\et
};
\endxy \   \quad\quad\text{and}\quad\quad
\xy
(0,0)*{
\bt[scale=.35, color=\clr]
	\draw [ thick, looseness=2, directed=0.99 ] (1,2.5) to [out=270,in=270] (-1,2.5);
	\draw [ thick, directed=1 ] (-1,2.5) to (1,4.5);
	\draw [ thick, rdirected=-0.85] (1,2.5) to (0.15,3.35);
	\draw [ thick] (-0.15,3.65) to (-1,4.5);
\et
};
\endxy
 \ = \
\xy
(0,0)*{
\bt[scale=.35, color=\clr]
	\draw [ thick, looseness=2, rdirected=-0.95 ] (1,2.5) to [out=270,in=270] (-1,2.5);
\et
};
\endxy .
\end{equation}

\item  
\begin{equation}\label{E:leftcirclewithdotiszero}
\begin{tikzpicture}[scale=.3, color=\clr,baseline = 0]
	\draw[thick, -, looseness=1.5] (-1,0) to [out=270, in=270] (1,0);
	\draw[thick, <-, looseness=1.5] (-1,0) to [out=90, in=90] (1,0);
	\draw (1,0.15) \wdot ;
\end{tikzpicture} = 0  \ .
\end{equation}

\item 
\begin{equation}\label{E:dotthroughupupwardcrossing}
\xy
(0,0)*{\reflectbox{
\begin{tikzpicture}[scale=.3, color=\clr]
	\draw [ thick, -] (-1,-1) to (-0.25,-0.25);
	\draw [ thick, ->] (0.25,0.25) to (1,1);
	\draw [ thick, ->] (1,-1) to (-1,1);
	\draw (0.55,0.55) \wdot ;
	\end{tikzpicture}
}};
\endxy
 \ = \ 
\xy
(0,0)*{\reflectbox{
\begin{tikzpicture}[scale=.3, color=\clr]
	\draw [ thick, -] (-1,-1) to (-0.25,-0.25);
	\draw [ thick, ->] (0.25,0.25) to (1,1);
	\draw [ thick, ->] (1,-1) to (-1,1);
	\draw (-0.55,-0.55) \wdot ;
	\end{tikzpicture}
}};
\endxy
 \ + \tq  \ 
\left( 
 \xy
(0,0)*{
\begin{tikzpicture}[scale=.3, color=\clr]
	\draw [ thick, ->] (-.5,-1) to (-.5,1);
	\draw [ thick, ->] (.5,-1) to (.5,1);
	\draw (-.5,0) \wdot ;
	\end{tikzpicture}
};
\endxy
\ - \
  \xy
(0,0)*{
\begin{tikzpicture}[scale=.3, color=\clr]
	\draw [ thick, ->] (-.5,-1) to (-.5,1);
	\draw [ thick, ->] (.5,-1) to (.5,1);
	\draw (.5,0) \wdot ;
	\end{tikzpicture}
};
\endxy
 \right) \ .
\end{equation}

\begin{equation}\label{E:dotthroughupupwardcrossing2}
\xy
(0,0)*{\reflectbox{
\begin{tikzpicture}[scale=.3, color=\clr]
	\draw [ thick, ->] (-1,-1) to (1,1);
	\draw [ thick, ->] (-0.25,0.25) to (-1,1);
	\draw [ thick, -] (0.25,-0.25) to (1,-1);
	\draw (-0.55,0.55) \wdot ;
	\end{tikzpicture}
}};
\endxy
 \ = \ 
\xy
(0,0)*{\reflectbox{
\begin{tikzpicture}[scale=.3, color=\clr]
	\draw [ thick, ->] (-1,-1) to (1,1);
	\draw [ thick, ->] (-0.25,0.25) to (-1,1);
	\draw [ thick, -] (0.25,-0.25) to (1,-1);
	\draw (0.55,-0.55) \wdot ;
	\end{tikzpicture}
}};
\endxy
 \ + \tq  \ 
\left( 
 \xy
(0,0)*{
\begin{tikzpicture}[scale=.3, color=\clr]
	\draw [ thick, ->] (-.5,-1) to (-.5,1);
	\draw [ thick, ->] (.5,-1) to (.5,1);
	\draw (-.5,0) \wdot ;
	\end{tikzpicture}
};
\endxy
\ - \
  \xy
(0,0)*{
\begin{tikzpicture}[scale=.3, color=\clr]
	\draw [ thick, ->] (-.5,-1) to (-.5,1);
	\draw [ thick, ->] (.5,-1) to (.5,1);
	\draw (.5,0) \wdot ;
	\end{tikzpicture}
};
\endxy
 \right) \ .
\end{equation}

\item 
\begin{equation}\label{E:dotthroughleftwardcrossing}
\xy
(0,0)*{\reflectbox{
\begin{tikzpicture}[scale=.3, color=\clr]
	\draw [ thick, ->] (-1,-1) to (1,1);
	\draw [ thick, -] (-0.25,0.25) to (-1,1);
	\draw [ thick, ->] (0.25,-0.25) to (1,-1);
	\draw (0.55,-0.55) \wdot ;
	\end{tikzpicture}
}};
\endxy
 \ = \ 
\xy
(0,0)*{\reflectbox{
\begin{tikzpicture}[scale=.3, color=\clr]
	\draw [ thick, ->] (-1,-1) to (1,1);
	\draw [ thick, -] (-0.25,0.25) to (-1,1);
	\draw [ thick, ->] (0.25,-0.25) to (1,-1);
	\draw (-0.6,0.6) \wdot ;
	\end{tikzpicture}
}};
\endxy
 \ + \tq  \ 
\left( 
 \xy
(0,0)*{
\begin{tikzpicture}[scale=.3, color=\clr]
	\draw [ thick, <-, looseness=1.8] (-1,1.25) to [out=270, in=270] (1,1.25);
	\draw [ thick, <-, looseness=1.8] (-1,-1.25) to [out=90, in=90] (1,-1.25);
	\draw (-0.85,-0.55) \wdot ;
	\end{tikzpicture}
};
\endxy
\ - \
  \xy
(0,0)*{
\begin{tikzpicture}[scale=.3, color=\clr]
	\draw [ thick, <-, looseness=1.8] (-1,1.25) to [out=270, in=270] (1,1.25);
	\draw [ thick, <-, looseness=1.8] (-1,-1.25) to [out=90, in=90] (1,-1.25);
	\draw (0.85,0.55) \wdot ;
	\end{tikzpicture}
};
\endxy
 \right) \ .
\end{equation}
\begin{equation}\label{E:dotthroughleftwardcrossing2}
\xy
(0,0)*{\reflectbox{
\begin{tikzpicture}[scale=.3, color=\clr]
	\draw [ thick, -] (-1,-1) to (-.25,-.25);
	\draw [ thick, ->] (.25,.25) to (1,1);
	\draw [ thick, ->] (-1,1) to (1,-1);
	\draw (0.55,0.55) \wdot ;
	\end{tikzpicture}
}};
\endxy
 \ = \ 
\xy
(0,0)*{\reflectbox{
\begin{tikzpicture}[scale=.3, color=\clr]
	\draw [ thick, -] (-1,-1) to (-.25,-.25);
	\draw [ thick, ->] (.25,.25) to (1,1);
	\draw [ thick, ->] (-1,1) to (1,-1);
	\draw (-0.6,-0.6) \wdot ;
	\end{tikzpicture}
}};
\endxy
 \ + \tq  \ 
\left( 
 \xy
(0,0)*{
\begin{tikzpicture}[scale=.3, color=\clr]
	\draw [ thick, <-, looseness=1.8] (-1,1.25) to [out=270, in=270] (1,1.25);
	\draw [ thick, <-, looseness=1.8] (-1,-1.25) to [out=90, in=90] (1,-1.25);
	\draw (-0.85,-0.55) \wdot ;
	\end{tikzpicture}
};
\endxy
\ - \
  \xy
(0,0)*{
\begin{tikzpicture}[scale=.3, color=\clr]
	\draw [ thick, <-, looseness=1.8] (-1,1.25) to [out=270, in=270] (1,1.25);
	\draw [ thick, <-, looseness=1.8] (-1,-1.25) to [out=90, in=90] (1,-1.25);
	\draw (0.85,0.55) \wdot ;
	\end{tikzpicture}
};
\endxy
 \right) \ .
\end{equation}

\item  
\begin{equation}\label{E:dotthroughrightwardcrossing}
\xy
(0,0)*{\reflectbox{
\begin{tikzpicture}[scale=.3, color=\clr]
	\draw [ thick, <-] (-1,-1) to (-0.25,-0.25);
	\draw [ thick, -] (0.25,0.25) to (1,1);
	\draw [ thick, ->] (1,-1) to (-1,1);
	\draw (0.55,0.55) \wdot ;
	\end{tikzpicture}
}};
\endxy
 \ = \ 
\xy
(0,0)*{\reflectbox{
\begin{tikzpicture}[scale=.3, color=\clr]
	\draw [ thick, <-] (-1,-1) to (-0.25,-0.25);
	\draw [ thick, -] (0.25,0.25) to (1,1);
	\draw [ thick, ->] (1,-1) to (-1,1);
	\draw (-0.55,-0.55) \wdot ;
	\end{tikzpicture}
}};
\endxy
 \ + \tq  \ 
\left( 
 \xy
(0,0)*{
\begin{tikzpicture}[scale=.3, color=\clr]
	\draw [ thick, ->, looseness=1.8] (-1,1.25) to [out=270, in=270] (1,1.25);
	\draw [ thick, ->, looseness=1.8] (-1,-1.25) to [out=90, in=90] (1,-1.25);
	\draw (-0.85,-0.55) \wdot ;
	\end{tikzpicture}
};
\endxy
\ - \
  \xy
(0,0)*{
\begin{tikzpicture}[scale=.3, color=\clr]
	\draw [ thick, ->, looseness=1.8] (-1,1.25) to [out=270, in=270] (1,1.25);
	\draw [ thick, ->, looseness=1.8] (-1,-1.25) to [out=90, in=90] (1,-1.25);
	\draw (0.85,0.55) \wdot ;
	\end{tikzpicture}
};
\endxy
 \right) \ .
\end{equation}
\begin{equation}\label{E:dotthroughrightwardcrossing2}
\xy
(0,0)*{\reflectbox{
\begin{tikzpicture}[scale=.3, color=\clr]
	\draw [ thick, <-] (-1,-1) to (1,1);
	\draw [ thick, ->] (-0.25,0.25) to (-1,1);
	\draw [ thick, -] (0.25,-0.25) to (1,-1);
	\draw (0.55,-0.55) \wdot ;
	\end{tikzpicture}
}};
\endxy
 \ = \ 
\xy
(0,0)*{\reflectbox{
\begin{tikzpicture}[scale=.3, color=\clr]
	\draw [ thick, <-] (-1,-1) to (1,1);
	\draw [ thick, ->] (-0.25,0.25) to (-1,1);
	\draw [ thick, -] (0.25,-0.25) to (1,-1);
	\draw (-0.55,0.55) \wdot ;
	\end{tikzpicture}
}};
\endxy
 \ + \tq  \ 
\left( 
 \xy
(0,0)*{
\begin{tikzpicture}[scale=.3, color=\clr]
	\draw [ thick, ->, looseness=1.8] (-1,1.25) to [out=270, in=270] (1,1.25);
	\draw [ thick, ->, looseness=1.8] (-1,-1.25) to [out=90, in=90] (1,-1.25);
	\draw (-0.85,-0.55) \wdot ;
	\end{tikzpicture}
};
\endxy
\ - \
  \xy
(0,0)*{
\begin{tikzpicture}[scale=.3, color=\clr]
	\draw [ thick, ->, looseness=1.8] (-1,1.25) to [out=270, in=270] (1,1.25);
	\draw [ thick, ->, looseness=1.8] (-1,-1.25) to [out=90, in=90] (1,-1.25);
	\draw (0.85,0.55) \wdot ;
	\end{tikzpicture}
};
\endxy
 \right) \ .
\end{equation}

\item \begin{equation}\label{E:dotthroughdownwardcrossing}
\xy
(0,0)*{\reflectbox{
\begin{tikzpicture}[scale=.3, color=\clr]
	\draw [ thick, <-] (-1,-1) to (-0.25,-0.25);
	\draw [ thick, -] (0.25,0.25) to (1,1);
	\draw [ thick, <-] (1,-1) to (-1,1);
	\draw (0.55,0.55) \wdot ;
	\end{tikzpicture}
}};
\endxy
 \ = \ 
\xy
(0,0)*{\reflectbox{
\begin{tikzpicture}[scale=.3, color=\clr]
	\draw [ thick, <-] (-1,-1) to (-0.25,-0.25);
	\draw [ thick, -] (0.25,0.25) to (1,1);
	\draw [ thick, <-] (1,-1) to (-1,1);
	\draw (-0.55,-0.55) \wdot ;
	\end{tikzpicture}
}};
\endxy
 \ + \tq  \ 
\left( 
 \xy
(0,0)*{
\begin{tikzpicture}[scale=.3, color=\clr]
	\draw [ thick, <-] (-.5,-1) to (-.5,1);
	\draw [ thick, <-] (.5,-1) to (.5,1);
	\draw (-.5,0) \wdot ;
	\end{tikzpicture}
};
\endxy
\ - \
  \xy
(0,0)*{
\begin{tikzpicture}[scale=.3, color=\clr]
	\draw [ thick, <-] (-.5,-1) to (-.5,1);
	\draw [ thick, <-] (.5,-1) to (.5,1);
	\draw (.5,0) \wdot ;
	\end{tikzpicture}
};
\endxy
 \right) \ .
\end{equation}

\begin{equation}\label{E:dotthroughdownwardcrossing2}
\xy
(0,0)*{\reflectbox{
\begin{tikzpicture}[scale=.3, color=\clr]
	\draw [ thick, <-] (-1,-1) to (1,1);
	\draw [ thick, -] (-0.25,0.25) to (-1,1);
	\draw [ thick, ->] (0.25,-0.25) to (1,-1);
	\draw (-0.55,0.55) \wdot ;
	\end{tikzpicture}
}};
\endxy
 \ = \ 
\xy
(0,0)*{\reflectbox{
\begin{tikzpicture}[scale=.3, color=\clr]
	\draw [ thick, <-] (-1,-1) to (1,1);
	\draw [ thick, -] (-0.25,0.25) to (-1,1);
	\draw [ thick, ->] (0.25,-0.25) to (1,-1);
	\draw (0.55,-0.55) \wdot ;
	\end{tikzpicture}
}};
\endxy
 \ + \tq  \ 
\left( 
 \xy
(0,0)*{
\begin{tikzpicture}[scale=.3, color=\clr]
	\draw [ thick, <-] (-.5,-1) to (-.5,1);
	\draw [ thick, <-] (.5,-1) to (.5,1);
	\draw (-.5,0) \wdot ;
	\end{tikzpicture}
};
\endxy
\ - \
  \xy
(0,0)*{
\begin{tikzpicture}[scale=.3, color=\clr]
	\draw [ thick, <-] (-.5,-1) to (-.5,1);
	\draw [ thick, <-] (.5,-1) to (.5,1);
	\draw (.5,0) \wdot ;
	\end{tikzpicture}
};
\endxy
 \right) \ .
\end{equation}

\item 
\begin{equation}\label{thickness-one-rightward-cup-and-cap-with-dot}
\xy
(0,0)*{
\bt[scale=.35, color=\clr]
	\draw [ thick, looseness=2, rdirected=-0.95] (1,4.5) to [out=90,in=90] (-1,4.5);
	\draw (-0.9,5) \wdot ;
\et
};
\endxy =
\xy
(0,0)*{
\bt[scale=.35, color=\clr]
	\draw [ thick, looseness=2, rdirected=-0.95] (1,4.5) to [out=90,in=90] (-1,4.5);
	\draw (0.9,5) \wdot ;
\et
};
\endxy \   \quad\quad\text{and}\quad\quad
\xy
(0,0)*{
\bt[scale=.35, color=\clr]
	\draw [ thick, looseness=2, rdirected=-0.95 ] (1,2.5) to [out=270,in=270] (-1,2.5);
	\draw (-0.9,2) \wdot ;
\et
};
\endxy = 
\xy
(0,0)*{
\bt[scale=.35, color=\clr]
	\draw [ thick, looseness=2, rdirected=-0.95 ] (1,2.5) to [out=270,in=270] (-1,2.5);
	\draw (0.9,2) \wdot ;
\et
};
\endxy \ .
\end{equation}
\end{enumerate} 
\end{lemma}

\begin{proof} To prove \cref{E:inverseleftwardcrossing,E:inversedownwardcrossing} apply leftward cup(s) and cap(s) to \cref{E:inverseupwardcrossing}.  To prove \cref{E:inverserightwardcrossing}, vertically compose a rightward under-crossing and a rightward over-crossing above and below, respectively, on \cref{E:inverseleftwardcrossing}.  Relation \cref{E:dotthroughupupwardcrossing} is precisely the Hecke-Clifford relation from \cref{L:srelations}(g.).  Relation \cref{E:dotthroughupupwardcrossing2} follows from this one by composing on the top and bottom with the upward oriented under-crossing.     

To prove \cref{E:leftcircleiszero} compose a leftward cup and cap on the bottom and top, respectively, of \cref{E:inverserightwardcrossing}. Doing so yields
\begin{equation}\label{E:bubbleequation}
\begin{tikzpicture}[scale=.3, color=\clr,baseline = 0]
	\draw[thick, -, looseness=1.5] (-1,0) to [out=270, in=270] (1,0);
	\draw[thick, ->, looseness=1.5] (-1,0) to [out=90, in=90] (1,0);
\end{tikzpicture} 
=
\begin{tikzpicture}[scale=.3, color=\clr,baseline = 0]
	\draw[thick, -, looseness=1.5] (-1,0) to [out=270, in=270] (1,0);
	\draw[thick, <-, looseness=1.5] (-1,0) to [out=90, in=90] (1,0);
\end{tikzpicture}
- \tq 
\begin{tikzpicture}[scale=.3, color=\clr,baseline = 0]
	\draw[thick, -, looseness=1.5] (-1,0) to [out=270, in=270] (1,0);
	\draw[thick, ->, looseness=1.5] (-1,0) to [out=90, in=90] (1,0);
\end{tikzpicture} \;\;
\begin{tikzpicture}[scale=.3, color=\clr,baseline = 0]
	\draw[thick, -, looseness=1.5] (-1,0) to [out=270, in=270] (1,0);
	\draw[thick, <-, looseness=1.5] (-1,0) to [out=90, in=90] (1,0);
\end{tikzpicture} .
\end{equation}  This equation along with  \cref{delete-bubble} implies \cref{E:leftcircleiszero}.   Both equalities in \cref{thickness-one-rightward-cup-and-cap} follow from applying a leftward cap or cup to \cref{E:inverserightwardcrossing} and using \cref{delete-bubble} or \cref{E:leftcircleiszero} to simplify.  Using the fact dots move freely through overcrossings and leftward cups and caps, \cref{thickness-one-rightward-cup-and-cap} and \cref{delete-bubble} imply \cref{E:leftcirclewithdotiszero}.    

To prove \cref{E:dotthroughleftwardcrossing,E:dotthroughdownwardcrossing}  apply leftward cups and caps to \cref{E:dotthroughupupwardcrossing,E:dotthroughupupwardcrossing2}.   Using relation \cref{E:dotthroughupupwardcrossing2} instead of \cref{E:dotthroughupupwardcrossing} and making similar arguments verifies all the remaining relations except \cref{thickness-one-rightward-cup-and-cap-with-dot}.    Applying a leftward cap (resp.\ a leftward cup) to \cref{E:dotthroughrightwardcrossing}, the fact that dots move freely over leftward cups, caps, and over-crossings along with the previous parts of the lemma yields the left (resp.\ right) equality in \cref{thickness-one-rightward-cup-and-cap-with-dot}.

\end{proof}

\begin{lemma}\label{L:straighteningtwists} For any $k \geq 1$, the following formulas involving rightward cups and caps hold in $\qwebsupdown$:

\begin{enumerate}
\item 
\begin{equation}\label{E:straighten-twists1}
\xy
(0,0)*{
\bt[scale=1.25, color=\clr]
        \draw[thick]  (1,0) to (1,0.35);
	\draw[thick] plot [smooth, tension=1.1] coordinates {(1,0.35) (0.95,0.45)};
	\draw[thick] plot [smooth, tension=1.1] coordinates {(0.9,0.55) (0.8,0.65) (0.625,0.5) (0.8,0.35) (1.0,0.65)};
        \draw[thick, directed=1]  (1,0.65) to (1,1);
	\node at (1,-0.15) {\scriptsize $k$};
\et
};
\endxy = q^{k(k-1)}
\xy
(0,0)*{
\bt[scale=1.25, color=\clr]
	\draw[thick, directed=1] (1,0) to (1,1);
	\node at (1,-0.15) {\scriptsize $k$};
\et
};
\endxy
= 
\xy
(0,0)*{
\bt[scale=1.25, color=\clr]
        \draw[thick]  (1,0) to (1,0.35);
	\draw[thick] plot [smooth, tension=1.1] coordinates {(1,0.35) (1.2,0.65) (1.375,0.5) (1.25,0.35) (1.10,0.45)};
	\draw[thick] plot [smooth, tension=1] coordinates {(1,0.65) (1.05,0.55)};
        \draw[thick, directed=1]  (1,0.65) to (1,1);
	\node at (1,-0.15) {\scriptsize $k$};
\et
};
\endxy \ ,  
\end{equation}

\item 
\begin{equation}\label{E:straighten-twists2}
\xy
(0,0)*{
\bt[scale=1.25, color=\clr]
        \draw[thick]  (1,0) to (1,0.35);
	\draw[thick] plot [smooth, tension=1.1] coordinates {(1,0.35) (0.8,0.65) (0.625,0.5) (0.75,0.35) (0.9,0.45)};
	\draw[thick] plot [smooth, tension=1] coordinates {(1,0.65) (0.95,0.55)};
        \draw[thick, directed=1]  (1,0.65) to (1,1);
	\node at (1,-0.15) {\scriptsize $k$};
\et
};
\endxy = q^{-k(k-1)}
\xy
(0,0)*{
\bt[scale=1.25, color=\clr]
	\draw[thick, directed=1] (1,0) to (1,1);
	\node at (1,-0.15) {\scriptsize $k$};
\et
};
\endxy 
=
\xy
(0,0)*{
\bt[scale=1.25, color=\clr]
        \draw[thick]  (1,0) to (1,0.35);
	\draw[thick] plot [smooth, tension=1.1] coordinates {(1.1,0.55) (1.25,0.65) (1.375,0.5) (1.25,0.35) (1.10,0.45) (1,0.65)};
	\draw[thick] plot [smooth, tension=1] coordinates {(1,0.35) (1.05,0.45)};
        \draw[thick, directed=1]  (1,0.65) to (1,1);
	\node at (1,-0.15) {\scriptsize $k$};
\et
};
\endxy \ ,
\end{equation}

\item 
\begin{equation}\label{dot-through-rightward-cup-and-cap}
\xy
(0,0)*{
\bt[scale=.35, color=\clr]
	\draw [ thick, looseness=2, rdirected=-0.95] (1,4.5) to [out=90,in=90] (-1,4.5);
	\node at (-1.75,4.75) {\scriptsize $k$};
	\draw (-0.85,5) \wdot ;
\et
};
\endxy =
\xy
(0,0)*{
\bt[scale=.35, color=\clr]
	\draw [ thick, looseness=2, rdirected=-0.95] (1,4.5) to [out=90,in=90] (-1,4.5);
	\node at (-1.75,4.75) {\scriptsize $k$};
	\draw (0.85,5) \wdot ;
\et
};
\endxy \   \quad\quad\text{and}\quad\quad
\xy
(0,0)*{
\bt[scale=.35, color=\clr]
	\draw [ thick, looseness=2, rdirected=-0.95 ] (1,2.5) to [out=270,in=270] (-1,2.5);
	\node at (-1.75,2.5) {\scriptsize $k$};
	\draw (-0.85,2) \wdot ;
\et
};
\endxy = 
\xy
(0,0)*{
\bt[scale=.35, color=\clr]
	\draw [ thick, looseness=2, rdirected=-0.95 ] (1,2.5) to [out=270,in=270] (-1,2.5);
	\node at (-1.75,2.5) {\scriptsize $k$};
	\draw (0.85,2) \wdot ;
\et
};
\endxy \ .
\end{equation}
\item 
\begin{equation}\label{E:pitchforkrelations}
\xy
(0,0)*{
\begin{tikzpicture}[baseline=0,color=\clr, scale=.25]
		\node [style=none] (0) at (-2, 0) {};
		\node [style=none] (1) at (2, 0) {};
		\node [style=none] (2) at (0, 0) {};
		\node [style=none] (3) at (-2.25, 2.25) {};
		\node [style=none] (4) at (-1.5, 1.5) {};
		\node [style=none] (5) at (-1, 1) {};
		\node at (2.75,0) {\scriptsize $k$};
		\node at (.75,0) {\scriptsize $h$};
		\draw [thick, <-,looseness=1.50] (0.center) to [out=90,in=90] (1.center);
		\draw [thick, ->] (4.center) to  (3.center);
		\draw [thick, -] (2.center) to  (5.center);
\end{tikzpicture}}
\endxy
=
\xy
(0,0)*{
\begin{tikzpicture}[baseline=0,color=\clr, scale=.25]
		\node [style=none] (0) at (-2, 0) {};
		\node [style=none] (1) at (2, 0) {};
		\node [style=none] (2) at (0, 0) {};
		\node [style=none] (3) at (1.5, 1.5) {};
		\node [style=none] (4) at (2.25, 2.25) {};
		\node [style=none] (5) at (1, 1) {};
		\node at (2.75,0) {\scriptsize $k$};
		\node at (-0.75,0) {\scriptsize $h$};
		\draw [thick, <-, looseness=1.50] (0.center) to [out=90,in=90] (1.center);
		\draw [thick, <-] (4.center) to (3.center);
		\draw [thick, -] (2.center) to (5.center);
\end{tikzpicture}}
\endxy
\   \quad\quad\text{and}\quad\quad
\xy
(0,0)*{
\begin{tikzpicture}[baseline=0,color=\clr, scale=.25]
		\node [style=none] (0) at (-2, 0) {};
		\node [style=none] (1) at (2, 0) {};
		\node [style=none] (2) at (-1, -1) {};
		\node [style=none] (3) at (-2.25, -2.25) {};
		\node [style=none] (4) at (-1.5, -1.5) {};
		\node [style=none] (5) at (0, 0) {};
		\node at (2.75,0) {\scriptsize $k$};
		\node at (.75,0) {\scriptsize $h$};
		\draw [thick, <-, looseness=1.50] (0.center) to [out=270,in=270] (1.center);
		\draw [thick, ->] (4.center) to (3.center);
		\draw [thick, -] (2.center) to (5.center);
\end{tikzpicture}}
\endxy
=
\xy
(0,0)*{
\begin{tikzpicture}[baseline=0,color=\clr, scale=.25]
		\node [style=none] (0) at (-2, 0) {};
		\node [style=none] (1) at (2, 0) {};
		\node [style=none] (2) at (1, -1) {};
		\node [style=none] (3) at (2.25, -2.25) {};
		\node [style=none] (4) at (1.5, -1.5) {};
		\node [style=none] (5) at (0, 0) {};
		\node at (2.75,0) {\scriptsize $k$};
		\node at (-0.75,0) {\scriptsize $h$};
		\draw [thick, <-, looseness=1.50] (0.center) to [out=270,in=270] (1.center);
		\draw [thick, ->] (4.center) to (3.center);
		\draw [thick, -] (2.center) to (5.center);
\end{tikzpicture}}
\endxy   \ .
\end{equation} 
\item  
\begin{equation}\label{straighten-right-zigzag}
\xy
(0,0)*{\reflectbox{\rotatebox{180}{
\bt[scale=1, color=\clr]
	\draw [ thick, looseness=2, ] (1,0) to [out=90,in=90] (0.5,0);
	\draw [ thick, looseness=2, ] (1.5,0) to [out=270,in=270] (1,0);
	\draw[ thick, ] (0.5,0) to (0.5,-0.5);
	\draw[ thick, ->] (1.5,0) to (1.5,0.5);
	\node at (0.5,-0.65) {\scriptsize \reflectbox{\rotatebox{180}{$k$}}};
\et
}}};
\endxy=
\xy
(0,0)*{
\bt[scale=1, color=\clr]
	\draw[thick, <-] (1,0) to (1,1);
	\node at (1,1.15) {\scriptsize $k$};
\et
};
\endxy \ \quad\quad\text{and}\quad\quad
\xy
(0,0)*{
\bt[scale=1, color=\clr]
	\draw [ thick, looseness=2, ] (1,0) to [out=90,in=90] (0.5,0);
	\draw [ thick, looseness=2, ] (1.5,0) to [out=270,in=270] (1,0);
	\draw[ thick, -] (0.5,0) to (0.5,-0.5);
	\draw[ thick, ->] (1.5,0) to (1.5,0.5);
	\node at (0.5,-0.65) {\scriptsize $k$};
\et
};
\endxy=
\xy
(0,0)*{
\bt[scale=1,  color=\clr]
	\draw[thick, <-] (1,1) to (1,0);
	\node at (1,-0.15) {\scriptsize $k$};
\et
};
\endxy \ .
\end{equation}
\end{enumerate}
The relations obtained by from (d.) by changing under/over-crossings to over/under-crossings and/or changing the strand orientations also hold.
\end{lemma}

\begin{proof}  The second equality of \cref{E:straighten-twists1}  follows from the definition of the rightward cap and the fact the rightward over-crossing is the inverse of the leftward under-crossing, along with applications of \cref{straighten-zigzag}.  Similarly, the left equality in \cref{E:straighten-twists2} follows from the fact the rightward under-crossing is the inverse of the leftward over-crossing.

For the other equalities of parts (a.) and (b.), we argue via induction on $k$.  The case $k=1$ follows from using \cref{E:inverseupwardcrossing}  to replace the over-crossing with an under-crossing and then applying the case considered in the previous paragraph and the fact undotted bubbles labelled by $1$ are zero.  We explain the inductive step for the left equality of \cref{E:straighten-twists1}. The right equality of \cref{E:straighten-twists2} is entirely similar.  First, use \cref{digon-removal} to explode the strand of this leftward under-twist into two strands, one strand of thickness $k-1$ and one strand of thickness $1$.  The fact merges and splits freely move over cups, caps, and through crossings allows us to draw the left most diagram of \cref{E:straighten-twists1} as a pair of parallel strands of the same shape as the original diagram along  with a split at the bottom and a merge at the top.  Applying the inductive assumption to simplify, the result is a web with upwards over-crossings between a merge and split.  Using \cref{L:untwist-permutation} to remove these crossings and \cref{digon-removal} to close up the strands results in a scalar multiple of an upward arrow labelled by $k$.  After simplification, the scalar which results is on the one given in \cref{E:straighten-twists1}. 

To prove (c.) we induct on $k$ with the base case of $k = 1$ given by rewriting the rightward cup and cap according to \cref{rightward-cup-and-cap}, then dragging the dot along the over-strand of the leftward crossing and the leftward cup or cap.  To pull the dot along the under-strand of the crossing, we use \cref{E:dotthroughleftwardcrossing} or \cref{E:dotthroughleftwardcrossing2}.  We explain the inductive step for the left equality and leave the similar right equality to the reader.  For $k>1$, we proceed by exploding the thickness $k$ strand just above the dot on the leftmost diagram into strands of thickness $1$ and $k-1$ and use \cref{dot-on-k-strand}(b.) to move the dot onto these strand labelled by $1$.  Since the merge and split can be moved freely over rightward cap and since by the base case the dot can as well, we can move each of these to the right hand side of the cap.  Applying leftward cups and caps to \cref{dot-on-k-strand}(b.) shows there is an analogous relation for downward oriented strands and, consequently, we can close up our exploded strands to obtain the right hand side of the first equality. 

To prove the left equality in (d.), we draw the leftward over-crossing using the definition given in \cref{E:leftwardcrossing}, use \cref{straighten-zigzag} to simplify, and the result is the desired equality.  Similar arguments apply for the other versions of the left equality which involve a leftward cap.  To prove those which involve a rightward cap, we draw the rightward cap using the definition given in \cref{rightward-cup-and-cap} and the desired equality then follows from the previously proven leftward cap cases along with applications of the appropriately oriented analogues of the relations given in \cref{L:braidingforups}. The various versions of the second equality involving a cup are proved identically.

To prove (e.), draw the rightward cups and caps according to their definitions as in \cref{rightward-cup-and-cap} and use the previous parts of the lemma along with \cref{straighten-zigzag} to straighten.
\end{proof}

Given a crossing of a strand labelled by $k$ and a strand labelled by $l$, we say it has \emph{crossing degree} $kl$ regardless of the orientation of the strands or which strand is the over-crossing strand.  More generally, the \emph{total crossing degree} of a web is the sum of the crossing degrees of all the crossings in the web.  The total crossing degree does depend on how the web is drawn.

In what follows we adopt the convention that orientations of edges are omitted whenever the statement is true regardless of orientation.
\begin{lemma}\label{L:dotsthroughcrossings}  let $h, k \geq 0$.  Then,
\begin{equation}\label{E:dotthroughcrossing}
\xy
(0,0)*{\reflectbox{
\begin{tikzpicture}[scale=.3, color=\clr]
	\draw [ thick, -] (-1,-1) to (-0.25,-0.25);
	\draw [ thick, -] (0.25,0.25) to (1,1);
	\draw [ thick, -] (1,-1) to (-1,1);
	\node at (-1,-1.5) {\reflectbox{\scriptsize $k$}};
	\node at (1,-1.5) {\reflectbox{\scriptsize $h$}};
	\draw (0.55,0.55) \wdot ;
	\end{tikzpicture}
}};
\endxy
 \ = \ 
\xy
(0,0)*{\reflectbox{
\begin{tikzpicture}[scale=.3, color=\clr]
	\draw [ thick, -] (-1,-1) to (-0.25,-0.25);
	\draw [ thick, -] (0.25,0.25) to (1,1);
	\draw [ thick, -] (1,-1) to (-1,1);
	\node at (-1,-1.5) {\reflectbox{\scriptsize $k$}};
	\node at (1,-1.5) {\reflectbox{\scriptsize $h$}};
	\draw (-0.55,-0.55) \wdot ;
	\end{tikzpicture}
}};
\endxy
 \ +  \ 
\left( ***
 \right) \ ,
\end{equation}

\begin{equation}\label{E:dotthroughcrossing2}
\xy
(0,0)*{\reflectbox{
\begin{tikzpicture}[scale=.3, color=\clr]
	\draw [ thick, -] (-1,-1) to (1,1);
	\draw [ thick, -] (-0.25,0.25) to (-1,1);
	\draw [ thick, -] (0.25,-0.25) to (1,-1);
	\node at (-1,-1.5) {\reflectbox{\scriptsize $k$}};
	\node at (1,-1.5) {\reflectbox{\scriptsize $h$}};
	\draw (-0.55,0.55) \wdot ;
	\end{tikzpicture}
}};
\endxy
 \ = \ 
\xy
(0,0)*{\reflectbox{
\begin{tikzpicture}[scale=.3, color=\clr]
	\draw [ thick, -] (-1,-1) to (1,1);
	\draw [ thick, -] (-0.25,0.25) to (-1,1);
	\draw [ thick, -] (0.25,-0.25) to (1,-1);
	\node at (-1,-1.5) {\reflectbox{\scriptsize $k$}};
	\node at (1,-1.5) {\reflectbox{\scriptsize $h$}};
	\draw (0.55,-0.55) \wdot ;
	\end{tikzpicture}
}};
\endxy
 \ +  \ 
\left( ***
 \right) \ ,
\end{equation} where $(***)$ is a linear combination of diagrams of total crossing degree strictly less than $hk$. 
\end{lemma}

\begin{proof}  Consider the diagram on the left side of the equality in \cref{E:dotthroughcrossing}.  Using \cref{dot-on-k-strand} or its downward analogue, we can replace the dot with a split into strands labelled by $k-1$ and $1$, a dot on the strand labelled by $1$, followed by a merge.  We can then move the merge through the crossing using the appropriately oriented version of \cref{L:braidingforups}(c.).  On the other hand, after scaling by $1/[h]_{q}!$ the strand labelled by $h$ can blown apart into strands of thickness $1$ at a point below the initial crossing thanks to the appropriately oriented version of \cref{digon-removal}.  The  merge which was just created can also be moved over the parallel strands labelled by $1$ and $k-1$.  Having done so the result is a web which consisting of a strand labelled by $k-1$ passing under $l$ parallel strands labelled by $1$ along with a strand with a dot labelled by $1$ and parallel to the strand labelled by $k-1$ also passing under those same $l$ strands labelled by $1$.  Furthermore, the  strand labelled by $1$ which is parallel to the strand labelled by $k-1$ has a dot which is to the northwest of all the crossings.   Note that the resulting diagram still has total crossing degree $kl$.

Using the appropriately oriented Hecke-Clifford relation given by \cref{L:thicknessonerelations} the dot can be moved successively through the crossings of strands of thickness $1$ modulo linear combinations of diagrams with total crossing degree $kl-1$.  After moving the dot through all the crossings, the result will one web with total crossing degree $kl$ which can be simplified to the first diagram on the right hand side of the equality in \cref{E:dotthroughcrossing} plus a linear combination of diagrams of total crossing degree strictly less than $kl$, as claimed.  

The second equality is argued similarly.
\end{proof}

\begin{proposition}\label{L:bubblespop} If $d$ is a nonempty web of type $\unit \to \unit$, then $d=0$.  In particular,  $\End_{\qwebsupdown}(\unit ) = \K$.  
\end{proposition}
\begin{proof} Let $d$ be a nonempty web of type $\unit  \to \unit  $.  We will show $d=0$ in  $\End_{\qwebsupdown}(\unit )$.  Throughout we will ignore the scalars which  appear as they will not play a role.  Using \cref{straighten-zigzag,digon-removal} we can ``explode'' each strand of thickness $k>1$ into strands of thickness $1$.  For the moment let us assume there are no dots.  Since merges and splits move freely through crossings, cups, and caps it follows the merge at the top of the point of explosion can be freely moved around $d$ until it arrives at the bottom of a merge or split.  Similarly, the split at the bottom of each point of explosion can moved freely moved around $d$ until it arrives at the top of a merge or split. As discussed at the end of \cref{SS:UpwardWebs}, after moving all merges and splits until they are adjacent to other merges and splits, the end result is a collection of clasp idempotents connected together by strands of thickness $1$.  Applying \cref{L:claspsum} to these clasp idempotents reduces $d$ to a linear combination of webs consisting solely of strands of thickness $1$. See \cref{sergeev-generators} for another example of this technique.  If there are dots on strands, then \cref{dots-past-merges} allows us first to move dots past merges and splits and onto the strands of thickness one. Thereafter we can argue as above as dots on strands of thickness $1$ can remain fixed and one can still apply the above argument.  Thus we may reduce all webs to a linear combination of webs with strands of thickness $1$.

We have reduced ourself to showing $d=0$ when $d$ is a link constructed of oriented strands of thickness $1$, possibly labelled with dots.  Using the relations established in \cref{L:srelations,L:thicknessonerelations,L:straighteningtwists}, such a web can have over-crossings changed to under-crossings, have twists straightened, and have dots moved through crossings, all modulo linear combinations of similar diagrams but with fewer crossings.  Arguing by induction on the number of crossings shows $d$ can be reduced to a linear combination of products of concentric circles consisting of oriented strands of thickness $1$, possibly labelled with dots.  But any such web is zero by \cref{delete-bubble,E:leftcircleiszero}.   
\end{proof}

\subsection{A Ribbon Category}\label{SS:RibbonCategory}

The objects of $\qwebsupdown$ are words in $\left\{\uparrow_{k}, \downarrow_{k} \mid k\geq 0 \right\}$.  We extend our earlier notation by writing $\ob{a}$ for an object in $\qwebsupdown$ (e.g.\ $\ob{a} = \uparrow_{a_{1}} \downarrow_{a_{2}}\downarrow_{a_{3}}$).  Since a tuple of integers does not determine an object of $\qwebsupdown$,  it is convenient to write $\uddoublearrow_{k}$ for either $\up_{k}$ or $\down_{k}$ when the orientation is not important. More generally, for any tuple of nonnegative integers ${a}=(a_{1}, \dotsc , a_{r})$ we write $\uddoublearrow_{{a}} = \uddoublearrow_{a_{1}}\uddoublearrow_{a_{2}}\dotsb \uddoublearrow_{a_{r}}$ for an object in $\qwebsupdown$ when the orientation of the objects is not needed. 

 Given a nonnegative integer $a$, define the \emph{twist morphism} $\theta_{a} = q^{a(a-1)}\Id_{\uddoublearrow_{a}}$.  More generally, recursively define the \emph{twist morphism} $\theta_{\ob{c}}:{\ob{c}} \to {\ob{c}}$ for an arbitrary object ${\ob{c}}$ to be the unique morphism which satisfies the formula 
\[
\theta_{{\ob{a}}{\ob{b}}} = \left( \theta_{{\ob{a}}} \otimes \theta_{{\ob{b}}}\right) \beta_{{\ob{b}},{\ob{a}}}\beta_{{\ob{a}},{\ob{b}}}
\] for all objects ${\ob{a}},{\ob{b}}$.

Given $k \geq 1$, let $\uddoublearrow'_{k}$ denote the object with the orientation reversed (ie., $\up'_{k}=\down_{k}$ and $\down'_{k}=\up_{k}$).
 Given a tuple of nonnegative integers $(a_{1}, \dotsc , a_{t})$ there are obvious morphisms
\begin{align*}
(\uddoublearrow'_{a_{t}}\dotsb \uddoublearrow'_{a_{1}})(\uddoublearrow_{a_{1}}\dotsb \uddoublearrow_{a_{t}}) &\to \unit, \\
(\uddoublearrow_{a_{1}}\dotsb \uddoublearrow_{a_{t}})(\uddoublearrow'_{a_{t}}\dotsb \uddoublearrow'_{a_{1}}) &\to \unit,
\end{align*}
and
\begin{align*}
\unit  &\to (\uddoublearrow_{a_{1}}\dotsb \uddoublearrow_{a_{t}})(\uddoublearrow'_{a_{t}}\dotsb \uddoublearrow'_{a_{1}}), \\
\unit  &\to (\uddoublearrow'_{a_{t}}\dotsb \uddoublearrow'_{a_{1}})(\uddoublearrow_{a_{1}}\dotsb \uddoublearrow_{a_{t}}),
\end{align*} given by drawing a ``rainbow''  of appropriately oriented caps and cups, respectively.  One can easily check they satisfy analogues of the straightening rules given in \cref{straighten-zigzag,straighten-right-zigzag}. 

Since we have crossings of any two strands of arbitrary orientation and thickness, given a tuple of nonnegative integers $\ob{a}=(a_{1}, \dotsc , a_{t})$ and $\ob{b}=(b_{1}, \dotsc , b_{u})$ we can use the self-evident generalization of \cref{D:generalupbraiding} to define a braiding morphism between arbitrary objects in $\qwebsupdown$,
\[
\beta_{\ob{a},\ob{b}}: \ob{a}\ob{b} \to \ob{b}\ob{a}.
\]
Since the individual crossings satisfy the analogue of \cref{L:braidingforups}, it is not hard to verify these new morphisms satisfy the multi-strand analogue of \cref{L:braidingforups}. In short, for arbitrary objects in $\qwebsupdown$ we have constructed over- and under-crossings, twist maps, cups, and caps. 

\begin{theorem}\label{T:ribboncategory}  Let $\qwebsupdowneven$ be the subcategory of $\qwebsupdown$ consisting of all objects and all morphisms which can be written as linear combinations of webs without dots. Then  $\qwebsupdowneven$ is a ribbon category. 

Moreover, any two webs in $\qwebsupdown$  which are equal via planar isotopies which leave dots fixed along with applications of the relations given in \ref{L:straighteningtwists} are equal as morphisms in $\qwebsupdown$. 
\end{theorem}

\begin{proof}  The category $\qwebsupdowneven$ is a monoidal category with braidings, twists, and dualities.  By the relations verified above these morphisms satisfy the axioms of a ribbon category as given in, for example, \cite[Section 1.4]{Turaev} and this proves the first statement.  As described in \cite[Section 3]{Turaev} there is a diagrammatic calculus for such categories.  Since $\qwebsupdowneven$ satisfies the relations listed in \cite[Lemma 3.4]{Turaev}, it follows this calculus applies to $\qwebsupdowneven$ where we use the blackboard framing to view webs as ribbon graphs.  In particular, by working locally and leaving dots fixed, this implies any two webs which are equal via planar isotopies which leave dots fixed are equal as morphisms in $\qwebsupdown$, as claimed.
\end{proof}

\subsection{Isomorphisms between morphism spaces in \texorpdfstring{$\qwebsupdown $}{Oriented Webs} }\label{SS:Useful isomorphisms}

The symmetric group on $t$ letters, $S_{t}$, has a right action by place permutation on the set of all objects of $\qwebsupdown$ which are the tensor product of $t$ generating objects. Namely,  $\ob{a}\cdot \sigma = \uddoublearrow_{{a}}  \cdot \sigma  = \uddoublearrow_{a_{\sigma^{-1}(1)}} \uddoublearrow_{a_{\sigma^{-1}(2)}} \dotsb \uddoublearrow_{a_{\sigma^{-1}(t)}}$ for $\ob{a}=\uddoublearrow_{{a}} = \uddoublearrow_{a_{1}} \uddoublearrow_{a_{2}} \dotsb \uddoublearrow_{a_{t}}$ and $\sigma \in S_{t}$.  Moreover, if ${\ob{a}}$ is the tensor product of $t$ generating objects and $\sigma \in S_{t}$, then we can use a reduced expression for $\sigma$ and appropriately oriented over-crossings to define a morphism $d_{\sigma}  \in \Hom_{\qwebsupdown}\left( \ob{a},  \ob{a} \cdot \sigma \right)$.  The morphism $d_{\sigma}$ is invertible with the inverse given by using a reduced expression for $\sigma^{-1}$ and the under-crossing braidings.  

More generally, given objects ${\ob{a}}$ and ${\ob{b}}$ which are the tensor product of $t$ and $u$ generating objects, respectively, and any $\sigma \in S_{t}$ and $\tau \in S_{u}$, then the map $ w \mapsto d_{\tau} \circ w  \circ d^{-1}_{\sigma}$ defines a superspace isomorphism,
\begin{equation}\label{E:twistingstrandisomorphism}
\Hom_{\qwebsupdown}\left({\ob{a}}, {\ob{b}} \right) \xrightarrow{\cong} \Hom_{\qwebsupdown} \left({\ob{a}} \cdot \sigma, {\ob{b}} \cdot \tau  \right).
\end{equation} 

Given objects of the form ${\ob{a}} = \down_{a_{1}}\dotsb \down_{a_{r}}\up_{a_{r+1}}\dotsb \up_{a_{t}}$ and ${\ob{b}} = \up_{b_{1}}\dotsb \up_{b_{s}}\down_{b_{s+1}}\dotsb \down_{b_{u}}$, there is a parity preserving superspace isomorphism,
\begin{equation}\label{E:cupcapisomorphism}
\Hom_{\qwebsupdown}({\ob{a}}, { \ob{b}}) \xrightarrow{\cong} \Hom_{\qwebsupdown}\left( \up_{a_{r+1}}\dotsb \up_{a_{t}}\up_{b_{u}}\dotsb \up_{b_{s+1}}, \up_{a_{r}}\dotsb \up_{a_{1}}\up_{b_{1}}\dotsb \up_{b_{s}} \right),
\end{equation}

 given on diagrams by
\[
\xy
(0,0)*{
\begin{tikzpicture}[scale=.3, color=\clr]
	\draw [ thick] (-2.3,-8) rectangle (4.3,-6);
	\node at (1,-7) {$w$};
	\draw [ thick, rdirected=0.75] (-2,-9.5) to (-2,-8);
	\draw [ thick, rdirected=0.75] (0,-9.5) to (0,-8);
	\draw [ thick, directed=0.75] (2,-9.5) to (2,-8);
	\draw [ thick, directed=0.75] (4,-9.5) to (4,-8);
	\node at (2,-10.25) {\scriptsize $a_{r+1}$};
	\node at (4,-10.25) {\scriptsize $a_t$};
	\node at (-2,-10.25) {\scriptsize $a_1$};
	\node at (-1,-9.125) {\,$\cdots$};
	\node at (3,-9.125) {\,$\cdots$};
	\node at (0,-10.25) {\scriptsize $a_r$};
	\draw [ thick, rdirected=0.75] (2,-6) to (2,-4.5);
	\draw [ thick, rdirected=0.75] (4,-6) to (4,-4.5);
	\draw [ thick, directed=0.75] (-2,-6) to (-2,-4.5);
	\draw [ thick, directed=0.75] (0,-6) to (0,-4.5);
	\node at (2,-3.75) {\scriptsize $b_{s+1}$};
	\node at (3,-5.625) {\,$\cdots$};
	\node at (-1,-5.625) {\,$\cdots$};
	\node at (4,-3.75) {\scriptsize $b_u$};
	\node at (-2,-3.75) {\scriptsize $b_1$};
	\node at (0,-3.75) {\scriptsize $b_s$};
\end{tikzpicture}
};
\endxy\mapsto
\xy
(0,0)*{
\begin{tikzpicture}[scale=.3, color=\clr]
	\draw [ thick] (-2.3,-8) rectangle (4.3,-6);
	\draw [ thick, looseness=2, directed=0.5] (6,-6) to [out=90,in=90] (4,-6);
	\draw [ thick, looseness=1.75, directed=0.5] (8,-6) to [out=90,in=90] (2,-6);
	\node at (1,-7) {$w$};
	\draw [ thick, looseness=2, rdirected=0.5] (-4,-8) to [out=270,in=270] (-2,-8);
	\draw [ thick, looseness=1.75, rdirected=0.5] (-6,-8) to [out=270,in=270] (0,-8);
	\draw [ thick, directed=0.4] (6,-11) to (6,-6);
	\draw [ thick, directed=0.4] (8,-11) to (8,-6);
	\draw [ thick, directed=0.65] (4,-11) to (4,-8);
	\draw [ thick, directed=0.65] (2,-11) to (2,-8);
	\node at (2,-11.75) {\scriptsize $a_{r+1}$};
	\node at (-5,-5.5) { \ $\cdots$};
	\node at (4,-11.75) {\scriptsize $a_{t}$};
	\node at (6,-11.75) {\scriptsize $b_{u}$};
	\node at (-1,-5.5) {\,$\cdots$};
	\node at (8,-11.75) {\scriptsize $b_{s+1}$};
	\draw [ thick, directed=0.5] (-2,-6) to (-2,-3);
	\draw [ thick, directed=0.5] (0,-6) to (0,-3);
	\draw [ thick, directed=0.7] (-6,-8) to (-6,-3);
	\draw [ thick, directed=0.7] (-4,-8) to (-4,-3);
	\node at (-6,-2.25) {\scriptsize $a_{r}$};
	\node at (3,-8.5) {\,$\cdots$};
	\node at (-4,-2.25) {\scriptsize $a_{1}$};
	\node at (-2,-2.25) {\scriptsize $b_{1}$};
	\node at (7,-8.5) {\,$\cdots$};
	\node at (0,-2.25) {\scriptsize $b_{s}$};
\end{tikzpicture}
};
\endxy.
\]
The inverse is given by a similar map thanks to relation \cref{straighten-zigzag}.

As in the proof of \cref{T:fullness}, whenever a morphism between $\Hom$-spaces in $\qwebsupdown$ is defined by applying a combination of compositions and tensor products of morphisms we can apply the functor $\Psi$ given by \cref{T:Psiupdown} to obtain a corresponding homomorphism of $\U_{q}(\fq_{n} )$-supermodules.  In particular, commuting diagrams go to commuting diagrams and isomorphisms go to isomorphisms.

\subsection{Relating the two diagrammatic supercategories}\label{SS:FurtherFunctors}

\begin{theorem}\label{T:Thetafunctor}  There is a fully faithful functor of monoidal supercategories 
\[
\qwebs \to \qwebsupdown 
\] which takes objects and morphisms in $\qwebs$ to objects and morphisms of the same name in $\qwebsupdown$.

\end{theorem}
\begin{proof}  By construction the defining relations of $\qwebs$ hold in $\qwebsupdown$.  Consequently there is a well-defined functor of monoidal supercategories.  We next consider the question of the functor being full and faithful. Say $\ob{a}$ and $\ob{b}$ are objects of $\qwebs$ and say $d$ is a web in $\Hom_{\qwebsupdown}(\ob{a}, \ob{b})$.  That is, $d$ is an oriented web of type $\ob{a}\to \ob{b}$ where its bottom and top are upward oriented.  Between the bottom and top, $d$ will consist of dots, merges, splits, cups, caps, and crossings, and some of the internal edges in $d$ may well be downward oriented.

Merges and splits can freely move through crossings, cups, and caps. Dots can move freely over cups and caps.  These facts along with an argument by induction on the total crossing number of a diagram shows every morphism in $\qwebsupdown$ of type $\ob{a} \to \ob{b}$ can be written as a linear combination of diagrams where all merges, splits, and dots are upward oriented and above all crossings, cups, and caps.  That is, $d$ is equal to a linear combination of diagrams of type $\ob{a} \to \ob{b}$ where each diagram is the vertical composition of two diagrams: both have upward oriented bottoms and tops, the upper diagrams consist of only upward oriented merges, splits, and dots, and the lower diagrams consist of only cups, caps, and crossings.  

In particular, the upper diagrams are morphisms in $\qwebs$.  On the other hand, the lower diagrams are (labelled) braids which lie in the  subcategory $\qwebsupdowneven$.  However, this is a ribbon category and so satisfies the Reidemeister moves (where \cref{L:straighteningtwists}(a.)~and~(b.) stands in for the first Reidemeister move).  Thus, up to scaling by a power of $q$, the lower diagrams can be replaced with any diagram which is isotopic via boundary preserving isotopies.  In particular, the strands in the lower diagrams can be replaced by diagrams consisting solely of upward oriented crossings.  Thus each lower diagram is equal to a scalar multiples of some web which lies in $\qwebs$.  

In short, every diagram in $\qwebsupdown$ of type of $\ob{a} \to \ob{b}$ can written as a linear combination of diagrams of type $\ob{a} \to \ob{b}$ which lie in $\qwebs$.  Since every morphism in $\Hom_{\qwebsupdown}(\ob{a}, \ob{b})$ can be expressed as a linear combination of webs which lie in $\Hom_{\qwebs}(\ob{a}, \ob{b})$, fullness follows.  Similarly, any relation among morphisms in $\Hom_{\qwebsupdown}(\ob{a}, \ob{b})$ can be expressed using morphisms in $\Hom_{\qwebs}(\ob{a}, \ob{b})$ and so faithfulness also follows.
\end{proof}
We will identify $\qwebs$ as a full subcategory of $\qwebsupdown$ via this functor.

\section{Main Theorems and Applications}

\subsection{The oriented \texorpdfstring{$\Psi$}{Psi} functor}\label{SS:orientedpsifunctor}

 Let $\modsupdown$ denote the full monoidal subsupercategory of $U_{q}(\fq_{n})$-supermodules generated by the objects 
\[
\left\{S_{q}^{p}\left( V_{n}\right), S_{q}^{p} \left(V_{n}  \right)^{*} \mid p\geq 1 \right\}.
\]
 Define  $\K$-linear maps by
\[
\begin{array}{c}
\ev_{k }:  S_{q}^{k}(V_{n})^{*}\otimes S_{q}^{k}(V_{n}) \to \K, \\
 f \otimes x \mapsto f(x),
\end{array}
\]
and
\[
\begin{array}{c}
\coev_{k }:  \K \to S_{q}^{k}(V_{n})\otimes S_{q}^{k}(V_{n})^{*}, \\
1 \mapsto \sum_{} w_{i} \otimes w^{*}_{i},
\end{array}
\]
where $\left\{w_{i} \right\}$ is a homogeneous basis for $S_{q}^{k}(V_{n})$ and $w^{*}_{j} \in S_{q}^{k}(V_{n})^{*}$ is defined by $w^{*}_{j}(v_{i}) = \delta_{i,j}$. These are $U_{q}(\fq_{n})$-supermodule maps and are known to be independent of choice of basis. We call them evaluation and coevaluation, respectively.

It is convenient to extend our earlier notation by writing $S_{q}^{\uparrow_{d}}\left(V_{n} \right) = S_{q}^{d}\left(V_{n} \right)$ and $S_{q}^{\downarrow_{d}}\left(V_{n} \right) = S_{q}^{d}\left(V_{n} \right)^{*}$.  More generally, given an object in $\qwebsupdown$, $\ob{a} = \uddoublearrow_{a}= \uddoublearrow_{a_{1}}\dotsb \uddoublearrow_{a_{r}}$, we write 
\[
S_{q}^{\ob{a}}\left(V_{n} \right) =S_{q}^{\uddoublearrow_{a}}\left(V_{n} \right) = S_{q}^{\uddoublearrow_{a_{1}}}\left(V_{n} \right) \otimes \dotsb \otimes S_{q}^{\uddoublearrow_{a_{r}}}\left(V_{n} \right).
\]
The following result shows the functor $\Psi_n : \qwebs \to \mods$ can be extended to $\Psi_n: \qwebsupdown\to \modsupdown$.

\begin{theorem}\label{T:Psiupdown} There is an essentially surjective, full functor of symmetric monoidal supercategories
\[
\Psi_n : \qwebsupdown \to \modsupdown
\] given on generating objects by 
\begin{align*}
\Psi_n \left(\up_{k} \right) &= S^{\uparrow_{k}}(V_{n}),\\
 \Psi_n \left(\down_{k} \right) &= S^{\downarrow_{k}}(V_{n})
\end{align*}
and on generating morphisms by defining it on the dot, merge, and split as in \cref{psi-functor} and
\[
\Psi_n\left( 
\xy
(0,0)*{
\bt[scale=.35, color=\clr]
	\draw [ thick, looseness=2, directed=0.99] (1,2.5) to [out=270,in=270] (-1,2.5);
	\node at (-1,3) {\scriptsize $k$};
	\node at (1,3) {\scriptsize $k$};
\et
};
\endxy \right)=\coev_k,\quad\quad
\Psi_n\left(
\xy
(0,0)*{
\bt[scale=.35, color=\clr]
	\draw [ thick, looseness=2, directed=0.99] (1,2.5) to [out=90,in=90] (-1,2.5);
	\node at (-1,2) {\scriptsize $k$};
	\node at (1,2) {\scriptsize $k$};
\et
};
\endxy
\right)=\ev_k.
\]
\end{theorem}

To prove the functor given in \cref{T:Psiupdown} is well-defined requires verifying the defining relations of $\qwebsupdown$.  To do so it is helpful to first compute what happens on strands of thickness $1$.  

\begin{lemma}\label{L:onecases} The following statements hold:
\begin{enumerate}
\item 
\begin{equation*}
\Phi_{n}\left(\xy
(0,0)*{
\bt[scale=1.25, color=\clr]
	\draw[thick, rdirected=0.2] (0,0) to (0.2,0.2);
	\draw[thick ] (0.3,0.3) to (0.5,0.5);
	\draw[thick, directed=1] (0.5,0) to (0,0.5);
	\node at (0,-0.15) {\scriptsize $1$};
	\node at (0,0.65) {\scriptsize $1$};
	\node at (0.5,-0.15) {\scriptsize $1$};
	\node at (0.5,0.65) {\scriptsize $1$};
\et
};
\endxy  \right): V^{*}_{n} \otimes V_{n} \to V_{n} \otimes V_{n}^{*}
\end{equation*} is given for all $a,b \in I_{n|n}$ by 
\begin{equation*}
v_a^* \otimes v_b \mapsto (-1)^{\p{a}\p{b}} q^{\varphi(b,a)} v_b \otimes v_a^* + \delta_{a,b} \tq \sum_{\substack{k \in I_{n|n} \\ b < k}} v_k \otimes v_k^* + \delta_{a,-b} \sum_{\substack{k \in I_{n|n} \\ -b < k}} (-1)^{\p{k}} v_{-k} \otimes v_k^*.
\end{equation*}  Moreover this map is invertible.
\item 
\begin{equation*}
\Phi_{n}\left(\xy
(0,0)*{
\bt[scale=1.25, color=\clr]
	\draw[thick, rdirected=0.1] (0,0) to (0.5,0.5);
	\draw[thick ] (0.5,0) to (0.33,0.18);
	\draw[thick, directed=1] (0.18,.33) to (0,0.5);
	\node at (0,-0.15) {\scriptsize $1$};
	\node at (0,0.65) {\scriptsize $1$};
	\node at (0.5,-0.15) {\scriptsize $1$};
	\node at (0.5,0.65) {\scriptsize $1$};
\et
};
\endxy  \right): V^{*}_{n} \otimes V_{n} \to V_{n} \otimes V_{n}^{*}
\end{equation*}
is given  for all $a,b \in I_{n|n}$ by  
\[
v_a^* \otimes v_b \mapsto  (-1)^{\p{a}\p{b}} q^{\varphi(b,a)} v_b \otimes v_a^* - \delta_{a,b} \tq \sum_{\substack{k \in I_{n|n}\\  k \leq b}} v_k \otimes v_k^* + \delta_{a,-b} \tq \sum_{\substack{k \in I_{n|n}\\ -b < k}} (-1)^{\p{k}}  v_{-k} \otimes v_k^*.
\]  Moreover this map is invertible.
\item 
\begin{equation*}
\Psi_{n}\left( \xy
(0,0)*{
\bt[scale=.35, color=\clr]
	\draw [ thick, looseness=2, rdirected=-0.95] (1,4.5) to [out=90,in=90] (-1,4.5);
	\node at (-1.5,4.75) {\scriptsize $1$};
\et
};
\endxy\right) : V_{n} \otimes V^{*}_{n} \to \K 
\end{equation*} is given for all $a,b \in I_{n|n}$ by 
\[
v_{a} \otimes v^{*}_{b} \mapsto \delta_{a,b}(-1)^{\p{a}}q^{(-1)^{\p{a}}2a-(2n+1)}. 
\]
\end{enumerate}
\end{lemma}

\begin{proof}  The formulas in statements (a.) and (b.) follow by straightforward calculations which we leave to the reader.  In both cases the formulas show we can choose an appropriate ordering of the indices so the maps are upper triangular matrices with $\pm  q^{\pm 1}$ on the diagonal when written as a matrices with respect to this ordered basis.  Consequently both maps are invertible. 

We now prove (c.).  In \cite[Lemma 3.13]{BGJKW} it is shown the formula given in (c.) defines a $\U_{q}(\fq_{n})$-supermodule homomorphism.  Precomposing this with the map given in (a.) yields the evaluation map $\Psi_{n}(\lcap) :V_{n}^{*}\otimes V_{n} \to \K$.  From this and the definition of the rightward cap we deduce this morphism is the rightward cap on a strand labelled by $1$.
\end{proof}

We now prove \cref{T:Psiupdown}.  

\begin{proof} Since $\Psi_{n}$ is already known to preserve the defining relations of $\qwebs$ it suffices to verify relations \cref{straighten-zigzag,delete-bubble} and the invertibility of \cref{E:leftwardcrossing}. Since leftward cups and caps go to coevaluation and evaluation, a standard calculation verifies \cref{straighten-zigzag}. 

The second equality in \cref{delete-bubble} follows from parity considerations since the trivial $U_{q}(\fq)$-module has no nontrivial odd endomorphisms.  To verify the first equality requires we compute the composite $ \Psi_{n}\left({}_{\color{\clr }{1}} \rcap \right)  \circ \coev_{1}: \K \to V_{n} \otimes V_{n}^{*} \to \K$.  That this is zero is verified in the proof of \cite[Theorem 3.16]{BGJKW}; more specifically, it follows from \cref{L:onecases}(c.) and an identity given in \cite[Lemma 3.11]{BGJKW}.

We next verify the leftward crossings defined by \cref{E:leftwardcrossing} are sent to invertible maps in $\modsupdown$. From \cref{L:onecases}(a.)~and~(b.) this is true when the strands are labelled by $1$'s. We define $\Psi_{n}$ on the rightward over-crossing (resp. rightward under-crossing) of strands labelled by $1$ by sending it to the inverse of the map in \cref{L:onecases}(a.) (resp.\ \cref{L:onecases}(b.)). 

We now consider the general case.  Define
\[
\tilde{\beta}_{\down_{k}, \up_{l}} :=\frac{1}{[k]_{q}!\,[l]_{q}!}
\xy
(0,0)*{
\bt[color=\clr, scale=.35]
	\draw [ thick, rdirected=1] (0, .75) to (0,1.5);
	\draw [ thick, rdirected=0.75] (1,-1) to [out=90,in=330] (0,.75);
	\draw [ thick, rdirected=0.75] (-1,-1) to [out=90,in=210] (0,.75);
	\draw [ thick, directed=1] (4, .75) to (4,1.5);
	\draw [ thick, directed=0.75] (5,-1) to [out=90,in=330] (4,.75);
	\draw [ thick, directed=0.75] (3,-1) to [out=90,in=210] (4,.75);
	\draw [ thick, directed=0.75] (0,-6.5) to (0,-5.75);
	\draw [ thick, ] (0,-5.75) to [out=30,in=270] (1,-4);
	\draw [ thick, ] (0,-5.75) to [out=150,in=270] (-1,-4); 
	\draw [ thick, rdirected=0.75] (4,-6.5) to (4,-5.75);
	\draw [ thick, ] (4,-5.75) to [out=30,in=270] (5,-4);
	\draw [ thick, ] (4,-5.75) to [out=150,in=270] (3,-4); 
	\draw [ thick ] (5,-1) to (1,-4);
	\draw [ thick ] (3,-1) to (-1,-4);
	\draw [ thick ] (5,-4) to (3.4,-2.8);
	\draw [ thick ] (2.6,-2.2) to (2.2,-1.9); 
	\draw [ thick ] (1.8,-1.6) to (1,-1);
	\draw [ thick ] (3,-4) to (2.2,-3.4);
	\draw [ thick ] (1.8,-3.1) to (1.4,-2.8);
	\draw [ thick ] (0.6,-2.2) to (-1,-1);
	\node at (2.6, -0.5) {\scriptsize $1$};
	\node at (5.4, -0.5) {\scriptsize $1$};
	\node at (1.4, -0.5) {\scriptsize $1$};
	\node at (-1.4, -0.5) {\scriptsize $1$};
	\node at (2.6, -4.5) {\scriptsize $1$};
	\node at (5.4, -4.5) {\scriptsize $1$};
	\node at (1.4, -4.5) {\scriptsize $1$};
	\node at (-1.4, -4.5) {\scriptsize $1$};
	\node at (0.1, -0.65) { $\cdots$};
	\node at (4.1, -0.65) { $\cdots$};
	\node at (0.1, -4.6) { $\cdots$};
	\node at (4.1, -4.6) { $\cdots$};
	\node at (0,2) {\scriptsize $l$};
	\node at (4,2) {\scriptsize $k$};
	\node at (0,-7) {\scriptsize $k$};
	\node at (4,-7) {\scriptsize $l$};
\et
};
\endxy \ .
\]  This diagram is constructed from the tensor product and composition of upward and downward merges and splits, as well as rightward crossings on strands labelled by $1$'s.  As the functor $\Psi_{n}$ is defined on these we can construct a morphism in $\modsupdown$ by taking the corresponding tensor product and composition of the images under $\Psi_{n}$ of these constituent morphisms. One can verify the resulting map is the inverse under composition for the leftward over-crossing given in \cref{E:leftwardcrossing} by computing diagrammatically and noting the calculation only uses relations which have already been verified.  That is, 
\begin{equation*}
\Phi_{n}\left(\xy
(0,0)*{
\bt[scale=1.25, color=\clr]
	\draw[thick, rdirected=0.2] (0,0) to (0.2,0.2);
	\draw[thick ] (0.3,0.3) to (0.5,0.5);
	\draw[thick, directed=1] (0.5,0) to (0,0.5);
	\node at (0,-0.15) {\scriptsize $k$};
	\node at (0.5,-0.15) {\scriptsize $l$};
\et
};
\endxy  \right) \in \End_{U_{q}(\fq_{n})}\left(S^{k}(V_{n})^{*} \otimes S_{q}^{l}(V_{n}),  S_{q}^{l}(V_{n}) \otimes S^{k}(V_{n})^{*} \right)
\end{equation*} is invertible.  A similar construction provides an inverse for
\begin{equation*}
\Phi_{n}\left(\xy
(0,0)*{
\bt[scale=1.25, color=\clr]
	\draw[thick, rdirected=0.1] (0,0) to (0.5,0.5);
	\draw[thick ] (0.5,0) to (0.33,0.18);
	\draw[thick, directed=1] (0.18,.33) to (0,0.5);
	\node at (0,-0.15) {\scriptsize $k$};
	\node at (0.5,-0.15) {\scriptsize $l$};
\et
};
\endxy  \right) \in \End_{U_{q}(\fq_{n})}\left(S^{k}(V_{n})^{*} \otimes S_{q}^{l}(V_{n}),  S_{q}^{l}(V_{n}) \otimes S^{k}(V_{n})^{*} \right).
\end{equation*}

That the functor is essentially surjective is obvious.  It remains to prove the functor is full.  Using the isomorphisms defined in \cref{SS:Useful isomorphisms} it suffices to prove fullness for objects which lie in $\qwebs$.  By \cref{T:Thetafunctor} we have $\Hom_{\qwebs}\left(\uparrow_{{a}}, \uparrow_{{b}} \right) \cong \Hom_{\qwebsupdown}\left(\uparrow_{{a}}, \uparrow_{{b}} \right) $ for any pair of objects $\uparrow_{{a}}, \uparrow_{{b}}$ which lie in $\qwebs$.  Fullness of $\Psi_{n}$ then follows from the fullness of $\Psi_{n}$ proven in  \cref{T:fullness} for the upward oriented case.
\end{proof}

\subsection{Equivalences of Categories}\label{SS:EquivalencesofCategories}  Recall in \cref{SS:Olshanskiduality} we fixed an even element 
\[
e_{\lambda} \in \HC_{k}(q) \cong \End_{\qwebs}\left(\up_{1}^{k} \right)
\]  for each strict partition $\lambda \in \SP (k)$.  Applying the isomorphisms given by \cref{P:sergeev-isomorphism,T:Thetafunctor}, we abuse notation by writing $e_{\lambda} \in \End_{\qwebs}\left(\up_{1}^{k} \right) \cong \End_{\qwebsupdown}\left(\up_{1}^{k} \right)$ for any $\lambda \in \SP (k)$.  In particular, for $n \geq 1$ let $\lambda(n)$ be the strict partition of $k=(n+1)(n+2)/2$ given by $\lambda(n) = (n+1, n, n-1, \dotsc , 1)$ and let $e_{\lambda(n)} \in \End_{\qwebsupdown}\left(\up_{1}^{k} \right) $
 be the corresponding morphism.

\begin{definition}\label{D:qnWebs} Define $\webs$ (resp.\ $\websupdown$) to be the monoidal supercategory given by the same generators and relations as $\qwebs$ (resp.\ $\qwebsupdown$) along with the relation $e_{\lambda(n)}=0$.
\end{definition}

We are now prepared to state and prove one of the main theorems of the paper.  

\begin{theorem}\label{T:maintheorem} The functor $\Psi_{n}: \qwebsupdown \to \modsupdown$ induces functors 
\begin{align*}
\Psi_{n}: & \webs \to \mods, \\
\Psi_{n}: & \websupdown \to \modsupdown.
\end{align*}
These functors are equivalences of monoidal supercategories.
\end{theorem}

\begin{proof}  In both cases $\Psi_{n}:  \qwebs  \to \mods$ and $\Psi_{n}:  \qwebsupdown \to \modsupdown$ are full functors of monoidal categories.  Furthermore, in both cases $\Psi_{n}(e_{\lambda(n)})=0$ by \cref{P:OlshanskiKernel}.  Therefore in both cases $\Psi_{n}$ induces a full functor of monoidal categories as in the statement of the theorem.  We abuse notation by also calling these functors $\Psi_{n}$.  By construction they are essentially surjective.  

It remains to show it is faithful in both cases.  Using the isomorphisms in \cref{SS:Useful isomorphisms} we may assume without loss of generality that both objects consist of only upward oriented arrows.  Combining \cref{T:Thetafunctor,T:fullness} we may assume without loss of generality  we are in the case of $\Psi_{n}:  \qwebsupdown \to \modsupdown$.

Let $\mathcal{I}$ be the tensor ideal of $\qwebsupdown$ generated by $e_{\lambda(n)}$.  Since $\websupdown = \qwebsupdown / \mathcal{I}$, to show faithfulness amounts to showing if $f \in \Hom_{\qwebsupdown}(\uparrow_{{a}}, \uparrow_{{b}})$ and $\Psi_{n}(f)=0$ for tuples $a=(a_{1}, \dotsc , a_{r})$ and $b=(b_{1}, \dotsc , b_{s})$, then $f \in \mathcal{I}(\uparrow_{{a}}, \uparrow_{{b}})$. Just as in the proof of \cref{T:fullness}, we may assume $|{a}| = |{b}|$ since otherwise $\Psi_{n}$ is trivially faithful. Set $d=|{a}|=|{b}|$.  

As in the proof of \cref{T:fullness}, we can compose with merges and splits to induce the following commutative diagram of morphisms:

\begin{equation}\label{E:PsiBox2}
\begin{tikzcd}
\Hom_{\qwebsupdown   }(\up_{{a}}, \up_{{b}} )  \arrow[r, hook, "\alpha"] \arrow[d,  "\Psi_{n}"] & \Hom_{\qwebsupdown   }\left( \up_{1}^{d}, \up_{1}^{d}\right) \arrow[d,  "\Psi_{n}"]   \\
\Hom_{U_{q}\left(\fq_{n} \right)}\left( S^{\uparrow_{a}}(V_{n}),S^{\uparrow_{b}}(V_{n})\right)\arrow[r, hook, "\tilde{\alpha}"] & \Hom_{U_{q}\left(\fq_{n} \right)}\left( V_{n}^{\otimes d}, V_{n}^{\otimes d}\right).
\end{tikzcd}
\end{equation} 

Thus for $f \in \Hom_{\qwebsupdown   }(\uparrow_{{a}}, \uparrow_{{b}} ) $ we have $\Psi_{n}(f) =0$ if and only if $\Psi_{n}(\alpha(f)) =0$.  We claim $\alpha(f) \in \mathcal{I}\left(\up_{1}^{d}, \up_{1}^{d} \right)$.  Once the claim is proven, the proof of faithfulness will follow from the fact $\alpha$ has a left inverse given by composing by merges and splits and so, since $\mathcal{I}$ is a tensor ideal, the left inverse of $\alpha$ takes elements of $\mathcal{I}(\uparrow_{1}^{d}, \uparrow_{1}^{d})$ to $\mathcal{I}(\uparrow_{{a}}, \uparrow_{{b}})$.  In particular, we will have $f  \in \mathcal{I}(\uparrow_{{a}}, \uparrow_{{b}})$ as desired.

Therefore we have reduced ourselves to showing the claim that  $\Psi_{n}(\alpha(f)) =0$ implies $\alpha(f) \in \mathcal{I}\left(\up_{1}^{d}, \up_{1}^{d} \right)$.  If  $\Psi_{n}(\alpha(f)) =0$, then from \cref{P:sergeev-isomorphism,P:OlshanskiKernel} it follows 
\[
\alpha(f) = \sum_{\gamma_{i}} g_{i}e_{\gamma_{i}}h_{i},
\] where the sum is over all strict partitions $\gamma_{i}$ of $d$ with $\ell (\gamma_{i})>n$ and $g_{i},h_{i} \in \End_{\qwebsupdown}(\uparrow_{1}^{d})$.  Since $\mathcal{I}$ is a tensor ideal, if $e_{\gamma_{i}} \in \mathcal{I}\left(\up_{1}^{d}, \up_{1}^{d} \right)$ for all $i$, then $\alpha(f) \in \mathcal{I}\left(\up_{1}^{d}, \up_{1}^{d} \right)$ as desired.

That is, we have reduced ourselves to proving if $\gamma$ is a strict partition of $d$ with $\ell(\gamma) > n$, then $e_{\gamma} \in \mathcal{I}(\uparrow_{1}^{d}, \uparrow_{1}^{d})$.  Fix such a partition $\gamma$.  Let $k =|\lambda(n)|= (n+1)(n+2)/2$.  Since $\gamma$ is a partition of $d$ with $\ell (\gamma) > n$, we have $d \geq k$. Also, from \cite[Corollary 8.1.5]{BrownKujawa} we know in the classical setting $L(\gamma)$ appears as a direct summand of $L(\lambda(n)) \otimes V_{n}^{\otimes (d-k)}$.  By \cref{T:Grantcharovtheorem}(d.) the same is true in the quantum setting.  It follows from considering the total weight that $L_{q}((\lambda(n))$ is a direct summand of $V_{n}^{\otimes k}$ and, hence, $L_{q}(\gamma)$ appears as a direct summand in $L_{q}(\lambda(n)) \otimes V_{n}^{\otimes (d-k)}\subseteq V_{n}^{\otimes k} \otimes V^{\otimes d-k} \subseteq V_{n}^{\otimes d}$.    From the discussion after \cref{E:Olshanskidecomposition} this implies
\[
\left( e_{\lambda(n)} \otimes \Id_{V_{n}}^{\otimes (d-k)}\right) V_{n}^{\otimes r} = \left( e_{\lambda(n)}V_{n}^{\otimes k} \right) \otimes V_{n}^{\otimes (d-k)}
\] contains $L_{q}(\gamma)$ as a direct summand.  That is, if we use the decomposition given in \cref{E:ArtinWedderburn} and write 
\[
e_{\lambda(n)} \otimes \Id_{V_{n}}^{\otimes (d-k)} = \sum_{\mu\in \SP (d)} x_{\mu},
\] with $x_{\mu}$ an even element in $M(E_{\mu})$, then the discussion after \cref{E:Olshanskidecomposition} shows $x_{\gamma} \neq 0$.  Let $\iota_{\gamma} \in \HC_{k}(q)$ denote the idempotent corresponding to the identity of $M(E_{\gamma})$.  Then $\iota_{\gamma}\left( e_{\lambda(n)} \otimes \Id_{V_{n}}^{\otimes (d-k)}\right) = x_{\gamma}$ is a nonzero even element of $M(E_{\gamma})$.  Since  $M(E_{\gamma})$ is a simple superalgebra it follows that $e_{\gamma} = \sum_{t} a_{t}x_{\gamma}b_{t}$ for some $a_{t},b_{t} \in \HC_{d}(q) \cong \End_{\qwebsupdown}\left(\uparrow_{1}^{d} \right)$. Putting these together yields
\[
e_{\gamma} = \sum_{t} a_{t}\iota_{\gamma}\left( e_{\lambda(n)} \otimes \Id_{V_{n}}^{\otimes (r-k)}\right)b_{t}.
\]
Recall, $\mathcal{I}$ is a tensor ideal which contains $e_{\lambda(n)}$.  Since it is closed under tensor products, compositions, and linear combinations, the previous equation demonstrates $e_{\gamma} \in \mathcal{I}(\uparrow_{1}^{d},\uparrow_{1}^{d})$, as desired. 
\end{proof}

\subsection{The field of rational functions}\label{SS:rationalfunctions}  Before continuing we explain how to deduce analogues of our main results for the ground field $\k = \C (q)$ of rational functions in the variable $q$.  We first observe that all the mathematical objects introduced so far can be defined over $\k$.  In this section we write a subscript $\k$ or $\K$ to indicate the relevant base field.  For example, $U_{q}(\fq_{n})_{\k}$ and $U_{q}(\fq_{n})_{\K}$, $\HC_{q}(k)_{\k}$ and $\HC_{q}(k)_{\K}$, $\AA_{q, \k}$ and $\AA_{q, \K}$,  $\qwebsk$ and $\qwebsK$,  $\qwebsupdownk$ and $\qwebsupdownK$, and so on.  In nearly every case base change behaves well in the sense there are canonical isomorphisms $U_{q}(\fq_{n})_{\k} \otimes_{\k} \K \cong  U_{q}(\fq_{n})_{\K}$, $\AA_{q, \k} \otimes_{\k} \otimes \K \cong \AA_{q, \K}$, $\qwebsk \otimes_{\k} \K \cong \qwebsK$, and so forth. 

There is one crucial exception.  By character considerations and the discussion in \cite{GJKK,GJKKashK} it follows that $L_{q}(\lambda)_{\k} \otimes_{\k} \K$ is isomorphic to the direct sum of one or two simple $U_{q}(\fq_{n})_{\K}$-supermodules, both isomorphic to $L_{q}(\lambda)_{\K}$.  However, there is an important exception to the exception: the results of \cite{GJKK,GJKKashK} imply $S_{q}^{d}\left(V_{n} \right)_{\k} \otimes_{\k} \K \cong S_{q}^{d}\left(V_{n} \right)_{\K}$ for all $d \geq 0$.  From this it follows  we have
\begin{equation}\label{E:keyisom}
S_{q}^{\ob{a}}\left(V_{n} \right)_{\k} \otimes_{\k} \K \cong S_{q}^{\ob{a}}\left(V_{n} \right)_{\K}
\end{equation}
for any object $\ob{a} \in \qwebsupdownk$. 

 A base change argument shows \cref{L:weight-space-isomorphism} also holds over $\k$ and since the relevant algebraic and diagrammatic calculations could have been done over $\k$ one can define the functors $\Phi_{m,n}$, $\Pi_{m}$, and $\Psi_{n}$ over $\k$ without difficulty.  By \cite[Proposition 2.1]{JN} $\HC_{k}(q)_{\k}$ is semisimple and by Olshanki's double centralizer theorem (e.g.\ see \cite[Theorem 3.28]{BGJKW}) the map 
\[
\psi: \HC_{k} \to \End_{U_{q}(\fq_{n})_{\k}}\left(V_{n, \k}^{\otimes k} \right)
\] is surjective and an isomorphism if $n \gg k$.  Consequently the functor $\Psi_{n}$ is still essentially surjective and full.  

 Moreover, by a standard result regarding base change (c.f.\ \cite[Lemma 29.3]{CR}) and \cref{E:keyisom} we have, for all objects $\ob{a}$ and $\ob{b}$ in $\qwebsk$, the following isomorphisms of superspaces:
\begin{align*}
\Hom_{U_{q}(\fq_{n})_{\k}}\left(S^{\ob{a}}_{q}(V_{n})_{\k}, S_{q}^{\ob{b}}\left(V_{n} \right)_{\k} \right) \otimes_{\k} \K &\cong \Hom_{U_{q}(\fq_{n})_{\K}}\left(S^{\ob{a}}_{q}(V_{n})_{\k} \otimes_{\k} \K, S_{q}^{\ob{b}}\left(V_{n} \right)_{\k} \otimes_{\k} \K \right) \\
& \cong  \Hom_{U_{q}(\fq_{n})_{\K}}\left(S^{\ob{a}}_{q}(V_{n})_{\K}, S_{q}^{\ob{b}}\left(V_{n} \right)_{\K} \right). 
\end{align*}
 From these it follows a morphism is in the kernel of the functor $\Psi_{n}: \qwebsupdownk \to \modsupdownk$ if and only if it is in the kernel of $\Psi_{n}: \qwebsupdownK \to \modsupdownK$ after base change.

Since $\HC_{q}(k)_{\k}$ is semisimple but not split semisimple, a decomposition as in \cref{E:ArtinWedderburn} still holds by the graded version of the Artin-Wedderburn theorem, but now each simple summand is of the form $M(\tilde{E}_{\lambda})$ for some division superalgebra $\tilde{E}_{\lambda}$ (that is, a superalgebra in which every homogeneous element is invertible).  Nevertheless, one can still define the even elements $e_{\lambda} \in M(\tilde{E}_{\lambda})_{\k}$ as before. Since under base change $M(\tilde{E}_{\lambda}) \otimes_{\k} \K \cong M(E_{\lambda})$ for all $\lambda \in \SP (k)$, we could just as well have used the elements $e_{\lambda} \otimes 1 \in M(E_{\lambda})$ when we worked over $\K$.  If we define $\websk$ and $\websupdownk$ as before by imposing the relation $e_{\lambda(n)}=0$, then again $\Psi_{n}$ induces well-defined essentially surjective functors $\Psi_{n}: \websk \to \modsk$ and $\Psi_{n}: \websupdownk \to \modsupdownk$.  Applying the discussion in the previous paragraph along with base change arguments show these are equivalences of monoidal supercategories, as desired.

We remark the defining relations of $\qwebs$ only require coefficients from $\Z[q,q^{-1}]$.  One could choose to define $\qwebs$ as a $\Z[q,q^{-1}]$-linear monoidal supercategory.  Moreover, the crossing morphisms can be defined in this category using formulas analogous to those in \cite[Corollary 6.2.3]{CKM} and, hence, $\qwebsupdown$ can also be defined as a  $\Z[q,q^{-1}]$-linear monoidal supercategory.  These supercategories would then relate to the $\Z[q,q^{-1}]$-form for $U_{q}(\fq_{n})$ introduced in \cite[Section 8]{DW}. 

\subsection{Endomorphism superalgebras and the quantum walled Brauer-Clifford superalgebra}\label{SS:quantumBCalgebra}

Recall, given an object $\ob{a}$ in $\qwebsupdown$ we write $S^{\ob{a}}_{q}\left(V_{n} \right)$ for $\Psi_{n}\left(\ob{a} \right)$.  If $\ob{a}= \uddoublearrow_{a_{1}} \dotsb \uddoublearrow_{a_{t}}$, then we write $\tilde{\ob{a}}$ for the object $\uddoublearrow_{a_{t}}\dotsb \uddoublearrow_{a_{1}}$ and write $|\ob{a}|$ for $\sum_{i=1}^{t} a_{i}$.

\begin{theorem}\label{T:EndomorphismAlgebras}  Let $\ob{c}$ be an arbitrary object of $\qwebsupdown$. Then 
\[
\Psi_{n}: \End_{\qwebsupdown}\left(\ob{c} \right) \to \End_{U_{q}\left(\fq_{n} \right)}\left(S_{q}^{\ob{c}}\left(V_{n} \right) \right)
\] is a surjective superalgebra homomorphism.  Moreover this map is an isomorphism whenever $|\ob{c}| < (n+1)(n+2)/2$.

\end{theorem}

\begin{proof}  That the map is a surjective superalgebra homomorphism follows from the fact $\Psi_{n}$ is a full superfunctor.  The only thing which needs to be verified is injectivity.

Recalling our convention that edges labelled by $0$ may be included or omitted at will, we may choose tuples of nonnegative integers ${a}=(a_{1}, \dotsc , a_{r})$ and  ${b}=(b_{1}, \dotsc , b_{r})$ so the object can be written in the form 
\[
\ob{c}=\uparrow_{a_{1}}\downarrow_{b_{1}}\uparrow_{a_{2}}\downarrow_{b_{2}}\dotsb \uparrow_{a_{r}}\downarrow_{b_{r}}.
\]  Having done so there is the following commutative diagram:
\begin{equation}\label{E:AnotherPsibox}
\begin{tikzcd}
\End_{\qwebsupdown}\left( \ob{c}\right) \arrow[d,  "\Psi_{n}"]  \arrow[r, "\cong"] & \End_{\qwebsupdown}\left(\uparrow_{{a}}\downarrow_{\tilde{{b}}}\right)  \arrow[d,  "\Psi_{n}"]   \arrow[r, "\cong"] & \End_{\qwebsupdown}\left(\uparrow_{\ob{a}}\uparrow_{\ob{b}}\right)  \arrow[d,  "\Psi_{n}"]  \arrow[r, hook, "\alpha"]&  \End_{\qwebsupdown}\left( \up_{1}^{|{a}|+|{b}|}\right) \arrow[d,  "\Psi_{n}"]   \\
\End_{U_{q}\left(\fq_{n} \right)}\left(S_{q}^{\ob{c}}(V_{n}) \right) \arrow[r, "\cong"]&\End_{U_{q}\left(\fq_{n} \right)}\left(S_{q}^{\uparrow_{{a}}\downarrow_{\tilde{{b}}}}(V_{n}) \right)  \arrow[r, "\cong"]&\End_{U_{q}\left(\fq_{n} \right)}\left( S_{q}^{\uparrow_{{a}}\uparrow_{{b}}}(V_{n})\right)  \arrow[r, hook, "\tilde{\alpha}"]&   \End_{U_{q}\left(\fq_{n} \right)}\left( V_{n}^{|{a}|+|{b}|}\right).
\end{tikzcd}
\end{equation} The first isomorphism in the top row is the superalgebra isomorphism given by $w \mapsto d_{\sigma} \circ w \circ d^{-1}_{\sigma}$, where $d_{\sigma}$ is as in \cref{SS:Useful isomorphisms}. The second isomorphism in the top row is the map of superspaces given on diagrams by
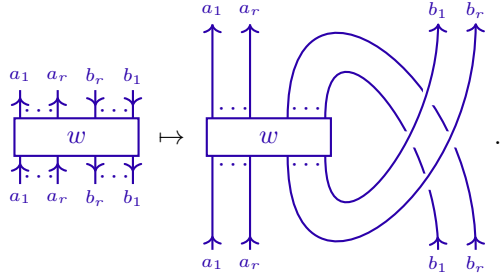
\[
\xy
(0,0)*{
\begin{tikzpicture}[scale=.25, color=\clr]
	\draw [ thick] (-2.3,-8) rectangle (4.3,-6);
	\node at (1,-7) {$w$};
	\draw [ thick, directed=0.75] (-2,-9.5) to (-2,-8);
	\draw [ thick, directed=0.75] (0,-9.5) to (0,-8);
	\draw [ thick, rdirected=0.75] (2,-9.5) to (2,-8);
	\draw [ thick, rdirected=0.75] (4,-9.5) to (4,-8);
	\node at (2,-10.25) {\scriptsize $b_{r}$};
	\node at (4,-10.25) {\scriptsize $b_{1}$};
	\node at (-2,-10.25) {\scriptsize $a_1$};
	\node at (-1,-9.125) {\,$\cdots$};
	\node at (3,-9.125) {\,$\cdots$};
	\node at (0,-10.25) {\scriptsize $a_r$};
	\draw [ thick, rdirected=0.75] (2,-6) to (2,-4.5);
	\draw [ thick, rdirected=0.75] (4,-6) to (4,-4.5);
	\draw [ thick, directed=0.75] (-2,-6) to (-2,-4.5);
	\draw [ thick, directed=0.75] (0,-6) to (0,-4.5);
	\node at (2,-3.75) {\scriptsize $b_{r}$};
	\node at (3,-5.625) {\,$\cdots$};
	\node at (-1,-5.625) {\,$\cdots$};
	\node at (4,-3.75) {\scriptsize $b_{1}$};
	\node at (-2,-3.75) {\scriptsize $a_1$};
	\node at (0,-3.75) {\scriptsize $a_r$};
\end{tikzpicture}
};
\endxy\mapsto
\xy
(0,0)*{
\begin{tikzpicture}[scale=.25, color=\clr]
	\draw [ thick] (-2.3,-8) rectangle (4.3,-6);
	\draw [ thick, looseness=1.75, directed=0.05] (10,-13) to [out=90,in=90] (4,-6);
	\draw [ thick, looseness=2, directed=0.035] (12,-13) to [out=90,in=90] (2,-6);
	\node at (1,-7) {$w$};
	%
	\node at (12,-0.25) {\scriptsize $b_{r}$};
	\node at (-1,-8.5) {\ $\cdots$};
	\node at (10,-0.25) {\scriptsize $b_{1}$};
	\node at (10,-13.75) {\scriptsize $b_{1}$};
	\node at (-1,-5.5) {\ $\cdots$};
	\node at (12,-13.75) {\scriptsize $b_{r}$};
	\draw [ thick, directed=1] (-2,-6) to (-2,-1);
	\draw [ thick, directed=1] (0,-6) to (0,-1);
	\draw [ thick, directed=0.15] (-2,-13) to (-2,-8);
	\draw [ thick, directed=0.15] (0,-13) to (0,-8);
	\node at (0,-13.75) {\scriptsize $a_{r}$};
	\node at (3,-8.5) {\,$\cdots$};
	\node at (-2,-13.75) {\scriptsize $a_{1}$};
	\node at (-2,-0.25) {\scriptsize $a_{1}$};
	\node at (3,-5.5) {\,$\cdots$};
	\node at (0,-0.25) {\scriptsize $a_{r}$};
	\draw [color=white, fill=white,  thick ] (9.25,-4.5) rectangle (10,-5.25);
	\draw [color=white, fill=white,  thick ] (8.25,-6.75) rectangle (9,-7.5);
	\draw [color=white, fill=white,  thick ] (10.25,-6.75) rectangle (11,-7.5);
	\draw [color=white, fill=white,  thick ] (9.25,-9) rectangle (10,-9.75);
	\draw [ thick, looseness=2, rdirected=0.02] (12,-1) to [out=270,in=270] (2,-8);
	\draw [ thick, looseness=1.75, rdirected=0.03] (10,-1) to [out=270,in=270] (4,-8);
\end{tikzpicture}
};
\endxy.
\]
 The third map in the top row is the embedding of superspaces of the same name defined in the proof of \cref{T:fullness}.  All three maps and their inverses are given by tensoring and composing with webs.  Since $\Psi_{n}$ is a monoidal superfunctor, applying this functor yields the maps in the bottom row and, moreover, the diagram commutes.  Combining \cref{P:OlshanskiKernel,P:sergeev-isomorphism} shows the rightmost $\Psi_{n}$ is injective whenever $|\ob{c}|=|a|+|b| < (n+1)(n+2)/2$.  By the commutativity of the diagram so are the other $\Psi_{n}$'s.  This proves the claimed result.
\end{proof}
We now consider the special case when $\ob{c} = \uparrow_{1}^{ r} \otimes \downarrow_{1}^{ s}$.  Given $r,s \geq 0$, the quantum walled Brauer-Clifford superalgebra $\BC_{r,s}(q)$ is defined in \cite[Definition 3.4]{BGJKW} to be the associative superalgebra on even generators $t_{1}, \dotsc , t_{r-1}, t^{*}_{1}, \dotsc , t^{*}_{s-1},e$, and odd generators $c_{1}, \dotsc , c_{r}, c^{*}_{1}, \dotsc , c^{*}_{s}$ subject to the following relations (for all admissible $i,j$):

\begin{enumerate}
\item $t_{i}^{2} - \tq t_{i}-1=0$ and $(t^{*}_{i})^{2} - \tq t^{*}_{i}-1=0$;
\item $t_{i}t_{i+1}t_{i}=t_{i+1}t_{i}t_{i+1}$ and $t^{*}_{i}t^{*}_{i+1}t^{*}_{i}=t^{*}_{i+1}t^{*}_{i}t^{*}_{i+1}$;
\item $t_{i}t_{j}=t_{j}t_{i}$ and $t^{*}_{i}t^{*}_{j}=t^{*}_{j}t^{*}_{i}$ when $|i-j| > 1$;
\item $et_{r-1}e = e$ and $et^{*}_{1}e=e$;
\item $et_{j}=t_{j}e$ when $j \neq t-1$ and $et^{*}_{j} = t^{*}_{j}e$ when $j \neq 1$;
\item $e^{2}=0$;
\item $et^{-1}_{r-1}t^{*}_{1}t^{-1}_{r-1} = t^{-1}_{r-1}t^{*}_{1}et^{*}_{1}t^{-1}_{r-1}e$;
\item $c_{i}^{2}=1$ and $(c^{*}_{i})^{2}=-1$;
\item $c_{i}c_{j}=-c_{j}c_{i}$ and $c^{*}_{i}c^{*}_{j}=-c^{*}_{j}c^{*}_{i}$ when $i \neq j$;
\item $c_{i}c^{*}_{j}=c^{*}_{j}c_{i}$;
\item $c_{i}t_{i}=t_{i}c_{i+1}+\tq \left(c_{i}-c_{i+1} \right)$ and $c^{*}_{i}t_{i}^{*}= t^{*}_{i}c^{*}_{i+1}+\tq \left(c^{*}_{i}-c^{*}_{i+1} \right)$;
\item  $t_{i}c_{j}= c_{j}t_{i}$ and  $t^{*}_{i}c^{*}_{j}= c^{*}_{j}t^{*}_{i}$ when $j \neq i, i+1$;
\item $t_{i}c^{*}_{j}=c_{j}^{*}t_{i}$ and  $t^{*}_{i}c_{j}=c_{j}t^{*}_{i}$;
\item $c_{r}e=c_{1}^{*}e$ and $ec_{r}=ec_{1}^{*}$;
\item $c_{i}e=ec_{i}$ when $i \neq r$ and $c^{*}_{i}e=ec^{*}_{i}$ when $i \neq  1$;
\item $ec_{r}e=0$.
\end{enumerate}

Note, we choose a different normalization in (h.) compared with the analogous relation in \cite{BGJKW}. Rescaling their Clifford generators by $\sqrt{-1}$ will yield the ones used here.  Also, there is a typographic error in relation (k.) as given in \cite[Definition 3.4]{BGJKW}.  Their Hecke-Clifford relations should be replaced with the ones used here.

\begin{definition}\label{BC-webs} Fix $r,s\in\Z_{\geq 0}$. All strands are labelled by $1$ in what follows.
\begin{enumerate}
\item For $i=1, \dotsc , r-1$ and $j = 1, \dotsc , s-1$, define morphisms $T_{i}, T_{j}^* \in\End_{\qwebs}(\up_{1}^{r}\down_{1}^{s})$ by
\[
T_{j}=
\xy
(0,0)*{
\bt[color=\clr, scale=1.25]
	\draw[thick, directed=1] (-0.75,0) to (-0.75,0.5);
	\draw[thick, directed=1] (-0.25,0) to (-0.25,0.5);
	\draw[thick, directed=1] (0.75,0) to (0.75,0.5);
	\draw[thick, directed=1] (1.25,0) to (1.25,0.5);
	\draw[thick, directed=1] (0,0) to (0.5,0.5);
	\draw[thick, directed=1] (0.2,0.3) to (0,0.5);
	\draw[thick, ] (0.5,0) to (0.3,0.2);
	\node at (-0.5,0.25) { \ $\cdots$};
	\node at (1,0.25) { \ $\cdots$};
	\draw[thick, <-] (1.6,0) to (1.6,0.5);
	\draw[thick, <-] (2.1,0) to (2.1,0.5);
	\node at (1.85,0.25) { \ $\cdots$};
\et
};
\endxy \ ,
\quad T^{*}_j=
\xy
(0,0)*{
\bt[color=\clr, scale=1.25]
	\draw[thick, ->] (-1.1,0) to (-1.1,0.5);
	\draw[thick, ->] (-1.6,0) to (-1.6,0.5);
	\node at (-1.35,0.25) { \ $\cdots$};
	\draw[thick, <- ] (-0.75,0) to (-0.75,0.5);
	\draw[thick, <-] (-0.25,0) to (-0.25,0.5);
	\draw[thick, <-] (0.75,0) to (0.75,0.5);
	\draw[thick, <-] (1.25,0) to (1.25,0.5);
	\draw[thick, <-] (0,0) to (0.5,0.5);
	\draw[thick, -] (0.2,0.3) to (0,0.5);
	\draw[thick, <-] (0.5,0) to (0.3,0.2);
	\node at (-0.5,0.25) { \ $\cdots$};
	\node at (1,0.25) { \ $\cdots$};
\et
};
\endxy ,
\]  where the upward crossing is between the $i$ and $(i+1)$-st upward strands and the downward crossing is between the $j$ and $(j+1)$-st downward strands.

\item For $i=1, \dotsc , r$ and $j=1, \dotsc , s$, define $C_i, C_j^* \in \End_{\qwebs}(\up_{1}^{r}\down_{1}^{s})$ by
\[
C_i=
\xy
(0,0)*{
\bt[color=\clr, scale=1.25]
	\draw[thick, directed=1] (-0.75,0) to (-0.75,0.5);
	\draw[thick, directed=1] (-0.25,0) to (-0.25,0.5);
	\draw[thick, directed=1] (0.25,0) to (0.25,0.5);
	\draw[thick, directed=1] (0.75,0) to (0.75,0.5);
	\draw[thick, directed=1] (0,0) to (0,0.5);
	\node at (-0.5,0.25) { \ $\cdots$};
	\node at (0.5,0.25) { \ $\cdots$};
	\draw (0,0.25) \wdot;
	\draw[thick, <-] (1,0) to (1,0.5);
	\draw[thick, <-] (1.5,0) to (1.5,0.5);
	\node at (1.25,0.25) { \ $\cdots$};
\et
};
\endxy \ ,
\quad C^{*}_j=
\xy
(0,0)*{
\bt[color=\clr, scale=1.25]
	\draw[thick, ->] (-1,0) to (-1,0.5);
	\draw[thick, ->] (-1.5,0) to (-1.5,0.5);
	\node at (-1.25,0.25) { \ $\cdots$};
	\draw[thick, <-] (-0.75,0) to (-0.75,0.5);
	\draw[thick, <-] (-0.25,0) to (-0.25,0.5);
	\draw[thick, <-] (0.25,0) to (0.25,0.5);
	\draw[thick, <-] (0.75,0) to (0.75,0.5);
	\draw[thick, <-] (0,0) to (0,0.5);
	\node at (-0.5,0.25) { \ $\cdots$};
	\node at (0.5,0.25) { \ $\cdots$};
	\draw (0,0.25) \wdot;
\et ,
};
\endxy
\]  where the dot in $C_{i}$ is on the $i$th upward strand and  the dot in $C^{*}_{j}$ is on the $j$th downward strand.

\item Let $E \in \End_{\qwebs}(\up_{1}^{r}\down_{1}^{s})$ denote the web
\[
E=
\xy
(0,0)*{
\bt[color=\clr, scale=1.25]
	\draw[thick, -> ] (-0.75,0) to (-0.75,0.5);
	\draw[thick, ->] (-0.25,0) to (-0.25,0.5);
	\draw[thick, <-] (0.75,0) to (0.75,0.5);
	\draw[thick, <-] (1.25,0) to (1.25,0.5);
	\draw[thick, ->, looseness=1.5] (0,0) to [out=90, in=90] (0.5,0);
	\draw[thick, <-, looseness=1.5] (0,0.5) to [out=270, in=270] (0.5,0.5);
	\node at (-0.5,0.25) { \ $\cdots$};
	\node at (1,0.25) { \ $\cdots$};
\et
};
\endxy .
\]

\end{enumerate}

\end{definition}

\begin{theorem}\label{T:quantumBCalgebra}  There is a superalgebra isomorphism 
\[
\phi: \BC_{r,s}(q) \to \End_{\qwebsupdown}\left(\up_{1}^{r}\down_{1}^{s} \right) 
\] given on generators by $\phi \left(t_{i} \right) = T_{i}$, $\phi \left(t^{*}_{i} \right) = T^{*}_{i}$, $\phi \left(c_{i} \right) = C_{i}$, $\phi \left(c^{*}_{i} \right) = C^{*}_{i}$,  $\phi \left(e \right) = E$.
\end{theorem}

\begin{proof}  It is straightforward using the relations given in \cref{L:srelations,L:thicknessonerelations} to verify the defining relations of $\BC_{r,s}(q)$ hold in $\End_{\qwebsupdown}\left(\up_{1}^{r}\down_{1}^{s} \right)$.  Thus $\phi$ gives a well-defined superalgebra homomorphism.

We next prove $\phi$ is surjective.  Arguing as in the proof of \cref{sergeev-generators}, every web in $\End_{\qwebsupdown}\left(\up_{1}^{r}\down_{1}^{s} \right) $ can be written as a linear combination of webs consisting of only thickness $1$ strands. Using \cref{L:thicknessonerelations} and arguing by induction on the number of crossings it follows any such web can be written as a linear combination of webs where all strands have thickness $1$ and are of the form $AXB$ where (1) $A = a\downarrow^{s}$ for some web $a$ of type $\uparrow_{1}^{r} \to \uparrow_{1}^{r}$, (2) $X = \uparrow_{1}^{r-p}x_{2}x_{1}\downarrow_{1}^{s-p}$ where $0 \leq p \leq \operatorname{min}(r,s)$ and $x_{1}$ is of type $\uparrow_{1}^{p}\downarrow_{1}^{p} \to \unit$,  $x_{2}$ is of type $\unit  \to \uparrow_{1}^{p}\downarrow_{1}^{p}$ and neither have dots or crossings, and (3) $B=\uparrow^{r}b$ for some web $b$ of type $\downarrow_{1}^{s} \to \downarrow_{1}^{s}$.  Applying \cref{sergeev-generators} to $a$ it follows $A$ can be written as a linear combination of compositions of $T_{i}$'s and $C_{i}$'s.  Similarly, applying leftward cups and caps to \cref{sergeev-generators} yields an analogous theorem for downward arrows which can be applied to $b$.  From this we see $B$ can be obtained as a linear combination of compositions of $T_{i}^{*}$'s and $C_{i}^{*}$'s.  It remains to observe the web $X$ can be obtained by pre- and post-composing $p$ $E$'s along with $T_{i}$'s and $T_{i}^{*}$'s.  In short, every element of $\End_{\qwebsupdown}\left(\up_{1}^{r}\down_{1}^{s} \right) $ can be written as a linear combination of webs generated by $E$'s, $T_{i}$'s, $C_{i}$'s, $T_{i}^{*}$'s, and $C_{i}^{*}$'s.  Consequently, $\phi$ is surjective.

Using the superspace isomorphisms established in \cref{SS:Useful isomorphisms} along with \cref{P:sergeev-isomorphism} it follows the dimension of $\End_{\qwebsupdown}\left(\up_{1}^{r}\down_{1}^{s} \right)$ equals the dimension of $\HC_{r+s}(q)$, which is known to equal $(r+s)!2^{r+s}$.  But by \cite[Corollary 3.25]{BGJKW} this is the dimension of $\BC_{r,s}(q)$.  Therefore $\phi$ is an isomorphism.
\end{proof}

The superalgebra $\BC_{r,s}(q)$ is given a combinatorial description in \cite[Section 4]{BGJKW} whose diagrammatics essentially agrees with the one used here for $\End_{\qwebsupdown}\left(\up_{1}^{r}\down_{1}^{s} \right)$.  Combining the previous two theorems shows there is a surjective superalgebra map $BC_{r,s}(q) \to \End_{U_{q}(\fq_{n})}\left(V_{n}^{\otimes r} \otimes (V_{n}^{*})^{\otimes s} \right)$ which is an isomorphism when $r+s < (n+1)(n+2)/2$.  This recovers \cite[Theorem 3.28]{BGJKW} with an improved bound for when injectivity holds.  Moreover, an examination of the proof of \cref{T:EndomorphismAlgebras} shows this bound is sharp.

\subsection{Modified traces}\label{SS:mtraces}  Recall the ribbon category $\qwebsupdowneven$ from \cref{SS:RibbonCategory}.  This is the subcategory of $\qwebsupdown$ consisting of all objects and the morphisms which can be written as a linear combination of webs with  no dots.  For a $\K$-linear ribbon category such as $\qwebsupdowneven$ there is a notion of \emph{modified trace functions} or \emph{m-traces} for short.  We first summarize the relevant definitions and results from \cite{GKPM} as they apply to $\qwebsupdowneven$. 

Given a subcategory $\II$ of $\qwebsupdown$ and an object ${\ob{a}}$ in $\qwebsupdowneven$, we write ${\ob{a}} \in \II$ as a shorthand for the assertion that ${\ob{a}}$ is an object of $\II$. A \emph{thick tensor ideal} or just \emph{ideal} of $\qwebsupdowneven$ is a full subcategory $\II$ which satisfies the following two conditions:
\begin{itemize}
\item If ${\ob{a}} \in \II$ and ${\ob{b}} \in \qwebsupdowneven$, then ${\ob{a}}\otimes {\ob{b}}  \in \II$. 
\item The subcategory $\II$ is closed under retracts;  that is, if ${\ob{a}} \in \II$, ${\ob{b}} \in \qwebsupdowneven$, and there are morphisms $\alpha: {\ob{b}} \to {\ob{a}}$ and $\beta: {\ob{a}} \to {\ob{b}}$ such that $\beta \circ \alpha = \Id_{{\ob{b}}}$, then ${\ob{b}} \in \II$. 
\end{itemize} 
Given an object ${\ob{a}}$ the \emph{ideal generated by ${\ob{a}}$} is the full subcategory with objects
\begin{equation*}
\II_{{\ob{a}}} = \left\{ {\ob{b}} \mid {\ob{b}} \text{ is a retract of ${\ob{a}} \otimes {\ob{c}}$ for some object ${\ob{c}} \in \qwebsupdowneven$  }  \right\}.
\end{equation*}  That this is an ideal is a straightforward check.  Similarly, it is straightforward to check $\II_{{\ob{b}}}$ is a full subcategory of $\II_{{\ob{a}}}$ for all ${\ob{b}} \in \II_{{\ob{a}}}$.

Let $\II$ be an ideal of $\qwebsupdowneven$.  An \emph{m-trace} on $\II$ is a family of $\K$-linear functions 
\[
\left\{\mt_{{\ob{a}} } : \End_{\qwebsupdowneven}({\ob{a}}) \to \K  \right\}_{{\ob{a}} \in \II}
\] which satisfy the following two conditions:
\begin{itemize}
\item For all ${\ob{a}} \in \II$ and ${\ob{b}} \in \qwebsupdowneven$ and all $g \in \End_{\qwebsupdowneven}\left({\ob{a}}\otimes {\ob{b}} \right)$, 
\[
\mt_{{\ob{a}}\otimes{\ob{b}}}\left(g \right) 
= 
\mt_{\ob{a}}\left( 
\mathord{\begin{tikzpicture}[baseline=0,scale=.3, color=\clr]
		\node [style=none] (0) at (0, 0) {};
		\node [rectangle, draw, thick] (1) at (0, 0) {$g$};
		\node [style=none] (2) at (-0.4, 0.75) {};
		\node [style=none] (3) at (0.4, 0.75) {};
		\node [style=none] (4) at (0.4, -0.75) {};
		\node [style=none] (5) at (-0.4, -0.75) {};
		\node [style=none] (6) at (-0.4, 2) {};
		\node [style=none] (7) at (2, 0.7) {};
		\node [style=none] (8) at (2, -0.7) {};
		\node [style=none] (9) at (-0.4, -2) {};
		\node [style=none] (10) at (-1, -2) {\scriptsize{$\ob{a}$}};
		\node [style=none] (11) at (-1, 2) {\scriptsize{$\ob{a}$}};
		\node [style=none] (13) at (0.2, 1.4) {\scriptsize{$\ob{b}$}};
		\node [style=none] (14) at (0.2, -1.4) {\scriptsize{$\ob{b}$}};
		\draw [thick,-, looseness=1.50] (3.center) to [out=90,in=90] (7.center);
		\draw  [thick, -] (7.center) to (8.center);
		\draw [thick, -, looseness=1.5] (8.center) to [out=270,in=270] (4.center);
		\draw  [thick, -] (2.center) to (6.center);
		\draw  [thick, -] (5.center) to (9.center);
\end{tikzpicture}}
 \right).
\]  Here and below we draw an unoriented cup/cap labelled by $\ob{b}$ to represent the appropriately oriented ``rainbow'' morphisms defined in \cref{SS:RibbonCategory}. 
\item For all ${\ob{a}}, {\ob{b}} \in \II$ and all $f \in \Hom_{\qwebsupdowneven}\left({\ob{a}},{\ob{b}} \right)$ and $g \in \Hom_{\qwebsupdowneven}\left({\ob{b}},{\ob{a}} \right)$,
\[
\mt_{{\ob{a}}}\left(g \circ f \right) = \mt_{{\ob{b}}}\left(f \circ g \right).
\]
\end{itemize}

Given an object ${\ob{a}}$ in $\qwebsupdowneven$, we call a $\K$-linear map $\mt : \End_{\qwebsupdowneven}\left( {\ob{a}}\right) \to \K$ an \emph{ambidextrous trace} if for all $g \in \End_{\qwebsupdowneven}({\ob{a}}\otimes {\ob{a}})$ the map satisfies 
\[
\mt_{}\left( 
\mathord{\reflectbox{\begin{tikzpicture}[baseline=0,scale=.3, color=\clr]
		\node [style=none] (0) at (0, 0) {};
		\node [rectangle, draw, thick] (1) at (0, 0) {\reflectbox{$g$}};
		\node [style=none] (2) at (-0.4, 0.75) {};
		\node [style=none] (3) at (0.4, 0.75) {};
		\node [style=none] (4) at (0.4, -0.75) {};
		\node [style=none] (5) at (-0.4, -0.75) {};
		\node [style=none] (6) at (-0.4, 2) {};
		\node [style=none] (7) at (2, 0.7) {};
		\node [style=none] (8) at (2, -0.7) {};
		\node [style=none] (9) at (-0.4, -2) {};
		\node [style=none] (10) at (-1, -2) {\reflectbox{\scriptsize{$\ob{a}$}}};
		\node [style=none] (11) at (-1, 2) {\reflectbox{\scriptsize{$\ob{a}$}}};
		\node [style=none] (13) at (0.2, 1.4) {\reflectbox{\scriptsize{$\ob{a}$}}};
		\node [style=none] (14) at (0.2, -1.4) {\reflectbox{\scriptsize{$\ob{a}$}}};
		\draw [thick,-, looseness=1.50] (3.center) to [out=90,in=90] (7.center);
		\draw  [thick, -] (7.center) to (8.center);
		\draw [thick, -, looseness=1.5] (8.center) to [out=270,in=270] (4.center);
		\draw  [thick, -] (2.center) to (6.center);
		\draw  [thick, -] (5.center) to (9.center);
\end{tikzpicture}}}
 \right)
=
\mt_{}\left( 
\mathord{\begin{tikzpicture}[baseline=0,scale=.3, color=\clr]
		\node [style=none] (0) at (0, 0) {};
		\node [rectangle, draw, thick] (1) at (0, 0) {$g$};
		\node [style=none] (2) at (-0.4, 0.75) {};
		\node [style=none] (3) at (0.4, 0.75) {};
		\node [style=none] (4) at (0.4, -0.75) {};
		\node [style=none] (5) at (-0.4, -0.75) {};
		\node [style=none] (6) at (-0.4, 2) {};
		\node [style=none] (7) at (2, 0.7) {};
		\node [style=none] (8) at (2, -0.7) {};
		\node [style=none] (9) at (-0.4, -2) {};
		\node [style=none] (10) at (-1, -2) {\scriptsize{$\ob{a}$}};
		\node [style=none] (11) at (-1, 2) {\scriptsize{$\ob{a}$}};
		\node [style=none] (13) at (0.2, 1.4) {\scriptsize{$\ob{a}$}};
		\node [style=none] (14) at (0.2, -1.4) {\scriptsize{$\ob{a}$}};
		\draw [thick,-, looseness=1.50] (3.center) to [out=90,in=90] (7.center);
		\draw  [thick, -] (7.center) to (8.center);
		\draw [thick, -, looseness=1.5] (8.center) to [out=270,in=270] (4.center);
		\draw  [thick, -] (2.center) to (6.center);
		\draw  [thick, -] (5.center) to (9.center);
\end{tikzpicture}}
 \right).
\] Combining \cite[Theorems 3.3.1 and 3.3.2]{GKPM} yields the following result.

\begin{theorem}\label{T:mtracegeneralities}  If $\II$ is an ideal of $\qwebsupdowneven$ and $\left\{\mt_{{\ob{a}}}: \End_{\qwebsupdowneven}\left({\ob{a}} \right) \to \K  \right\}_{{\ob{a}} \in \II}$ is an m-trace on $\II$, then each $\mt_{{\ob{a}}}$ is an ambidextrous trace.  On the other hand, if ${\ob{a}}$ is an object with ambidextrous trace $\mt : \End_{\qwebsupdowneven}\left({\ob{a}} \right) \to \K$, then there is a unique m-trace on $\II_{{\ob{a}}}$ such that $\mt_{{\ob{a}}}=\mt$. 
\end{theorem}
Given an m-trace $\mt$ on an ideal $\II$, define the \emph{modified dimension} of an object ${\ob{b}} \in \II$ to be 
\[
\md_{\II}\left(\ob{b} \right) = \mt_{{\ob{b}}}\left(\Id_{{\ob{b}}} \right).
\]
We call ${\ob{a}} \in \qwebsupdowneven$ \emph{absolutely irreducible} if $\End_{\qwebsupdowneven}\left({\ob{a}} \right) = \K \Id_{{\ob{a}}}$ and call the $\K$-linear map $\mt: \End_{\qwebsupdowneven}\left({\ob{a}} \right) \to \K$ defined by $g = \mt(g)\Id_{{\ob{a}}}$ the \emph{canonical trace}. If ${\ob{a}}$ is an absolutely irreducible object of $\qwebsupdowneven$ with ambidextrous canonical trace, then by the previous theorem this extends uniquely to an m-trace $\left\{\mt_{\ob{b}} : \End_{\qwebsupdowneven}\left({\ob{b}} \right) \to \K  \right\}_{{\ob{b}} \in \II_{{\ob{a}}}}$.  Applying \cite[Lemma 4.2]{GKPM} in this context yields the following result.
\begin{proposition}\label{T:dimensionsandideals}  Let ${\ob{a}}$ be an absolutely irreducible object of $\qwebsupdowneven$ with ambidextrous canonical trace and let ${\ob{b}} \in \II_{{\ob{a}}}$ be an absolutely irreducible object.  Then $\II_{{\ob{b}}}=\II_{{\ob{a}}}$ if and only if $\md_{\II_{{\ob{a}}}}\left({\ob{b}} \right) \neq  0$.
\end{proposition}

\begin{theorem}\label{T:Generatingobjectsareabsolutelyirreducible} For all $k \geq 1$ the objects $\up_{k}$ and $\down_{k}$  are absolutely irreducible in $\qwebsupdowneven$.
\end{theorem}

\begin{proof}The case of $\down_{k}$ follows directly from the $\up_{k}$ case by the application of leftward cups and caps and \cref{straighten-zigzag} so we only prove the $\up_{k}$ case.  We may assume without loss $d$ is a web of type $\up_{k} \to \up_{k}$. Using \cref{digon-removal} yields the first equality and the others are explained below:
\[
\mathord{\begin{tikzpicture}[baseline=0,color=\clr, scale=.5]
		\node [style=none] (1) at (0, 3) {};
		\node [style=none] (2) at (0, -3) {};
		\node [rectangle, draw, thick] (3) at (0, 0) {$d$};
		\node [style=none] (6) at (-0.75, -3) {\scriptsize{$k$}};
		\node [style=none] (8) at (-0.75, 3) {\scriptsize{$k$}};
		\node [style=none] (9) at (0, 2.5) {};
		\node [style=none] (10) at (0, 1) {};
		\node [style=none] (13) at (0, -1) {};
		\node [style=none] (15) at (0, -0.45) {};
		\node [style=none] (16) at (0, 0.45) {};
		\draw [thick, <-] (1.center) to (16.center);
		\draw [thick, -, rdirected=.5]  (15.center) to (2.center);
\end{tikzpicture}}
 =  \frac{1}{[k]_{q}![k]_{q}!}
\mathord{\begin{tikzpicture}[baseline=0,color=\clr, scale=.5]
		\node [style=none] (0) at (0, 0) {};
		\node [style=none] (1) at (0, 3) {};
		\node [style=none] (2) at (0, -3) {};
		\node [rectangle, draw, thick] (3) at (0, 0) {$d$};
		\node [style=none] (6) at (-0.75, -3) {\scriptsize{$k$}};
		\node [style=none] (8) at (-0.75, 3) {\scriptsize{$k$}};
		\node [style=none] (9) at (0, 2.5) {};
		\node [style=none] (10) at (0, 1) {};
		\node [style=none] (13) at (0, -1) {};
		\node [style=none] (14) at (0, -2.5) {};
		\node [style=none] (15) at (0, -0.45) {};
		\node [style=none] (16) at (0, 0.45) {};
		\node [style=none] (17) at (0.05, -1.75) {\scriptsize{$\cdots$}};
		\node [style=none] (18) at (0.05, 1.75) {\scriptsize{$\cdots$}};
		\node [style=none] (20) at (0.75, -1.75) {\scriptsize{$1$}};
		\node [style=none] (21) at (0.75, 1.75) {\scriptsize{$1$}};
		\node [style=none] (22) at (-0.75, 1.75) {\scriptsize{$1$}};
		\node [style=none] (23) at (-0.75, -1.75) {\scriptsize{$1$}};
		\draw [thick, <-] (1.center) to (9.center);
		\draw [thick, -, rdirected=.7]  (14.center) to (2.center);
		\draw [thick, -, looseness=1.25] (10.center) to [out=135, in=225] (9.center);
		\draw [thick, -, looseness=1.25] (14.center) to [out=45, in=315] (13.center);
		\draw [thick, -, rdirected=.75] (10.center) to (16.center);
		\draw [thick, -, looseness=1.25] (14.center) to [out=135, in=225] (13.center);
		\draw [thick, -, looseness=1.20] (9.center) to [out=315, in=45] (10.center);
		\draw [thick, -, directed=.75] (13.center) to (15.center);
\end{tikzpicture}} 
 =  c 
\mathord{\begin{tikzpicture}[baseline=0,color=\clr, scale=.5]
		\node [style=none] (0) at (0, 0) {};
		\node [style=none] (1) at (0, 3) {};
		\node [style=none] (2) at (0, -3) {};
		\node [style=none] (6) at (-0.75, -3) {\scriptsize{$k$}};
		\node [style=none] (8) at (-0.75, 3) {\scriptsize{$k$}};
		\node [style=none] (9) at (0, 2.5) {};
		\node [style=none] (10) at (0, 1) {};
		\node [style=none] (13) at (0, -1) {};
		\node [style=none] (14) at (0, -2.5) {};
		\node [style=none] (15) at (0, -0.45) {};
		\node [style=none] (16) at (0, 0.45) {};
		\node [style=none] (17) at (0.05, 0) {\scriptsize{$\cdots$}};
		\node [style=none] (20) at (0.75, -1.75) {\scriptsize{$1$}};
		\node [style=none] (21) at (0.75, 1.75) {\scriptsize{$1$}};
		\node [style=none] (22) at (-0.75, 1.75) {\scriptsize{$1$}};
		\node [style=none] (23) at (-0.75, -1.75) {\scriptsize{$1$}};
		\draw [thick, <-] (1.center) to (9.center);
		\draw [thick, -, rdirected=.7]  (14.center) to (2.center);
		\draw [thick, -, looseness=.5] (14.center) to [out=135, in=225] (9.center);
		\draw [thick, -, looseness=.5] (14.center) to [out=45, in=315] (9.center);
\end{tikzpicture}}
 =  c[k]_{q}!
\mathord{
\begin{tikzpicture}[baseline=0,color=\clr, scale=.5]
		\node [style=none] (0) at (0, 0) {};
		\node [style=none] (1) at (0, 3) {};
		\node [style=none] (2) at (0, -3) {};
		\node [style=none] (6) at (-0.75, -3) {\scriptsize{$k$}};
		\node [style=none] (9) at (0, 2.5) {};
		\node [style=none] (10) at (0, 1) {};
		\node [style=none] (13) at (0, -1) {};
		\node [style=none] (14) at (0, -2.5) {};
		\node [style=none] (15) at (0, -0.45) {};
		\node [style=none] (16) at (0, 0.45) {};
		\draw [thick, <-] (1.center) to (2.center);
\end{tikzpicture}}.
\]  In the middle of the second diagram we see a web of type $\up_{1}^{k} \to \up_{1}^{k}$.  By \cref{T:Thetafunctor} we may assume this web can be rewritten as a linear combination of upward oriented webs and since $d$ did not have dots these upward oriented webs will not have dots.  By \cref{P:sergeev-isomorphism} these webs may be written as a linear combination of upward oriented crossings.  By \cref{L:untwist-permutation} these crossings can be replaced with scalars and the end result is the third diagram (where $c \in \K$).  Finally, using \cref{digon-removal} yields the last equality and shows the web $d$ lies in $\K \Id_{\up_{k}}$, as desired.  

\end{proof}

\begin{theorem}\label{T:allidealsareequal}  For all $k \geq 1$ we have 
\[
\II_{\up_{k}} = \II_{\up_{1}} = \II_{\down_{1}}=\II_{\down_{k}}.
\]
\end{theorem}

\begin{proof} Consider the morphisms $\up_{k} \to \up_{k-1}\up_{1}$ and $\up_{k-1}\up_{1} \to \up_{k}$ given by the upward oriented split and merge, respectively.  Using an appropriate rescaling of these maps and applying \cref{digon-removal} shows $\up_{k}$ is a retract of $\up_{k-1}\up_{1}$ and, hence, $\up_{k} \in \II_{\up_{k-1}}$ and $\II_{\up_{k}} \subseteq \II_{\up_{k-1}}$.  On the other hand, for $k \geq 2$ let $\alpha_{k-1}: \up_{k-1} \to \up_{k}\down_{1}$ and $\beta_{k-1}: \up_{k}\down_{1} \to \up_{k-1}$ be given by 
\begin{equation*}
 \alpha_{k-1} \ = \ 
\mathord{ 
\begin{tikzpicture}[baseline=0,color=\clr, scale=.3]
		\node [style=none] (0) at (0, 0) {};
		\node [style=none] (1) at (-1, -1) {};
		\node [style=none] (2) at (1, -1) {};
		\node [style=none] (3) at (3, -1) {};
		\node [style=none] (4) at (3, 2) {};
		\node [style=none] (6) at (0, 2) {};
		\node [style=none] (9) at (-1, -3) {};
		\node [style=none] (10) at (-2.5, -3) {\scriptsize{$k-1$}};
		\node [style=none] (11) at (3.5, 1) {\scriptsize{$1$}};
		\node [style=none] (12) at (-0.5, 1) {\scriptsize{$k$} };
		\draw [thick, -, directed=.5] (9.center) to (1.center);
		\draw [thick, -, directed=.5, looseness=1.25] (1.center) to [out=90, in=225](0.center);
		\draw [thick, -, rdirected=.5, looseness=1.25] (0.center) to [out=315, in=90] (2.center);
		\draw [thick, -, looseness=1.5] (2.center) to [out=270, in=270] (3.center);
		\draw [thick, -, rdirected=.7] (3.center) to (4.center);
		\draw [thick, <-] (6.center) to (0.center);
\end{tikzpicture}
}
\;\;   \text{and}  \;\;  \beta_{k-1} \ = \ 
\mathord{
\begin{tikzpicture}[baseline=0,color=\clr, scale=.3]
		\node [style=none] (0) at (0, -3) {};
		\node [style=none] (3) at (3, -3) {};
		\node [style=none] (4) at (3, 1) {};
		\node [style=none] (5) at (1, 1) {};
		\node [style=none] (6) at (0, 0) {};
		\node [style=none] (7) at (-1, 1) {};
		\node [style=none] (8) at (-1, 3) {};
		\node [style=none] (11) at (3.5, -1) {\scriptsize{$1$}};
		\node [style=none] (12) at (-0.5, -1) {\scriptsize{$k$} };
		\node [style=none] (13) at (-2.5, 3) {\scriptsize{$k-1$}};
		\draw [thick, -, directed=.45] (3.center) to (4.center);
		\draw [thick, looseness=1.5] (4.center) to [out=90, in=90] (5.center);
		\draw [thick, -, directed=.5, looseness=1.25]  (6.center) to [out=45, in=270] (5.center);
		\draw [thick, -, directed=.5, looseness=1.25] (6.center) to [out=135, in=270] (7.center);
		\draw [thick, ->] (7.center) to (8.center);
		\draw [thick, -, rdirected=.5] (6.center) to (0.center);
\end{tikzpicture}
} \ .
\end{equation*}

Computing $\beta_{k-1} \circ \alpha_{k-1}$, using \cref{square-switch} for the second equality, \cref{E:overcrossingdef,delete-bubble} for the third equality, \ref{E:straighten-twists1} for the fourth equality, and \cref{digon-removal} for the last equality yields:
\[
\mathord{
\begin{tikzpicture}[baseline=0,color=\clr, scale=.3]
		\node [style=none] (0) at (0, 0) {};
		\node [style=none] (1) at (-1, -1) {};
		\node [style=none] (2) at (1, -1) {};
		\node [style=none] (3) at (3, -1) {};
		\node [style=none] (4) at (3, 3) {};
		\node [style=none] (5) at (1, 3) {};
		\node [style=none] (6) at (0, 2) {};
		\node [style=none] (7) at (-1, 3) {};
		\node [style=none] (8) at (-1, 5) {};
		\node [style=none] (9) at (-1, -3) {};
		\node [style=none] (10) at (-2.5, -3) {\scriptsize{$k-1$}};
		\node [style=none] (11) at (3.5, 1) {\scriptsize{$1$}};
		\node [style=none] (12) at (-0.5, 1) {\scriptsize{$k$} };
		\node [style=none] (13) at (-2.5, 5) {\scriptsize{$k-1$}};
		\draw [thick, -, directed=.5] (9.center) to (1.center);
		\draw [thick, -, directed=.5, looseness=1.25] (1.center) to [out=90, in=225](0.center);
		\draw [thick, -, rdirected=.5, looseness=1.25] (0.center) to [out=315, in=90] (2.center);
		\draw [thick, -, looseness=1.5] (2.center) to [out=270, in=270] (3.center);
		\draw [thick, -] (3.center) to (4.center);
		\draw [thick, looseness=1.5] (4.center) to [out=90, in=90] (5.center);
		\draw [thick, -, directed=.5, looseness=1.25]  (6.center) to [out=45, in=270] (5.center);
		\draw [thick, -, directed=.5, looseness=1.25] (6.center) to [out=135, in=270] (7.center);
		\draw [thick, ->] (7.center) to (8.center);
		\draw [thick, -] (6.center) to (0.center);
\end{tikzpicture}
}
=
\mathord{
\begin{tikzpicture}[baseline=0,color=\clr, scale=.3]
		\node [style=none] (0) at (1, 0) {};
		\node [style=none] (1) at (-1, 0) {};
		\node [style=none] (2) at (1, -1) {};
		\node [style=none] (3) at (3, -1) {};
		\node [style=none] (4) at (3, 3) {};
		\node [style=none] (5) at (1, 3) {};
		\node [style=none] (6) at (1, 1.75) {};
		\node [style=none] (7) at (-1, 1.75) {};
		\node [style=none] (8) at (-1, 5) {};
		\node [style=none] (9) at (-1, -3) {};
		\node [style=none] (10) at (-2.5, -3) {\scriptsize{$k-1$}};
		\node [style=none] (11) at (4, 1) {\scriptsize{$1$}};
		\node [style=none] (12) at (-1.5, 1) {\scriptsize{$k$}};
		\node [style=none] (13) at (-2.5, 5) {\scriptsize{$k-1$}};
		\node [style=none] (14) at (0, 2.5) {\scriptsize{$1$}};
		\node [style=none] (15) at (0, -.75) {\scriptsize{$1$}};
		\draw  [thick, directed=.3, -] (9.center) to (1.center);
		\draw  [thick, -] (0.center) to (2.center);
		\draw [thick, -, looseness=1.5] (2.center) to [out=270, in=270] (3.center);
		\draw [thick, -, rdirected=.5] (3.center) to (4.center);
		\draw [thick, -, looseness=1.5] (4.center) to [out=90,in=90] (5.center);
		\draw [thick, -] (5.center) to (6.center);
		\draw [thick, rdirected=.5,-] (6.center) to (7.center);
		\draw [thick, ->] (7.center) to (8.center);
		\draw [thick, directed=.5, -] (0.center) to (1.center);
		\draw [thick, -] (7.center) to (1.center);
\end{tikzpicture}}
=
\mathord{\begin{tikzpicture}[baseline=0,color=\clr, scale=.3]
		\node [style=none] (0) at (1, 0) {};
		\node [style=none] (1) at (-1, 0) {};
		\node [style=none] (2) at (1, -1) {};
		\node [style=none] (3) at (3, -1) {};
		\node [style=none] (4) at (3, 3) {};
		\node [style=none] (5) at (1, 3) {};
		\node [style=none] (6) at (1, 1.75) {};
		\node [style=none] (7) at (-1, 1.75) {};
		\node [style=none] (8) at (-1, 5) {};
		\node [style=none] (9) at (-1, -3) {};
		\node [style=none] (10) at (-2.5, -3) {\scriptsize{$k-1$}};
		\node [style=none] (11) at (3.5, 1.5) {\scriptsize{$1$}};
		\node [style=none] (12) at (-2.5, 1) {\scriptsize{$k-2$}};
		\node [style=none] (13) at (-2.5, 5) {\scriptsize{$k-1$}};
		\node [style=none] (14) at (0, 2.5) {\scriptsize{$1$}};
		\node [style=none] (15) at (0, -.75) {\scriptsize{$1$}};
		\node [style=none] (16) at (1.5, 1.25) {\scriptsize{$2$}};
		\draw  [thick, directed=.3, -] (9.center) to (1.center);
		\draw  [thick, -] (0.center) to (2.center);
		\draw [thick, -, looseness=1.5] (2.center) to [out=270, in=270] (3.center);
		\draw [thick, -, rdirected=.5] (3.center) to (4.center);
		\draw [thick, -, looseness=1.5] (4.center) to [out=90,in=90] (5.center);
		\draw [thick, -] (5.center) to (6.center);
		\draw [thick, directed=.5,-] (6.center) to (7.center);
		\draw [thick, ->] (7.center) to (8.center);
		\draw [thick, rdirected=.5, -] (0.center) to (1.center);
		\draw [thick, -] (7.center) to (1.center);
		\draw [thick, -, directed=.5] (0.center) to (6.center);
\end{tikzpicture}}
=
\mathord{\begin{tikzpicture}[baseline=0,color=\clr, scale=.3]
		\node [style=none] (0) at (1, 0) {};
		\node [style=none] (1) at (-1, 0) {};
		\node [style=none] (2) at (1, -1) {};
		\node [style=none] (3) at (3, -1) {};
		\node [style=none] (4) at (3, 3) {};
		\node [style=none] (5) at (1, 3) {};
		\node [style=none] (6) at (1, 1.75) {};
		\node [style=none] (7) at (-1, 1.75) {};
		\node [style=none] (8) at (-1, 5) {};
		\node [style=none] (9) at (-1, -3) {};
		\node [style=none] (10) at (-2.5, -3) {\scriptsize{$k-1$}};
		\node [style=none] (11) at (3.5, 1.5) {\scriptsize{$1$}};
		\node [style=none] (12) at (-2.5, 1) {\scriptsize{$k-2$}};
		\node [style=none] (13) at (-2.5, 5) {\scriptsize{$k-1$}};
		\node [style=none] (14) at (0.25, 0.75) {};
		\node [style=none] (15) at (0, 1) {};
		\draw [thick, -, directed=.3] (9.center) to (1.center);
		\draw [thick, -] (0.center) to (2.center);
		\draw [thick, -, looseness=1.5] (2.center) to [out=270, in=270] (3.center);
		\draw [thick, -, rdirected=.5] (3.center) to (4.center);
		\draw [thick,-, looseness=1.5] (4.center) to [out=90, in=90] (5.center);
		\draw [thick, -] (5.center) to (6.center);
		\draw [thick, ->] (7.center) to (8.center);
		\draw [thick, -] (7.center) to (1.center);
		\draw [thick, -]  (1.center) to [out=90,in=270] (6.center);
		\draw [thick, -, looseness=1.25] (0.center) to [out=90, in=305](14.center);
		\draw [thick, -, looseness=1.25] (15.center) to [out=135, in=305] (7.center);
\end{tikzpicture}}
=
\mathord{\begin{tikzpicture}[baseline=0,color=\clr, scale=.3]
		\node [style=none] (1) at (-1, -1) {};
		\node [style=none] (7) at (-1, 3) {};
		\node [style=none] (8) at (-1, 5) {};
		\node [style=none] (9) at (-1, -3) {};
		\node [style=none] (10) at (-2.5, -3) {\scriptsize{$k-1$}};
		\node [style=none] (12) at (-3.5, 1) {\scriptsize{$k-2$}};
		\node [style=none] (13) at (-2.5, 5) {\scriptsize{$k-1$}};
		\node [style=none] (16) at (0.5, 1) {\scriptsize{$1$}};
		\draw [thick, -, directed=.5] (9.center) to (1.center);
		\draw [thick, -, looseness=1] (1.center) to [out=45, in=315] (7.center);
		\draw [thick, -, looseness=1] (1.center) to [out=135,in=225] (7.center);
		\draw [thick, ->] (7.center) to (8.center);
\end{tikzpicture}}
= [k-1]_{q} \; 
\mathord{\begin{tikzpicture}[baseline=0,color=\clr, scale=.3]
		\node [style=none] (8) at (-1, 5) {};
		\node [style=none] (9) at (-1, -3) {};
		\node [style=none] (10) at (0.5, -3) {\scriptsize{$k-1$}};
		\draw [thick, ->] (9.center) to (8.center);
\end{tikzpicture}}  .
\]  This proves $\up_{k-1}$ is a retract of $\up_{k}\down_{1}$, hence $\up_{k-1} \in \II_{\up_{k}}$ and $\II_{\up_{k-1}} \subseteq \II_{\up_{k}}$.  In short, we have proven $\II_{\up_{k-1}}=\II_{\up_{k}}$ for all $k \geq 2$. 

Similar but easier arguments using \cref{straighten-zigzag,straighten-right-zigzag} show $\II_{\up_{k}} = \II_{\down_{k}}$ for all $k \geq 1$.  Taking this together with the previous paragraph proves the claimed result.
\end{proof}

\begin{theorem}\label{T:uponeisambi}  The canonical trace for $\up_{1}$ is ambidextrous.  In particular, it extends uniquely to an m-trace on $\II =\II_{\up_{1}}$ and the induced $\K$-linear maps  $\mt_{\up_{k}}: \End_{\qwebsupdowneven} (\up_{k}) \to \K$ and $\mt_{\down_{k}}: \End_{\qwebsupdowneven} (\down_{k}) \to \K$ are nontrivial ambidextrous traces for all $k \geq 1$. 

\end{theorem}

\begin{proof} Since the canonical trace is $\K$-linear it suffices to check the ambidextrous condition on a basis for $\End_{\qwebsupdowneven}\left(\up_{1}^{2} \right)$.  By \cref{T:Thetafunctor} this space is spanned by upward oriented webs of type $\up_{1}^{2} \to \up_{1}^{2}$.  By \cref{P:sergeev-isomorphism} it is spanned by the identity and the upward oriented over-crossing.  But in both cases the ambidextrous condition is easily seen to be satisfied using \cref{delete-bubble,E:leftcircleiszero,E:straighten-twists1}.  

By \cref{T:mtracegeneralities} the canonical trace extends to an m-trace on $\II$.  Combining \cref{T:dimensionsandideals,T:allidealsareequal} proves $\mt_{\up_{k}}(\Id_{\up_{k}}) = \md_{\II}(\up_{k}) \neq  0$ and, hence, the induced ambidextrous trace $\mt_{\up_{k}}$ is nontrivial.  The same argument applies to $\mt_{\down_{k}}$ as well.
\end{proof}  Note that since $\up_{k}$ and $\down_{k}$ are absolutely irreducible the maps $\mt_{\up_{k}}$ and $\mt_{\down_{k}}$ are necessarily nonzero scalar multiples of the canonical trace.  In particular, up to scaling by a constant one could use the canonical traces in what follows.

One important application of the existence of a nontrivial m-trace on an ideal of a ribbon category is their use in defining invariants of knots, links, 3-manifolds, and other objects in low-dimensional topology.  We give one example of this below following \cite{GPMV}, referring the reader therein for further details.  Let $\mathcal{C}$ denote the category of $\qwebsupdowneven$-colored ribbon graphs in $\mathbb{R} \times [0,1]$.  Let $F: \mathcal{C} \to \qwebsupdowneven$ be the canonical functor.  We say such a $\qwebsupdowneven$-colored ribbon graph $\Gamma$ is \emph{$\II$-admissible} if there is a $1$--$1$ $\qwebsupdowneven$-colored ribbon graph $T_{\Gamma}$ which is an endomorphism in $\mathcal{C}$ for some object ${\ob{a}} \in \II$ and in $\mathcal{C}$ we have
\[
\Gamma \; = \; 
\mathord{\begin{tikzpicture}[baseline=-1ex,color=\clr, scale=.5]
		\node [style=none] (0) at (0, 0) {};
		\node [rectangle, draw, thick] (1) at (0, 0) {$T_{\Gamma}$};
		\node [style=none] (2) at (0, 0.5) {};
		\node [style=none] (5) at (0, -0.5) {};
		\node [style=none] (6) at (0, -1) {};
		\node [style=none] (7) at (1, -1) {};
		\node [style=none] (8) at (0, 1) {};
		\node [style=none] (9) at (1, 1) {};
		\node [style=none] (10) at (-0.5, 1) {\scriptsize{$\ob{a}$}};
		\node [style=none] (11) at (-0.5, -1) {\scriptsize{$\ob{a}$}};
		\draw [thick,-, looseness=1.50] (9.center) to [out=90, in=90] (8.center);
		\draw [thick, -, looseness=1.50] (7.center) to [out=270,in=270] (6.center);
		\draw [thick, -](9.center) to (7.center);
		\draw [thick,-](5.center) to (6.center);
		\draw [thick,-](2.center) to (8.center);
\end{tikzpicture}} \ . \ 
\] 
By \cite[Theorem 5]{GPMV} we have the following result.
\begin{theorem}\label{T:mtraceinvariant}  Let $\Gamma$ be a $\qwebsupdowneven$-colored ribbon graph which is $\II$-admissible.  Then the function 
\[
\Gamma \mapsto \mt_{{\ob{a}}}\left(F\left( T_{\Gamma}\right) \right)
\] is a well-defined (i.e. does not depend on the choice of ${ \ob{a}}$ or $T_{\Gamma}$) isotopy invariant of $\II$-admissible graphs.
\end{theorem}

For example, if $K$ is an oriented link diagram given as the closure of a braid with edges labelled by $1$, then by cutting the rightmost edge of the diagram and bending the upper free strand upward and the other free strand downward, we obtain $T_{K}$, a $1$--$1$ $\qwebsupdowneven$-colored ribbon graph $T_{\Gamma}$ which is an endomorphism of $\up_{1}$ in $\qwebsupdowneven$.  The function 
\[
K \mapsto \mt_{\up_{1}}\left(F \left(T_{K} \right) \right)
\] is then an isotopy invariant.  In this case, since the unknot is sent to $1$ and relation \cref{E:inversedownwardcrossing} is the standard skein relation, this is precisely the Alexander-Conway polynomial of $K$ evaluated at $\tq$.

More generally, if $K$ is a framed, oriented link diagram given as the closure of a braid with edges labelled by positive integers, then by cutting the rightmost edge of the diagram and bending the upper free strand upward and the other free strand downward, we obtain $T_{K}$, a $1$--$1$ $\qwebsupdowneven$-colored ribbon graph $T_{\Gamma}$ which is an endomorphism of $\up_{k}$ in $\qwebsupdowneven$ for some $k \geq 1$.  The function 
\[
K \mapsto \mt_{\up_{k}}\left(F \left(T_{K} \right) \right)
\] is an isotopy invariant.  A comparison with the web-calculus description of the colored HOMFLY-PT polynomial given in \cite[Section 4]{TVW} shows the above invariant recovers the renormalized version of the colored HOMFLY-PT invariant where the edges are colored by partitions with one part.  

We briefly sketch the procedure for recovering the colored HOMFLY-PT invariant. Given a strand of $K$ labelled by an integer greater than $1$, we can pick a point on the strand just below the braid which defines $K$ and explode it into strands of thickness $1$ thanks to relation \cref{digon-removal}.  Since merges and splits can be freely moved through crossings, cups, and caps, the top half of the exploded diagram (which looks like \cref{E:explosion}) can be moved around the strand until it is just below the bottom half of the explosion.  In this way each strand labelled by an integer $k>1$ can be replaced with the $k$th cabling of the strand containing a single upward-oriented clasp idempotent, $\Cl_{k}$, just below the braid.  The end result of this procedure is the closure of a braid consisting of strands labelled by $1$.  Comparing the formula given in \cref{clasp-sum} with the formula given for the red clasps in \cite[Remark 2.21]{TVW} we see this procedure yields the same final result as the procedure described in \cite[Section 4]{TVW}.   Therein they denote the braid which results by $\rho_{K'}\left(\tilde{b}_{D}^{K'} \right)e_{q}(k_{1}, \dotsc , k_{r})$, where $k_{1}, \dotsc , k_{r}$ are the labels of the strands at the bottom of the initial braid, read left to right and viewed as single-row partitions.  If we write $tr$ for the Jones-Ocneanu trace, then the colored HOMFLY-PT polynomial is given in \cite[Definition 4.2]{TVW} as 
\begin{equation*}
tr\left(\rho_{K'}\left(\tilde{b}_{D}^{K'} \right)e_{q}(k_{1}, \dotsc , k_{r}) \right).
\end{equation*} 
On the other hand, by \ref{T:mtraceinvariant} we have 
\begin{equation*}
\mt_{\up_{k}}\left(F \left(T_{K} \right) \right) = \mt_{\up_{1}}\left(F \left( T_{\rho_{K'}\left(\tilde{b}_{D}^{K'} \right)e_{q}(k_{1}, \dotsc , k_{r})} \right) \right).
\end{equation*}
A direct calculation verifies  $b \mapsto \mt_{\up_{1}} \left(F\left(T_{b} \right) \right)$ satisfies the defining properties of the Jones-Ocneanu trace and, hence, the resulting invariants coincide.

We end this discussion by recalling the m-trace is defined for all objects of $\II = \II_{\up_{1}}$.  Consequently, we could label the strands by any non-unit object of $\qwebsupdowneven$ to define further isotopy invariants.  More generally one could use cablings and the elements $e_{\lambda} \in \HC_{k}(q)$ discussed in \cref{SS:Olshanskiduality} to define invariants indexed by strict partitions.  It would be interesting to describe them.

\subsection{An alternate specialization of the undotted bubble}\label{}\label{SS:bubblespop}  In the classical case the left equality of \cref{delete-bubble} can be omitted as it follows from the other relations (see \cite[Lemma 6.2.1]{BrownKujawa}). As we explain, if we attempt a similar argument a second possible specialization of the clockwise bubble reveals itself.

If we do not assume \cref{delete-bubble}, then \cref{E:dotthroughrightwardcrossing} becomes 
\begin{equation*}
\xy
(0,0)*{\reflectbox{
\begin{tikzpicture}[scale=.3, color=\clr]
	\draw [ thick, <-] (-1,-1) to (-0.25,-0.25);
	\draw [ thick, -] (0.25,0.25) to (1,1);
	\draw [ thick, ->] (1,-1) to (-1,1);
	\draw (0.55,0.55) \wdot ;
	\end{tikzpicture}
}};
\endxy
 \ = \ 
\xy
(0,0)*{\reflectbox{
\begin{tikzpicture}[scale=.3, color=\clr]
	\draw [ thick, <-] (-1,-1) to (-0.25,-0.25);
	\draw [ thick, -] (0.25,0.25) to (1,1);
	\draw [ thick, ->] (1,-1) to (-1,1);
	\draw (-0.55,-0.55) \wdot ;
	\end{tikzpicture}
}};
\endxy
 \ + \tq  \ 
\left( 
 \xy
(0,0)*{
\begin{tikzpicture}[scale=.3, color=\clr]
	\draw [ thick, ->, looseness=1] (-1,1.25) to [out=90, in=270] (1,3.25);
	\draw [ thick, -, looseness=1] (1,1.25) to [out=90, in=315] (.2,2);
	\draw [ thick, -, looseness=1] (-.3,2.4) to [out=135, in=270] (-1,3.25);
	\draw [ thick, <-, looseness=1.8] (-1,1.25) to [out=270, in=270] (1,1.25);
	\draw [ thick, ->, looseness=1.8] (-1,-1.25) to [out=90, in=90] (1,-1.25);
	\draw (-0.85,-0.55) \wdot ;
	\end{tikzpicture}
};
\endxy
\ - \
  \xy
(0,0)*{
\begin{tikzpicture}[scale=.3, color=\clr]
	\draw [ thick, ->, looseness=1] (-1,1.25) to [out=90, in=270] (1,3.25);
	\draw [ thick, -, looseness=1] (1,1.25) to [out=90, in=315] (.2,2);
	\draw [ thick, -, looseness=1] (-.3,2.4) to [out=135, in=270] (-1,3.25);
	\draw [ thick, <-, looseness=1.8] (-1,1.25) to [out=270, in=270] (1,1.25);
	\draw [ thick, ->, looseness=1.8] (-1,-1.25) to [out=90, in=90] (1,-1.25);
	\draw (-0.85,0.55) \wdot ;
	\end{tikzpicture}
};
\endxy
 \right) \ .
\end{equation*}
Compose on the top with a leftward cap. Compose on the bottom with a dot on the left strand followed by a leftward cup. Simplifying using the definition of the rightward cup and cap along with the super interchange law yields:
\begin{equation*}\label{E:newdotthroughrightwardcrossing}
\xy
(0,0)*{\reflectbox{
\begin{tikzpicture}[scale=.3, color=\clr]
	\draw [ thick, <-, looseness=1.7] (-1,0) to [out=90, in=90] (1,0);
	\draw [ thick, -, looseness=1.7] (-1,0) to [out=270, in=270] (1,0);
	\end{tikzpicture}
}};
\endxy
 \ = \ -
\xy
(0,0)*{\reflectbox{
\begin{tikzpicture}[scale=.3, color=\clr]
	\draw [ thick, <-, looseness=1.7] (-1,0) to [out=90, in=90] (1,0);
	\draw [ thick, -, looseness=1.7] (-1,0) to [out=270, in=270] (1,0);
	\end{tikzpicture}
}};
\endxy
 \ + \tq  \ 
 \xy
(0,0)*{
\begin{tikzpicture}[scale=.3, color=\clr]
	\draw [ thick, ->, looseness=1.7] (-1,0) to [out=90, in=90] (1,0);
	\draw [ thick, -, looseness=1.7] (-1,0) to [out=270, in=270] (1,0);
	\draw [ thick, ->, looseness=1.7] (2,0) to [out=90, in=90] (4,0);
	\draw [ thick, -, looseness=1.7] (2,0) to [out=270, in=270] (4,0);
	\end{tikzpicture}
};
\endxy \ .
\end{equation*}
Consequently,
\begin{equation}\label{E:alternateclockwisebubblevalue}
\xy
(0,0)*{\reflectbox{
\begin{tikzpicture}[scale=.3, color=\clr]
	\draw [ thick, <-, looseness=1.7] (-1,0) to [out=90, in=90] (1,0);
	\draw [ thick, -, looseness=1.7] (-1,0) to [out=270, in=270] (1,0);
	\end{tikzpicture}
}};
\endxy  = 0 \; \text{ or } \; 
\xy
(0,0)*{\reflectbox{
\begin{tikzpicture}[scale=.3, color=\clr]
	\draw [ thick, <-, looseness=1.7] (-1,0) to [out=90, in=90] (1,0);
	\draw [ thick, -, looseness=1.7] (-1,0) to [out=270, in=270] (1,0);
	\end{tikzpicture}
}};
\endxy
=
\frac{2}{\tq} =:\kappa
\end{equation} Using the second value leads to the following alternate relations in \cref{L:thicknessonerelations}.  Relations \cref{E:inverseleftwardcrossing,E:inverserightwardcrossing,E:inversedownwardcrossing,E:dotthroughleftwardcrossing,E:dotthroughdownwardcrossing} are unchanged.   Using $\kappa$ in \cref{E:bubbleequation} shows \cref{E:leftcircleiszero} becomes
\begin{equation}\label{E:alternateleftcircleiszero}
\xy
(0,0)*{\reflectbox{
\begin{tikzpicture}[scale=.3, color=\clr]
	\draw [ thick, ->, looseness=1.7] (-1,0) to [out=90, in=90] (1,0);
	\draw [ thick, -, looseness=1.7] (-1,0) to [out=270, in=270] (1,0);
	\end{tikzpicture}
}};
\endxy  = - \frac{2}{\tq} \ .
\end{equation}  The equalities in \cref{thickness-one-rightward-cup-and-cap}  become 
\begin{equation}\label{alternate-thickness-one-rightward-cup-and-cap}
\xy
(0,0)*{
\bt[scale=.35, color=\clr]
	\draw [ thick, looseness=2, directed=0.99] (1,4.5) to [out=90,in=90] (-1,4.5);
	\draw [ thick, directed=1 ] (0.15,3.65) to (1,4.5);
	\draw [ thick,  ] (-1,2.5) to (-0.15,3.35);
	\draw [ thick, rdirected=-.9 ] (1,2.5) to (-1,4.5);
\et
};
\endxy 
 \ = \ -
\xy
(0,0)*{
\bt[scale=.35, color=\clr]
	\draw [ thick, looseness=2, rdirected=-0.95] (1,4.5) to [out=90,in=90] (-1,4.5);
\et
};
\endxy \   \quad\quad\text{and}\quad\quad
\xy
(0,0)*{
\bt[scale=.35, color=\clr]
	\draw [ thick, looseness=2, directed=0.99 ] (1,2.5) to [out=270,in=270] (-1,2.5);
	\draw [ thick, directed=1 ] (-1,2.5) to (1,4.5);
	\draw [ thick, rdirected=-0.85] (1,2.5) to (0.15,3.35);
	\draw [ thick] (-0.15,3.65) to (-1,4.5);
\et
};
\endxy
 \ = \ -
\xy
(0,0)*{
\bt[scale=.35, color=\clr]
	\draw [ thick, looseness=2, rdirected=-0.95 ] (1,2.5) to [out=270,in=270] (-1,2.5);
\et
};
\endxy .
\end{equation} The relation given in \ref{E:dotthroughrightwardcrossing} becomes
\begin{equation}\label{E:alternatedotthroughrightwardcrossing}
\xy
(0,0)*{\reflectbox{
\begin{tikzpicture}[scale=.3, color=\clr]
	\draw [ thick, <-] (-1,-1) to (-0.25,-0.25);
	\draw [ thick, -] (0.25,0.25) to (1,1);
	\draw [ thick, ->] (1,-1) to (-1,1);
	\draw (0.55,0.55) \wdot ;
	\end{tikzpicture}
}};
\endxy
 \ = \ 
\xy
(0,0)*{\reflectbox{
\begin{tikzpicture}[scale=.3, color=\clr]
	\draw [ thick, <-] (-1,-1) to (-0.25,-0.25);
	\draw [ thick, -] (0.25,0.25) to (1,1);
	\draw [ thick, ->] (1,-1) to (-1,1);
	\draw (-0.55,-0.55) \wdot ;
	\end{tikzpicture}
}};
\endxy
 \ - \tq  \ 
\left( 
 \xy
(0,0)*{
\begin{tikzpicture}[scale=.3, color=\clr]
	\draw [ thick, ->, looseness=1.8] (-1,1.25) to [out=270, in=270] (1,1.25);
	\draw [ thick, ->, looseness=1.8] (-1,-1.25) to [out=90, in=90] (1,-1.25);
	\draw (-0.85,-0.55) \wdot ;
	\end{tikzpicture}
};
\endxy
\ - \
  \xy
(0,0)*{
\begin{tikzpicture}[scale=.3, color=\clr]
	\draw [ thick, ->, looseness=1.8] (-1,1.25) to [out=270, in=270] (1,1.25);
	\draw [ thick, ->, looseness=1.8] (-1,-1.25) to [out=90, in=90] (1,-1.25);
	\draw (0.85,0.55) \wdot ;
	\end{tikzpicture}
};
\endxy
 \right) \ .
\end{equation}
At this point we assume the clockwise dotted bubble is zero (i.e.\ the second equality of \cref{delete-bubble} holds) and continue to assume the clockwise undotted bubble equals $\kappa$.  Then \ref{E:leftcirclewithdotiszero} still holds. Finally, \cref{thickness-one-rightward-cup-and-cap-with-dot} becomes 
\begin{equation}\label{alternate-thickness-one-rightward-cup-and-cap-with-dot}
\xy
(0,0)*{
\bt[scale=.35, color=\clr]
	\draw [ thick, looseness=2, rdirected=-0.95] (1,4.5) to [out=90,in=90] (-1,4.5);
	\draw (-0.9,5) \wdot ;
\et
};
\endxy = \ - \ 
\xy
(0,0)*{
\bt[scale=.35, color=\clr]
	\draw [ thick, looseness=2, rdirected=-0.95] (1,4.5) to [out=90,in=90] (-1,4.5);
	\draw (0.9,5) \wdot ;
\et
};
\endxy \   \quad\quad\text{and}\quad\quad
\xy
(0,0)*{
\bt[scale=.35, color=\clr]
	\draw [ thick, looseness=2, rdirected=-0.95 ] (1,2.5) to [out=270,in=270] (-1,2.5);
	\draw (-0.9,2) \wdot ;
\et
};
\endxy = \ - \
\xy
(0,0)*{
\bt[scale=.35, color=\clr]
	\draw [ thick, looseness=2, rdirected=-0.95 ] (1,2.5) to [out=270,in=270] (-1,2.5);
	\draw (0.9,2) \wdot ;
\et
};
\endxy \ .
\end{equation}  When the clockwise bubble is equal to $\kappa$, arguments similar to those used to prove \cref{L:straighteningtwists} show \ref{E:straighten-twists1} becomes
\begin{equation}\label{E:alternate-straighten-twists1}
\ (-1)^{k} \ 
\xy 
(0,0)*{
\bt[scale=1.25, color=\clr]
        \draw[thick]  (1,0) to (1,0.35);
	\draw[thick] plot [smooth, tension=1.1] coordinates {(1,0.35) (0.95,0.45)};
	\draw[thick] plot [smooth, tension=1.1] coordinates {(0.9,0.55) (0.8,0.65) (0.625,0.5) (0.8,0.35) (1.0,0.65)};
        \draw[thick, directed=1]  (1,0.65) to (1,1);
	\node at (1,-0.15) {\scriptsize $k$};
\et
};
\endxy = q^{k(k-1)}
\xy
(0,0)*{
\bt[scale=1.25, color=\clr]
	\draw[thick, directed=1] (1,0) to (1,1);
	\node at (1,-0.15) {\scriptsize $k$};
\et
};
\endxy
= 
\xy
(0,0)*{
\bt[scale=1.25, color=\clr]
        \draw[thick]  (1,0) to (1,0.35);
	\draw[thick] plot [smooth, tension=1.1] coordinates {(1,0.35) (1.2,0.65) (1.375,0.5) (1.25,0.35) (1.10,0.45)};
	\draw[thick] plot [smooth, tension=1] coordinates {(1,0.65) (1.05,0.55)};
        \draw[thick, directed=1]  (1,0.65) to (1,1);
	\node at (1,-0.15) {\scriptsize $k$};
\et
};
\endxy \ ,  
\end{equation}
\begin{equation} 
\xy
(0,0)*{
\bt[scale=1.25, color=\clr]
        \draw[thick]  (1,0) to (1,0.35);
	\draw[thick] plot [smooth, tension=1.1] coordinates {(1,0.35) (0.8,0.65) (0.625,0.5) (0.75,0.35) (0.9,0.45)};
	\draw[thick] plot [smooth, tension=1] coordinates {(1,0.65) (0.95,0.55)};
        \draw[thick, directed=1]  (1,0.65) to (1,1);
	\node at (1,-0.15) {\scriptsize $k$};
\et
};
\endxy = q^{-k(k-1)}
\xy
(0,0)*{
\bt[scale=1.25, color=\clr]
	\draw[thick, directed=1] (1,0) to (1,1);
	\node at (1,-0.15) {\scriptsize $k$};
\et
};
\endxy 
= \ (-1)^{k} \ 
\xy
(0,0)*{
\bt[scale=1.25, color=\clr]
        \draw[thick]  (1,0) to (1,0.35);
	\draw[thick] plot [smooth, tension=1.1] coordinates {(1.1,0.55) (1.25,0.65) (1.375,0.5) (1.25,0.35) (1.10,0.45) (1,0.65)};
	\draw[thick] plot [smooth, tension=1] coordinates {(1,0.35) (1.05,0.45)};
        \draw[thick, directed=1]  (1,0.65) to (1,1);
	\node at (1,-0.15) {\scriptsize $k$};
\et
};
\endxy \ ,
\end{equation}
Finally, \ref{dot-through-rightward-cup-and-cap} becomes
\begin{equation}\label{alternate-dot-through-rightward-cup-and-cap}
\xy
(0,0)*{
\bt[scale=.35, color=\clr]
	\draw [ thick, looseness=2, rdirected=-0.95] (1,4.5) to [out=90,in=90] (-1,4.5);
	\node at (-1.75,4.75) {\scriptsize $k$};
	\draw (-0.85,5) \wdot ;
\et
};
\endxy = \ - \
\xy
(0,0)*{
\bt[scale=.35, color=\clr]
	\draw [ thick, looseness=2, rdirected=-0.95] (1,4.5) to [out=90,in=90] (-1,4.5);
	\node at (-1.75,4.75) {\scriptsize $k$};
	\draw (0.85,5) \wdot ;
\et
};
\endxy \   \quad\quad\text{and}\quad\quad
\xy
(0,0)*{
\bt[scale=.35, color=\clr]
	\draw [ thick, looseness=2, rdirected=-0.95 ] (1,2.5) to [out=270,in=270] (-1,2.5);
	\node at (-1.75,2.5) {\scriptsize $k$};
	\draw (-0.85,2) \wdot ;
\et
};
\endxy = \ - \
\xy
(0,0)*{
\bt[scale=.35, color=\clr]
	\draw [ thick, looseness=2, rdirected=-0.95 ] (1,2.5) to [out=270,in=270] (-1,2.5);
	\node at (-1.75,2.5) {\scriptsize $k$};
	\draw (0.85,2) \wdot ;
\et
};
\endxy \ .
\end{equation}

Denote by $\qwebsupdown (\kappa)$ the monoidal supercategory defined with the same generators and relations as $\qwebsupdown$ but with the left equality in \ref{delete-bubble} replaced with the relation that the clockwise undotted circle labelled by $1$ is equal to $\kappa$.  Let $\qwebsupdowneven (\kappa)$ denote the full subcategory of $\qwebsupdown (\kappa)$ consisting of all objects and all morphisms without dots, and let $\qwebsupdowneven (\kappa)_{2}$ denote the monoidal subcategory of $\qwebsupdowneven (\kappa)$ generated as a monoidal category by $\left\{\uparrow_{2k}, \downarrow_{2k} \mid k \in \Z_{\geq 0} \right\}$ and by all morphisms which can be written as a linear combination of diagrams which have all strands labelled by even integers. Then  $\qwebsupdowneven (\kappa)_{2}$ is again a ribbon category. 

We end with a calculation in $\qwebsupdown (\kappa)$.  In this category we have the following recursive formulas:
\begin{align*}\label{E:thickclockwisecircleinkappacategory}
\xy 
(0,0)*{\reflectbox{
\begin{tikzpicture}[scale=.3, color=\clr]
	\draw [ thick, <-, looseness=1.7] (-1,0) to [out=90, in=90] (1,0);
	\draw [ thick, -, looseness=1.7] (-1,0) to [out=270, in=270] (1,0);
	\node at (1.5,0) {\reflectbox{\scriptsize $k$}};
	\end{tikzpicture}
}};
\endxy
&= \left(\frac{[k-1]_{q}q^{-1}(\kappa +q)}{[k]_{q}} - \frac{\kappa [k-2]_{q}}{[k]_{q}} \right)
\xy 
(0,0)*{\reflectbox{
\begin{tikzpicture}[scale=.3, color=\clr]
	\draw [ thick, <-, looseness=1.7] (-1,0) to [out=90, in=90] (1,0);
	\draw [ thick, -, looseness=1.7] (-1,0) to [out=270, in=270] (1,0);
	\node at (-2.75,0) {\reflectbox{\scriptsize $k-1$}};
	\end{tikzpicture}
}};
\endxy \\
&= \left(\frac{\left(q^{k-1}+q^{-(k-1)} \right)}{[k]_{q}\tq} \right)
\xy 
(0,0)*{\reflectbox{
\begin{tikzpicture}[scale=.3, color=\clr]
	\draw [ thick, <-, looseness=1.7] (-1,0) to [out=90, in=90] (1,0);
	\draw [ thick, -, looseness=1.7] (-1,0) to [out=270, in=270] (1,0);
	\node at (-2.75,0) {\reflectbox{\scriptsize $k-1$}};
	\end{tikzpicture}
}};
\endxy .
\end{align*}   We sketch how to obtain the above formulas.  First, use \cref{digon-removal} to explode the upward oriented part of the circle into strands labelled by $1$'s and then move the merges around the circle until they are below the splits.  This expresses the clockwise oriented circled labelled by $k$ as the clasp $\Cl_{k}$ closed up by $k$ clockwise oriented arcs labelled by $1$'s.  The first equality then follows by using the clasp recursion given in \cref{clasp-recursion} followed by rewriting $\Cl_{2}$ in terms of crossings using \ref{E:overcrossingdef}, and simplifying. The second equality follows from the first by algebraic manipulations.  This recursive formula along with \cref{E:alternateclockwisebubblevalue} yields the following formula for all $k \geq 1$:

\begin{equation*}\label{E:thickclockwisecircleinkappacategory2}
\xy
(0,0)*{\reflectbox{
\begin{tikzpicture}[scale=.3, color=\clr]
	\draw [ thick, <-, looseness=1.7] (-1,0) to [out=90, in=90] (1,0);
	\draw [ thick, -, looseness=1.7] (-1,0) to [out=270, in=270] (1,0);
	\node at (1.5,0) {\reflectbox{\scriptsize $k$}};
	\end{tikzpicture}
}};
\endxy  = \frac{\prod_{t=0}^{k-1}\left(q^{t}+q^{-t} \right)}{[k]_{q}!\tq^{k}} = \frac{\prod_{t=1}^{k}\left([t]_{q} - [t-2]_{q} \right)}{[k]_{q}!\tq^{k}} \ .
\end{equation*}
It would be interesting to investigate $\qwebsupdown (\kappa)$ further as well as the topological invariants coming from $\qwebsupdowneven (\kappa)_{2}$.

\makeatletter
\renewcommand*{\@biblabel}[1]{\hfill#1.}
\makeatother

\bibliographystyle{halpha}

\bibliography{biblio}

\end{document}

%% file: UpwardQuantumWebRelations
la%
\beq\label{associativity}
\xy
(0,0)*{
\begin{tikzpicture}[color=\clr, scale=.3]
	\draw [ color=\clr, thick, directed=.55] (0,.75) to [out=90,in=220] (1,2.5);
	\draw [ color=\clr, thick, directed=.55] (1,-1) to [out=90,in=330] (0,.75);
	\draw [ color=\clr, thick, directed=.55] (-1,-1) to [out=90,in=210] (0,.75);
	\draw [ color=\clr, thick, directed=.55] (3,-1) to [out=90,in=330] (1,2.5);
	\draw [ color=\clr, thick, directed=1] (1,2.5) to (1,3.75);
	\node at (-1,-1.5) {\scriptsize $h$};
	\node at (1,-1.5) {\scriptsize $k$};
	\node at (-1,1.75) {\scriptsize $h\! +\! k$ \ \ };
	\node at (3,-1.5) {\scriptsize $l$};
	\node at (1,4.25) {\scriptsize $h\! +\! k\! +\! l$};
\end{tikzpicture}
};
\endxy=
\xy,
(0,0)*{\reflectbox{
\begin{tikzpicture}[color=\clr, scale=.3]
	\draw [ color=\clr, thick, directed=.55] (0,.75) to [out=90,in=220] (1,2.5);
	\draw [ color=\clr, thick, directed=.55] (1,-1) to [out=90,in=330] (0,.75);
	\draw [ color=\clr, thick, directed=.55] (-1,-1) to [out=90,in=210] (0,.75);
	\draw [ color=\clr, thick, directed=.55] (3,-1) to [out=90,in=330] (1,2.5);
	\draw [ color=\clr, thick, directed=1] (1,2.5) to (1,3.75);
	\node at (-1,-1.5) {\scriptsize \reflectbox{$l$}};
	\node at (1,-1.5) {\scriptsize \reflectbox{$k$}};
	\node at (-1,1.75) {\scriptsize \reflectbox{ \ $k\! +\! l$}};
	\node at (3,-1.5) {\scriptsize \reflectbox{$h$}};
	\node at (1,4.25) {\scriptsize \reflectbox{$h\! +\! k\! +\! l$}};
\end{tikzpicture}
}};
\endxy\quad,\quad
\xy
(0,0)*{\rotatebox{180}{
\begin{tikzpicture}[color=\clr, scale=.3]
	\draw [ color=\clr, thick, rdirected=.55] (0,.75) to [out=90,in=220] (1,2.5);
	\draw [ color=\clr, thick, rdirected=.1] (1,-1) to [out=90,in=330] (0,.75);
	\draw [ color=\clr, thick, rdirected=.1] (-1,-1) to [out=90,in=210] (0,.75);
	\draw [ color=\clr, thick, rdirected=.05] (3,-1) to [out=90,in=330] (1,2.5);
	\draw [ color=\clr, thick, rdirected=.55] (1,2.5) to (1,3.75);
	\node at (-1,-1.75) {\rotatebox{180}{\scriptsize $l$}};
	\node at (1,-1.75) {\rotatebox{180}{\scriptsize $k$}};
	\node at (-1,1.75) {\rotatebox{180}{\scriptsize \ $k\! +\! l$}};
	\node at (3,-1.75) {\rotatebox{180}{\scriptsize $h$}};
	\node at (1,4.25) {\rotatebox{180}{\scriptsize $h\! +\! k\! +\! l$}};
\end{tikzpicture}
}};
\endxy=
\xy
(0,0)*{\reflectbox{\rotatebox{180}{
\begin{tikzpicture}[color=\clr, scale=.3]
	\draw [ color=\clr, thick, rdirected=.55] (0,.75) to [out=90,in=220] (1,2.5);
	\draw [ color=\clr, thick, rdirected=.1] (1,-1) to [out=90,in=330] (0,.75);
	\draw [ color=\clr, thick, rdirected=.1] (-1,-1) to [out=90,in=210] (0,.75);
	\draw [ color=\clr, thick, rdirected=.05] (3,-1) to [out=90,in=330] (1,2.5);
	\draw [ color=\clr, thick, rdirected=.55] (1,2.5) to (1,3.75);
	\node at (-1,-1.75) {\reflectbox{\rotatebox{180}{\scriptsize $h$}}};
	\node at (1,-1.75) {\reflectbox{\rotatebox{180}{\scriptsize $k$}}};
	\node at (-1,1.75) {\reflectbox{\rotatebox{180}{\scriptsize $h\! +\! k$ \ }}};
	\node at (3,-1.75) {\reflectbox{\rotatebox{180}{\scriptsize $l$}}};
	\node at (1,4.25) {\reflectbox{\rotatebox{180}{\scriptsize $h\! +\! k\! +\! l$}}};
\end{tikzpicture}
}}};
\endxy,
\eeq
\beq\label{digon-removal}
\xy
(0,0)*{
\begin{tikzpicture}[color=\clr, scale=.3]
	\draw [ color=\clr, thick, directed=1] (0,.75) to (0,2);
	\draw [ color=\clr, thick, directed=.55] (0,-2.75) to [out=30,in=330] (0,.75);
	\draw [ color=\clr, thick, directed=.55] (0,-2.75) to [out=150,in=210] (0,.75);
	\draw [ color=\clr, thick, directed=.65] (0,-4) to (0,-2.75);
	\node at (0,-4.5) {\scriptsize $k\! +\! l$};
	\node at (0,2.5) {\scriptsize $k\! +\! l$};
	\node at (-1.5,-1) {\scriptsize $k$ \ \ };
	\node at (1.5,-1) {\scriptsize  \ \ $l$};
\end{tikzpicture}
};
\endxy=\left[\nfrac{k+l}{l}\right]_{q}
\xy
(0,0)*{
\begin{tikzpicture}[color=\clr, scale=.3]
	\draw [ color=\clr, thick, directed=1] (0,-4) to (0,2);
	\node at (0,-4.5) {\scriptsize $k\! +\! l$};
	\node at (0,2.5) {\scriptsize $k\! +\! l$};
\end{tikzpicture}
};
\endxy,
\eeq
\beq\label{dot-collision}
\xy
(0,0)*{
\begin{tikzpicture}[color=\clr, scale=.3] 
	\draw [ color=\clr, thick, directed=1] (1,-2) to (1,2);
	\node at (1,2.5) {\scriptsize $k$};
	\node at (1,-2.5) {\scriptsize $k$};
	\draw (1,0.5) \wdot; 
	\draw (1,-0.5) \wdot;
\end{tikzpicture}
};
\endxy=[k]_{q^2}
\xy
(0,0)*{
\begin{tikzpicture}[color=\clr, scale=.3] 
	\draw [ color=\clr, thick, directed=1] (1,-2) to (1,2);
	\node at (1,2.5) {\scriptsize $k$};
	\node at (1,-2.5) {\scriptsize $k$};
\end{tikzpicture}
};
\endxy,
\eeq
\beq\label{dots-past-merges}
\xy
(0,0)*{
\bt[color=\clr, scale=.35]
	\draw [ color=\clr, thick, directed=1] (0, .75) to (0,2.5);
	\draw [ color=\clr, thick, directed=.6] (1,-1) to [out=90,in=330] (0,.75);
	\draw [ color=\clr, thick, directed=.6] (-1,-1) to [out=90,in=210] (0,.75);
	 \draw (0,1.5) \wdot;
	\node at (0, 3) {\scriptsize $1\! +\! k$};
	\node at (-1,-1.5) {\scriptsize $1$};
	\node at (1,-1.5) {\scriptsize $k$};
\et
};
\endxy= q^{k}
\xy
(0,0)*{
\bt[color=\clr, scale=.35]
	\draw [ color=\clr, thick, directed=1] (0, .75) to (0,2);
	\draw [ color=\clr, thick, directed=.3] (1,-1) to [out=90,in=330] (0,.75);
	\draw [ color=\clr, thick, directed=.3] (-1,-1) to [out=90,in=210] (0,.75);
	 \draw (-0.75,0) \wdot;
	\node at (0, 2.5) {\scriptsize $1\! +\! k$};
	\node at (-1,-1.5) {\scriptsize $1$};
	\node at (1,-1.5) {\scriptsize $k$};
\et
};
\endxy+ q^{-1}
\xy
(0,0)*{
\bt[color=\clr, scale=.35]
	\draw [ color=\clr, thick, directed=1] (0, .75) to (0,2);
	\draw [ color=\clr, thick, directed=.3] (1,-1) to [out=90,in=330] (0,.75);
	\draw [ color=\clr, thick, directed=.3] (-1,-1) to [out=90,in=210] (0,.75);
	 \draw (0.75,0) \wdot;
	\node at (0, 2.5) {\scriptsize $1\! +\! k$};
	\node at (-1,-1.5) {\scriptsize $1$};
	\node at (1,-1.5) {\scriptsize $k$};
\et
};
\endxy,\quad\quad
\xy
(0,0)*{
\bt[color=\clr, scale=.35]
	\draw [ color=\clr, thick, directed=.3] (0,-1) to (0,.75);
	\draw [ color=\clr, thick, directed=1] (0,.75) to [out=30,in=270] (1,2.5);
	\draw [ color=\clr, thick, directed=1] (0,.75) to [out=150,in=270] (-1,2.5);
	 \draw (0,0) \wdot;
	\node at (0, -1.5) {\scriptsize $1\! +\! k$};
	\node at (-1,3) {\scriptsize $1$};
	\node at (1,3) {\scriptsize $k$};
\et
};
\endxy= q^{-k}
\xy
(0,0)*{
\bt[color=\clr, scale=.35]
	\draw [ color=\clr, thick, directed=.65] (0,-0.5) to (0,.75);
	\draw [ color=\clr, thick, directed=1] (0,.75) to [out=30,in=270] (1,2.5);
	\draw [ color=\clr, thick, directed=1] (0,.75) to [out=150,in=270] (-1,2.5);
	 \draw (-0.75,1.5) \wdot;
	\node at (0, -1) {\scriptsize $1\! +\! k$};
	\node at (-1,3) {\scriptsize $1$};
	\node at (1,3) {\scriptsize $k$};
\et
};
\endxy+ q
\xy
(0,0)*{
\bt[color=\clr, scale=.35]
	\draw [ color=\clr, thick, directed=.65] (0,-0.5) to (0,.75);
	\draw [ color=\clr, thick, directed=1] (0,.75) to [out=30,in=270] (1,2.5);
	\draw [ color=\clr, thick, directed=1] (0,.75) to [out=150,in=270] (-1,2.5);
	 \draw (0.75,1.5) \wdot;
	\node at (0, -1) {\scriptsize $1\! +\! k$};
	\node at (-1,3) {\scriptsize $1$};
	\node at (1,3) {\scriptsize $k$};
\et
};
\endxy,
\eeq
\beq\label{dumbbell-relation}
\xy
(0,0)*{
\begin{tikzpicture}[color=\clr, scale=.3]
	\draw [ color=\clr, thick, directed=.55] (0,-1) to (0,.75);
	\draw [ color=\clr, thick, directed=1] (0,.75) to [out=30,in=270] (1,2.5);
	\draw [ color=\clr, thick, directed=1] (0,.75) to [out=150,in=270] (-1,2.5); 
	\draw [ color=\clr, thick, directed=.65] (1,-2.75) to [out=90,in=330] (0,-1);
	\draw [ color=\clr, thick, directed=.65] (-1,-2.75) to [out=90,in=210] (0,-1);
	\node at (-1,3) {\scriptsize $1$};
	\node at (1,3) {\scriptsize $1$};
	\node at (-1,-3.35) {\scriptsize $1$};
	\node at (1,-3.35) {\scriptsize $1$};
	\node at (0.75,-0.25) {\scriptsize $2$};
\end{tikzpicture}
};
\endxy
-
\xy
(0,0)*{
\begin{tikzpicture}[color=\clr, scale=.3]
	\draw [ color=\clr, thick, directed=.55] (0,-1) to (0,.75);
	\draw [ color=\clr, thick, directed=1] (0,.75) to [out=30,in=270] (1,2.5);
	\draw [ color=\clr, thick, directed=1] (0,.75) to [out=150,in=270] (-1,2.5); 
	\draw [ color=\clr, thick, directed=.75] (1,-2.75) to [out=90,in=330] (0,-1);
	\draw [ color=\clr, thick, directed=.75] (-1,-2.75) to [out=90,in=210] (0,-1);
	\node at (-1,3) {\scriptsize $1$};
	\node at (1,3) {\scriptsize $1$};
	\node at (-1,-3.35) {\scriptsize $1$};
	\node at (1,-3.35) {\scriptsize $1$};
	\node at (0.75,-0.25) {\scriptsize $2$};
	\draw (-0.9,-2.15) \wdot; 
	\draw (0.9,-2.15) \wdot; 
	\draw (0.85,1.65) \wdot; 
	\draw (-0.85,1.65) \wdot;
\end{tikzpicture}
};
\endxy= [2]_{q}
\xy
(0,0)*{
\begin{tikzpicture}[color=\clr, scale=.3] 
	\draw [ color=\clr, thick, directed=1] (1,-2.75) to (1,2.5);
	\draw [ color=\clr, thick, directed=1] (-1,-2.75) to (-1,2.5);
	\node at (-1,3) {\scriptsize $1$};
	\node at (1,3) {\scriptsize $1$};
	\node at (-1,-3.35) {\scriptsize $1$};
	\node at (1,-3.35) {\scriptsize $1$};
\end{tikzpicture}
};
\endxy 
+ \tq \xy
(0,0)*{
\begin{tikzpicture}[color=\clr, scale=.3] 
	\draw [ color=\clr, thick, directed=1] (1,-2.75) to (1,2.5);
	\draw [ color=\clr, thick, directed=1] (-1,-2.75) to (-1,2.5);
	\node at (-1,3) {\scriptsize $1$};
	\node at (1,3) {\scriptsize $1$};
	\node at (-1,-3.35) {\scriptsize $1$};
	\node at (1,-3.35) {\scriptsize $1$};
	\draw (-1,0) \wdot; 
	\draw (1,0) \wdot; 
\end{tikzpicture}
};
\endxy \ ,
\eeq
\beq\label{square-switch}
\xy
(0,0)*{
\bt[color=\clr]
	\draw[ color=\clr, thick, directed=.15, directed=1, directed=.55] (0,0) to (0,1.75);
	\node at (0,-0.2) {\scriptsize $k$};
	\node at (0,1.9) {\scriptsize $k$};
	\draw[ color=\clr, thick, directed=.15, directed=1, directed=.55] (1,0) to (1,1.75);
	\node at (1,-0.2) {\scriptsize $l$};
	\node at (1,1.9) {\scriptsize $l$};
	\draw[ color=\clr, thick, directed=.55] (0,0.5) to (1,0.5);
	\node at (0.5,0.25) {\scriptsize$1$};
	\node at (-0.4,0.875) {\scriptsize $k\!-\!1$};
	\draw[ color=\clr, thick, directed=.55] (1,1.25) to (0,1.25);
	\node at (0.5,1.5) {\scriptsize$1$};
	\node at (1.4,0.875) {\scriptsize $l\!+\!1$};
\et
};
\endxy-
\xy
(0,0)*{
\bt[color=\clr]
	\draw[ color=\clr, thick, directed=.15, directed=1, directed=.55] (0,0) to (0,1.75);
	\node at (0,-0.2) {\scriptsize $k$};
	\node at (0,1.9) {\scriptsize $k$};
	\draw[ color=\clr, thick, directed=.15, directed=1, directed=.55] (1,0) to (1,1.75);
	\node at (1,-0.15) {\scriptsize $l$};
	\node at (1,1.9) {\scriptsize $l$};
	\draw[ color=\clr, thick, directed=.55] (1,0.5) to (0,0.5);
	\node at (0.5,0.25) {\scriptsize$1$};
	\node at (-0.4,0.875) {\scriptsize $k\!+\!1$};
	\draw[ color=\clr, thick, directed=.55] (0,1.25) to (1,1.25);
	\node at (0.5,1.5) {\scriptsize$1$};
	\node at (1.4,0.875) {\scriptsize $l\!-\!1$};
\et
};
\endxy=[k-l]_{q}
\xy
(0,0)*{
\begin{tikzpicture}[color=\clr, scale=.3] 
	\draw [ color=\clr, thick, directed=1] (1,-2.75) to (1,2.5);
	\draw [ color=\clr, thick, directed=1] (-1,-2.75) to (-1,2.5);
	\node at (-1,3) {\scriptsize $k$};
	\node at (1,3) {\scriptsize $l$};
	\node at (-1,-3.35) {\scriptsize $k$};
	\node at (1,-3.35) {\scriptsize $l$};
\end{tikzpicture}
};
\endxy \ ,
\eeq
\begin{eqnarray}\label{square-switch-dots}
\xy
(0,0)*{
\bt[color=\clr]
	\draw[ color=\clr, thick, directed=.15, directed=1, directed=.55] (0,0) to (0,1.75);
	\node at (0,-0.2) {\scriptsize $k$};
	\node at (0,1.9) {\scriptsize $k$};
	\draw[ color=\clr, thick, directed=.15, directed=1, directed=.55] (1,0) to (1,1.75);
	\node at (1,-0.2) {\scriptsize $l$};
	\node at (1,1.9) {\scriptsize $l$};
	\draw[ color=\clr, thick, directed=.55] (0,0.5) to (1,0.5);
	\node at (0.5,0.25) {\scriptsize$1$};
	\node at (-0.4,0.875) {\scriptsize $k\!-\!1$};
	\draw[ color=\clr, thick, directed=.55] (1,1.25) to (0,1.25);
	\node at (0.5,1.5) {\scriptsize$1$};
	\node at (1.4,0.875) {\scriptsize $l\!+\!1$};
	 \draw (0.25,1.25) \wdot;
\et
};
\endxy-
\xy
(0,0)*{
\bt[color=\clr]
	\draw[ color=\clr, thick, directed=.15, directed=1, directed=.55] (0,0) to (0,1.75);
	\node at (0,-0.2) {\scriptsize $k$};
	\node at (0,1.9) {\scriptsize $k$};
	\draw[ color=\clr, thick, directed=.15, directed=1, directed=.55] (1,0) to (1,1.75);
	\node at (1,-0.2) {\scriptsize $l$};
	\node at (1,1.9) {\scriptsize $l$};
	\draw[ color=\clr, thick, directed=.55] (1,0.5) to (0,0.5);
	\node at (0.5,0.25) {\scriptsize$1$};
	\node at (-0.4,0.875) {\scriptsize $k\!+\!1$};
	\draw[ color=\clr, thick, directed=.55] (0,1.25) to (1,1.25);
	\node at (0.5,1.5) {\scriptsize$1$};
	\node at (1.4,0.875) {\scriptsize $l\!-\!1$};
	 \draw (0.75,0.5) \wdot; 
\et
};
\endxy &=& q^{-l}
\xy
(0,0)*{
\begin{tikzpicture}[color=\clr, scale=.3] 
	\draw [ color=\clr, thick, directed=1] (1,-2.75) to (1,2.5);
	\draw [ color=\clr, thick, directed=1] (-1,-2.75) to (-1,2.5);
	\node at (-1,3) {\scriptsize $k$};
	\node at (1,3) {\scriptsize $l$};
	\node at (-1,-3.35) {\scriptsize $k$};
	\node at (1,-3.35) {\scriptsize $l$};
	 \draw (-1,-0.125) \wdot;
\end{tikzpicture}
};
\endxy- q^{-k}
\xy
(0,-1)*{
\begin{tikzpicture}[color=\clr, scale=.3] 
	\draw [ color=\clr, thick, directed=1] (1,-2.75) to (1,2.5);
	\draw [ color=\clr, thick, directed=1] (-1,-2.75) to (-1,2.5);
	\node at (-1,3) {\scriptsize $k$};
	\node at (1,3) {\scriptsize $l$};
	\node at (-1,-3.35) {\scriptsize $k$};
	\node at (1,-3.35) {\scriptsize $l$};
	 \draw (1,-0.125) \wdot;
	\end{tikzpicture}
};
\endxy \\
\xy
(0,0)*{
\bt[color=\clr]
	\draw[ color=\clr, thick, directed=.15, directed=1, directed=.55] (0,0) to (0,1.75);
	\node at (0,-0.2) {\scriptsize $k$};
	\node at (0,1.9) {\scriptsize $k$};
	\draw[ color=\clr, thick, directed=.15, directed=1, directed=.55] (1,0) to (1,1.75);
	\node at (1,-0.2) {\scriptsize $l$};
	\node at (1,1.9) {\scriptsize $l$};
	\draw[ color=\clr, thick, directed=.55] (0,0.5) to (1,0.5);
	\node at (0.5,0.25) {\scriptsize$1$};
	\node at (-0.4,0.875) {\scriptsize $k\!-\!1$};
	\draw[ color=\clr, thick, directed=.55] (1,1.25) to (0,1.25);
	\node at (0.5,1.5) {\scriptsize$1$};
	\node at (1.4,0.875) {\scriptsize $l\!+\!1$};
	 \draw (0.25,0.5) \wdot; 
\et
};
\endxy- 
\xy
(0,0)*{
\bt[color=\clr]
	\draw[ color=\clr, thick, directed=.15, directed=1, directed=.55] (0,0) to (0,1.75);
	\node at (0,-0.2) {\scriptsize $k$};
	\node at (0,1.9) {\scriptsize $k$};
	\draw[ color=\clr, thick, directed=.15, directed=1, directed=.55] (1,0) to (1,1.75);
	\node at (1,-0.2) {\scriptsize $l$};
	\node at (1,1.9) {\scriptsize $l$};
	\draw[ color=\clr, thick, directed=.55] (1,0.5) to (0,0.5);
	\node at (0.5,0.25) {\scriptsize$1$};
	\node at (-0.4,0.875) {\scriptsize $k\!+\!1$};
	\draw[ color=\clr, thick, directed=.55] (0,1.25) to (1,1.25);
	\node at (0.5,1.5) {\scriptsize$1$};
	\node at (1.4,0.875) {\scriptsize $l\!-\!1$};
	 \draw (0.75,1.25) \wdot; 
\et
};
\endxy &=& q^{l}
\xy
(0,0)*{
\begin{tikzpicture}[color=\clr, scale=.3] 
	\draw [ color=\clr, thick, directed=1] (1,-2.75) to (1,2.5);
	\draw [ color=\clr, thick, directed=1] (-1,-2.75) to (-1,2.5);
	\node at (-1,3) {\scriptsize $k$};
	\node at (1,3) {\scriptsize $l$};
	\node at (-1,-3.35) {\scriptsize $k$};
	\node at (1,-3.35) {\scriptsize $l$};
	 \draw (-1,-0.125) \wdot;
\end{tikzpicture}
};
\endxy- q^{k}
\xy
(0,-1)*{
\begin{tikzpicture}[color=\clr, scale=.3] 
	\draw [ color=\clr, thick, directed=1] (1,-2.75) to (1,2.5);
	\draw [ color=\clr, thick, directed=1] (-1,-2.75) to (-1,2.5);
	\node at (-1,3) {\scriptsize $k$};
	\node at (1,3) {\scriptsize $l$};
	\node at (-1,-3.35) {\scriptsize $k$};
	\node at (1,-3.35) {\scriptsize $l$};
	 \draw (1,-0.125) \wdot;
	\end{tikzpicture}
};
\endxy
 \ ,\nonumber
\end{eqnarray}
\beq\label{double-rungs-1}
\xy
(0,0)*{
\bt[color=\clr]
	\draw[ color=\clr, thick, directed=.15, directed=1] (-1,0) to (-1,1.75);
	\node at (-1,-0.15) {\scriptsize $h$};
	\node at (-1,1.9) {\scriptsize $h\!+\!1$};
	\draw[ color=\clr, thick, directed=.15, directed=1, directed=.55] (0,0) to (0,1.75);
	\node at (0,-0.15) {\scriptsize $k$};
	\node at (0,1.9) {\scriptsize $k$};
	\draw[ color=\clr, thick, directed=.15, directed=1] (1,0) to (1,1.75);
	\node at (1,-0.15) {\scriptsize $l$};
	\node at (1,1.9) {\scriptsize $l\!-\!1$};
	\draw[ color=\clr, thick, directed=.55] (1,0.5) to (0,0.5);
	\node at (0.5,0.25) {\scriptsize$1$};
	\draw[ color=\clr, thick, directed=.55] (0,1.25) to (-1,1.25);
	\node at (-0.5,1.5) {\scriptsize$1$};
\et
};
\endxy- q^{-1}
\xy
(0,0)*{
\bt[color=\clr]
	\draw[ color=\clr, thick, directed=.15, directed=1] (-1,0) to (-1,1.75);
	\node at (-1,-0.15) {\scriptsize $h$};
	\node at (-1,1.9) {\scriptsize $h\!+\!1$};
	\draw[ color=\clr, thick, directed=.15, directed=1, directed=.55] (0,0) to (0,1.75);
	\node at (0,-0.15) {\scriptsize $k$};
	\node at (0,1.9) {\scriptsize $k$};
	\draw[ color=\clr, thick, directed=.15, directed=1] (1,0) to (1,1.75);
	\node at (1,-0.15) {\scriptsize $l$};
	\node at (1,1.9) {\scriptsize $l\!-\!1$};
	\draw[ color=\clr, thick, directed=.55] (0,0.5) to (-1,0.5);
	\node at (-0.5,0.25) {\scriptsize$1$};
	\draw[ color=\clr, thick, directed=.55] (1,1.25) to (0,1.25);
	\node at (0.5,1.5) {\scriptsize$1$};
\et
};
\endxy=
\xy
(0,0)*{
\bt[color=\clr]
	\draw[ color=\clr, thick, directed=.15, directed=1] (-1,0) to (-1,1.75);
	\node at (-1,-0.15) {\scriptsize $h$};
	\node at (-1,1.9) {\scriptsize $h\!+\!1$};
	\draw[ color=\clr, thick, directed=.15, directed=1, directed=.55] (0,0) to (0,1.75);
	\node at (0,-0.15) {\scriptsize $k$};
	\node at (0,1.9) {\scriptsize $k$};
	\draw[ color=\clr, thick, directed=.15, directed=1] (1,0) to (1,1.75);
	\node at (1,-0.15) {\scriptsize $l$};
	\node at (1,1.9) {\scriptsize $l\!-\!1$};
	\draw[ color=\clr, thick, directed=.55] (1,0.5) to (0,0.5);
	\node at (0.5,0.25) {\scriptsize$1$};
	\draw[ color=\clr, thick, directed=.55] (0,1.25) to (-1,1.25);
	\node at (-0.5,1.5) {\scriptsize$1$};
	\draw (0.75,0.5) \wdot;
	\draw (-0.75,1.25) \wdot; 
\et
};
\endxy+ q^{-1}
\xy
(0,0)*{
\bt[color=\clr]
	\draw[ color=\clr, thick, directed=.15, directed=1] (-1,0) to (-1,1.75);
	\node at (-1,-0.15) {\scriptsize $h$};
	\node at (-1,1.9) {\scriptsize $h\!+\!1$};
	\draw[ color=\clr, thick, directed=.15, directed=1, directed=.55] (0,0) to (0,1.75);
	\node at (0,-0.15) {\scriptsize $k$};
	\node at (0,1.9) {\scriptsize $k$};
	\draw[ color=\clr, thick, directed=.15, directed=1] (1,0) to (1,1.75);
	\node at (1,-0.15) {\scriptsize $l$};
	\node at (1,1.9) {\scriptsize $l\!-\!1$};
	\draw[ color=\clr, thick, directed=.55] (0,0.5) to (-1,0.5);
	\node at (-0.5,0.25) {\scriptsize$1$};
	\draw[ color=\clr, thick, directed=.55] (1,1.25) to (0,1.25);
	\node at (0.5,1.5) {\scriptsize$1$};
	\draw (0.75,1.25) \wdot;
	\draw (-0.75,0.5) \wdot; 
\et
};
\endxy,
\eeq
\beq\label{double-rungs-2}
\xy
(0,0)*{
\bt[color=\clr]
	\draw[ color=\clr, thick, directed=.15, directed=1] (-1,0) to (-1,1.75);
	\node at (-1,-0.15) {\scriptsize $h$};
	\node at (-1,1.9) {\scriptsize $h\!+\!1$};
	\draw[ color=\clr, thick, directed=.15, directed=1, directed=.55] (0,0) to (0,1.75);
	\node at (0,-0.15) {\scriptsize $k$};
	\node at (0,1.9) {\scriptsize $k$};
	\draw[ color=\clr, thick, directed=.15, directed=1] (1,0) to (1,1.75);
	\node at (1,-0.15) {\scriptsize $l$};
	\node at (1,1.9) {\scriptsize $l\!-\!1$};
	\draw[ color=\clr, thick, directed=.55] (1,0.5) to (0,0.5);
	\node at (0.5,0.25) {\scriptsize$1$};
	\draw[ color=\clr, thick, directed=.55] (0,1.25) to (-1,1.25);
	\node at (-0.5,1.5) {\scriptsize$1$};
	 \draw (0.75,0.5) \wdot; 
\et
};
\endxy- q^{-1}
\xy
(0,0)*{
\bt[color=\clr]
	\draw[ color=\clr, thick, directed=.15, directed=1] (-1,0) to (-1,1.75);
	\node at (-1,-0.15) {\scriptsize $h$};
	\node at (-1,1.9) {\scriptsize $h\!+\!1$};
	\draw[ color=\clr, thick, directed=.15, directed=1, directed=.55] (0,0) to (0,1.75);
	\node at (0,-0.15) {\scriptsize $k$};
	\node at (0,1.9) {\scriptsize $k$};
	\draw[ color=\clr, thick, directed=.15, directed=1] (1,0) to (1,1.75);
	\node at (1,-0.15) {\scriptsize $l$};
	\node at (1,1.9) {\scriptsize $l\!-\!1$};
	\draw[ color=\clr, thick, directed=.55] (0,0.5) to (-1,0.5);
	\node at (-0.5,0.25) {\scriptsize$1$};
	\draw[ color=\clr, thick, directed=.55] (1,1.25) to (0,1.25);
	\node at (0.5,1.5) {\scriptsize$1$};
	 \draw (0.75,1.25) \wdot; 
\et
};
\endxy=
\xy
(0,0)*{
\bt[color=\clr]
	\draw[ color=\clr, thick, directed=.15, directed=1] (-1,0) to (-1,1.75);
	\node at (-1,-0.15) {\scriptsize $h$};
	\node at (-1,1.9) {\scriptsize $h\!+\!1$};
	\draw[ color=\clr, thick, directed=.15, directed=1, directed=.55] (0,0) to (0,1.75);
	\node at (0,-0.15) {\scriptsize $k$};
	\node at (0,1.9) {\scriptsize $k$};
	\draw[ color=\clr, thick, directed=.15, directed=1] (1,0) to (1,1.75);
	\node at (1,-0.15) {\scriptsize $l$};
	\node at (1,1.9) {\scriptsize $l\!-\!1$};
	\draw[ color=\clr, thick, directed=.55] (1,0.5) to (0,0.5);
	\node at (0.5,0.25) {\scriptsize$1$};
	\draw[ color=\clr, thick, directed=.55] (0,1.25) to (-1,1.25);
	\node at (-0.5,1.5) {\scriptsize$1$};
	 \draw (-0.75,1.25) \wdot; 
\et
};
\endxy- q^{-1}
\xy
(0,0)*{
\bt[color=\clr]
	\draw[ color=\clr, thick, directed=.15, directed=1] (-1,0) to (-1,1.75);
	\node at (-1,-0.15) {\scriptsize $h$};
	\node at (-1,1.9) {\scriptsize $h\!+\!1$};
	\draw[ color=\clr, thick, directed=.15, directed=1, directed=.55] (0,0) to (0,1.75);
	\node at (0,-0.15) {\scriptsize $k$};
	\node at (0,1.9) {\scriptsize $k$};
	\draw[ color=\clr, thick, directed=.15, directed=1] (1,0) to (1,1.75);
	\node at (1,-0.15) {\scriptsize $l$};
	\node at (1,1.9) {\scriptsize $l\!-\!1$};
	\draw[ color=\clr, thick, directed=.55] (0,0.5) to (-1,0.5);
	\node at (-0.5,0.25) {\scriptsize$1$};
	\draw[ color=\clr, thick, directed=.55] (1,1.25) to (0,1.25);
	\node at (0.5,1.5) {\scriptsize$1$};
	 \draw (-0.75,0.5) \wdot; 
\et
};
\endxy
\eeq